\documentclass{article}
\usepackage{geometry}
\usepackage{amsmath, amsfonts, bm, mathrsfs}
\usepackage{amsthm,graphicx}
\usepackage[dvipsnames]{xcolor}

\usepackage[defaultcolor=red]{changes}
\usepackage[normalem]{ulem}
\newcommand{\mathsout}[1]{%
	\ifmmode
	\text{\sout{\ensuremath{#1}}}
	\else
	\sout{#1}
	\fi
}
\setdeletedmarkup{\mathsout{#1}}
\setauthormarkup{}

\usepackage[unicode=true,
bookmarks=false,
breaklinks=false,
pdfborder={0 0 1},colorlinks=false
]
{hyperref}
\hypersetup{
	colorlinks,citecolor=blue,filecolor=blue,linkcolor=blue,urlcolor=blue}

\allowdisplaybreaks

\usepackage{float}
\usepackage{multirow}
\usepackage{footnote}
\usepackage{dsfont}
\usepackage{booktabs}

\usepackage[linesnumbered,ruled,vlined]{algorithm2e}
\usepackage{algorithmic}

\definecolor{cc}{RGB}{0,127,0}

\usepackage{dsfont}

\title{The high-dimensional asymptotics of\\ 
	first order methods with random data}

\author{Michael Celentano\thanks{Department of Statistics, University of California Berkeley} \and Chen Cheng\thanks{Department of Statistics, Stanford University} \and 
	Andrea Montanari\thanks{Department of Statistics and  Department of Mathematics, 
		Stanford University}}

\newcommand{\Cov}{\operatorname{Cov}}
\newcommand{\Tr}{\operatorname{Tr}}

\newcommand{\de}{\mathrm{d}}

\newcommand{\E}{\mathbb{E}}

\newcommand{\opnorm}[1]{{\left\vert\kern-0.25ex\left\vert\kern-0.25ex\left\vert #1 
		\right\vert\kern-0.25ex\right\vert\kern-0.25ex\right\vert}}

\newcommand{\ba}{\boldsymbol{a}}
\newcommand{\bb}{\boldsymbol{b}}
\newcommand{\flr}[1]{{\lfloor {#1} \rfloor}}
\newcommand{\cil}[1]{{\lceil {#1} \rceil}}
\newcommand{\integersp}{\integers_{\geq 0}}
\newcommand{\wb}[1]{\overline{#1}}

\newcommand{\Ep}{\mathbb{E}}

\newcommand{\real}{\mathbb{R}} 
\newcommand{\prn}[1]{\left({#1}\right)} 
\newcommand{\brk}[1]{\left[{#1}\right]} 
\newcommand{\brc}[1]{\left\{{#1}\right\}} 
\newcommand{\norm}[1]{\left\|{#1}\right\|} 
\newcommand{\normtwo}[1]{\norm{#1}_2}

\newcommand{\normlbd}[2]{\norm{#1}_{\lambda, #2}}
\newcommand{\intnormlbd}[2]{\opnorm{#1}_{\lambda, #2}}
\newcommand{\est}[1]{\widehat{#1}}

\newcommand{\wdst}[2]{\mathrm{W}_{2}\prn{#1,#2}}
\newcommand{\dW}[2]{d_{{\rm\scriptsize W}}\prn{#1,#2}}
\newcommand{\lbddst}[2]{\mathsf{dist}_{\lambda,T}\prn{#1,#2}}
\newcommand{\ddt}{\frac{\de}{\de t}}

\newcommand{\vt}{\theta}
\newcommand{\vr}{r}
\newcommand{\vLambda}{\Lambda}
\newcommand{\vu}{u}
\newcommand{\vw}{w}
\newcommand{\vl}{\ell}
\newcommand{\veps}{z}

\newcommand{\trsfrm}{\mathcal{T}}
\newcommand{\trsfrmA}{\mathcal{T}_{\mathcal{S} \to \wb{\mathcal{S}}}}
\newcommand{\trsfrmB}{\mathcal{T}_{\wb{\mathcal{S}} \to \mathcal{S}}}

\newcommand{\cstloss}{M_\ell}
\newcommand{\cstspaceS}{M_{\mathcal{S}}}
\newcommand{\cstdualspaceS}{M_{\wb{\mathcal{S}}}}
\newcommand{\cstlbd}{M_\Lambda}
\newcommand{\cstthetaz}{M_{\theta^0, z}}

\newcommand{\Rt}{R_\theta}
\newcommand{\Rl}{R_{\ell}}
\newcommand{\Ct}{C_\theta}
\newcommand{\Cl}{C_{\ell}}
\newcommand{\vRt}{R_\theta}
\newcommand{\vRl}{R_{\ell}}
\newcommand{\vCt}{C_\theta}
\newcommand{\vCl}{C_{\ell}}
\newcommand{\bRt}{\wb{R}_\theta}
\newcommand{\bRl}{\wb{R}_\ell}
\newcommand{\bCt}{\wb{C}_\theta}
\newcommand{\bCl}{\wb{C}_\ell}
\newcommand{\bGamma}{\wb{\Gamma}}
\newcommand{\pRt}{\Phi_{\Rt}}
\newcommand{\pRl}{\Phi_{\Rl}}
\newcommand{\pCt}{\Phi_{\Ct}}
\newcommand{\pCl}{\Phi_{\Cl}}

\newcommand{\vSl}{\vRl}
\newcommand{\vGamma}{\Gamma}
\newcommand{\bvCl}{\bCl}
\newcommand{\bvCt}{\bCt}
\newcommand{\bvRl}{\bRl}
\newcommand{\bvRt}{\bRt}
\newcommand{\bvSl}{\bRl}
\newcommand{\bvGamma}{\bGamma}

\def\Loss{{\sf L}}
\def\oX{{\overline{X}}}
\def\naturals{{\mathbb N}}
\def\reals{{\mathbb R}}

\def\sBL{\mbox{\scriptsize\rm BL}}
\def\sW{\mbox{\scriptsize\rm W}}
\def\sLip{\mbox{\scriptsize\rm Lip}}
\DeclareMathOperator*{\plim}{p-lim}
\DeclareMathOperator*{\plimsup}{p-lim\,sup}

\DeclareMathOperator*{\argmin}{arg\,min}

\newtheoremstyle{myexample} 
    {\topsep}                    
    {\topsep}                    
    {\rm }                   
    {}                           
    {\bf }                   
    {.}                          
    {.5em}                       
    {}  

\newtheoremstyle{myremark} 
    {\topsep}                    
    {\topsep}                    
    {\rm}                        
    {}                           
    {\bf}                        
    {.}                          
    {.5em}                       
    {}  

\newtheorem{claim}{Claim}[section]
\newtheorem{lemma}[claim]{Lemma}

\newtheorem{assumption}{Assumption}

\newtheorem{theorem}{Theorem}
\newtheorem{proposition}[claim]{Proposition}
\newtheorem{corollary}[claim]{Corollary}
\newtheorem{definition}[claim]{Definition}

\theoremstyle{myremark}
\newtheorem{remark}{Remark}[section]

\theoremstyle{myremark}

\theoremstyle{myexample}
\newtheorem{example}[remark]{Example}

\definecolor{darkblue}{rgb}{0, 0, 0.5}
\newcommand{\rev}[1]{\textcolor{black}{#1}}

\def\<{\langle}
\def\>{\rangle}

\def\argmin{{\rm argmin}}

\def\eps{{\varepsilon}}

\def\id{{\boldsymbol I}}

\def\sT{{\sf T}}

\def\cH{{\cal H}}

\def\hR{\hat{R}}

\def\complex{\mathbb{C}}

\def\Prox{{\rm Prox}}

\def\normal{{\sf N}}

\def\cS{{\cal S}}

\def\Cov{{\rm Cov}}

\def\cL{{\cal L}}

\def\Tr{{\rm {Tr}}}

\def\de{{\rm d}}

\def\htheta{\hat{\theta}}

\def\bfone{{\bf 1}}
\def\bfzero{{\bf 0}}

\def\bb{{\boldsymbol b}}

\def\ed{\stackrel{{\rm d}}{=}}

\def\hS{\widehat{S}}

\def\naturals{\mathbb{N}}
\def\integers{\mathbb{Z}}
\def\reals{\mathbb{R}}

\def\realsp{\mathbb{R}_{\ge 0}}

\def\bB{{\boldsymbol B}}

\def\bX{{\boldsymbol X}}

\def\be{{\boldsymbol e}}

\def\bu{{\boldsymbol u}}

\def\bx{{\boldsymbol x}}

\def\by{{\boldsymbol y}}
\def\bz{{\boldsymbol z}}

\def\btheta{{\boldsymbol \theta}}

\def\bv{{\boldsymbol v}}

\def\bF{{\boldsymbol F}}

\def\hbtheta{\hat{\boldsymbol \theta}}

\def\hr{\hat{r}}
\def\bell{{\boldsymbol \ell}}

\def\br{{\boldsymbol r}}
\def\bsigma{{\boldsymbol \sigma}}

\def\prob{{\mathbb P}}
\def\E{{\mathbb E}}

\def\Treg[#1]{T^{{\rm reg},#1}}
\def\GW[#1]{{\rm GW}(#1)}
\def\MGW[#1]{{\rm MGW}(#1)}

\def\tpsi{\widetilde{\psi}}

\def\P{{\rm P}}

\def\Loss{{\sf Loss}}
\def\Ls{{\sf L}}

\def\btheta{{\boldsymbol \theta}}

\def\diag{{\rm diag}}

\def\op{{\rm op}}

\date{\today}

\begin{document}
	\maketitle
	
	\begin{abstract}
		We study a class of deterministic flows in $\reals^{d\times k}$, parametrized 
		by a random matrix $\bX\in \reals^{n\times d}$ with i.i.d. centered subgaussian entries.
		We characterize the asymptotic behavior of these flows over bounded time horizons, in
		the high-dimensional limit in which $n,d\to\infty$ with $k$ fixed and 
		converging aspect ratios $n/d\to\delta$. The asymptotic characterization we prove is in 
		terms of a system of nonlinear stochastic processes in $k$ dimensions, whose parameters
		are determined by a fixed point condition. This type of characterization is known in physics as
		dynamical mean field theory. Rigorous results of this type
		have been obtained in the past for a few spin glass models.
		Our proof is based on time discretization and a reduction to certain iterative 
		schemes known as approximate message passing (AMP) algorithms, as opposed to 
		earlier work that was based on large deviations theory and stochastic processes
		theory. The new approach provides a unified view of a general 
		class of algorithms  and implies that 
		the high-dimensional behavior of the flow is universal with respect to
		the distribution of the entries of $\bX$.  
		
		As specific applications, we obtain high-dimensional characterizations of gradient flow 
		in some classical models from statistics and machine learning, under a random design assumption.
	\end{abstract}
	
	\tableofcontents
	
	\section{Introduction}
	\label{sec:Introduction}
	
	\subsection{Motivation}

	Understanding the behavior of gradient descent dynamics in non-convex random
	energy landscapes is a central problem in a number of disciplines, ranging from statistical physics
	to applied mathematics, machine learning and statistics. 
	\rev{Consider for instance
	the problem of fitting $n$ data points $\bx_i\in\reals^d$, $y_i\in \reals$  using a superposition of
	$k$ functions $\varphi(\<\btheta_1,\bx\>)$, \dots $\varphi(\<\btheta_k,\bx\>)$:
	\begin{align}
		\mbox{Solve:} &\;\;\;\;\;\sum_{a=1}^kc_a\varphi(\<\bx_i,\btheta_a\>)= y_i\, , 
		 \;\;\; \forall i\le n\, ,\label{eq:Problem1}\\
		 \mbox{Subject to:} &\;\;\;\;\;\btheta_1,\dots,\btheta_k\in\reals^d\, .\nonumber
	\end{align}
	Here $\varphi:\reals\to\reals$ is a known function, $c_a$, 
	$\bx_i\in\reals^d$, $y_i\in\reals$ are known data of the problem, and $\<\bu,\bv\>$
	denotes the standard scalar product of $\bu,\bv\in\reals^d$.
	We are given the data points $y_i$ and $\bx_i$, $i\le n$, the coefficients 
	$c_a$, $a\le k$, and the function $\varphi$,
	and would like to solve  Problem \eqref{eq:Problem1}  for 
	$\btheta_1,\dots,\btheta_k\in\reals^d$.}
	
	\rev{The classical problem of 
	representing a sub-sampled signal in $d$ dimension as a sum of a 
	small number of Fourier waves (with unknown wave-vectors $\btheta_a$) 
	reduces to problem \eqref{eq:Problem1} with $\varphi(t)=e^{it}$ \cite{plonka2018numerical}. 
	Fitting a two-layer
	neural network with $k$ hidden neurons to $n$ data points also takes the same form with $\varphi$
	the activation function (e.g. $\varphi(t) = \tanh(t)$)
	 \cite{pinkus1999approximation}. Other special cases of the above include 
	 linear regression \cite{skouras1994estimation} and
	 phase retrieval \cite{fienup1982phase,chen2017solving}.
	}
	
	\rev{An interesting way to explore the space of solutions and near-solutions of
	Problem \ref{eq:Problem1}  is to consider a gradient flow that converges to solutions.
	In order to define this gradient flow, it is convenient to 
	introduce the notation $\btheta = [\btheta_1,\dots ,\btheta_k]\in\reals^{d\times k}$ 
	and define the function  $\Loss:\reals^k\times \reals \to \reals$
	\begin{align}
	\Loss(u;y) = \frac{1}{2}\Big(\sum_{a=1}^kc_a\varphi(u_a)-y\Big)^2\, ,
	\end{align}
	(other smooth functions could replace the square $(\hat y-y)^2/2$.)}
   \rev{ The gradient flow of interest then reads 
	\begin{align}
		\frac{\de\btheta^t}{\de t} &= -\nabla\cL_n(\btheta^t)\, ,\label{eq:FirstFlow}\\
		\cL_n(\btheta) & := \frac{d}{n}\sum_{i=1}^n\Loss(\btheta^{\sT}\bx_i,;y_i)\, .\label{eq:FirstLoss}
	\end{align}
	We will use this as a running example in what follows.}

	\rev{The main objective of this paper is to establish a characterization
	of a class of flows including the one in Eq.~\eqref{eq:FirstFlow} as a special case, which holds
	for certain distributions of random matrices $\bX$, under the 
	high-dimensional asymptotics $n,d\to\infty$, when $n/d\to \delta\in (0,\infty)$.
	This characterization (known in physics as `dynamical mean field theory' (DMFT))
	is amenable to both numerical and mathematical analysis and indeed has been used in 
	a large number of works in statistics and machine learning,
	 both before and after our results (see Section \ref{sec:OtherConsequences} for a few pointers). 
	 This work provides the first rigorous foundation for a number of  these applications.}
	
	\rev{Gradient} flow can be discretized to 
	yield a gradient descent algorithm 
	(typically initialized at an uninformative position, e.g. $\btheta^0=0$ or 
	$\btheta^0\sim\normal(0,c_0\id_d)$):
	\begin{align}
		\btheta^{k+1}= \btheta^k-\eta\nabla\cL_n(\btheta^k)\, ,\label{eq:FirstDiscr}
	\end{align}
	where $\eta$ is a stepsize parameter. For small enough $\eta$, this algorithm should 
	closely track  gradient flow, and (alongside its many variants)
	 is broadly used in practice because of its scalability. 
	
	Similar  flows are studied in statistical physics.
	For instance, in the Hopfield model of associative memories,
	the memory retrieval dynamics is closely related to the following flow
	\begin{align}
		\frac{\de\bsigma^t}{\de t} & = -\nabla\cH(\bsigma^t)\, ,\\
		\cH(\bsigma) & := \frac{1}{2}\sum_{i=1}^n\<\bx_i,\bsigma\>^2+ V(\bsigma)\, ,
	\end{align}
	where $\bx_i$ are the `patterns' memorized by the network, $V(\bsigma) := \sum_{i=1}^dV(\sigma_i)$,
	and $\cH(\bsigma)$ is the Hamiltonian or energy function.
	This type of dynamics was studied in the context of
	the Sherrington-Kirkpatrick model \cite{sompolinsky1982relaxational},
	the spherical $p$-spin glass model \cite{crisanti1993sphericalp,cugliandolo1993analytical},
	and the Hopfield model \cite{rieger1988glauber}.
	
	All of these dynamics can be  modified by introducing a noise term, thus yielding 
	a Langevin dynamics. For instance, Eq.~\eqref{eq:FirstFlow} can be modified to
	$\de\btheta^t = -\nabla\cL(\btheta^t)\de t + \alpha\de\bB^t$, with $(\bB_t)_{t\ge 0}$
	a standard $d$-dimensional Brownian motion. We believe that this generalization can be treated using
	our approach, but our formal results are established for the `zero temperature' case $\alpha=0$.
	
	\subsection{Dynamical mean field theory}

	We will study the behavior of the flow \eqref{eq:FirstFlow}
	(and indeed a significant generalization of this flow), when the matrix
	$\bX$ is random, with i.i.d. centered subgaussian entries\footnote{Our main theorem can 
		be proved in a slightly stronger form for the case of Gaussian entries, 
		because of available theorems in the literature.}.
	As already mentioned, we will focus on the proportional asymptotics $n,d\to\infty$ 
	with $n/d\to \delta\in (0,\infty)$. 
	Indeed, the regime $n\asymp d$ is the richest (and most challenging) from a mathematical viewpoint.
	If  $n/d\to 0$ then the optimization problem \eqref{eq:FirstLoss} is strongly overparametrized,
	and gradient flow quickly converges to a global minimizer for most reasonable loss
	functions $\Loss(\bz;y)$ \cite{bartlett2021deep}. On the other hand, if $n/d\to\infty$,
	then gradient flow  \eqref{eq:FirstFlow} converges to gradient flow with respect to the population 
	error $\cL(\btheta) = \E[\cL_n(\btheta)]$ which is much simpler.
	
	The mainstream approach to the analysis of the flow
	\eqref{eq:FirstFOM} in statistics and applied
	mathematics is to study the landscape of the cost function $\cL_n(\btheta)$,
	\rev{and compare it with its expectation $\cL(\btheta) = \E[\cL_n(\btheta)]$.}
	One then relates the properties of gradient flow to such landscape properties (e.g., the absence
	of `bad' local minima) via a deterministic argument. \rev{This approach
	has two weaknesses: $(i)$~It is accurate only for $n\gg d$, because otherwise the random landscape 
	$\cL_n(\btheta)$ will not converge uniformly around its expectation;
	$(ii)$~It is inherently a `worst case' analysis, and does not capture situations in which
	 bad local minima exist but are avoided by the dynamics.} 
	 We refer also to Section \ref{sec:Interpretation} for a comparison of our DMFT characterization to 
	 the $n/d\to\infty$ limit.

	In contrast, within statistical physics, there exists a well established
	approach to the analysis of gradient or Langevin flows for 
	spin glass Hamiltonians. One takes the limit $n,d\to\infty$ at $n/d=\delta$ 
	and $t\le T$ fixed and uses a non-rigorous argument to derive an asymptotic characterization
	also known as `dynamical mean field theory' (DMFT).
\rev{Consider ---to be definite--- the case $k=1$ of the flow \eqref{eq:FirstFlow}:
this describes the evolution of a $d$-dimensional vector $\btheta^t$ with
coordinates $\theta^t_1,\dots\theta^t_n$, each tracing a trajectory $\theta^{[0,T]}_i$
over the time horizon $t\in [0,T]$. DMFT predicts the distribution of this trajectory for a typical
coordinate $i$. Explicitly, for any test function $\psi$ (which takes as input a one-dimensional 
trajectory, i.e. a function in $C([0,T])$)}
\begin{align}
\plim_{n,d\to\infty}\frac{1}{d}\sum_{i=1}^d\psi(\theta^{[0,T]}_i) = \E\big\{\psi(\theta^{[0,T]})\big\}\,.
\end{align}
\rev{Here $\theta^{[0,T]}$ on the right-hand side is the asymptotic process. The crucial point is that this 
process is one-dimensional, and is completely defined by a certain stochastic differential 
equation with memory. As such it lends itself both to sharp analysis and \emph{dimension-independent}
numerical approximation.}

\rev{As a concrete example ---by considering the test function 
$\psi(\theta^{[0,T]}_i) =(\theta^t_i-\theta^s_i)^2$---
DMFT allows to compute the high-dimensional asymptotics of the distance between the state at time $t$ 
and time $s$}
	\begin{align}
		D_{\theta}(t,s) := \lim_{n,d\to\infty}\frac{1}{d}\|\btheta^{t}-\btheta^s\|^2 \, .
	\end{align}
	\rev{It is important to emphasize that the limit $n,d\to\infty$ is taken at fixed $t$, $s$.}

	At first sight, studying gradient flow or similar flows over a time horizon
	$T=O(1)$ as $n,d\to \infty$ might seem to have little use. Instead it turns out
	that, in many problems of interest, a non-trivial evolution takes place on this time scale,
	and a near optimum is achieved.  Examples from the literature will be provided
	in the next sections. (An important role is of course played by the scaling of the 
	cost function in Eq.~\eqref{eq:FirstLoss}). 
	We will also prove that gradient descent, as is defined in Eq.~\eqref{eq:FirstDiscr}, closely tracks gradient
	flow (and the same happens for more general flows) when $\eta$ is small, but is still of the size 
	$O(1)$ as $n,d\to\infty$. This means the theory developed here concerns algorithms whose complexity is
	of the order of $T/\eta=O(1)$ matrix--vector multiplications.
	
	\rev{There are interesting examples in which gradient-based methods only achieve non-trivial
	learning after $t\ge C\log d$. As demonstrated recently in \cite{troiani2025fundamental,montanari2026phase}, 
	even in those cases the DMFT theory developed here is a very useful tool (both papers 
	build upon the results proven here.) Further, we expect that the proof technique used here can be generalized 
	beyond $O(1)$ timescales.}
	
	DMFT asymptotics have been proved in the past for Langevin 
	dynamics on several spin glass models 
	\cite{arous1995large,arous1997symmetric,arous2001aging,arous2006cugliandolo}. 
	These proofs were based either on a large deviations argument or on 
	stochastic processes and weak convergence theory. 
	Over the last few years, physicists have applied
	the DMFT approach to analyze gradient flow algorithms in several problems from 
	high-dimensional statistics and machine learning (see next section for some pointers). 
	While a DMFT characterization was not proven for these applications, 
	several insights were extracted from the analysis of DMFT systems. 
	The present paper aims at filling this gap.
	
	\subsection{Technical contributions}

	We report contributions in several directions:
	\begin{description}
		\item[Asymptotic characterization.] We prove an asymptotic DMFT characterization of
		a class of flows including \eqref{eq:FirstLoss} as a special case.
		Our setting includes cases in  which  $\btheta^t\in\reals^{d\times k}$ is a matrix with a 
		fixed number $k$ of columns (with $k$ independent of $d,n$). Further, the flow
		can depend on time, and the function $\Loss'$ in Eq.~\eqref{eq:FirstFOM}
		is replaced by a general function $\ell:\reals^{k+1}\to \reals^k$.
		
		These generalizations allow us to cover a range of applications.
		In particular, our results characterize the asymptotic distribution:
		\begin{align}
			\frac{1}{d}\sum_{i=1}^d \delta_{\rev{\theta_i^{[0,T]}}} \Rightarrow {\rm P}_T
		\end{align}
		where $\rev{\theta_i^{[0,T]}}$ is the trajectory of row $i$ of $\btheta^t$ (seen as a function from $[0,T]$
		to $\reals^k$), and $\Rightarrow$ denotes weak convergence in the space of probability distributions
		over $C([0,T],\reals^k)$.
		In particular, we prove existence and  uniqueness of the solution of the  
		DMFT  equations. Our DMFT characterization generalizes earlier results obtained by 
		non-rigorous physics techniques.
		\item[Proof technique.] We introduce a new proof technique that is based on a 
		3 step  procedure: $(1)$ Discretize time; $(2)$ Show that the discrete-time flow can
		be obtained  by applying a simple change of variables to the iterates of an approximate message passing (AMP)
		algorithm;
		$(3)$ Apply an existing asymptotic characterization of AMP algorithms, known as `state evolution'.
		
		Given the special structure of AMP algorithms, 
		this approach has advantages over alternative ones.
		We note that the reduction of discrete time flows to AMP was already elucidated
		in \cite{celentano2020estimation}.
		The main technical challenge addressed in this paper is to show that 
		that the time step can be taken to zero, to yield the DMFT.  
		We show this  by establishing  a contraction property
		in a suitable function space. This contraction property has other useful consequences:
		among them, it implies existence and uniqueness of solutions of the asymptotic dynamics.
		\item[Beyond gradient flow.] In applications, gradient flow is only one among 
		other algorithms that we might be interested in. These algorithms need not be
		gradient flows with respect to a cost function. For instance, it is known that, among 
		first order methods for statistical estimation (algorithms that proceed by successive
		multiplication by $\bX$ or $\bX^{\sT}$), Bayes AMP achieves optimal statistical accuracy,
		under suitable assumptions  \cite{celentano2020estimation}.
		
		The proof by discretization and reduction to AMP makes transparent the relation between 
		various algorithms, and in particular the fact that each of these algorithms can be 
		viewed as AMP plus some post-processing. 
		\item[Non-vanishing step size.] As a byproduct of our analysis, the 
		asymptotic characterization does not apply only to the continuous time flow,
		but also to its discretization (e.g. gradient descent \eqref{eq:FirstDiscr})
		for any stepsize $\eta=\eta_n\to 0$ as $n,d\to\infty$.
		This follows from the fact that our proof technique is based on time discretization.
		
		For the case of non-vanishing stepsize, our analysis still gives an asymptotic characterization 
		(with discrete time).
		\item[Universality.] As our approach leverages available results on the analysis of AMP algorithms, 
		we inherit the generality of those results. 
		In particular, we can establish universality of the $n,d\to\infty$ limit
		with respect to the distribution of the entries $X_{ij}$ of $\bX$. 
		\item[Stationary points.] In the case of gradient flows, we prove that a subset 
		of fixed points of the DMFT dynamics are in correspondence with 
		stationary points of an infinite-dimensional variational principle.
		This correspondence holds both for convex and non-convex optimization problems
		and --to the best of our knowledge-- was not mentioned even at heuristic level (except in special cases).
		
		In physics language, these stationary points correspond to `replica symmetric solutions'.
		A subset of them should describe the long-time asymptotics of the flow.
		However, we leave for future work the study of how and when these stationary points
		do actually control the long-time asymptotics.
		Statistical physics predicts that other types of asymptotic behaviors
		are possible as well in non-convex problems  \cite{cugliandolo1993analytical}.
	\end{description}

	The rest of the paper is organized as follows. We briefly survey related work in 
	Section \ref{sec:Related}. We then state our general results in Section \ref{sec:Main}.
	We specialize the general result to a few cases of interest in Section \ref{sec:Examples}.
	We finally present our proofs in Sections \ref{sec:proof:UniqueExist} and 
	\ref{sec:proof:StateEvolution}, 
	with several technical lemmas
	deferred to the appendices.

	\section{Related work}
	\label{sec:Related}
	
	As mentioned in the introduction, DMFT was used by physicists for a long time to characterize
	the high-dimensional behavior of Langevin dynamics in mean field spin glasses
	\cite{sompolinsky1981dynamic,sompolinsky1982relaxational}.
	The asymptotic characterization is given, as in our paper, by a correlation and
	response function. In some cases, these functions are determined by a set of 
	integral-differential equations \cite{crisanti1993sphericalp,cugliandolo1993analytical}.
	More often, they solve a fixed point condition that is given in terms of \rev{an} one-dimensional 
	stochastic process with correlated noise and 
	memory~\cite{sompolinsky1982relaxational,cugliandolo2008out,altieri2020dynamical}. 
	
	Over the last few years physicists applied the same techniques to several problems
	in high-dimensional statistics and machine-learning. They studied the behavior of
	gradient flow learning and extracted useful insights from the DMFT characterization.
	An incomplete list of examples includes tensor principal component analysis
	\cite{mannelli2020marvels}, max margin linear classification \cite{agoritsas2018out,mignacco2020dynamical},
	Gaussian mixture models \cite{mignacco2020dynamical}.
	
	Some of these papers compare Langevin learning to Bayes optimal AMP, by solving
	numerically the corresponding high-dimensional characterizations. 
	They observe that Bayes AMP achieves superior accuracy and provide physics-based explanations
	for this phenomenon. 
	Our analysis (alongside the results of \cite{celentano2020estimation}) provides a  simple rigorous explanation
	of this observation. Langevin (as gradient flow and indeed any first order method) 
	is equivalent to a specific AMP plus post-processing. Bayes AMP is the optimal AMP
	algorithm in Bayesian estimation problems.
	
	DMFT characterizations for the Langevin dynamics of the Sherrington-Kirkpatrick (SK) model
	were first proved by Ben Arous and Guionnet  \cite{arous1995large,arous1997symmetric,guionnet1997averaged}
	(who considered continuous spins and Langevin dynamics) and by 
	Grunwald \cite{grunwald1996sanov} who instead considered Ising spins and Glauber dynamics.
	Spherical spin glasses (whereby the vector $\btheta$ lies on a sphere)
	were studied in \cite{arous2001aging} in the case of quadratic cost functions. 
	With respect to all other cases discussed here, the example of spherical spin glasses
	with quadratic activations is significantly simpler. In this case, the solutions to the flow
	can be written explicitly. 
	The case of spherical spin glasses with general polynomial interactions (the so called $p$-spin model) 
	was  studied in \cite{arous2006cugliandolo}.  This paper proved the DMFT equations using 
	a concentration technique and Girsanov formula, and leveraging in a crucial way the fact that 
	the energy function is a Gaussian process. 
	
	We notice in passing that all of the above approaches use in an important way the
	fact that the process studied is a non-degenerate diffusion, e.g. Langevin dynamics at 
	non-zero temperature. In contrast, we focus on the degenerate case of deterministic flow.
	We believe it is possible to apply our proof technique to non-zero temperature,
	by constructing the Brownian noise as a deterministic function of the noise vector $\bz$.
	We also note that the models we treat are analogous to the  SK model in that 
	the asymptotic characterization is given in terms of a stochastic process.
	
	Recently the mathematical study of DMFT asymptotics has attracted renewed interest
	to address the question of universality with respect to the distribution 
	of the underlying randomness (the matrix $\bX$ in our case). Dembo and Gheissari
	\cite{dembo2021diffusions} prove universality for a class of diffusions 
	parametrized by a random matrix. Finally, Dembo,  Lubetzky and  Zeitouni
	\cite{dembo2019universality} prove universality for a version of the SK model 
	Langevin dynamics in which the symmetric interaction matrix is replaced by an asymmetric matrix
	with independent entries. This additional independence allows to use a direct approach
	based on Girsanov formula.
	
	Our proof technique is based on a reduction to AMP, and hence allows us to
	leverage the wealth of results proved in that context.
	While most of these results
	\cite{bayati2011dynamics,javanmard2013state,berthier2020state}  
	were proven for Gaussian randomness, using a technique first introduced in 
	\cite{bolthausen2014iterative}, universality results were proven in 
	\cite{bayati2015universality,chen2021universality}. In particular, we 
	exploit the result of Chen and Lam \cite{chen2021universality} and deduce 
	universality for a large class of flows.
	
	Finally, we believe the same technique should be applicable to prove universality in other
	cases as well, e.g. for Langevin dynamics in the SK model which is not covered by our main theorem.
	
\subsection{Some applications of the DMFT equations}
\label{sec:OtherConsequences}

One important feature of the DMFT equations \eqref{eq:planted-se},
\eqref{eq:planted-derivs} is that they can be solved numerically.
As such these (or the analogous equations for related models) have been used
to derive predictions in a number of high-dimensional statistics problems.
Prior to our work, \cite{sarao2019afraid,sarao2020marvels,sarao2020complex} used (heuristically derived)  DMFT equations 
to derive weak recovery thresholds for gradient flow in tensor principal component analysis
and phase retrieval.
\cite{mignacco2022effective} used (heuristically derived) DMFT equations to study
learning in Gaussian mixture models,  in particular showing how overfitting depends on
initialization and early  stopping. 

After our work appeared as a preprint, our results were used in \cite{dandi2024benefits,arnaboldi2024repetita} 
to study the effect of reusing data in learning multi-index models using two-layer neural nets.
These authors showed that SGD and GD can learn from substantially smaller sample sizes than 
online SGD.

The authors of \cite{bellec2024uncertainty,tan2024estimating} built on our approach to 
quantify uncertainty along the trajectory of gradient descent in high-dimensional statistics
estimators.

Generalizing our work, \cite{fan2025dynamical_a,fan2025dynamical_b} studied Langevin sampling
for Bayesian linear regression. This was used as a building block in an empirical Bayes approach. The 
DMFT equations were crucial to prove high-dimensional consistency.

Finally, \cite{montanari2025dynamical} used DMFT to study the dynamics of learning in two-layer networks
when the network width (number of hidden neurons) grows. The authors showed that
features learning decouples from overfitting for large networks. 
\rev{Multi-index models were also studied in \cite{troiani2025fundamental,montanari2026phase},
which focused on feature learning beyond $t=O(1)$  time horizons: the results proven here were a useful tool in that context as well.}

\paragraph{Notational conventions.} We use boldface symbols for matrices or vectors whose dimensions 
diverge, e.g. $\btheta^t$, $\bX$ and so on. 
Also, we generally use upper case letters for matrices and lower case
letters for vectors, with the exception of $n\times k$ or $d\times k$ matrices such as $\btheta^t$.
We use $\normtwo{u}$ to denote the $\ell_2$ norm of a vector $u$. We also use 
$\norm{M}$ and $\norm{M}_F$ to denote the operator norm and Frobenius norm of a matrix 
$M$. \rev{
	For two vectors $u,v$ of the same dimension, we write $u\geq v$ or $u \leq v$ to
	represent entrywise inequality. For random variables $\xi$ and $\xi_1, \xi_2, \cdots, \xi_n, \cdots$ 
	defined on the same probability space, we denote convergence almost surely, in probability and 
	weakly by $\xi_n \stackrel{a.s.}\to \xi$, $\xi_n \stackrel{p}\to \xi$ and
	$\xi_n \Rightarrow \xi$, respectively. We will also use convergence in Wasserstein-$2$ 
	distance, denoted by $\mu_n \stackrel{\mathrm{W}_2}\to \mu$. In particular, 
	for two distributions $\mu, \nu$ on $\mathbb{R}^k$, 
	the Wasserstein-$2$ distance between them is defined by
	\begin{align*}
		\wdst{\mu}{\nu} := \inf_{\gamma \in \Gamma(\mu, \nu)} \sqrt{\int \normtwo{\xi - \eta}^2 \de \gamma(\xi, \eta)} \, ,
	\end{align*}
	where $\Gamma(\mu, \nu)$ denotes the collection of all couplings of $\mu$ and $\nu$.
}

	\section{Main results}
	\label{sec:Main}
	
	\newcommand{\bXt}{\bX^{\sT}}
	\newcommand{\opt}{^\star}
	\newcommand{\iidsim}{\stackrel{\mathrm{i.i.d.}}{\sim}}
	
	\subsection{Setting}

	\rev{We will next formally define the general setting of
	our work, and illustrate it with the running example of 
	Eqs.~\eqref{eq:FirstFlow} and \eqref{eq:FirstLoss}. 
	In the next subsection we will introduce the asymptotic
	characterization, and subsequently state our convergence results.}

\rev{As a motivation for our general setting, we reconsider the example of 
Eq.~\eqref{eq:FirstFlow}.
	Taking the derivative in Eq.~\eqref{eq:FirstLoss}, substituting back into Eq.~\eqref{eq:FirstFlow} and adopting the vector notation,
	gradient flow reads}
	\begin{align}
		\frac{\de\btheta^t}{\de t} = -\frac{d}{n}\bX^{\sT} \Loss'(\bX\btheta^t;\by)\, ,
		\label{eq:FirstFOM}
	\end{align}
	\rev{where $\bX\in\reals^{n\times d}$ is the matrix whose $i$-th row is given by vector $\bx_i$,
	$\by = (y_1,\dots,y_n)^{\sT}$, $\btheta=[\btheta^t_1,\dots,\btheta^t_k]\in\reals^{d\times k}$,
	and hence $\bX\btheta^t = [\bX\btheta^t_1,\dots,\bX\btheta^t_k]\in\reals^{n\times k}$. 
	Further $\Loss'(\,\cdot\,;y):\reals^k\to\reals^k$
is the derivative of $\Loss$ with respect to its first argument.
	We adopt the convention of applying $\Loss'$ to matrices $\bu\in\reals^{n\times k}$
row-wise, i.e. $\Loss'(\bu;\by) \in\reals^{n\times k}$ is the matrix whose $i$-th row is given by 
 $\Loss'(u_i;y_i)$ (where $u_i$ is the $i$-th row of $\bu$).
	}

	\paragraph{The general flow.} We can now state our general setting.
	Let $\ell:\reals^{k}\times \reals\times\reals_{\ge 0} \to \reals^k$,
	$(r,z,t)\mapsto  \ell_t(r;z)$ be a Lipschitz function,
	and $\Lambda:\reals_{\ge 0}\to \reals^{k\times k}$, $t \mapsto \Lambda^t$ be a bounded matrix-valued 
	function.
	
	Fixing a constant $\delta \in (0,\infty)$, we then consider the general flow over $\reals^{d\times k}$, defined via the 
	following ordinary differential equation, denoted by $\mathfrak{F}(\btheta^0, \bz, \Lambda, \ell)$,
	\begin{align}
\frac{\de \btheta^t}{\de t} = -\btheta^t\Lambda^{t,\sT} -\frac{1}{\delta}
\bX^{\sT}\bell_t(\bX\btheta^t;\bz)\, , \label{eq:GeneralFlow}
\end{align}
with the initial condition $\btheta^0\in \reals^{d\times k}$. We point out that $\bz \in \reals^n$ should be understood as a noise vector independent of the data matrix $\bX$--differing from the previous notation $\by$ which is a general response vector--as is clarified in the following when we specify distributional assumptions.
Here we follow the convention of applying functions to matrices row-wise.
In particular $\bell_t(\bX\btheta^t;\bz)\in\reals^{n\times k}$ is the matrix 
with rows
\begin{align*}
\bell_t(\bX\btheta^t;\bz) =
\left[\begin{matrix}
	\ell_t(\bx_1^{\sT}\btheta^t;z_1)\\
	\ell_t(\bx_2^{\sT}\btheta^t;z_2)\\
	\vdots\\
	\ell_t(\bx_n^{\sT}\btheta^t;z_n)\\
\end{matrix}\right]\, ,
\end{align*}
where we recall that $\bx_i$ is the $i$-th row of $\bX$. 
Notice that this setting generalizes the gradient flow equation \eqref{eq:FirstFlow}
in a few ways, apart from the fact that $\btheta^t$ can now have $k$
columns. First, the function $\ell_t$ can now depend on  the additional 
argument $z$ as well as on the time $t$; second, $\ell_t(\,\cdot\,; z)$ is not 
necessarily the gradient of a cost function; third, the additional term 
$-\btheta^t\Lambda^{t,\sT}$ allows us to include constraints on the norm
of $\btheta^t$, or regularization terms. In Section \ref{sec:Examples} we will illustrate how the additional flexibility introduced here allows
to capture applications in statistics and machine learning. 

\begin{example}[Fitting a two-layer neural network]
\rev{Consider the gradient flow of Eq.~\eqref{eq:FirstFlow}, which
is also equivalent to Eq.~\eqref{eq:FirstFOM}. 
It is clear that this is a special case of Eq.~\eqref{eq:GeneralFlow},
with the choices}
\begin{align}
\Lambda^t  = 0\, ,\;\;\; \bell_t(u;z) =\Loss'(u; z)\, .
\end{align}
\rev{where we recall that $\Loss'(\,\cdot\,;z):\reals^k\to\reals^k$ is 
the derivative of $\Loss$ with respect to its first argument. Explicitly 
using Eq.~\eqref{eq:FirstLoss}, we have
$\ell_{t}(u;z) = (\ell_{t,1}(u;z),\dots,\ell_{t,k}(u;z))^{\sT}$, where}
\begin{align}
\ell_{t,a}(u;z) = c_a\Big(z-\sum_{b=1}^k c_b\varphi(u_b)\Big)\cdot
\varphi'(u_a)\, .\label{eq:ComputationDerivative}
\end{align}
\end{example}

\begin{remark}
\rev{While in the previous example we assumed $\ell_t$ is independent of $t$ and given by a gradient,
our results accommodate situations in which we change the loss function 
$\Loss$ during the optimization process, as is sometimes done in machine learning. For instance,
the square of Eq.~\eqref{eq:FirstLoss} can be changed with another function with $t$.}
\end{remark}

\subsection{Dynamical Mean Field Theory}

\rev{Define $\br^t = \bX \btheta^t$. The main result of
our paper is that there exist low-dimensional stochastic processes 
$\theta^t$ and $r^t$ in $C([0,T],\reals^k)$,
such that, as $n,d\to\infty$, with $n/d\to \delta$,
\begin{align*}
	\frac{1}{d} \sum_{i=1}^d \theta_i^{[0,T]} \rightsquigarrow \theta^{[0, T]} \, , \qquad \frac{1}{n} \sum_{i=1}^n r_i^{[0,T]} \rightsquigarrow r^{[0,T]} \, ,
\end{align*}
where  $\theta^{[0, T]}$  denotes the function $[0,T]\ni t\mapsto \theta^t$ and 
similarly for $r^{[0, T]}$ and 
$\rightsquigarrow$ denotes convergence in distribution in
the space  $C([0,T],\reals^k)$, (see Theorem~\ref{thm:StateEvolution}).}

\rev{We will refer to the characterization of $\theta^t$ and $r^t$ as to
the \emph{Dynamical Mean Field Theory} (DMFT), and to the 
stochastic processes 
$\theta^t$ and $r^t$ themselves as to the \emph{DMFT processes}.
In words, for large $n,d$, a typical row of $\theta_i^{[0,T]}$ evolves
according to the DMFT process $\theta^t$. (Notationally we distinguish  the 
asymptotic process $\theta^t$ from the rows $\theta_i^t$ of $\btheta^t$, by the subscript.)}

Given random variables $(\theta^0, z) \in \reals^{k} \times \reals$, a matrix valued function 
$\Lambda:\reals_{\ge 0}\to \reals^{k\times k}$ and 
a Lipschitz function $\ell:\reals^{k}\times \reals\times\reals_{\ge 0} \to \reals^k$, 
\rev{the DMFT processes are defined as the unique (see 
Theorem~\ref{thm:UniqueExist} below)
solutions of the following system of equations (DMFT equations).}

The DMFT equations involve the unknown
deterministic functions 
$\Gamma: \realsp \to \reals^{k \times k}$, $R_\theta, R_\ell, C_\theta, C_\ell: 
\realsp \times \realsp \to \reals^{k \times k}$ and stochastic processes
$\theta, r: \realsp \to \reals^k$, 
which we will denote by $\mathfrak{S}:=\mathfrak{S}(\theta^0, z, \delta, \Lambda, \ell)$:\label{eq:int-diff-1}
\begin{subequations} \label{eq:DMFT}
\begin{align}
	\frac{\de}{\de t} \theta^t & = -(\Lambda^t + \Gamma^t) \theta^t - \int_0^t R_{\ell}(t,s) \theta^s \de s + u^t \, , & & u \sim \mathsf{GP}(0, \Cl / \delta) \, , \label{eq:def-theta} \\
	r^t & = - \frac{1}{\delta} \int_0^t R_{\theta}(t,s) \ell_s(r^s; z) \de s + w^t \, , & & w \sim \mathsf{GP}(0, \Ct) \, , \label{eq:def-r}  \\ 
	R_\theta(t,s) & = \Ep \brk{\frac{\partial \theta^t}{\partial u^s}} \, , & & 0 \leq s \leq t < \infty\, , \label{eq:def-R-t}  \\
	R_{\ell}(t,s) & = \Ep \brk{\frac{\partial \ell_t(r^t;z)}{\partial w^s}} \, , & &0 \leq s < t <\infty\, , \label{eq:def-R-l} \\
	\Gamma^t &= \Ep \brk{\nabla_r \ell_t(r^t;z)}\, , \label{eq:def-Gamma}  \\
	C_\theta(t,s) & = \Ep \brk{\theta^t {\theta^s}^{\sT}}\, , & & 0 \leq s , t < \infty\, ,  \label{eq:def-C-t} \\
	C_\ell(t,s) & =   \Ep \brk{\ell_t(r^t;z) \ell_s(r^s;z)^{\sT}} \, ,  & & 0 \leq s , t < \infty\, . \label{eq:def-C-l} 
\end{align}
\end{subequations}
Here the notation $u \sim \mathsf{GP}(0, \Cl / \delta)$,  $w \sim \mathsf{GP}(0, \Ct)$
means that $u, w$ are independent centered Gaussian processes with covariance kernels 
$C_{\ell} / \delta$ and $C_\theta$.
We set $R_\theta(t,s) = R_{\ell}(t,s) = 0$ for $t<s$. 

The quantities $\partial\theta^t / \partial u^s$ in Eq.~\eqref{eq:def-R-t} and 
$\partial \ell_t(r^t;z) /\partial w^s$ in Eq.~\eqref{eq:def-R-l} are stochastic processes defined via
the following equations:
\begin{subequations}\label{eq:int-diff-2}
\begin{align}
	\frac{\de}{\de t} \frac{\partial \theta^t}{\partial u^s} & = -(\Lambda^t + \Gamma^t) \frac{\partial \theta^t}{\partial u^s} - \int_s^t R_{\ell}(t,s') \frac{\partial \theta^{s'}}{\partial u^s} \de s'\, , & & 0 \leq s \leq t < \infty\, , \label{eq:def-derivative-t} \\
	\frac{\partial \ell_t(r^t;z)}{\partial w^s} & = \nabla_r \ell_t(r^t;z) \cdot \prn{- \frac{1}{\delta} \int_s^t R_{\theta}(t,s') \frac{\partial \ell_{s'}(r^{s'};z)}{\partial w^s}  \de s' - \frac 1 \delta R_\theta(t,s) \nabla_r \ell_s(r^s;z) } \, , & & 0 \leq s < t < \infty\, , \label{eq:def-derivative-l}
\end{align}
\end{subequations} 
with boundary condition $\partial\theta^t / \partial u^t = I$.
Note that the first one is a deterministic integral-differential equation 
and therefore $\partial\theta^t / \partial u^s$ is a deterministic function. 
On the other hand, $\partial \ell_t(r^t;z) /\partial w^s$ is in general a stochastic process because 
$r^t$ is. 

The rationale for the notations $\partial\theta^t / \partial u^s$ and
$\partial \ell_t(r^t;z) /\partial w^s$ is that Eqs.~\eqref{eq:int-diff-2} can be
heuristically derived as equations for such functional derivatives. However,
we will not need to prove that these are the actual functional derivatives. We set $\partial\theta^t / \partial u^s = 0$ if $s > t$ and $\partial \ell_t(r^t;z) /\partial w^s = 0$
if $s \geq t$.

\begin{remark}
Since $\partial \theta^t / \partial u^s$ is a deterministic function,
the expectation operator in Eq.~\eqref{eq:def-R-t} can be removed. However,
this is not immediately clear from Eq.~\eqref{eq:def-R-t} alone and also to keep uniformity in our definitions, we retain the expectation operator in Eq.~\eqref{eq:def-R-t}.
\end{remark}

\begin{remark}
The DMFT equations \eqref{eq:DMFT} generalize characterizations obtained heuristically
by physics methods, the most closely related earlier work being \cite{agoritsas2018out}
which studies the perceptron model.
Section \ref{sec:Interpretation} we will briefly discuss the interpretation of the DMFT process defined 
by Eqs.~\eqref{eq:DMFT}.
Additional intuition is provided by the cavity derivation in \cite{agoritsas2018out}.
\end{remark}

\begin{example}[Fitting a single neuron]
\rev{To illustrate the DMFT equations, consider the special case 
of $k=1$, $c_1=1$ of the flow defined in Eq.~\eqref{eq:FirstFlow}.
Using Eq.~\eqref{eq:ComputationDerivative}, we get that
Eqs.~\eqref{eq:def-theta}, \eqref{eq:def-r} reduce to}
\begin{align}
	\frac{\de}{\de t} \theta^t & = - \Gamma^t \theta^t - 
	\int_0^t R_{\ell}(t,s) \theta^s \de s + u^t \, , & & u \sim \mathsf{GP}(0, \Cl / \delta) \, , \label{eq:def-theta-example} \\
	r^t & = - \frac{1}{\delta} \int_0^t R_{\theta}(t,s) \, (z-\varphi(r^s)) \varphi'(r^s)\,
	\de s + w^t \, , & & w \sim \mathsf{GP}(0, \Ct) \, . \label{eq:def-r-example}  
\end{align}
\rev{It is important to emphasize that here (as in the general Eqs.~\eqref{eq:def-theta}, \eqref{eq:def-r})
$\Gamma^t$ $R_{\ell}(t,s)$ and $R_{\theta}(t,s)$, as well as the kernels $C_\theta(t,s)$ and $C_\ell(t,s)$
 are deterministic functions of $t$ and $s$. Once these are given, the 
 stochastic processes $\theta^t$ and $r^t$ are defined by solving the above equations.}
\end{example}

\subsection{Statement of main results}

\begin{assumption}\label{ass:Normal}
\begin{itemize}
	\item[(a)] The entries $\bX= (X_{ij})_{i\le n,j\le d}$ are given by $X_{ij}= \oX_{ij}/\sqrt{d}$,
	where $(\oX_{ij})_{i,j\ge 1}$ is a collection of i.i.d. random variables with distribution
	independent of $n,d$, such that $\E\{\oX_{ij}\}=0$, $\E\{\oX^2_{ij}\}=1$, 
	and $\|\oX_{ij}\|_{\psi_2}\le C$ for a constant $C$ (here $\|\, \cdot\,\|_{\psi_2}$
	denotes  the sub-Gaussian norm).
	\item[(b)] The function $\ell_t(r; z)$ is Lipschitz continuous 
	with Lipschitz
	continuous Jacobian in $t$ and $r$. Further, these Lipschitz constants are bounded uniformly over $t\in [0,T]$ and $z \in \reals$.
	Namely there exists some $\cstloss \in \realsp$ such that, for all $z$, all
	$r_1,r_2\in\reals^k$ and $t_1, t_2 \in [0, T]$, we have
	\begin{subequations}
		\begin{align}
			\big\|\ell_{t_1}(r_1;z)-\ell_{t_2}(r_2;z)\big\|_2&\le \cstloss(\|r_1-r_2\|_2 + |t_1 - t_2|)\, ,\\
			\big\|D\ell_{t_1}(r_1;z)-D\ell_{t_2}(r_2;z)\big\|&\le \cstloss(\|r_1-r_2\|_2 + |t_1 - t_2|)\, ,
		\end{align}
	\end{subequations}
	where
	\begin{align*}
		D\ell_t(r;z) = \begin{bmatrix} \nabla_r \ell_t(r;z) & \frac{\de}{\de t} \ell_t(r;z) \end{bmatrix}.
	\end{align*}
	\item[(c)] In addition, $\Lambda^t$ is Lipschitz continuous and symmetric. There exists some $\cstlbd \in \realsp$ such that $\|\Lambda^t\| \leq \cstlbd$ for all $t \in [0, T]$, and for all $t_1, t_2 \in [0,T]$
	\begin{align}
		\norm{\Lambda^{t_1} - \Lambda^{t_2}} \leq \cstlbd |t_1 - t_2| \, .
	\end{align}
\end{itemize}

\end{assumption}

Our first result establishes existence and uniqueness of solutions of the DMFT system
$\mathfrak{S}(\theta^0, z, \delta, \Lambda, \ell)$. 
\begin{theorem} \label{thm:UniqueExist}
Let Assumption~\ref{ass:Normal} hold, and suppose the random variables $(\theta^0, z) \in \reals^k \times \reals$ satisfy
\begin{align}
	\cstthetaz := \max \left\{\Ep \brk{\normtwo{\theta^0}^2}, \sup_{t \in \realsp} \Ep \brk{\normtwo{\ell_t(0; z)}^2} \right\} < \infty.
\end{align} 
Given any functions $\Lambda: \realsp \to \reals^{k \times k}$ and $\ell: \reals^k \times \reals \times \reals_{\ge 0} \to \reals^k$, 
there exists a sextet $(\theta, r, \Rt, \Rl, \Ct, \Cl)$ solving the DMFT system $\mathfrak{S} := \mathfrak{S}(\theta^0, z, \delta, \Lambda, \ell)$ 
defined through Eqs.~\eqref{eq:def-theta} to \eqref{eq:def-C-l} and Eqs.~\eqref{eq:def-derivative-t} to \eqref{eq:def-derivative-l}. 
The solution is also unique among all sextets whose components $(\Ct, \Rt)$ are bounded in all compact sets in $\realsp^2$. There further exists nondecreasing functions $\pRt, \pRl, \pCt, \pCl: \realsp \to \realsp$ satisfying
\begin{subequations}
	\begin{align}
		&\norm{\Rt(t,s)} \leq \pRt(t-s)\, , & & \norm{\Rl(t,s)} \leq \pRl(t-s)\, , & & \forall 0 \leq s \leq t < \infty\, , \\
		&\norm{\Ct(t,t)} \leq \pCt(t)\, , & &\norm{\Cl(t,t)} \leq \pCl(t)\, , & & \forall 0 \leq t < \infty\, ,\\
		&\norm{\Gamma^t} \leq \cstloss \, , & & & & \forall 0 \leq t < \infty\,.
	\end{align}
\end{subequations}
Further, the process $(\theta^t)_{t\in [0,T]}$ has continuous sample paths.
Finally, the functions $\pRt, \pRl, \pCt, \pCl$ are such that there exists 
$\lambda:= \lambda(\theta^0, z, \delta,\cstthetaz, \cstloss, \cstlbd) > 0$ such that
\begin{align}
	\lim_{t \to \infty} e^{-\lambda t} \max \left\{\pRt(t), \pRl(t), \pCt(t), \pCl(t) \right\} = 0. \label{eq:Phi-exponential-growth}
\end{align}
\end{theorem}
The proof and the definitions of $\pRt, \pRl, \pCt, \pCl$ are presented in Section~\ref{sec:proof:UniqueExist},
with most technical details deferred to the appendices.

We next prove that the original flow converges ---in a suitable sense--- to the unique solution
of the DMFT system in the proportional asymptotics $n,d\to\infty$, with $n/d\to \delta$.
\begin{theorem} \label{thm:StateEvolution}
Under Assumption~\ref{ass:Normal},
further assume that $n,d\to\infty$ with $n/d\to\delta\in (0,\infty)$.
Let $\bz,\btheta^0$ be independent of $\bX$, and assume that the empirical distributions
$\widehat{\mu}_{\theta^0} := d^{-1}\sum_{i=1}^d\delta_{\theta^0_i}$ and $\widehat{\mu}_{z} := n^{-1}\sum_{i=1}^n\delta_{z_i}$
converge 
weakly to $\mu_{\theta^0}$ and $\mu_{z}$, 
$\Ep_{\widehat{\mu}_{\theta^0}} [\|\theta^0\|^2] \to \Ep_{\mu_{\theta^0}} [\|\theta^0\|^2] < \infty$ and 
$\Ep_{\widehat{\mu}_{z}} [\|z\|^2] \to \Ep_{\mu_{z}} [\|z\|^2] < \infty$. Let $\rev{\theta_i^{[0,T]}}$ be 
the unique stochastic process that solves $\mathfrak{S}$ in Theorem~\ref{thm:UniqueExist}.
Finally, define $\br^t:=\bX\btheta^t\in \reals^{n\times k}$, $t\ge 0$.

Then, for any distance $d_{\sW}$ that metrizes weak convergence in  $C([0,T],\reals^k)$
(for instance $d_{\sW}=d_{\sBL}$ the bounded Lipschitz distance), 
\rev{\begin{align}
	d_{\sBL}(\mu,\nu):= \sup\big\{\int f \de\mu -\int f \de\nu:\; \|f\|_{\infty}\le 1,\;  
	\|f\|_{\sLip}\le 1 \big\}\, .
\end{align}}
we have
\begin{align}
	&\plim_{n,d\to\infty}d_{\sW}\Big(\frac{1}{d}\sum_{i=1}^d \delta_{\rev{\theta_i^{[0,T]}}},{\rm P}_{\rev{\theta^{[0,T]}}}\Big)=0\, ,\label{eq:LimTheta}\\
	&\plim_{n,d\to\infty}d_{\sW}\Big(\frac{1}{n}\sum_{i=1}^n \delta_{z_i,\rev{r_i^{[0,T]}}},{\rm P}_{z,\rev{r^{[0,T]}}}\Big)=0\, .\label{eq:LimR}
\end{align}
Here $\plim_{n,d\to\infty}$ denotes convergence in probability, 
${\rm P}_{\rev{\theta^{[0,T]}}}$ denotes the law of $\rev{\theta^{[0,T]}}:= (\theta^t)_{0\le t\le T}$,
and ${\rm P}_{z,\rev{r^{[0,T]}}}$ denotes the joint law of $z$ and 
$\rev{r^{[0,T]}}:= (r^t)_{0\le t\le T}$.
\end{theorem}
The proof of this theorem is presented in Section~\ref{sec:proof:StateEvolution}.

\begin{remark} \label{remark:MainThm}
Concretely, the convergence in Theorem \ref{thm:StateEvolution} implies the following.
For any $L\in \naturals$, any times $0\le t_1<\cdots\le t_L$, and any bounded continuous functions
$\psi:(\reals^k)^L\to\reals$, mapping $(x_1,\dots,x_L)\mapsto \psi(x_1,\dots,x_L)$,
for $x_i\in\reals^k$, and $\tilde\psi:(\reals^k)^L\times\reals\to\reals$, we have:
\begin{align}
	\plim_{n,d\to\infty}\frac{1}{d}\sum_{i=1}^d \psi(\theta_i^{t_1},\dots,\theta_i^{t_L}) = 
	\E\big\{\psi(\theta^{t_1},\dots,\theta^{t_L}) \big\}\, ,\label{eq:TestF_Rmk}\\
	\plim_{n,d\to\infty}\frac{1}{n}\sum_{i=1}^n \tilde\psi(r_i^{t_1},\dots,r_i^{t_L},z) = 
	\E\big\{\tilde\psi(r^{t_1},\dots,r^{t_L},z) \big\}\, .
\end{align}
The expectation on the right-hand side is with respect to the processes $(\theta^t)_{t\ge 0}$,
$(r^t)_{t\ge 0}$
defined by the DMFT system.	

In the case the matrix $\bX$ has i.i.d. Gaussian entries, the same proof of 
Section~\ref{sec:proof:StateEvolution} implies a somewhat stronger statement by leveraging the 
results of \cite{javanmard2013state}. Namely, Eq.~\eqref{eq:TestF_Rmk} holds for any 
continuous functions with at most quadratic growth $|\psi(x)|\le C(1+\|x\|_2^2)$,
$|\tilde\psi(x)|\le C(1+\|x\|_2^2)$.
\end{remark}

\subsection{Interpretation of the DMFT process}
\label{sec:Interpretation}

Theorem \ref{thm:StateEvolution} establishes that
the stochastic process $(\theta^t: t\ge 0)$  captures the $n,d\to\infty$ asymptotics of
$(\theta_i^t: t\ge 0)$  ($i$-th row of $\btheta$),  and $(r^t: t\ge 0)$  captures the
asymptotics of
$(r_i^t= (\bX\btheta)_i^t: t\ge 0)$  ($i$-th row of $\br^t = \bX\btheta$).
In order to develop some intuition of this characterization,
it is instructive \rev{to}  consider a modification of
Eq.~\eqref{eq:GeneralFlow} whose state $\hbtheta^t$ evolves according to the expectation
of the right-hand side (for $X_{ij}\sim \normal(0,1/d)$ independent of $\hbtheta$, which is deterministic in this case):
\begin{align}
\frac{\de \hbtheta^t}{\de t} = -\hbtheta^t\Lambda^{t,\sT} -\frac{1}{\delta}
\E\big\{\bX^{\sT}\bell_t(\bX\hbtheta^t;\bz)\big\}\, .\label{eq:SimplifiedFlow}
\end{align}
Considering the $i$-th row of $\hbtheta^t$ 
(and letting $\bx_i$ denote the $i$-th row
of $\bX$) we obtain:
\begin{align}
\frac{\de \htheta_i^t}{\de t} &= -\Lambda^{t}\htheta_i^t -\frac{1}{\delta}
\sum_{j=1}^n\E\big\{X_{ji}\ell_t(\hbtheta^{t,\sT}\bx_j;\bz)^{\sT}\big\}\nonumber\\
& \stackrel{\mathrm{(i)}}{=} -\Lambda^{t}\htheta_i^t -\frac{1}{n}
\sum_{j=1}^n\E\big\{\nabla_r\ell_t(\hr^t_j ;\bz)^{\sT}\big\}\htheta^t_i\, ,\label{eq:PopFlow}
\;\;\;\;\;\;
\hr^t_j := \hbtheta^{t,\sT}\bx_j\, ,
\end{align}
where $\mathrm{(i)}$ we used Stein's lemma \rev{and $n = d \delta$}. Note that the $(\hr^t_j:j\le n)$ are i.i.d. and let
$\hr^t\ed \hr^t_1$. We obtain 
\begin{align}
\frac{\de \htheta_i^t}{\de t} &= -(\Lambda^{t}+\hat\Gamma^t)\htheta_i^t\, ,
\;\;\;\;\; \hat\Gamma^t :=\E\big\{\nabla_r\ell_t(\rev{\hr^t} ;\bz)^{\sT}\big\}\, .
\end{align}
This coincides with Eqs.~\eqref{eq:def-theta} and \eqref{eq:def-Gamma} 
in which the terms  $\int_0^t R_{\ell}(t,s) \theta^s \de s$ and $u^t$ have been dropped. 
Notice that these terms are heuristically of order $1/\delta$ and $1/\sqrt{\delta}$,
which can be understood since the difference between 
$\bX^{\sT}\bell_t(\bX\hbtheta^t;\bz)$ (appearing in Eq.~\eqref{eq:GeneralFlow}) and its expectation 
(cf. Eq.~\eqref{eq:SimplifiedFlow}) is of order $1/\sqrt{\delta}$.

From Eq.~\eqref{eq:PopFlow},
$(\hr^{t_1}_j, \hr^{t_2}_j, \dots , \hr_j^{t_\ell})$ are jointly Gaussian for any fixed 
$t_1,\dots \rev{,} t_{\ell}$ (because $\bx_j$ \rev{is} Gaussian independent of $\hbtheta^t$).
Their covariance is
\begin{align}
\E[\hr^{t_1}_j(\hr^{t_2}_j)^{\sT}] = \E[ \hbtheta^{t_1,\sT}\bx_j \bx_j^{\sT}\hbtheta^{t_2}]
=\frac{1}{d}\hbtheta^{t_1,\sT}\hbtheta^{t_2}\to C_{\theta}(t_1,t_2)\, .
\end{align}
This matches Eq.~\eqref{eq:def-r} provided we drop the term proportional to
$\int_0^t R_{\theta}(t,s) \ell_s(r^s; z) \de s$, which is heuristically of order $1/\delta$. 

To summarize, the DMFT equations for the simplified (non-random) flow
\eqref{eq:SimplifiedFlow} are easy to derive and corresponds to 
Eqs.~\eqref{eq:def-theta}, \eqref{eq:def-r}, and \eqref{eq:def-Gamma} in which terms of order 
$1/\sqrt{\delta}$ and $1/\delta$ have been dropped. 
We finally note that these terms arise because of two reasons. 

First, 
the drift  $-\frac{1}{\delta}\bX^{\sT}\bell_t(\br^t;\bz)$ fluctuates around expectation,
which leads to the additional term $u^t$ in Eq.~\eqref{eq:def-theta}.
\rev{Its covariance would be the same as that of the $i$-th row of $-\frac{1}{\delta}\bX^{\sT}\bell_t(\br^t;\bz)$ if the process $\br^t$ were independent of $\bX$.}

Second, both Eq.~\eqref{eq:def-theta} and Eq.~\eqref{eq:def-r} contain memory terms,
which results in the asymptotic process $(\theta^t: t\ge 0)$ and $(r^t: t\ge 0)$ 
being non-Markovian. The non-Markovian nature is expected. Indeed, the original process
$(\btheta^t:\; t\ge 0)$ is Markov (indeed deterministic), conditionally on $\bX$.
However the marginal distribution of $\btheta^t$ is non-Markov, and so is the  marginal
distribution of any low-dimensional projection of $\btheta^t$ (e.g. its $i$-th row $\theta_i^t$, 
whose asymptotics is captured by Eq.~\eqref{eq:def-theta}).

\section{Applications}
\label{sec:Examples}

In this section,
we apply Theorem \ref{thm:StateEvolution} to prove a DMFT characterization of gradient 
flows for generalized linear models and shallow neural networks with a 
constant number of hidden neurons. 

Next, we define a notion of stationary-point solutions of the DMFT system 
$\mathfrak{S}(\theta^0, z, \delta, \Lambda, \ell)$.
We show that these stationary points are characterized by a system of five 
nonlinear equations, and that they are in correspondence with stationary points of a 
certain infinite-dimensional variational principle. 
This variational principle also emerges in the study of global minimizers
of the risk via Gordon's comparison inequality.    
As an illustration, we discuss the case of logistic regression.

We introduce two classes of applications that are covered by our general results, 
\rev{Theorems} \ref{thm:UniqueExist} and \ref{thm:StateEvolution}: gradient flow with respect 
to these cost functions are special cases of the general theorem stated below. 
\begin{description}
\item[\textbf{Generalized linear models}.] 
The statistician observes $n$ iid pairs $(y_i,\bx_i)$ where $y_i = \varphi(\< \bx_i , \btheta^* \>,z_i)$ and $z_i$ is noise drawn independently of $\bx_i$.
The goal is to estimate or recover the planted signal $\btheta^*$.
Ridge regularized empirical risk minimization attempts to minimize the objective 
\begin{equation}
	\label{eq:glm-erm}
	\cL_n(\btheta)
	:= 
	\frac{d}{n} \sum_{i=1}^n \Ls_0(\< \bx_i,\btheta\>;y_i)
	+ 
	\frac{\lambda}{2} \| \btheta \|_2^2.
\end{equation} 
When $\varphi(r,z) = r+z$, we recover a linear model (in this case, we might take $\Ls_0(r;y)=(r-y)^2$).
When $\varphi(r,z) = \boldsymbol{1}\{ r + z \geq 0 \}$ and $z \sim \mathrm{Logistic}$
(and for $\Ls_0$ the logistic loss), we recover logistic regression.
When $\varphi(r,z) = |r|^2 + z$, we recover a model of noisy phase-retrieval.
Alternative choices of $\varphi$ recover several other popular 
regression and classification models.

\item[\textbf{Shallow neural networks with constant number of hidden units}.]
For a fixed-constant $k$,
a two-layer neural network with width $k$ and activation $\sigma: \reals \rightarrow \reals$ is the function class containing functions of the form
\begin{equation}
	f_{\btheta}(\bx)
	:= 
	\sum_{a=1}^k
	\alpha_a \sigma(\< \bx , \btheta^{(a)}\>),
\end{equation}
where $\btheta \in \reals^{d \times k}$ has columns $\{ \btheta^{(a)}\}_{a \in [k]}$.
Given training data $\{(y_i,\bx_i)\}_{i \in [n]}$, 
the statistician may fit a neural network by gradient descent on the objective 
\begin{equation}
	\cL_n(\btheta)
	:= 
	\frac{d}{n} \sum_{i=1}^n \Ls_0\left( \sum_{a = 1}^k \alpha_a \sigma(\< \bx_i,\btheta^{(a)}\>);y_i\right)
	+ 
	\frac{\lambda}{2} \sum_{a = 1}^k \| \btheta^{(a)} \|_2^2.
\end{equation} 
\rev{As} in a multi-index model, the response is assumed to depend on a low-dimensional
projection of the data, for instance
$y_i = \varphi\big((\btheta^{*})^{\sT}\bx_i;z_i\big)$.
\end{description}


\subsection{Flow with planted signal}

At first glance,
it may appear that Theorem \ref{thm:StateEvolution} does not apply in general
to gradient flow in the examples above because $\by$ is not independent of $\bX$, 
whereas the noise $\bz$ in Theorem \ref{thm:StateEvolution} is independent of $\bX$.
In fact, \rev{as we shall explain below},
gradient flow in these examples can be represented as a cross-section of a flow of
the form \eqref{eq:GeneralFlow} on a higher dimensional space, so that 
its DMFT is an instance of Theorem \ref{thm:StateEvolution}.

First note that the cost functions above can be written as
\begin{align} 
\cL_n(\btheta)
:= 
\frac{d}{n} \sum_{i=1}^n \Ls(\btheta^{\sT} \bx_i,(\btheta^*)^{\sT} \bx_i;z_i)
+ 
\frac{\lambda}{2} \| \btheta \|_F^2\, ,\label{eq:GeneralCost}
\end{align} 	
for a suitable function $\Ls:(\reals^k)^2\to \reals$.  For instance,
in the generalized linear model, we set $k=1$ and
$\Ls(r,w;z) = \Ls_0(r,\varphi(w,z))$.

It is convenient to consider a more general (non-gradient, time-dependent) flow
of the form
\begin{align}
\frac{\de \btheta^t}{\de t} = -\btheta^t\Lambda^t -\frac{1}{\delta}
\bX^{\sT}\bell_t(\bX\btheta^t,\bX \btheta^*;\bz)\,. \label{eq:GeneralPlantedFlow}
\end{align}
We recover gradient flow with respect to the general cost \eqref{eq:GeneralCost} 
by setting $\Lambda^t = \lambda I_k$ and $\bell_t(\bX \btheta^t,\bX\btheta^*;\bz)_i = 
\nabla_r\Ls(r,w;z) |_{(r,w,z) = (\btheta^{\sT} \bx_i,(\btheta^*)^{\sT} \bx_i, z_i)}$.
We can put Eq.~\eqref{eq:GeneralPlantedFlow} into the form of flow \eqref{eq:GeneralFlow} 
by concatenating $\btheta^*$ to the iterates $\btheta^t$ and considering the flow 
\begin{align}
\label{eq:full-flow}
\frac{\de (\btheta^t,\btheta^*)}{\de t} 
= 
-(\btheta^t ,\btheta^*)
\begin{pmatrix} 
	\Lambda^t & 0 \\
	0 & 0
\end{pmatrix} 
- \frac{1}{\delta}
\bX^{\sT}
\begin{pmatrix} 
	\vert & \vert \\
	\bell_t(\bX\btheta^t,\bX \btheta^*;\bz) & 0 \\ 
	\vert & \vert
\end{pmatrix}.
\end{align}
Indeed, the final $k$-columns on the right-hand side are 0, so that $\btheta^*$ does 
not change along the trajectory.
Theorem \ref{thm:StateEvolution} can be applied directly to this flow.
Doing so leads to certain simplifications which allow us to represent the 
asymptotic characterization as a $k$-dimensional rather than $2k$-dimensional process.
Here we present this characterization as a corollary to Theorem \ref{thm:StateEvolution}.

Also, notice that there is no loss of generality in assuming that 
$\btheta^t$ and $\btheta^*$ have the same number of columns. Indeed, we can always
accommodate cases in which the number of columns is different by adding some zero columns, 
and redefining $\ell_t$ accordingly. 

We require the following assumption on the flow defined in Eq.~\eqref{eq:GeneralPlantedFlow},
defined in terms of a function $\ell_t:(\reals^k)^2\times\reals\to\reals^k$.
\begin{assumption}\label{ass:planted-Normal}
\begin{itemize}
\item[(a)] The same conditions of Assumption \ref{ass:Normal} are required on the random matrix $\bX$.

\item[(b)] The function $(r,w^*, z)\mapsto \ell_t(r,w^*; z)$ is assumed to be Lipschitz continuous 
with Lipschitz
continuous Jacobian in $t$, $r$, and $w^*$. 
Further, these Lipschitz constants are bounded uniformly over $t\in [0,T]$ and $z \in \reals$.
Explicitly, there exists $\cstloss \in \realsp$ such that, for all $z$, all
$\rev{w_1^*, w_2^*}, r_1,r_2\in\reals^k$ and $t_1, t_2 \in [0, T]$, we have
\begin{subequations}
	\begin{align}
		\big\|\ell_{t_1}(r_1,w_1^*;z)-\ell_{t_2}(r_2,w_2^*;z)\big\|_2&\le \cstloss(\|r_1-r_2\|_2 + \| w_1^* - w_2^* \|_2 + |t_1 - t_2|)\, ,\\
		\big\|D\ell_{t_1}(r_1,w_1^*;z)-D\ell_{t_2}(r_2,w_2^*;z)\big\|&\le \cstloss(\|r_1-r_2\|_2 + \| w_1^* - w_2^* \|_2 + |t_1 - t_2|)\, ,
	\end{align}
\end{subequations}
where
\begin{align*}
	D\ell_t(r,w^*;z) = \begin{bmatrix} \nabla_r \ell_t(r;z) & \nabla_{w^*} \ell_t(r,w^*;z) & \frac{\de}{\de t} \ell_t(r;z) \end{bmatrix}.
\end{align*}
\item[(c)] In addition, $\Lambda^t$ is Lipschitz continuous and symmetric. Explicitly, there exists  $\cstlbd \in \realsp$ such that $\|\Lambda^t\| \leq \cstlbd$ for all $t \in [0, T]$, and for all $t_1, t_2 \in [0,T]$
\begin{align}
	\norm{\Lambda^{t_1} - \Lambda^{t_2}} \leq \cstlbd |t_1 - t_2| \, .
\end{align}
\end{itemize}
\end{assumption}
Given random variables $(\theta^0,\theta^*,z) \in (\reals^k)^2 \times \reals$,
we consider the following system of equations
$\mathfrak{S}:=\mathfrak{S}(\theta^0, \theta^*, z, \delta, \Lambda^t, \ell_t)$ for 
unknown deterministic functions $\Lambda^t: \realsp \to \reals^{k \times k}$, 
$R_\ell, C_\theta: (\realsp \cup \{ * \}) \times (\realsp \cup \{ * \}) \to \reals^{k \times k}$,
$R_\theta,C_\ell : \realsp  \times \realsp \to \reals^{k \times k}$, 
and stochastic processes $\theta: (\realsp\cup \{*\}) \to \reals^k$, 
$r: \realsp \rightarrow \reals^k$:
\begin{subequations}\label{eq:planted-se}
\begin{align}
	\frac{\de}{\de t} \theta^t & = -(\Lambda^t + \Gamma^t) \theta^t - \int_0^t R_{\ell}(t,s) \theta^s \de s -  R_\ell(t,*) \theta^* + u^t \, , & & u^t \sim \mathsf{GP}(0, \Cl / \delta) \, , \label{eq:def-theta-plt} \\
	r^t & = - \frac{1}{\delta} \int_0^t R_{\theta}(t,s) \ell_s(r^s,w^*;z) \de s + w^t \, , & & w^t \sim \mathsf{GP}(0, \Ct) \, , \label{eq:def-r-plt}  \\ 
	R_\theta(t,s) & = \Ep \brk{\frac{\partial \theta^t}{\partial u^s}} \, , & & 0 \leq s \leq t < \infty\, , \label{eq:def-R-t-plt}  \\
	R_{\ell}(t,s) & = \Ep \brk{\frac{\partial \ell_t(r^t,w^*;z)}{\partial w^s}} \, , & &0 \leq s < t <\infty\, , \label{eq:def-R-l-plt} \\
	R_{\ell}(t,*) & = \Ep \brk{\frac{\partial \ell_t(r^t,w^*;z)}{\partial w^*}}\, ,\\
	\Gamma^t &= \Ep \brk{\nabla_r \ell_t(r^t,w^*;z)}\, , \label{eq:def-Gamma-plt}  \\
	C_\theta(t,s) & = \Ep \brk{\theta^t {\theta^s}^\sT}\, , & & 0 \leq s \leq t < \infty\;\text{or}\;s=*\,,  \label{eq:def-C-t-plt} \\
	C_\ell(t,s) & =   \Ep \brk{\ell_t(r^t,w^*;z) \ell_s(r^s,w^*;z)^\sT} \, ,  & & 0 \leq s \leq t < \infty\, . \label{eq:def-C-l-plt} 
\end{align}
\end{subequations}
Here $u^t, w^t$ are independent centered Gaussian processes with covariance kernels 
$C_{\ell} / \delta$ and $C_\theta$, and $y = \varphi(w^*,z)$ \rev{(recall that the cost function implicitly depends on $y$ through $\Ls(r,w;z) = \Ls_0(r,\varphi(w,z))$)}, with 
$w^*\sim\normal(0,\E[\theta^*(\theta^*)^{\sT}])$. 
As before, we have
$C_{\theta}(s,t) = C_{\theta}(t,s)$,  $C_{\ell}(s,t) = C_{\ell}(t,s)$,
and $R_\theta(t,s)= R_\ell(t,s)=0$ for $t<s$.

The quantities $\partial\theta^t / \partial u^s$, $\partial \ell(r^t,w^*;z) /\partial w^s$, and $\partial \ell(r^t,w^*;z) /\partial w^*$ are 
uniquely defined via the following integral-differential equations \rev{for $0 \leq s \leq t < \infty$,}
\begin{subequations}\label{eq:planted-derivs}
\begin{align}
	\frac{\de}{\de t} \frac{\partial \theta^t}{\partial u^s} & = -(\Lambda^t + \Gamma^t) \frac{\partial \theta^t}{\partial u^s} - \int_s^t R_{\ell}(t,s') \frac{\partial \theta^{s'}}{\partial u^s} \de s'\, ,  \label{eq:def-derivative-planted-t} \\
	\frac{\partial \ell(r^t,w^*;z)}{\partial w^s} & = \nabla_r \ell(r^t,w^*;z) \cdot \prn{- \frac{1}{\delta} \int_s^t R_{\theta}(t,s') \frac{\partial \ell(r^{s'},w^*;z)}{\partial w^s}  \de s' - \frac 1 \delta R_\theta(t,s) \nabla_r \ell(r^s,w^*;z) } \, ,  \label{eq:def-derivative-planted-l}\\ 
	\frac{\partial \ell(r^t,w^*;z)}{\partial w^*} & =  - \frac{1}{\delta} \nabla_r \ell(r^t,w^*;y) \int_0^t R_{\theta}(t,s') \frac{\partial \ell(r^{s'},w^*;z)}{\partial w^*}  \de s' + \nabla_{w^*} \ell(r^t,w^*;z) \, ,  \label{eq:def-derivative-dw}
\end{align}
\end{subequations} 
with boundary condition $\partial\theta^t / \partial u^t = I$. Similarly, we set $\partial\theta^t / \partial u^s = 0$ if $s > t$ and $\partial \ell(r^t;z) /\partial w^s = 0$ if $s \geq t$. 
As before, our notations point to the fact that these quantities are functional derivatives, although we do not 
prove it formally (and we do not need to).

\begin{corollary}
\label{cor:planted-se}
Under Assumption~\ref{ass:planted-Normal},
suppose the random variables $(\theta^0, \theta^*, z) \in (\reals^k)^2 \times \reals$ satisfy
\begin{align}
	M_{\rev{\theta^0},\theta^*,z} 
	:= 
	\max \left\{\Ep \brk{\normtwo{\theta^0}^2}, \Ep \brk{\normtwo{\theta^*}^2}, \sup_{t \in \realsp} \Ep \brk{\normtwo{\ell_t(0,0; z)}^2} \right\} < \infty.
\end{align}         
Then the system of equations $\mathfrak{S}(\theta^0, \theta^*, z, \delta, \Lambda^t, \ell_t)$ defined in Eqs.~\eqref{eq:planted-se} and \eqref{eq:planted-derivs} has a unique solution.

Moreover, assume that $n,d\to\infty$ with $n/d\to\delta\in (0,\infty)$.
Let $\bz,\btheta^0,\btheta^*$ be independent of $\bX$, and assume that the empirical distributions
$\widehat{\mu}_{\theta^0,\theta^*} := d^{-1}\sum_{i=1}^d\delta_{(\theta^0_i,\theta^*_i)}$ and 
$\widehat{\mu}_{z} := n^{-1}\sum_{i=1}^n\delta_{z_i}$ 
converge weakly to $\mu_{\theta^0, \theta^*}$ and $\mu_{z}$, 
$\Ep_{\widehat{\mu}_{\theta^0,\theta^*}} [\|\theta^0\|^2 + \|\theta^*\|^2]
\to \Ep_{\mu_{\theta^0,\theta^*}} [\|\theta^0\|^2 + \|\theta^*\|^2] < \infty$ 
and $\Ep_{\widehat{\mu}_{z}} [\|z\|^2] \to \Ep_{\mu_{z}} [\|z\|^2] < \infty$. 
Let $\rev{\theta^{[0,T]}}$ be the \rev{unique stochastic process that solves} the system
$\mathfrak{S}$ of Eq.~\eqref{eq:planted-se}.
Finally, define $\br^t:=\bX\btheta^t\in \reals^{n\times k}$, $t\ge 0$ \rev{and $\by = \varphi(\bX \btheta^*; \bz)$}.

Then, for any distance $d_{\sW}$ that metrizes weak convergence in  $C([0,T],\reals^{k_0})$,
for some fixed $k_0$ (for instance $d_{\sW}=d_{\sBL}$ the bounded Lipschitz distance), we have
\begin{align}
	&\plim_{n,d\to\infty}d_{\sW}\Big(\frac{1}{d}\sum_{i=1}^d \delta_{\theta^*_i,\rev{\theta_i^{[0,T]}}},{\rm P}_{\theta^*,\rev{\theta^{[0,T]}}}\Big)=0\, ,\\
	&\plim_{n,d\to\infty}d_{\sW}\Big(\frac{1}{n}\sum_{i=1}^n \delta_{y_i,\rev{r_i^{[0,T]}}},{\rm P}_{\varphi(w^*,z),\rev{r^{[0,T]}}}\Big)=0\, .
\end{align}
\end{corollary}

\noindent Corollary \ref{cor:planted-se} is proved in Appendix \ref{app:fixed-pt}.

\begin{remark}
The interpretation of the DMFT equations \eqref{eq:planted-se} is
analogous to the interpretations of Eqs.~\eqref{eq:DMFT} given in Section \ref{sec:Interpretation}.
Indeed,  the main difference is the \rev{appearance} of a term $- R_\ell(t,*) \theta^*$
in Eq.~\eqref{eq:def-theta-plt}. This is not unexpected since the data distribution is parametrized by 
$\theta^*$. Indeed it is straightforward to repeat the derivation in the simplified model of Section \ref{sec:Interpretation},
and recover  Eq.~\eqref{eq:def-theta-plt} without terms $- \int_0^t R_{\ell}(t,s) \theta^s \de s$
and $u^t$ and the term $R_\ell(t,*)$ replaced by $\E[\nabla_{w^*} \ell(r^t,w^*;z)]$. 
\end{remark}

\begin{remark}
We emphasize that here we study algorithms in which each data-point $(y_i,\bx_i)$
is used all along the gradient flow (or \rev{gradient} descent) trajectory. This ultimately gives
rise to the non-Markovian nature of the evolution \eqref{eq:planted-se}.
Similar behavior is obtained in stochastic gradient descent (SGD), 
when each data-point is revisited multiple times 
\cite{mignacco2020dynamical,mignacco2022effective} (although SGD is not formally 
covered by our results).

Several papers have studied statistical learning (for problems as the ones described above) using 
`online SGD,' whereby each data point is visited only once. In this case
the dynamics remains Markovian after marginalizing over the data thus allowing for 
simpler (or more detailed) study, see e.g. \cite{arous2021online,abbe2023sgd}. 
\end{remark}   
\subsection{Exponentially attractive fixed points}

In this section we focus on  the case in which $\ell_t,\Lambda^t$ do not depend on time $t$,
so that we omit the subscripts and instead write $\ell,\Lambda$.
We study the case in which the solutions of the DMFT equations \eqref{eq:planted-se} converge 
exponentially fast as $t \rightarrow \infty$,
and provide a system of non-linear equations which characterize  its limit. 

\begin{definition}
\label{def:exp-conv}
We say that the DMFT system $\mathfrak{S}:=\mathfrak{S}(\theta^0, \theta^*, z, \delta, \Lambda, \ell)$ given in
Eq.~\eqref{eq:planted-se} and \eqref{eq:planted-derivs} \emph{converges exponentially} if 
there \rev{exist}  deterministic constants $C,c > 0$, matrices
$\Gamma,R_\ell^* \in \reals^{k \times k}$, 
and functions  $R_\ell,R_\theta : \reals_{\geq 0} \rightarrow \reals^{k\times k}$, 
random functions $\widehat R_\ell : \reals_{\geq 0} \rightarrow \reals^{k \times k}$,
random matrix $\widehat R_\ell^* \in \reals^{k \times k}$,
and random variables $\theta^\infty,r^\infty,u^\infty,w^\infty \in \reals^k$ such that 
\begin{equation}
	\begin{gathered}
		\| r^t - r^\infty \|_{L^2} \leq C e^{-ct},
		\qquad 
		\| \theta^t - \theta^\infty \|_{L^2} \leq Ce^{-ct},
		\qquad 
		\| u^t - u^\infty \|_{L^2} \leq Ce^{-ct},
		\qquad 
		\| w^t - w^\infty \|_{L^2} \leq Ce^{-ct},
		\\
		\| R_\ell(\cdot) - R_\ell(t,t-\cdot) \|_{\infty} \rightarrow 0,
		\qquad 
		\| R_\theta(\cdot) - R_\theta(t,t-\cdot) \|_\infty \rightarrow 0,
		\qquad 
		R_\ell(t,*) \rightarrow R_\ell^*,
		\qquad 
		\Gamma^t \rightarrow \Gamma^\infty,
		\\
		\frac{\partial \ell(r^{t+s},w^*;z)}{\partial w^t}
		\stackrel{\rev{L^2}}\rightarrow 
		\widehat R_\ell(t),
		\qquad 
		\frac{\partial \ell(r^t,w^*;z)}{\partial w^*}
		\stackrel{\rev{L^2}}\rightarrow 
		\widehat R_\ell^*,          
	\end{gathered}
\end{equation}
where all limits are taken with $s$ fixed and $t \rightarrow \infty$,
and further \rev{for all $t, s$} (where the bound on $\widehat R_\ell(s)$ is understood to hold almost surely),
\begin{align}
	R_\ell(s),R_\ell(t+s,t),R_\theta(s),R_\theta(t+s,t),\widehat R_\ell(s) \leq Ce^{-cs}\, .
\end{align}
\end{definition}
Notice the abuse of notation in the last definition. For instance we use the
same notation for $C_{\theta}(t,t+s)$ and its limit $C_{\theta}(s):=
\lim_{t\to\infty} C_{\theta}(t,t+s)$. This should not cause confusion in what follows.

The next theorem establishes that, if the DMFT solution converges exponentially, then the 
limit quantities satisfy a set of nonlinear equations.
\begin{theorem}
\label{thm:fixed-pt}
Assume that the DMFT system $\mathfrak{S}:=\mathfrak{S}(\theta^0, \theta^*, z, \delta, \Lambda, \ell)$ given in Eq.~\eqref{eq:planted-se} and \eqref{eq:planted-derivs} \emph{converges exponentially}.
Then there exist matrices $R_\ell^\infty,R_\theta^\infty,C_\ell^\infty \in \reals^{k \times k}$ and $C_\theta^\infty \in \reals^{2k \times 2k}$ 
such that 
\begin{equation}\label{eq:fixed-pt-1}
	\begin{aligned}
		0 &= -(\Lambda + R_\ell^\infty) \theta^\infty - R_\ell^* \theta^* + u^\infty,
		&
		r^\infty &= - \frac1\delta R_\theta^\infty \ell(r^\infty,w^*;z) + w^\infty,
		\\
		C_\theta & = \Ep [ (\theta^{\infty\sT},\theta^{*\sT})^\sT (\theta^{\infty\sT},\theta^{*\sT})]\, , 
		&
		C_\ell & =   \Ep[\ell(r^\infty,w^*;z) \ell(r^\infty,w^*;z)^\sT] \, ,
	\end{aligned}
\end{equation}
where $u^\infty \sim \normal(0, \Cl / \delta)$ and $(w^\infty,w^*) \sim \normal(0, \Ct)$.
Here $R_\ell^*,\theta^\infty,u^\infty,r^\infty,w^\infty$ are the same as those which appear in Definition \ref{def:exp-conv}.
Moreover, $R_\ell^\infty,R_\theta^\infty,R_\ell^*$ satisfy
\begin{equation}\label{eq:fixed-pt-2}
	\begin{aligned}
		R_\ell^\infty
		&=
		\E\Big[ \Big(I_k + \frac1\delta\nabla_r\ell(r^\infty,w^*;z) R_\theta^\infty\Big)^{-1}\nabla_r\ell(r^\infty,w^*;z) \Big],
		&
		(R_\theta^\infty)^{-1} &= \Lambda + R_\ell^\infty,
		\\
		R_\ell^*
		&= \E\Big[\Big( I_k + \frac1\delta \nabla_r \ell(r^\infty,w^*;z) R_\theta^\infty\Big)^{-1} \nabla_{w^*} \ell(r^\infty,w^*;z)\Big],
	\end{aligned}
\end{equation}
assuming all inverted matrices are invertible.
\end{theorem}
Let us emphasize that, so far, we assumed $\ell_t(r,w;z) = \ell(r,w;z)$
to be independent of time, but not necessarily that it is the gradient of a cost function.

In the next theorem we specialize this to gradient flow with respect to convex losses,
and build on \cite{asgari2025local} to establish connection to global optima. 
For $f:\reals^k\to\reals$ proper, convex and $A\in\reals^{k\times k}$ positive semidefinite, we denote by 
$\Prox_f(\,\cdot\,;S):\reals^k\to\reals^k$ the corresponding proximal operator, namely
\begin{align}
\Prox_f(z;S):= \argmin_{x\in\reals^k}\Big\{ \frac{1}{2}\<x-z,S^{-1}(x-z)\> + f(x)\Big\}\, .
\end{align}
\begin{theorem}[Convergence to global minimizers]\label{thm:Variational-General}
Assume $\Lambda = \lambda I_k$ and $\ell(r,w^*;z) =  \nabla_r\Ls(r,w^*;z)$ for a $C^2$ 
convex function $\Ls$. Further assume $Q_{00}=\E[\theta^*\theta^{*,\sT}]$
and $\Cov_z(\nabla \Ls(r,w;z))$ to be strictly positive for any $r,w$. 
Then the solutions of Eqs.~\eqref{eq:fixed-pt-1},
\eqref{eq:fixed-pt-2}
are in one-to\rev{-}one correspondence with the solutions of the following system of equations \rev{over all} pairs of positive semidefinite matrices 
$S\in\reals^{k\times k}$, $Q\in\reals^{2k\times 2k}$:
\begin{align}
	\E\big[\nabla \Ls(r^{\infty},w^*;z)\nabla \Ls(r^{\infty},w^*;z)^{\sT}\big] & = \frac{1}{\delta}
	S^{-1}(Q\setminus Q_{00})S^{-1}\, ,\label{eq:QS1}\\
	\E\big[\nabla \Ls(r^{\infty},w^*;z)(r^{\infty},w_*)^{\sT}\big] +\lambda[Q_{11},Q_{10}] & = 0\, ,
	\label{eq:QS2}\\
	Q= \left(\begin{matrix}
		Q_{00} & Q_{01}\\
		Q_{10} & Q_{11}\end{matrix}\right)&\, ,\;\;\;\; Q_{00}=\E[\theta^*\theta^{*,\sT}]\, ,
\end{align} 
where $Q_{ab}\in\reals^{k\times k}$ for $a,b\in\{0,1\}$, $Q\setminus Q_{00} :=Q_{11}-Q_{10}Q_{00}^{-1}
Q_{01}$, and $(w^*,w^{\infty})\sim\normal(0,Q)$ and
\begin{align}
	r^{\infty} = \Prox_{ \Ls(\,\cdot\, ,w^*;z)}(w^{\infty};S)\, .
\end{align}
Further, the solution to Eqs.~\eqref{eq:QS1}, \eqref{eq:QS2} is unique if either 
$\lambda>0$ or $\lambda=0$ and $\Ls(,\cdot\, ,w^*;z)$ is strictly convex for all $w^*,z$
(and at least one solution exists). 
In this case, the solution gives the asymptotics of the empirical risk minimizer 
(see Remark \ref{rmk:CVX_ERM} below\rev{).}
\end{theorem}
We prove Theorem \ref{thm:fixed-pt} and Theorem 
\ref{thm:Variational-General} in Appendix \ref{app:fixed-pt}.

\begin{remark}[Exponential convergence assumption]
The limit quantities $R_\ell^\infty,R_\theta^\infty,C_\ell^\infty \in \reals^{k \times k}$ and 
$C_\theta^\infty \in \reals^{2k \times 2k}$ can be defined as long as the convergence in 
Definition \ref{def:exp-conv} takes place at any rate (not necessarily at exponential rate\rev{).}

However, exponential convergence is used in the proof of Theorem \ref{thm:fixed-pt} 
to control the behavior of the DMFT equations \eqref{eq:DMFT}.
More precisely, these equations involve time integrals whose convergence is controlled by the 
assumed exponential decay.

It is also worth mentioning that Definition \ref{def:exp-conv} is non-empty.
For instance, gradient flow with strongly convex risk functions
converges exponentially fast to the global minimizer: 
this translates into exponential convergence
of the DMFT dynamics, as stated formally below. 
\end{remark}
\begin{proposition} \label{prop:convergence-DMFT}
	\rev{Under the assumptions of Corollary \ref{cor:planted-se}, assume that $\ell(r,w^*;z)  = \nabla_r\Ls(r,w^*;z)$, and $\Lambda^t=\Lambda\succeq 0$ 
	symmetric positive semidefinite matrix. Further assume that either 
	$(i)$~$r \mapsto \Ls(r,r^*;z)$ is strongly convex and $\delta>1$; or 
	$(ii)$~$r \mapsto \Ls(r,r^*;z)$ is convex and $\Lambda\succ 0$.
	Then the corresponding DMFT system converges exponentially.}
\end{proposition}
\rev{A proof is in Appendix \ref{proof:convergence-DMFT}.} More generally, determining under which conditions exponential
convergence holds is an open problem.

\begin{remark}[Convex empirical risk minimization]\label{rmk:CVX_ERM}
Uniqueness of the Eqs.~\eqref{eq:QS1}, \eqref{eq:QS2} under the stated conditions
was proven in \cite{asgari2025local}. The connection between these equations and 
Eqs.~\eqref{eq:fixed-pt-1},  \eqref{eq:fixed-pt-2} is given by 
\begin{align}
	C_{\theta} &= Q\, ,\;\;\; C_{\ell} = \frac{1}{\delta}  S^{-1}(Q\setminus Q_{00})S^{-1}\, ,\nonumber\\
	R_{\theta}^{\infty} &= \delta\, S\, ,\;\;\; R_{\ell}^{\infty} = (\delta \, S)^{-1}-\lambda I_k\, ,
	\label{eq:FixedPointMatch}\\
	R_{\ell}^*& = -\frac{1}{\alpha} S^{-1} Q_{10}Q_{00}^{-1}\, .\nonumber
\end{align}
We note that, by rotation invariance, we can assume without loss of generality $\theta^*\sim
\normal(0,Q_{00})$.
With these definitions we have $(\theta^*,\theta^{\infty})\sim \normal(0,Q)$.

It is also proven in  \cite{asgari2025local} that, whenever the solution
of the Eqs.~\eqref{eq:QS1}, \eqref{eq:QS2} is unique, it captures the 
asymptotics of the empirical risk minimizer:
\begin{align}
	\hbtheta := \argmin_{\btheta\in \reals^{d\times k}}\left\{\frac{d}{n}\sum_{i=1}^n
	\Ls(\btheta^{\sT}\bx_i,\btheta^{*,\sT}\bx_i,z_i) +\frac{\lambda}{2} \|\btheta\|_F^2\right\}\, .
	\label{eq:CVX_ERM}
\end{align}
In particular:
\begin{align*}
	\frac{1}{d}\sum_{i=1}^d\delta_{\hbtheta^{\sT}\be_i,\btheta^{*,\sT}\be_i}\Rightarrow
	{\rm Law}(\theta^{\infty},\theta^*)\, ,\;\;\; \;\;\;
	\frac{1}{n}\sum_{i=1}^d\delta_{\hbtheta^{\sT}\bx_i,\btheta^{*,\sT}\bx_i,z_i}\Rightarrow
	{\rm Law}(r^{\infty},w^*,z)\, .
\end{align*}
\end{remark}

There is a substantial literature on the asymptotics of empirical risk minimization
problems under the proportional \rev{regime} whereby $n,d\to\infty$ with $n/d\to \alpha$,
see e.g.
\cite{BayatiMontanariLASSO,karoui2013asymptotic,stojnic2013framework,
thrampoulidis2015,Donoho2016,thrampoulidis2018,Sur2019}, as well as the more general
\cite{asgari2025local}. The above remark shows that all of these can be recovered as exponentially
attractive fixed points in DMFT. 

\subsection{Convergence rate}

The DMFT equations contain significant information about the dynamics, beyond the 
existence of attractive fixed points. As an illustration, we derive the rate of
convergence to the global minimum for gradient flow with strongly convex risk
functions. 

We consider gradient flow with respect to the risk function \eqref{eq:CVX_ERM}
whereby either $\Ls(\,\cdot\, \rev{,} \bv^*,z)$ is convex and $\lambda>0$, or $\Ls(\,\cdot\,\rev{,}  \bv^*,z)$
is strongly convex and $\alpha>1$. In both cases, the regularized risk function
$\cL_n(\btheta)$ of  \eqref{eq:CVX_ERM} is (with high probability)
$c_0$-strongly convex for some $c_0>0$ which is $n,d$-independent.
This in \rev{turn} implies (again with high probability over the choice of the initialization 
and the data)
\begin{align}
\frac{1}{\sqrt{d}}\|\btheta^t-\hbtheta\|\le C\, e^{-c_0t}\, ,\label{eq:StrongCVX-conv}
\end{align}
for some $d$-independent constants $C,c_0>0$.

\rev{Even in this simple setting, classical theory only yields crude bounds on the 
exponential convergence rate, of the type 
$c_0 \ge \inf_{\btheta\in\reals^d}\lambda_{\min}(\nabla^2\cL(\btheta))$. 
These bounds are typically too crude to compare different algorithms or
different choices of loss functions. In contrast, DMFT allows to compute the precise constant.
Since a rigorous proof requires a substantially longer paper, we only sketch
the basic idea.}

Hence, by  Corollary~\ref{cor:planted-se}, we have that 
\begin{align}
\|\ell(r^t,w^*;z)\|_{L^2}\le C'\, e^{-c_0t},\;\;\; \|r^t-r^{\infty}\|_{L^2}\le C'\, e^{-c_0t}\, \rev{.}
\end{align}
We note that these conclusion could have been derived directly from the DMFT equations,
but it is simpler to use the indirect approach outlined above.

We therefore heuristically replace $\nabla_r \ell(r^t,w^*;z)$,  $\nabla_r \ell(r^s,w^*;z)$
in Eq.~\eqref{eq:def-derivative-planted-l} by $D:= \nabla_r \ell(r^{\infty},w^*;z)$.
We also replace $\Gamma_t =  \E\{\nabla_r \ell_t(r^t,w^*;z)\}$ by $\Gamma^{\infty}= 
\E\{D\}$ in Eq.~\eqref{eq:def-derivative-planted-t}, to get 
\begin{subequations}\label{eq:planted-deris-Simpl}
\begin{align}
	\frac{\de}{\de t} R_{\theta}(t,s)& = -(\lambda\, I + \Gamma^t) 
	R_{\theta}(t,s) - \int_s^t R_{\ell}(t,s')  R_{\theta}(s',s) \de s'\, , & & 0 \leq s \leq t < \infty\, , \label{eq:def-derivative-planted-t-b} \\
	\hR_{\ell}(t,s) & = D \cdot \prn{- \frac{1}{\delta} \int_s^t R_{\theta}(t,s') 
		\hR_{\ell}(s',s) \de s' - \frac 1 \delta R_\theta(t,s) D} \, , & & 0 \leq s < t < \infty\, , \label{eq:def-derivative-planted-l-b}\\
	R_{\ell}(t,s) & = \E\{\hR_{\ell}(t,s)\}\, \rev{,}
\end{align}
\end{subequations} 
where we emphasize that   $(t,s)\mapsto\hR_{\ell}(t,s)$ is a random function.
The above equations  imply  (with an abuse of notation)
$R_{\theta}(t,s) = R_{\theta}(t-s)$ 
for all $t>s$ and $\hR_{\ell}(t,s)=  \hR_{\ell}(t-s)$
\rev{given the boundary condition $R_\theta(t,t) = I$}. We define the Laplace transforms
\begin{align}
S_{\theta}(\zeta) := \int_0^{\infty}e^{\zeta t} R_{\theta}(t)\, \de t\, ,\;\;\;\;\;
\hS_{\ell}(\zeta)  :=\int_0^{\infty}e^{\zeta t} \hR_{\ell}(t)\, \de t\, \rev{,}
\end{align}
where $\zeta\in\complex$, $\Re(\zeta)\le -M$ with $M$ a sufficiently large constant 
as to make the integral convergent. In terms of these quantities, Eqs.~\eqref{eq:planted-deris-Simpl}
read
\begin{subequations}\label{eq:planted-deris-laplace}
\begin{align}
	-I-\zeta S_{\theta}(\zeta)& = -(\lambda\, I + \Gamma^\infty) 
	S_{\theta}(\zeta) - S_{\ell}(\zeta)  S_{\theta}(\zeta) \, , \label{eq:def-derivative-planted-t-laplace} \\
	\hS_{\ell}(\zeta) & = D \cdot \prn{- \frac{1}{\delta} D S_{\theta}(\zeta)\hS_{\ell}(\zeta) 
		- \frac 1 \delta S_\theta(\zeta) D} \, ,
	\;\;\;\;\; S_{\ell}(\zeta) = \E\{\hS_{\ell}(\zeta)\} \, . \label{eq:def-derivative-planted-l-laplace}
\end{align}
\end{subequations}
Solving the second equation for  $\hS_{\ell}(\zeta)$, 
substituting in the first one, 
and using $\Gamma^\infty=\E\{D\}$,
we obtain that $S_{\theta}(\zeta)$ should satisfy
\begin{align}\label{eq:Stieltjes}
S_{\theta} = \E\Big\{\Big(I+\frac{1}{\delta}  D S_{\theta} \Big)^{-1}D\Big\}-(z-\lambda)\, I\, .
\end{align}
Note that, by Theorem \ref{thm:UniqueExist} $S_{\theta}(\zeta)$ is analytic for $\Re(\zeta)<-M$,
with $M$ a large enough constant. In fact by a consequence of the exponential
convergence of Eq.~\eqref{eq:StrongCVX-conv},  $S_{\theta}(\zeta)$ is analytic for $\Re(\zeta)<-c_0$.

In the case $k=1$, Eq.~\eqref{eq:Stieltjes} coincides with the equation for the
Stieltjes transform of generalized Marchenko-Pastur law \cite[Theorem 4.3]{BaiSilverstein}.
For general $k$\rev{, the} uniqueness of the solution of Eq.~\eqref{eq:Stieltjes}
\rev{is} proven, for instance in \cite{asgari2025local}.  
The same proof also yields that $S_{\theta}$ is analytic everywhere
except on an interval $[\lambda_m, \lambda_M]\subset \reals$, 
$0<\lambda_m\le\lambda_M\le \infty$, which we assume to be minimal.

We then get $\|R_{\theta}(t)\|= \exp(-\lambda_m t +o(t))$ and 
$\|R_{\ell}(t)\|= \exp(-\lambda_m t +o(t))$. We expect the same exponential convergence to be
inherited by $\theta^t, r^t$, namely
\begin{align}
\|r^t-r^{\infty}\|_{L^2} = e^{-\lambda_m t+o(t)}\, ,
\;\;\;\;\; \|\theta^t-\theta^{\infty}\|_{L^2} = e^{-\lambda_m t+o(t)}\, .
\end{align}
As shown in \cite{dobriban2015efficient,asgari2025local}, the solution to \eqref{eq:Stieltjes}
can be efficiently computed numerically, and hence $\lambda_m$ can be estimated accurately.

\section{Proof of Theorem \ref{thm:UniqueExist}}
\label{sec:proof:UniqueExist}

\subsection{Proof outline}

The proof is divided into the following parts. 
\begin{enumerate}
\item[I.] Define the auxiliary real-valued functions $\pRt, \pRl, \pCt$ and $\pCl$.
These satisfy a set equations that are constructed as to provide bounds on 
$R_{\theta}, C_{\theta}, R_{\ell}, C_{\ell}$.
\item[II.] Construct a metric space $\mathcal{S}:= \mathcal{S}(\pRt, \pRl, \pCt, \pCl, T)$ 
for the function triplet $(C_\ell, R_\ell, \Gamma)$ when $t, s \in [0, T]$ and also a 
space $\wb{\mathcal{S}}$ for $(\Ct, \Rt)$. Show that in these spaces, the stochastic 
processes $\vt^t, \vr^t$ and functional derivatives $\partial \vt^t / \partial \vu^s, \partial \ell_t(\vr^t; \veps) / \partial w^s$ 
are uniquely defined.
\item[III.]  Define a transformation $\mathcal{T}: \mathcal{S} \to \mathcal{S}$ 
such that for any solution of the DMFT system $\mathfrak{S}$ in Eq.~\eqref{eq:DMFT}, 
$(C_\ell, R_\ell, \Gamma)$ must be a fixed-point of $\mathcal{T}$. We then show $\trsfrm$ is a contraction mapping,
and conclude by Banach fixed-point theorem.
\end{enumerate}
For any real-valued function $f(t)$ on $[0, T]$, $\lambda \geq 0$ and $T \in [0, \infty)$, we define the following norms
\begin{align}
\intnormlbd{f}{T}&  := \int_0^T e^{-\lambda t} |f(t)| \de t\, , \\
\normlbd{f}{T}& := \sup_{0 \leq t \leq T} e^{-\lambda t} |f(t)|\, .
\end{align}
If $f: \realsp \to \real^{k_1 \times k_2}$ is vector- or matrix-valued, we 
define $\intnormlbd{f}{T}$, $\normlbd{f}{T}$ analogously, with $|f(t)|$ replaced by $\|f(t)\|$
on the right-hand side. Finally, we let
$\intnormlbd{f}{\infty} = \lim_{T \to \infty} \intnormlbd{f}{T}$, 
$\normlbd{f}{\infty} = \lim_{T \to \infty} \normlbd{f}{T}$. 

\subsection{Part I: The auxiliary  functions} 

We will develop 
bounds in terms of solutions of the following ODEs.
\begin{lemma} \label{lem:Exist-General-ODE}
Consider the following system of ODEs for
nonnegative functions $\pRt, \pRl, \pCt$ and $\pCl$ on $[0, \infty)$.
\begin{subequations}
	\begin{align}
		\frac{\de}{\de t}\pRt(t) & = (\cstlbd + \cstloss)\pRt(t) + \int_0^t \pRl(t-s) \pRt(s) \de s \, , \label{eq:def-pRt}\\
		\pRl(t) & = \frac{\cstloss}{\delta} \cdot \brc{\cstloss \pRt(t) + \int_0^t \pRt(t-s) \pRl(s) \de s}\, , \label{eq:def-pRl} \\
		\frac{\de}{\de t} \sqrt{\pCt(t)} & = \sqrt{3 \cdot \brc{ (\cstlbd + \cstloss)^2 \pCt(t) + \frac{k}{\delta} \pCl(t) + \int_0^t (t-s+1)^2 \pRl(t-s)^2 \pCt(s) \de s}}\, , \label{eq:def-pCt}\\
		\pCl(t) & = 3 \cdot \brc{\cstthetaz + k\cstloss^2 \pCt(t) +  \frac{\cstloss^2}{\delta^2} \int_0^t (t-s+1)^2 \pRt(t-s)^2 \pCl(s) \de s}\, , \label{eq:def-pCl}
	\end{align}
\end{subequations}
For any given $\pRt(0) >0, \pCt(0) > 0$, 
these equations have a unique solution in the space of locally integrable functions on $\realsp$.
Further, there exists some $\lambda> 0$, dependent on $\cstlbd, \cstloss,\cstthetaz ,\delta,k,\pRt(0), \pCt(0)$  such that
\begin{align}
	\lim_{t \to \infty} e^{-\lambda t} \max \left\{\pRt(t) , \pRl(t), \pCt(t),
	\pCl(t) \right\} = 0\, .
\end{align}
\end{lemma}
\rev{We defer its proof to Appendix~\ref{proof:Exist-General-ODE}.}

\subsection{Part II: The function space $\mathcal{S}$, $\wb{\mathcal{S}}$} 
We next define a function space in which we solve  $\mathfrak{S}$.
\begin{definition}[The function triplet spaces $\mathcal{S}$ and $\mathcal{S}_{\mathrm{cont}}$] \label{def:space-S}
For $T>0$, denote by $X = (C_\ell, R_\ell, \Gamma)$ the function triplet
$C_\ell, R_\ell: [0,T]^2 \to \reals^{k \times k}$ and $\Gamma^t: [0,T] \to \reals^{k \times k}$.
Define the following space parametrized by a constant  $\cstspaceS$,
\begin{align}
	\mathcal{S}:= \mathcal{S}(\pRt, \pRl, \pCt, \pCl,\cstspaceS, T)
\end{align} 
of all $X$ such that
\begin{enumerate}
	\item $C_\ell(t,s)$ is a covariance kernel (satisfying in particular 
	$C_{\ell}(t,s)^{\sT} = C_{\ell}(s,t)$),
	such that 
	$\norm{C_\ell(t,t)} \leq \pCl(t)$
	for $t \in [0,T]$ and
	\begin{align}
		\Cl(0,0) = 	\mathbb{E} \brk{\ell_0(r^0;\veps)\vl_0(r^0;\veps)^\sT} \, , \quad r^0 \sim \normal\prn{0, \Ep\brk{\theta^0 {\theta^0}^\sT}} \, .
	\end{align} 
	Further $C_\ell(t,s)$ is continuous for $s \leq t$ and $s, t \in [0,T] \backslash P$ where $P$ is a finite set. 
	Moreover, for any $s \leq t$ such that $C_\ell$ is continuous in $[s,t]^2$,
	%
	\begin{align}
		\norm{C_\ell(t,t) - C_\ell(t,s)- C_\ell(s,t)+ C_\ell(s,s)} \leq \cstspaceS (t-s)^2 \, . \label{eq:space-condition-Cl}
	\end{align}
	\item $\Rl(t,s)$ is measurable and $\Rl(t,s) = 0$ when $t \leq s$. Further for any $s \leq t$ and $s,t \in [0,T]$
	\begin{align}
		\norm{\Rl(t,s)} \leq \pRl(t-s) \, .\label{eq:Rl-Bound}
	\end{align} 
	$\Gamma^t$ is measurable in $[0,T]$ such that
	\begin{align}
		\norm{\Gamma^t} \leq \cstloss\, ,\label{eq:Gamma-Bound-Ass}
	\end{align}
	and
	\begin{align}
		\Gamma_0 = \mathbb{E} \brk{\nabla_r \vl(r^0, \veps)}\, , \quad r^0 \sim \normal\prn{0, \Ep\brk{\theta^0 {\theta^0}^\sT}} \, .
	\end{align}
\end{enumerate}
Moreover, we define the space $\mathcal{S}_{\mathrm{cont}} \subset \mathcal{S}$ of all $X$ such that $P= \emptyset$ in the first condition and for all $s, s' \in [0,t]$,
\begin{align}
	\norm{\Cl(t,s) - \Cl(t,s')} & \leq \sqrt{\pCl(T)\cstspaceS } \cdot |s-s'|\, . \label{eq:lip-continuity-Cl}
\end{align}
\end{definition}
Next we consider the function pairs $(\Ct, \Rt)$ when $t, s \in[ 0,T]$ and $\Ct, \Rt: [0, T]^2 \to \reals^{k \times k}$. 
\begin{definition}[The function pair spaces $\wb{\mathcal{S}}$ and $\wb{\mathcal{S}}_{\mathrm{cont}}$] \label{def:space-S-dual}
Letting $Y = (\Ct, \Rt)$, we consider the following  space  (depending on a constant $\cstdualspaceS$)
\begin{align}
	\wb{\mathcal{S}}:= \wb{\mathcal{S}}(\pRt, \pRl, \pCt, \pCl, \cstdualspaceS, T)
\end{align} 
for all $Y$ such that
\begin{enumerate}
	\item $\Ct(t,s)$ is a covariance kernel  (satisfying in particular 
	$\Ct(t,s)^{\sT} = \Ct(s,t)$), such that $\norm{\Ct(t,t)} \leq \pCt(t)$
	for all $t \in [0, T]$ and
	\begin{align}
		\Ct(0,0) = 	\mathbb{E} \brk{\vt^0 {\vt^0}^{\sT}} \, .
	\end{align} 
	$\Ct(t,s)$ is continuous for all $s \leq t$ and $s \leq t$ and $s, t \in [0,T] \backslash P$
	where $P$ is a finite set. Moreover, for any $s \leq t$ such that $\bCt(t,s)$ is continuous in 
	$[s,t]^2$, 
	\begin{align}
		&\norm{\Ct(t,t) - \Ct(t,s)-\Ct(s,t) + \Ct(s,s)} \leq \cstdualspaceS (t-s)^2 \, .  \label{eq:space-condition-Ct}
	\end{align}
	\item $\Rt(t,s)$ is measurable and $\Rl(t,s) = 0$ when $t < s$. It  also satisfies for any $s \leq t$ and $s,t \in [0,T]$
	\begin{align}
		\norm{\Rt(t,s)} \leq \pRt(t-s) \, .\label{eq:RtBoundDef}
	\end{align} 
\end{enumerate}
Moreover, we define the space $\wb{\mathcal{S}}_{\mathrm{cont}} \subset \wb{\mathcal{S}}$ of all $X$ such that $P= \emptyset$ in the first condition and for all $s, s' \in [0,t]$,
\begin{align}
	\norm{\bCt(t,s) - \bCt(t,s')} & \leq  \sqrt{\pCt(T)\cstdualspaceS } \cdot |s-s'|\, . \label{eq:lip-continuity-Ct}
\end{align}
\end{definition}
The following lemma shows the stochastic processes $\theta^t, r^t$ and the 
functional derivatives $\partial \theta^t/\partial u^s, \partial \ell_t(r^t;z) / \partial w^s$ 
are well-defined whenever $(\Cl, \Rl, \Gamma) \in \mathcal{S}$ and $(\Ct, \Rt) \in \wb{\mathcal{S}}$.
Its proof can be found in Appendix~\ref{proof:well-defined-paths}.
\begin{lemma} \label{lem:well-defined-paths}
For any fixed $T > 0$, initialization $\theta^0$ and  $(\Cl, \Rl, \Gamma) \in \mathcal{S}$,
$(\Ct, \Rt) \in \wb{\mathcal{S}}$, the functions $\theta^t, r^t, \partial \theta^t/\partial u^s, \partial \ell_t(r^t;z) / \partial w^s$ are uniquely defined by Eqs.~\eqref{eq:def-theta}, \eqref{eq:def-r}, \eqref{eq:def-derivative-t} and \eqref{eq:def-derivative-l}.
\end{lemma}
We endow $\mathcal{S}$ and $\wb{\mathcal{S}}$ with distances. To begin with, we define the $(\lambda,T)$-distance for two Gaussian processes $\vu_1$ and $\vu_2$ on $\reals^k$ by the following formula
\begin{align}
\lbddst{\vu_1}{\vu_2} := \inf_{(\vu_1,\vu_2) \sim \gamma \in \Gamma(\vu_1,\vu_2)} \sup_{t \in [0, T]} e^{-\lambda t} \sqrt{\mathbb{E} \brk{\norm{\vu^t_1 - \vu^t_2}_2^2}} \, , 
\end{align}
where $\lambda > 0$ and $\Gamma(u_1, u_2)$ is the collection of all couplings between the two Gaussian processes $u_1, u_2$. \rev{For} any pair of positive semi-definite kernel functions $C^1, C^2 :[0,T]^2 \to \reals^{k \times k}$, we define their $(\lambda, T)$-distance by
\begin{align}
\lbddst{C^1}{C^2} := \lbddst{g_1}{g_2} \, ,
\end{align}
where $g_1$ and $g_2$ are two centered Gaussian processes with covariance kernels 
$C^1$ and $C^2$. This distance satisfies the triangle inequality by
Minkowski inequality.
For function pairs $\vRl^i, \vGamma_i, \vRt^i$, $i=1,2$, we overload the notation 
$\lbddst{\cdot }{\cdot }$ to define distances
\begin{subequations}
\begin{align}
	\lbddst{\vRl^1}{\vRl^2} &:= \sup_{0 \leq s<t \leq T} e^{-\lambda t} \norm{\vRl^1(t,s) - \vRl^2(t,s)} \, , \label{eq:def-distance-Rl}\\
	\lbddst{\vGamma_1}{\vGamma_2} & := \sup_{0 \leq t \leq T} e^{-\lambda t} \norm{\vGamma^t_1 - \vGamma^t_2} \, , \label{eq:def-distance-Gamma}
\end{align}
\end{subequations}
and similarly for $\lbddst{\vRt^1}{\vRt^2}$.
It can be seen the triangle inequality still holds.
Finally, for any $X_i = (\Cl^i, \Rl^i, \Gamma_i) \in \mathcal{S}$ and 
$Y_i = (\Ct^i, \Rt^i) \in \wb{\mathcal{S}}$ and $i=1,2$, we define the distances
\begin{subequations}
\begin{align}
	\lbddst{X_1}{X_2} & := \lbddst{\vCl^1}{\vCl^2} + \lbddst{\vRl^1}{\vRl^2} + \lbddst{\vGamma_1}{\vGamma_2} \, \label{eq:distance-spaceS} , \\
	\lbddst{Y_1}{Y_2} & := \lbddst{\vCt^1}{\vCt^2} + \lbddst{\vRt^1}{\vRt^2} \, . \label{eq:distance-dualspaceS}
\end{align}
\end{subequations}

\subsection{Part III: The contraction mapping $\mathcal{T}$} 
In this part, we will define the mapping $\mathcal{T}: \mathcal{S} \to \mathcal{S}$ 
such that any solution of $\mathfrak{S}$ must be a fixed-point of $\mathcal{T}$ \rev{and} show the mapping is a contraction. We will define $\trsfrm$ by
\begin{align}
\trsfrm (X) := \trsfrmB \circ \trsfrmA (X) \, ,
\end{align}
where $\trsfrmA: (\Cl, \Rl, \Gamma) \mapsto (\bCt, \bRt)$ and $\trsfrmB: (\bCt, \bRt) \mapsto \wb{X}: = (\bCl, \bRl, \wb{\Gamma})$, so that $\mathcal{T}: X = (\Cl, \Rl, \Gamma) \mapsto \wb{X} = (\bCl, \bRl, \wb{\Gamma})$. Specifically, the mapping $\trsfrmA$ is defined by first solving Eqs.~\eqref{eq:def-theta} and \eqref{eq:def-derivative-t}
\begin{align}
\frac{\de}{\de t} \theta^t & = -(\Lambda^t + \Gamma^t) \theta^t - \int_0^t R_{\ell}(t,s) \theta^s \de s + u^t \, , & & u^t \sim \mathsf{GP}(0, \Cl/\delta) \, , \label{eq:trsfrm-t}\\
\frac{\de}{\de t} \frac{\partial \theta^t}{\partial u^s} & = -(\Lambda^t + \Gamma^t) \frac{\partial \theta^t}{\partial u^s} - \int_s^t R_{\ell}(t,s') \frac{\partial \theta^{s'}}{\partial u^s} \de s'\, , & & 0 \leq s \leq t \leq T \, , \label{eq:trsfrm-derivative-t}
\end{align}
with boundary condition $\partial\theta^t / \partial u^t = I$. We know $\vt^t$ and $\partial \vt^t / \partial u^s$ are uniquely defined by Lemma~\ref{lem:well-defined-paths}. This allows us to define the mapping $\trsfrmA(\Cl, \Rl, \Gamma) = (\bCt, \bRt)$ via
\begin{align}
\bCt(t,s)  & = \Ep \brk{\theta^t {\theta^s}^{\sT}}\, , & & 0 \leq s \leq t \leq T \, , \label{eq:trsfrm-Ct}\\
\bRt(t,s)  & =\Ep \brk{\frac{\partial \theta^t}{\partial u^s}} \, , & & 0 \leq s \leq t \leq T \, ,  \label{eq:trsfrm-Rt}
\end{align}
with the convention $\bCt(t,s) = \bRt(t,s) = 0$ for all $t < s$. Similarly, we define $\trsfrmB (\bCt, \bRt) = (\bCl, \bRl, \wb{\Gamma})$ through Eqs.~\eqref{eq:def-r} and \eqref{eq:def-derivative-l}, namely
\begin{align}
r^t & = - \frac{1}{\delta} \int_0^t \bRt(t,s) \ell_s(r^s; z) \de s + w^t \, , & & w^t \sim \mathsf{GP}(0, \bCt) \, , \label{eq:trsfrm-r} \\
\frac{\partial \ell_t(r^t;z)}{\partial w^s} & = \nabla_r \ell_t(r^t;z) \cdot \prn{- \frac{1}{\delta} \int_s^t \bRt(t,s') \frac{\partial \ell_{s'}(r^{s'};z)}{\partial w^s}  \de s' - \frac 1 \delta \bRt(t,s) \nabla_r \ell_s(r^s;z) } \, , & & 0 \leq s < t \leq T \, . \label{eq:trsfrm-derivative-l}
\end{align}
The random functions $r^t, \partial \vl_t(r^t;z) / \partial w^s$ are uniquely defined by Lemma~\ref{lem:well-defined-paths}. The mapping then determined by setting
\begin{align}
\bCl(t,s) & =   \Ep \brk{\ell_t(r^t;z) \ell_s(r^s;z)^{\sT}} \, ,  & & 0 \leq s \leq t < T \, , \label{eq:trsfrm-Cl} \\
\bRl(t,s) & =  \Ep \brk{\frac{\partial \ell_t(r^t;z)}{\partial w^s}} \, , & &0 \leq s < t < T \, , \label{eq:trsfrm-Rl} \\
\bGamma^t & = \Ep \brk{\nabla_r \ell_t(r^t;z)}\, , & & 0 \leq t \leq T \, , \label{eq:trsfrm-Gamma}
\end{align}
and on regions in $\realsp \times \realsp$ that are not defined for $\bCl, \bRl$, we set their values to be zero. The next lemma shows we can choose $\pRt(0)$ and $\pCt(0)$ large enough such that $\trsfrmA: \mathcal{S} \to \wb{\mathcal{S}}$ and $\trsfrmB: \wb{\mathcal{S}} \to \mathcal{S}$. Its proof can be found in Appendix~\ref{proof:solution-in-space}.
\begin{lemma} \label{lem:solution-in-space}
Under the same assumptions of Theorem~\ref{thm:UniqueExist}, suppose $\pCt(0) >\cstthetaz$ and 
$\pRt(0) > 1 $. Then there exist constants 
$\cstspaceS$,  $\cstdualspaceS$ such that
$\trsfrmA$ maps $\mathcal{S}$ into $\wb{\mathcal{S}}_{\mathrm{cont}} \subset 
\wb{\mathcal{S}}$ and $\trsfrmB$ maps $\wb{\mathcal{S}}$ into $\mathcal{S}_{\mathrm{cont}} \subset \mathcal{S}$. In particular, this implies $\trsfrm = \trsfrmB \circ \trsfrmA$ maps $\mathcal{S}$ into $\mathcal{S}_{\mathrm{cont}} \subset \mathcal{S}$.
\end{lemma}
In what follows it will be understood that constants 
$\cstspaceS$,  $\cstdualspaceS$ are chosen as in the proof of Lemma~\ref{lem:solution-in-space}.

Next, we want to show $\mathcal{T}$ is a contraction mapping under the
$\lbddst{\cdot}{\cdot}$ metric defined in Eq.~\eqref{eq:distance-spaceS}. 
To this end, we need the following lemmas for the transformation $\trsfrmA$.

\begin{lemma} \label{lem:diff-Cl-global}
Suppose $X_1 = (\Cl^1, \Rl^1, \Gamma_1), X_2 = (\Cl^2, \Rl^2, \Gamma_2) \in \mathcal{S}$, and further $\vRl^1 = \vRl^2$ on $[0,T]^2$ and $\Gamma_1 = \Gamma_2$ on $[0,T]$. Let $\trsfrmA(X_i) = (\bCt^i, \bRt^i)$ for $i=1,2$, then we have for any $\epsilon > 0$,
\begin{subequations}
	\begin{align}
		\lbddst{\bCt^1}{\bCt^2} & \leq \epsilon \cdot \lbddst{\vCl^1}{\vCl^2} \, , \\
		\bRt^1 & = \bRt^2 \, ,
	\end{align}
\end{subequations}
for all $\lambda \geq \wb{\lambda}_1 := \wb{\lambda}_1(\eps, \mathcal{S})$.
\end{lemma}
We defer its proof to Appendix~\ref{proof:diff-Cl-global}.
\begin{lemma} \label{lem:diff-Rl-global}
Suppose $X_1 = (\Cl^1, \Rl^1, \Gamma_1), X_2 = (\Cl^2, \Rl^2, \Gamma_2) \in \mathcal{S}$ and $\vCl^1 = \vCl^2$ on $[0, T]^2$. Let $\trsfrmA(X_i) = (\bCt^i, \bRt^i)$ for $i=1,2$, then we have for any $\epsilon > 0$,
\begin{subequations}
	\begin{align}
		\lbddst{\bCt^1}{\bCt^2} &\leq  \epsilon \cdot \prn{\lbddst{\vRl^1}{\vRl^2} + \lbddst{\vGamma_1}{\vGamma_2}}  \, , \\
		\lbddst{\bRt^1}{\bRt^2} & \leq  \epsilon \cdot \prn{ \lbddst{\vRl^1}{\vRl^2} + \lbddst{\vGamma_1}{\vGamma_2}} \, .
	\end{align}
\end{subequations}
for all $\lambda \geq \wb{\lambda}_2 := \wb{\lambda}_2(\eps, \mathcal{S})$.
\end{lemma}
We defer the proof to Appendix~\ref{proof:diff-Rl-global}. We next derive the lemmas for the transformation $\trsfrmB$.
\begin{lemma} \label{lem:diff-Ctheta-global}
Suppose $Y_1 = (\bCt^1, \bRt^1), Y_2 = (\bCt^2, \bRt^2) \in \wb{\mathcal{S}}$ and $\bvRt^1 = \bvRt^2$ on $[0,T]^2$. Let $\trsfrmB(Y_i) = (\bCl^i, \bRl^i, \bGamma_i)$ for $i=1,2$, then there exists a constant $M := M(\mathcal{S})$ such that
\begin{align}
	\lbddst{\bvCl^1}{\bvCl^2} \leq &  M \cdot \lbddst{\bvCt^1}{\bvCt^2} \, , \\
	\lbddst{\bvSl^1}{\bvSl^2} \leq &   M \cdot \lbddst{\bvCt^1}{\bvCt^2} \, , \\
	\lbddst{\bvGamma^1}{\bvGamma^2} \leq &  M \cdot \lbddst{\bvCt^1}{\bvCt^2} \, , 
\end{align}
for all $\lambda \geq \wb{\lambda}_3 := \wb{\lambda}_3(\mathcal{S})$.
\end{lemma}
We defer the proof to Appendix~\ref{proof:diff-Ctheta-global}.

\begin{lemma} \label{lem:diff-Rtheta-global}
Suppose $Y_1 = (\bCt^1, \bRt^1), Y_2 = (\bCt^2, \bRt^2) \in \wb{\mathcal{S}}$ and $\bvCt^1 = \bvCt^2$ on $[0, T]^2$. Let $\trsfrmB(Y_i) = (\bCl^i, \bRl^i, \bGamma_i)$ for $i=1,2$, then there exists a constant $M := M(\mathcal{S})$ such that
\begin{align}
	\lbddst{\bvCl^1}{\bvCl^2} \leq &  M \cdot \lbddst{\bvRt^1}{\bvRt^2} \, , \\
	\lbddst{\bvSl^1}{\bvSl^2} \leq &   M \cdot \lbddst{\bvRt^1}{\bvRt^2} \, , \\
	\lbddst{\bvGamma^1}{\bvGamma^2} \leq &  M \cdot \lbddst{\bvRt^1}{\bvRt^2} \, , 
\end{align}
for all $\lambda \geq \wb{\lambda}_4 := \wb{\lambda}_4(\mathcal{S})$.
\end{lemma}
We defer the proof to Appendix~\ref{proof:diff-Rtheta-global}. Now, we are ready to show that $\mathcal{T}$ is a contraction. We take the constant $M$ to be the maximum one among that of Lemma~\ref{lem:diff-Ctheta-global} and \ref{lem:diff-Rtheta-global}. Then we take $\epsilon = (12M)^{-1}$, and any $\lambda \geq \max \left\{\wb{\lambda}_1, \wb{\lambda}_2, \wb{\lambda}_3, \wb{\lambda}_4\right\}$, where $\wb{\lambda}_i$ are defined in above lemmas. For any $X_1 = (\vCl^1, \vRl^1, \Gamma_1), X_2 = (\vCl^2, \vRl^2, \Gamma_2) \in \mathcal{S}$, we set $Y_i = \trsfrmA(X_i) = (\bCt^i, \bRt^i)$ for $i=1,2$. We also define $Y_3 = (\bCt^1, \bRt^2)$, and thus
\begin{align}
\lbddst{\trsfrm(X_1)}{\trsfrm(X_2)} & = \lbddst{\trsfrmB(Y_1)}{\trsfrmB(Y_2)} \nonumber \\
& \leq \lbddst{\trsfrmB(Y_1)}{\trsfrmB(Y_3)} + \lbddst{\trsfrmB(Y_2)}{\trsfrmB(Y_3)} \, .
\end{align}
We can then control $\lbddst{\trsfrmB(Y_1)}{\trsfrmB(Y_3)}$ by Lemma~\ref{lem:diff-Rtheta-global} and $\lbddst{\trsfrmB(Y_2)}{\trsfrmB(Y_3)}$ by Lemma~\ref{lem:diff-Ctheta-global} which further gives
\begin{align}
\lbddst{\trsfrm(X_1)}{\trsfrm(X_2)}
& \le 3M \lbddst{Y_1}{Y_2} \nonumber \\
& =  3M \lbddst{\trsfrmA(X_1)}{\trsfrmA(X_2)} \, . \label{eq:ineq-contraction-mapping-1}
\end{align}
Then, we take $X_3 = (\vCl^1, \vRl^2, \Gamma_2)$ and apply Lemma~\ref{lem:diff-Cl-global} and \ref{lem:diff-Rl-global},
\begin{align}
\lbddst{\trsfrmA(X_1)}{\trsfrmA(X_2)} & \leq \lbddst{\trsfrmA(X_1)}{\trsfrmA(X_3)} + \lbddst{\trsfrmA(X_2)}{\trsfrmA(X_3)} \nonumber \\
& \leq 2 \epsilon  \cdot  \lbddst{X_1}{X_2} \, .
\end{align}
\rev{Substituting} into Eq.~\eqref{eq:ineq-contraction-mapping-1}, and choosing $\epsilon = (12M)^{-1}$, we get
\begin{align}
\lbddst{\trsfrm(X_1)}{\trsfrm(X_2)} & \le \frac 1 2\lbddst{X_1}{X_2} \, . \label{eq:T-contraction}
\end{align}

For any finite $T$, the distance $\lbddst{}{}$ is equivalent to 
an $L^\infty$ distance hence complete,
implying existence and uniqueness 
the fixed point $\Rl, \Rt$ and $\Gamma$ in $\mathcal{S}$ and $\wb{\mathcal{S}}$.

To prove uniqueness of $\Cl, \Ct$
we note that  $\lbddst{C_1}{C_2} =0$ implies that there exist Gaussian processes
$(g^t_1)_{t\in [0,T]}$, $(g^t_2)_{t\in [0,T]}$ such that $\E[\|\rev{g_1^t-}g_2^t\|^2] =0$ for all $t\le T$,
and therefore $C_1(t,s) = C_2(t,s)$ for all $t,s\le T$. 
To see the uniqueness when $\Ct, \Rt$ are bounded functions in any compact set, we can simply take $\pCt(0) \to \infty, \pRt(0) \to \infty$.

\section{Proof of Theorem \ref{thm:StateEvolution}}
\label{sec:proof:StateEvolution}

\subsection{Proof outline}

We first present a proof roadmap.
\begin{enumerate}
\item[I.] Discretization. For the general flow system $\mathfrak{F}$ in Eq.~\eqref{eq:GeneralFlow} we 
construct a Euler's discretization with step size $\eta > 0$, 
$\mathfrak{F}^\eta$. We show that $\mathfrak{F}^\eta$ approximates $\mathfrak{F}$ 
uniformly with respect to  $n,d\to\infty$, $n/d \to \delta$.
\item[II.] We introduce a discrete time approximation $\mathfrak{S}^\eta$ for the 
DMFT  system $\mathfrak{S}$. We prove that $\mathfrak{S}^\eta$  characterizes
the asymptotics of $\mathfrak{F}^\eta$ when $n/d \to \delta$, by showing
that the discretized system is equivalent to an AMP algorithm plus post-processing.
\item[III.] We prove that the unique solution of $\mathfrak{S}^\eta$  
converges to the unique solution of $\mathfrak{S}$ as $\eta\to 0$.
The latter therefore characterizes the general flow system $\mathfrak{F}$.
\end{enumerate}
Throughout this section, we denote by $d_{\sW}$ any distance that metrizes 
weak convergence of probability distributions in $\reals^m$ and with an abuse of notation,
weak convergence in $C([0,T],\reals^k)$. For instance, we can take 
$d_{\sW}=d_{\sBL}$ the bounded Lipschitz distance
\begin{align*}
d_{\sBL}(\mu,\nu):= \sup\big\{\int f \de\mu -\int f \de\nu:\; \|f\|_{\infty}\le 1,\;  
\|f\|_{\sLip}\le 1 \big\}\, .
\end{align*}
We further denote by $\wdst{\mu}{\nu}$ the Wasserstein-$2$ distance between $\mu$ and $\nu$.
Also, we will focus on proving Eq.~\eqref{eq:LimTheta}, since \eqref{eq:LimR} follows 
by repeating the same argument.

\subsection{Part I: Discrete time approximation of the flow} 
For the general flow $\mathfrak{F}$ 
of Eq.~\eqref{eq:GeneralFlow},
%
we consider a discrete time approximation with step size $\eta >0$. For all $t_i = i \eta$ and $i \in \integersp$, we set $\btheta^0_\eta = \btheta^0$ and
\begin{align}
\btheta^{t_{i+1}}_\eta = \btheta^{t_i}_\eta + \eta \cdot \brc{ -\btheta^{t_i}_\eta \Lambda^{t_i, \sT} -\frac{1}{\delta} \bX^{\sT}\bell_{t_i}(\bX\btheta^{t_i}_\eta;\bz)}\, .
\end{align}
This defines $\btheta_\eta^t$ on all $t_i = i \eta$. 
We extend it to $t\to \realsp$ as a piecewise linear function. Specifically, we define
$\flr{t} := \max \brc{i\eta \mid i\eta \leq t, i \in \integersp }$,
and the flow $\mathfrak{F}^\eta$ is given by
\begin{align}
\ddt \btheta^t_\eta = -\btheta^{\flr{t}}_\eta \Lambda^{\flr{t}, \sT} -\frac{1}{\delta} \bX^{\sT}\bell_{\flr{t}}(\bX\btheta^{\flr{t}}_\eta;\bz) \, . \label{eq:def-system-F-eta}
\end{align}
Consider the empirical distributions of the rows of $\btheta^{\tau_1}, \cdots, \btheta^{\tau_m}$ and $\btheta^{\tau_1}_\eta, \cdots, \btheta^{\tau_m}_\eta$ for any $\tau_1, \cdots, \tau_m \in \realsp$, denoted by
\begin{align}
\est{\mu}_{\theta^{\tau_1}, \cdots, \theta^{\tau_m}} & := \frac{1}{d} \sum_{j=1}^d \delta_{\prn{\btheta^{\tau_1}_j, \cdots, \btheta^{\tau_m}_j}} \label{eq:dstrb-empirical-flow} \, , \\
\est{\mu}_{\theta^{\tau_1}_\eta, \cdots, \theta^{\tau_m}_\eta} & := \frac{1}{d} \sum_{j=1}^d \delta_{\prn{\prn{\btheta^{\tau_1}_\eta}_j, \cdots, \prn{\btheta^{\tau_m}_\eta}_j}} \, , \label{eq:dstrb-empirical-discrete-flow}
\end{align}
where $\est{\mu}_{\theta^{\tau_1}, \cdots, \theta^{\tau_m}}$ and $\est{\mu}_{\theta^{\tau_1}_\eta, \cdots, \theta^{\tau_m}_\eta}$ are probability distributions in $\reals^{km}$. The following lemma controls the
distance between the two distributions uniformly with respect to
$n,d\to\infty$. We defer its proof to Appendix~\ref{proof:flow-approximation}.
\begin{lemma} \label{lem:flow-approximation}
Under the same assumptions of Theorem~\ref{thm:StateEvolution}, consider the gradient flow 
$\btheta^t$ and its piecewise linear approximation $\btheta^t_\eta$ by forward Euler 
with step size $\eta > 0$ and the same initialization $\btheta^0$.  Then, almost surely,
for any $t\ge 0$:
\begin{align}
	\lim_{\eta \to 0} \limsup_{n \to \infty}\frac{1}{\sqrt{d}}	\sup_{0 \leq t \leq T} \norm{\btheta^t - \btheta^t_\eta} = 0\, .
	\label{eq:FlowApproxFirstClaim}
\end{align}
As a consequence, for any $\tau_1, \cdots, \tau_m \in [0,T]$, we have almost surely that
\begin{align}
	\lim_{\eta \to 0} \limsup_{n \to \infty} \wdst{\est{\mu}_{\theta^{\tau_1}, \cdots, \theta^{\tau_m}}}{\est{\mu}_{\theta^{\tau_1}_\eta, \cdots, \theta^{\tau_m}_\eta}} = 0 \, .
\end{align}
\end{lemma}

\subsection{Part II: Characterizing the discrete flow $\mathfrak{F}^\eta$}

In \cite{celentano2020estimation}, the authors show that
a general first order method of the type $\mathfrak{F}^\eta$ can be
reduced to an AMP algorithm followed by a post-processing operation that
operates row-wise  on $\btheta^t_{\eta}$ and $\br^t_{\eta}$ (and across multiple times).
This allows us to leverage existing high-dimensional characterizations of AMP
that go under the name of `state evolution' \cite{bayati2011dynamics,bayati2015universality,javanmard2013state,chen2021universality}.

We will use the notation $\cil{t} := \flr{t} + \eta$. 
We introduce the following DMFT system 
$\mathfrak{S}^\eta:=\mathfrak{S}^\eta(\theta^0, \theta^*, z, \delta, \lambda, \ell)$:
\begin{subequations}
\begin{align}
	\frac{\de}{\de t} \theta^t_\eta & = -(\Lambda^{\flr{t}} + \Gamma^{\flr{t}}_\eta) \theta^{\flr{t}}_\eta - \int_0^{\flr{t}} R_{\ell}^\eta(\flr{t},\flr{s}) \theta^{\flr{s}}_\eta \de s + u^t_\eta \, , & & u^t_\eta \sim \mathsf{GP}(0, \Cl^\eta / \delta) \, , \label{eq:def-theta-eta} \\
	r^t_\eta & = - \frac{1}{\delta} \int_0^{\flr{t}} R_{\theta}^\eta(\flr{t}, \cil{s}) \ell_{\flr{s}}(r^{s}_\eta; z) \de s + w^t_\eta \, , & & w^t_\eta \sim \mathsf{GP}(0, \Ct^\eta) \, , \label{eq:def-r-eta}  \\ 
	R_\theta^\eta(t,s) & = \Ep \brk{\frac{\partial \theta^t_\eta}{\partial u^s_\eta}} \, , & & 0 \leq s \leq t < \infty\, , \label{eq:def-R-t-eta}  \\
	R_{\ell}^\eta(t,s) & = \Ep \brk{\frac{\partial \ell_{\flr{t}}(r^t_\eta;z)}{\partial w^s_\eta}} \, , & &0 \leq s < t <\infty\, , \label{eq:def-R-l-eta} \\
	\Gamma^t_\eta &= \Ep \brk{\nabla_r \ell_{\flr{t}}(r^t_\eta;z)}\, , \label{eq:def-Gamma-eta}  \\
	C_\theta^\eta(t,s) & = \Ep \brk{\theta^{\flr{t}}_\eta {\theta^{\flr{s}}_\eta}^{\sT}}\, , & & 0 \leq s \leq t < \infty\, ,  \label{eq:def-C-t-eta} \\
	C_\ell^\eta(t,s) & =   \Ep \brk{\ell_{\flr{t}}(r^{\flr{t}}_\eta;z) \ell_{\flr{s}}(r^{\flr{s}}_\eta;z)^{\sT}} \, ,  & & 0 \leq s \leq t < \infty\, , \label{eq:def-C-l-eta} 
\end{align}
\end{subequations}
where the functional derivatives are determined by
\begin{subequations}
\begin{align}
	\frac{\de}{\de t} \frac{\partial \theta^t_\eta}{\partial u^s_\eta} & = -(\Lambda^{\flr{t}} + \Gamma^{\flr{t}}_\eta) \frac{\partial \theta^{\flr{t}}_\eta}{\partial u^s_\eta} - \int_s^{\flr{t}} R_{\ell}^\eta(\flr{t},\flr{s'}) \frac{\partial \theta^{\flr{s'}}_\eta}{\partial u^{s}_\eta} \de s'\, , \label{eq:def-derivative-t-eta} \\
	\frac{\partial \ell_{\flr{t}}(r^t_\eta;z)}{\partial w^s_\eta} & = \nabla_r \ell_{\flr{t}}(r^t_\eta;z) \cdot \prn{- \frac{1}{\delta} \int_{\cil{s}}^{\flr{t}} R_{\theta}^\eta(\flr{t},\cil{s'}) \frac{\partial \ell_{\flr{s'}}(r^{s'}_\eta;z)}{\partial w^s_\eta}  \de s' - \frac 1 \delta R_\theta^\eta(\flr{t},\cil{s}) \nabla_r \ell_{\flr{s}}(r^s_\eta;z) } \, ,  \label{eq:def-derivative-l-eta}
\end{align}
\end{subequations} 
where Eq.~\eqref{eq:def-derivative-l-eta} is defined for $\cil{s} \leq \flr{t}$. For $\cil{s} > \flr{t}$ we set
\begin{align}
\frac{\partial \ell_{\flr{t}}(r^t_\eta;z)}{\partial w^s_\eta} & = -\frac 1 \delta \nabla_r \ell_{\flr{t}}(r^t_\eta;z)\nabla_r \ell_{\flr{t}}(r^s_\eta;z) \, . \label{eq:def-derivative-l-eta-2}
\end{align}
The boundary conditions are $\theta^0_\eta := \theta^0$ and $\partial \theta^s_\eta / \partial u^s_\eta = I$. 
The system $\mathfrak{S}^\eta$ can be viewed as a discrete approximation of $\mathfrak{S}$:
its solution is unique by induction over time.
The next lemma shows that the unique solution of
$\mathfrak{S}^\eta$ characterizes the asymptotic behavior of $\mathfrak{F}^\eta$
(see Appendix~\ref{proof:state-evolution-discrete-flow} for a proof). 
\begin{lemma} \label{lem:state-evolution-discrete-flow}
Under the assumptions of Theorem~\ref{thm:StateEvolution}, suppose $\pCt(0) >\cstthetaz$ and $\pRt(0) > 1 $, then the system $\mathfrak{S}^\eta$ has a unique solution. In particular, the function triplet $(\Cl^\eta, \Rl^\eta, \Gamma_\eta )$ that solves $\mathfrak{S}^\eta$ lies in the space $\mathcal{S}$ (cf.~Definition~\ref{def:space-S}). For any $\tau_1, \cdots, \tau_m \in [0,T]$, denote by $\mu_{\theta^{\tau_1}_\eta, \cdots, \theta^{\tau_m}_\eta}$ the joint distribution of $(\vt^{\tau_1}_\eta, \cdots, \vt^{\tau_m}_\eta)$\rev{. We} have
\begin{align}
	\plim_{n \to \infty} \dW{\widehat{\mu}_{\theta^{\tau_1}_\eta, \cdots, \theta^{\tau_m}_\eta}}{\mu_{\theta^{\tau_1}_\eta, \cdots, \theta^{\tau_m}_\eta}} = 0\, , \label{eq:LP-conv}
\end{align} 
where $\widehat{\mu}_{\theta^{\tau_1}_\eta, \cdots, \theta^{\tau_m}_\eta}$ is the empirical distribution of the discretized flow $\btheta^t_\eta$, defined in Eq.~\eqref{eq:dstrb-empirical-flow}.
\end{lemma}
\begin{remark}
In case we are interested in discrete-time flows, e.g. gradient descent with stepsize $\eta$,
Lemma \ref{lem:state-evolution-discrete-flow} provides the relevant characterization.
\end{remark}

\begin{remark}
In the proof of Lemma \ref{lem:state-evolution-discrete-flow} we use the results of 
\cite{chen2021universality} which establishes universality over the class of matrices satisfying the assumptions of
Theorem~\ref{thm:StateEvolution}.

If $\bX$ is a Gaussian matrix, we can use the results of \cite{javanmard2013state}
which imply convergence in Wasserstein-$2$ distance. Hence in this case, Theorem 
\ref{thm:StateEvolution} will hold with convergence of finite-dimensional distributions
in Wasserstein-$2$ distance, as stated
in Remark \ref{remark:MainThm}.
\end{remark}

\subsection{Part III: Approximating $\mathfrak{S}$ by $\mathfrak{S}^\eta$} 

We approximate the 
unique solution of the DMFT system 
$\mathfrak{S}$ by the unique solution of the discretized system $\mathfrak{S}^\eta$. 
In particular, we have the following lemma, whose proof is postponed to 
Appendix~\ref{proof:integral-differential-approximation}.
\begin{lemma} \label{lem:integral-differential-approximation}
Under the assumptions of Theorem~\ref{thm:StateEvolution}, suppose $\pCt(0) >\cstthetaz$ and $\pRt(0) > 1 $, the systems $\mathfrak{S}$ and $\mathfrak{S}^\eta$ both have unique solutions in $\mathcal{S}$. For any $\tau_1, \cdots, \tau_m \in [0,T]$, denote by $\mu_{\theta^{\tau_1}, \cdots, \theta^{\tau_m}}$ the distribution of $(\theta^{\tau_1}, \cdots, \theta^{\tau_m})$ and $\mu_{\theta^{\tau_1}_\eta, \cdots, \theta^{\tau_m}_\eta}$ the distribution of $(\vt^{\tau_1}_\eta, \cdots, \vt^{\tau_m}_\eta)$\rev{. We} have
\begin{align}
	\lim_{\eta \to \infty} \wdst{\mu_{\theta^{\tau_1}_\eta, \cdots, \theta^{\tau_m}_\eta}}{\mu_{\theta^{\tau_1}, \cdots, \theta^{\tau_m}}} = 0\, .
\end{align}
	\end{lemma}
	Using Lemma~\ref{lem:state-evolution-discrete-flow},
	\begin{align}
& \plimsup_{n \to \infty}\dW{\est{\mu}_{\theta^{\tau_1}, \cdots, \theta^{\tau_m}}}{\mu_{\theta^{\tau_1}, \cdots, \theta^{\tau_m}}} \nonumber \\
& \leq \plimsup_{n \to \infty}\prn{\dW{\est{\mu}_{\theta^{\tau_1}, \cdots, \theta^{\tau_m}}}{\est{\mu}_{\theta^{\tau_1}_\eta, \cdots, \theta^{\tau_m}_\eta}} + \dW{\widehat{\mu}_{\theta^{\tau_1}_\eta, \cdots, \theta^{\tau_m}_\eta}}{\mu_{\theta^{\tau_1}_\eta, \cdots, \theta^{\tau_m}_\eta}}  + \dW{\mu_{\theta^{\tau_1}_\eta, \cdots, \theta^{\tau_m}_\eta}}{\mu_{\theta^{\tau_1}, \cdots, \theta^{\tau_m}}}} \nonumber \\
& =   \plimsup_{n \to \infty}\dW{\est{\mu}_{\theta^{\tau_1}, \cdots, \theta^{\tau_m}}}{\est{\mu}_{\theta^{\tau_1}_\eta, \cdots, \theta^{\tau_m}_\eta}} +\dW{\mu_{\theta^{\tau_1}_\eta, \cdots, \theta^{\tau_m}_\eta}}{\mu_{\theta^{\tau_1}, \cdots, \theta^{\tau_m}}} \, .
\end{align}
Finally, \rev{taking} $\eta \to 0$ and \rev{combining} Lemma~\ref{lem:flow-approximation} and Lemma~\ref{lem:integral-differential-approximation}, we obtain
\begin{align}
& \plimsup_{n \to \infty}\dW{\est{\mu}_{\theta^{\tau_1}, \cdots, \theta^{\tau_m}}}{\mu_{\theta^{\tau_1}, \cdots, \theta^{\tau_m}}}  = 0.
\end{align}
Now, let $\mu$ be the \rev{probability} law of the DMFT process $\rev{\theta^{[0,T]}}$
on $C([0,T],\reals^k)$. Indeed, by condition \eqref{eq:space-condition-Ct} and Kolmogorov-Chentsov
theorem, $t \mapsto \theta^t$ is almost surely $\alpha$-H\"older for any $\alpha\in (0,1)$.
Also, let $\est{\mu}^{(n)} := d^{-1}\sum_{i=1}^d \delta_{\rev{\theta_i^{[0,T]}}}$.
Denoting by $\mu_{\tau_1,\dots,\tau_m}$ and $\est{\mu}^{(n)}_{\tau_1,\dots,\tau_n}$ the finite-dimensional marginals of these laws,
we proved that 
\begin{align}
\plim_{n\to\infty}
\dW{\est{\mu}^{(n)}_{\tau_1,\dots,\tau_m}}{\mu_{\tau_1,\dots,\tau_m}}
= 0\, .
\end{align}
We are left with the task of proving $\dW{\est{\mu}^{(n)}}{\mu}\to 0$ in probability.
Recall the following basic fact.
\begin{lemma}\label{lemma:Elementary}
For a sequence of random variable $(X_n)_{n\ge 1}$, we have 
$X_{n}\stackrel{p}{\to}0$ if and only if for each diverging subsequence $(n_\ell)$ there 
exists a refinement  $(n'_\ell)\subseteq(n_\ell)$, such that $X_{n'_{\ell}}\stackrel{a.s.}{\to}0$.
\end{lemma}

Let $(n_\ell)$ be a diverging sequence. 
Then for any $m$, and any $\tau_1,\dots,\tau_m\in[0,T]\cap{\mathbb Q}$, 
we can construct a subsequence along which 
$\dW{\est{\mu}^{(n'_{\ell})}_{\tau_1,\dots,\tau_n}}{\mu_{\tau_1,\dots,\tau_m}}
\stackrel{a.s.}{\to} 0$. By successive refinements and a diagonal argument,
we can assume that the subsequence is such that
\begin{align}
\prob\Big(\est{\mu}^{(n'_{\ell})}_{\tau_1,\dots,\tau_n}\Rightarrow \mu_{\tau_1,\dots,\tau_m}\;\;\;\;
\forall m, \;\;\forall \tau_1,\dots,\tau_m\in[0,T]\cap{\mathbb Q}\Big )=1\, .
\label{eq:SimultaneousFDD}
\end{align}
We finally need a tightness result, whose proof is presented in Appendix
\ref{app:Tightness}.
\begin{lemma}\label{lemma:Tight}
Under the assumptions of Theorem~\ref{thm:StateEvolution}, there exists $\alpha\in(0,1)$ and, 
for any $\eps>0$ there exists  $M(\eps)<\infty$ such that
\begin{align}
	\P\Big(\est{\mu}^{(n)}\big(\{\|\theta^0\|_2>M(\eps)\}\cup 
	\{\|\rev{\theta^{[0,T]}}\|_{C^{0,\alpha}}>M(\eps)\} \big)\ge \eps
	\mbox{ for infinitely many }n\Big) = 0\, .\label{eq:Tightness}
\end{align}
Here $\|f\|_{C^{0,\alpha}}$ denotes the $\alpha$-H\"older seminorm of function $f$.
\end{lemma}
By Eq.~\eqref{eq:SimultaneousFDD} and since finite-dimensional 
distributions on the rationals uniquely identify the limit on $C([0,T],\reals^k)$
\cite{billingsley2013convergence}, we proved $\dW{\est{\mu}^{(n'_{\ell})}}{\mu} \stackrel{a.s.}{\to}0$,
and therefore using Lemma \ref{lemma:Elementary} we obtain the desired claim.

\section*{Acknowledgments}
This work was supported by
NSF through award DMS-2031883 and from the Simons Foundation through Award 814639 for the
Collaboration on the Theoretical Foundations of Deep Learning. C. Cheng is supported by
the William R. Hewlett Stanford graduate fellowship. 
M. Celentano is supported by the Miller Institute for Basic Research in Science, University of California Berkeley.
We also acknowledge NSF grant CCF-2006489 and the ONR grant N00014-18-1-2729.

\bibliographystyle{amsalpha}
\newcommand{\etalchar}[1]{$^{#1}$}
\providecommand{\bysame}{\leavevmode\hbox to3em{\hrulefill}\thinspace}
\providecommand{\MR}{\relax\ifhmode\unskip\space\fi MR }
\providecommand{\MRhref}[2]{%
  \href{http://www.ams.org/mathscinet-getitem?mr=#1}{#2}
}
\providecommand{\href}[2]{#2}


\newpage
\appendix

%
%

\section{Auxiliary lemmas for the proof of Theorem \ref{thm:UniqueExist}}

\subsection{Proofs for the auxiliary real-valued system}

\subsubsection{Proof of Lemma~\ref{lem:Exist-General-ODE}} \label{proof:Exist-General-ODE}

To simplify notations, we recast the ODE system of Eqs.~\eqref{eq:def-pRt} to \eqref{eq:def-pCl} as
\begin{subequations}
\begin{align}
	\frac{\de}{\de t}f_1(t) & = \rev{ \alpha_1 f_1(t) +  \alpha_2 \int_0^t f_2(t-s) f_1(t) \de s \, } \label{eq:def-f1}\\
	f_2(t) & = \alpha_3 f_1(t) + \alpha_4 \int_0^t f_1(t-s) f_2(s) \de s\, , \label{eq:def-f2} \\
	\frac{\de}{\de t} \sqrt{f_3(t)} & = \sqrt{\alpha_5 f_3(t) + \alpha_6 f_4(t) + \alpha_7 \int_0^t (t-s+1)^2 f_2(t-s)^2  f_3(s) \de s  }\, , \label{eq:def-f3}\\
	f_4(t) & = \alpha_8 + \alpha_9 f_3(t) + \alpha_{10} \int_0^t (t-s+1)^2 f_1(t-s)^2 f_4(s) \de s\, , \label{eq:def-f4}
\end{align}
\end{subequations}
where $\alpha_1,\dots \rev{,}\alpha_{10}>0$
and we assume boundary conditions $f_1(0) = \beta_1, f_3(0) = \beta_2$, with $\beta_1,\beta_2>0$.

\paragraph{Existence and uniqueness of $f_1(t)$ and $f_2(t)$.}  These functions are
determined by Eqs.~\eqref{eq:def-f1} and \eqref{eq:def-f2}.
\begin{lemma}\label{lemma:f1-f2}
For any $\beta_1>0$,  Eqs.~\eqref{eq:def-f1} and \eqref{eq:def-f2} admit a unique solution $(f_1,f_2)\in C([0,\infty)\to\reals^2)$ with $f_i(t)>0$ for all $t$. Further, there exist
constants
$\lambda, C>0$  depending uniquely on $(\alpha_1,\dots,\alpha_4, \beta_1)$
such that 
$(f_1(t)\vee f_2(t))\le Ce^{\lambda t}$ for all $t\ge 0$. 
\end{lemma}
\begin{proof}
We can solve explicitly Eq.~\eqref{eq:def-f1} to yield $f_1=G(f_2)$ where we define the mapping $G:C([0,\infty))\to C([0,\infty) )$. \rev{Indeed, combining Eqs.~\eqref{eq:def-f1} and \eqref{eq:def-f2}, one has
	\begin{align}
		\frac{\de}{\de t}f_1(t) & =  \prn{\alpha_1 - \frac{\alpha_2 \alpha_3}{\alpha_4}} f_1(t) +  \frac{\alpha_2}{\alpha_4} f_2(t) \, 
	\end{align}
	and the mapping $G$ solving $f_1 = G(f_2)$ is then}

\begin{align}
	\rev{G(f)(t) := e^{ \prn{\alpha_1 - \frac{\alpha_2 \alpha_3}{\alpha_4}}  t}\beta_1+\frac{\alpha_2}{\alpha_4} \int_0^t e^{ \prn{\alpha_1 - \frac{\alpha_2 \alpha_3}{\alpha_4}}(t-s)}f(s) \, \de s\, .}
\end{align}
Hence we can rewrite Eq.~\eqref{eq:def-f2} as 
\begin{align}
	f_2(t) & = \alpha_3 G(f_2)(t) + \alpha_4 \int_0^t \rev{G(f_2)(t-s)} f_2(s) \de s\, .\label{eq:F1_F2}
\end{align}
Fixing $T>0$ arbitrarily, existence and uniqueness in $C([0,T])$ follows from
\cite[Theorem 1]{wolfersdorf1995class}, whereby $G_0 = -\alpha_3 G$, $G_1=-\alpha_4 G$,
$G_2=\id$ and the Lipschitz properties $(4)$, $(5)$ in the assumption of that theorem hold because,
in our notations, for all $\lambda$ large enough
\begin{align*}
	\big|[G(h_1)(t)-G(h_2)(t)]e^{-\lambda y}\big|\le \alpha_2e^{(\alpha_1-\lambda)t}
	\int_0^te^{-\alpha_1 s}|h_1-h_2|(s) \, \de s \, ,
\end{align*}
$\|G(h_1)-G(h_2)\|_{\lambda,\infty}\le (\alpha_2/(\lambda-\alpha_1))
\|h_1-h_2\|_{\lambda,\infty}$.  Since the solution exists unique on any $[0,T]$
it also exists unique on $[0\rev{,}\infty)$, and Eq.~\eqref{eq:F1_F2} holds because the unique solution
has $\|f_2\|_{\lambda,\infty}<\infty$ by the first part of \cite[Theorem 1]{wolfersdorf1995class}.
\end{proof}

\paragraph{Existence of $f_3(t)$ and $f_4(t)$.} Let $f_1(t)$ and $f_2(t)$ 
be given as per Lemma \ref{lemma:f1-f2}. 
We then seek measurable functions $f_3, f_4: \realsp \to \realsp$ solving 
Eqs.~\eqref{eq:def-f3} and \eqref{eq:def-f4}.
Consider the space
\begin{align}
\mathcal{S}_{\wb{\mathfrak{S}}, 2}(\lambda, \eps_3, \eps_4) := \left\{(f_3, f_4) \mid f_i : \realsp \to \realsp,  \normlbd{\sqrt{f_i}}{\infty} \leq \eps_i, i=3,4; f_3(0) = \beta_2 \right\} \, 
\end{align}
with the metric
\begin{align}
\mathsf{dist}_\lambda ((f_3, f_4), (g_3, g_4)) := 4 \sqrt{\alpha_9} \normlbd{\sqrt{f_3} - \sqrt{g_3}}{\infty} + \normlbd{\sqrt{f_4} - \sqrt{g_4}}{\infty}.
\end{align}
The space $\mathcal{S}_{\wb{\mathfrak{S}}, 2}(\lambda, \eps_3, \eps_4)$ is complete under the metric $\mathsf{dist}_\lambda$. Then we consider the transformation $\mathcal{T}_{\wb{\mathfrak{S}}, 2}(f_3, f_4) := (\wb{f}_3, \wb{f}_4)$ such that
\begin{align}
\frac{\de}{\de t} \sqrt{\wb{f}_3(t)} & = \sqrt{\alpha_5 f_3(t) + \alpha_6 f_4(t) + \alpha_7 \int_0^t (t-s+1)^2 f_2(t-s)^2  f_3(s) \de s  }\, , \label{eq:transform-f3}\\
\wb{f}_4(t) & = \alpha_8 + \alpha_9 f_3(t) + \alpha_{10} \int_0^t (t-s+1)^2 f_1(t-s)^2 f_4(s) \de s \, , \label{eq:transform-f4}
\end{align}
with $\wb{f}_3(0) = \beta_2$. Similarly, in the following lemma we show that for properly chosen $\eps_3, \eps_4$ and large enough $\lambda$, $\mathcal{T}_{\wb{\mathfrak{S}}, 2}$ is a contraction mapping. We postpone its proof to Appendix~\ref{proof:contraction-ODE-2}.
\begin{lemma} \label{lem:contraction-ODE-2}
There exists constants $\eps_i:=\eps_i(\alpha_1, \cdots, \alpha_8, \beta_1, \beta_2) > 0$ for $i=3,4$ such that for any $\lambda \geq \wb{\lambda} := \wb{\lambda}(\alpha_1, \cdots, \alpha_8, \beta_1, \beta_2)$, $\mathcal{T}_{\wb{\mathfrak{S}}, 2}$ is an operator that maps $\mathcal{S}_{\wb{\mathfrak{S}}, 2}(\lambda, \eps_3, \eps_4)$ into itself, and for any $(f_3, f_4), (g_3, g_4) \in \mathcal{S}_{\wb{\mathfrak{S}}, 2}$, the transformation $\mathcal{T}_{\wb{\mathfrak{S}}, 2}$ is a contraction
\begin{align}
	\mathsf{dist}_\lambda \big(\mathcal{T}_{\wb{\mathfrak{S}}, 2}(f_3, f_4) \, , \mathcal{T}_{\wb{\mathfrak{S}}, 2}(g_3, g_4)\big) \leq \frac 1 2 \mathsf{dist}_\lambda \big((f_3, f_4), (g_3, g_4)\big) \, .
\end{align}
\end{lemma}
The proof of existence is then concluded by the applying Banach fixed-point theorem and 
Lemma~\ref{lem:contraction-ODE-2}. 
We note that $\mathcal{T}_{\wb{\mathfrak{S}}, 2}$ maps $f_3, f_3$ into continuous function, hence the fixed point is continuous.

\subsubsection{Proof of Lemma~\ref{lem:contraction-ODE-2}} \label{proof:contraction-ODE-2}
We will use several times the following basic inequality for $I(f)(t):=\int_0^tf(s) \de s$
\begin{align}
\big\|I(f)\big\|_{\lambda, T}&=
\sup_{0\le t\le T}e^{-\lambda t}\left|\int_0^tf(s) \de s\right|\nonumber \\
&= \sup_{0\le t\le T}
\left|\int_0^te^{-\lambda(t-s)}f(s) e^{-\lambda s}\de s\right|\le \frac{1}{\lambda}|f\|_{\lambda,T}\, .
\label{eq:Int-Ineq}
\end{align}

In the first step, we show that for some properly chosen $\eps_3, \eps_4$ and large enough 
$\lambda$, $\mathcal{T}_{\wb{\mathfrak{S}}, 2}$ maps $\mathcal{S}_{\wb{\mathfrak{S}}, 2}(\lambda, \eps_3,
\eps_4)$ into itself. We set $F_i(\lambda) := \int_0^\infty e^{-\lambda s} (s+1)^2 f_i(s)^2 \de s$ for 
$i=1,2$ ($F_i(\lambda)$ is well-defined for any $\lambda$ large enough and
$F_i(\lambda) \to 0$ as $\lambda \to \infty$ by Lemma \ref{lemma:f1-f2}).
Letting $J(f_3)(t):=\sqrt{\int_0^t (t-s+1)^2 f_2(t-s)^2  f_3(s) \de s  }$,
we have 
\begin{align*}
\normlbd{J(f_3)}{\infty} &= \sup_{s\ge 0}
\sqrt{\int_0^s e^{-2\lambda(s-s')}(s-s'+1)^2 f_2(s-s')^2 \cdot e^{-2\lambda s'} f_3(s') \de s'} \\
&\le \sqrt{\alpha_7 F_2(2\lambda)} \normlbd{\sqrt{f_3}}{\infty}
\end{align*}
Hence, using Eq.~\eqref{eq:Int-Ineq}
\begin{align}
\normlbd{\sqrt{\wb{f}_3}}{\infty} &\le \sqrt{\beta_2}
+\frac{1}{\lambda}
\big[\sqrt{\alpha_5} \normlbd{\sqrt{f_3}}{\infty} +  \sqrt{\alpha_6} \normlbd{\sqrt{f_4}}{\infty}  +  \sqrt{\alpha_7}
\normlbd{J(f_3)}{\infty}\big]\nonumber\\
&\sqrt{\beta_2}
+\frac{\sqrt{\alpha_5}}{\lambda} \eps_3
+\frac{\sqrt{\alpha_6}}{\lambda} \eps_4
+ \frac{1}{\lambda}\sqrt{\alpha_7 F_2(2\lambda)} \eps_3\, .\label{eq:contraction-transform-2-1}
\end{align}
Proceeding analogously for $\wb{f}_4$, we get
\begin{align}
\normlbd{\sqrt{\wb{f}_4}}{\infty} &\leq \alpha_8 + \alpha_9 \eps_3 +\sqrt{\alpha_{10} F_1(2\lambda)} \eps_4 \, . \label{eq:contraction-transform-2-2}
\end{align}
By taking
\begin{align}
\eps_3 & \geq 2 \sqrt{\beta_2} \, , \qquad
\eps_4 \geq  2 \alpha_9 \eps_3 + 2 \alpha_8 \, , \nonumber
\end{align}
and large enough $\lambda$ 
we can then get from Eqs.~\eqref{eq:contraction-transform-2-1} and \eqref{eq:contraction-transform-2-2} 
$\normlbd{\wb{f}_3}{\infty}  \leq \eps_3$,  $\normlbd{\wb{f}_4}{\infty} \leq \eps_4$.

Next, we show $\mathcal{T}_{\wb{\mathfrak{S}}, 2}$ is a contraction mapping. Let $\mathcal{T}_{\wb{\mathfrak{S}}, 2}(g_3, g_4) = (\wb{g}_3, \wb{g}_4)$, we can get from Eq.~\eqref{eq:transform-f3}
\begin{align}
&\left|\sqrt{\wb{f}_3(t)} - \sqrt{\wb{g}_3(t)}\right| \nonumber \\
& \leq \int_0^t \Bigg(\sqrt{\alpha_5} \cdot \left|\sqrt{f_3(s)} - \sqrt{g_3(s)} \right| + \sqrt{\alpha_6} \cdot \left|\sqrt{f_4(s)} - \sqrt{g_4(s)} \right| \nonumber \\
& \qquad + \sqrt{\alpha_7} \cdot \left|\sqrt{\int_0^s (s-s'+1)^2 f_2(s-s')^2  f_3(s') \de s'} - \sqrt{\int_0^s (s-s'+1)^2 f_2(s-s')^2  g_3(s') \de s' }\right| \Bigg)  \de s \\
&\leq \int_0^t \Bigg(\sqrt{\alpha_5} \cdot \left|\sqrt{f_3(s)} - \sqrt{g_3(s)} \right| + \sqrt{\alpha_6} \cdot \left|\sqrt{f_4(s)} - \sqrt{g_4(s)} \right| \nonumber \\
& \qquad + \sqrt{\alpha_7} \sqrt{\int_0^s (s-s'+1)^2 f_2(s-s')^2  \prn{\sqrt{f_3(s')} - \sqrt{g_3(s')}}^2 \de s'} \Bigg)\de s\, .\label{eq:BoundKK}
\end{align}
Letting $K^2(f_3,g_3)(s):=   \int_0^s (s-s'+1)^2 f_2(s-s')^2  \prn{\sqrt{f_3(s')} - \sqrt{g_3(s')}}^2 \de s'$,
we have
\begin{align}
\normlbd{K(f_3,g_3)}{t}&\le  \sup_{0\le s\le t}
\sqrt{\int_0^s e^{-2\lambda(s-s')} (s-s'+1)^2 f_2(s-s')^2  e^{-2\lambda s'} 
	\prn{\sqrt{f_3(s')}
		-\sqrt{g_3(s')}} \de s'}\nonumber\\
&\le \sqrt{F_2(2\lambda)}   \normlbd{\sqrt{f_3} - \sqrt{g_3}}{t}\, .\label{eq:KK}
\end{align}
Using this bound and Eq.~\eqref{eq:Int-Ineq} in  Eq.~\eqref{eq:BoundKK}, we get
\begin{align}
\normlbd{\sqrt{\wb{f}_3} - \sqrt{\wb{g}_3}}{\infty}  
\leq \frac{1}{\lambda} \prn{\sqrt{\alpha_5} \normlbd{\sqrt{f_3} - 
		\sqrt{g_3}}{\infty} + \sqrt{\alpha_6} \normlbd{\sqrt{f_4} - 
		\sqrt{g_4}}{\infty} + \sqrt{\alpha_7  F_2(2\lambda)}   \normlbd{\sqrt{f_3} - \sqrt{g_3}}{\infty} }\, .
\label{eq:contraction-transform-2-3}
\end{align}
Next from Eq.~\eqref{eq:transform-f4} it follows similarly that
\begin{align}
\normlbd{\sqrt{\wb{f}_4} - \sqrt{\wb{g}_4}}{\infty} \leq \sqrt{\alpha_9} \normlbd{\sqrt{f_3} - \sqrt{g_3}}{\infty} +\sqrt{\alpha_{10}  F_1(2\lambda)}   \normlbd{\sqrt{f_4} - \sqrt{g_4}}{\infty} \, . \label{eq:contraction-transform-2-4}
\end{align}
We take $\lambda$ large enough such that
\begin{align*}
\frac{\sqrt{\alpha_5} + \sqrt{\alpha_7  F_2(2\lambda)} }{\lambda} & \leq \frac 1 4 \, , \qquad	\frac{4\sqrt{\alpha_6 \alpha_9}}{\lambda}   \leq \frac 1 4 \, , \qquad \sqrt{\alpha_{10}  F_1(2\lambda)} \leq \frac1 4 \, ,
\end{align*}
and further with Eqs.~\eqref{eq:contraction-transform-2-3} and \eqref{eq:contraction-transform-2-4}, 
\begin{align}
&\mathsf{dist}_\lambda \big(\mathcal{T}_{\wb{\mathfrak{S}}, 2}(f_3, f_4) \, , \mathcal{T}_{\wb{\mathfrak{S}}, 2}(g_3, g_4)\big)  = \mathsf{dist}_\lambda \big((\wb{f}_3, \wb{f}_4), (\wb{g}_3, \wb{g}_4)\big)  =4 \sqrt{\alpha_9} \normlbd{\sqrt{\wb{f}_3} - \sqrt{\wb{g}_3}}{\infty}  + \normlbd{\sqrt{\wb{f}_4} - \sqrt{\wb{g}_4}}{\infty} \nonumber \\
& \leq 2 \cdot \prn{\sqrt{\alpha_9} \normlbd{\sqrt{f_3} - \sqrt{g_3}}{\infty} + \frac 1 4  \normlbd{\sqrt{f_4} - \sqrt{g_4}}{\infty}}  = \frac 1 2 \mathsf{dist}_\lambda \big((f_3, f_4), (g_3, g_4)\big).
\end{align}

\subsection{Proof of Lemma~\ref{lem:well-defined-paths}} \label{proof:well-defined-paths}
We first show $\vt^t$ is uniquely defined. Note that $u^t$ has covariance kernel $\Cl/\delta$,
which implies for any $0 \leq s \leq t \leq T$,
\begin{align}
\Ep \brk{\normtwo{u^t - u^s}^2} & = \frac 1 \delta \cdot \Tr \prn{\Cl(t,t) - 2\Cl(t,s) + \Cl(s,s)} 
\leq \frac{k}{\delta} \norm{\Cl(t,t) - 2\Cl(t,s) + \Cl(s,s)} \, .
\end{align}
By the definition of $\wb{\mathcal{S}}$ and in particular Eq.~\eqref{eq:space-condition-Cl},
we can invoke Kolmogorov continuity theorem (cf.~\cite[Cor.~2.1.4]{stroock1997multidimensional}) 
and conclude that the sample path $u^t$ is 
locally $\alpha$-H\"{o}lder continuous for any $\alpha \in (0, 1)$. 

Since the sample path $u^t$ is continuous almost surely, 
we can solve Eq.~\eqref{eq:def-theta} per each given realization of $u^t$. Namely,
we rewrite Eq.~\eqref{eq:def-theta} as 
\begin{align}
&\frac{\de}{\de t} \wb{\theta}^t  +\int_0^t \wb{R}_{\ell}(t,s) \wb{\theta}^s \de s = \wb{u}^t\, ,\\
&\wb{\theta}^t := A(t)^{-1} \theta^t\, ,\;\;\;\; \wb{R}_{\ell}(t,s) := A(t)^{-1}R_{\ell}(t,s)A(s)
\, ,\;\;\;\; \wb{u}^t := A(t)^{-1}u^t\, ,\\
& \frac{\de A}{\de t}(t) := - (\Lambda^t + \Gamma^t) A(t)\, ,
\end{align}
or, in \rev{the} integral form,
\begin{align}
&\wb{\theta}^t  +\int_0^t K_{\ell}(t,s) \wb{\theta}^s \de s = v^t\, ,\label{eq:WbTheta}\\
&K_{\ell}(t,s) := \int_s^t \wb{R}_{\ell}(z,s)\, \de z\, ,\;\;\; v^t:= \theta^0+\int_0^t\wb{u}^s\de s\, .
\end{align} 
This is a linear Volterra integral equation of second kind, with 
domain $[0,T]$ and kernel $K_{\ell} : [0,T]\times[0,T]$ by assumptions
\eqref{eq:Rl-Bound}
\eqref{eq:Gamma-Bound-Ass} in the definition of $\cS$ (and using Lemma \ref{lem:Exist-General-ODE}). Since $v^t$ is continuous (and hence integrable)
over the same domain. 
By Theorem 3.6 and Corollary 4.3 in \cite[Chapter 9]{gripenberg1990volterra} admits a unique 
continuous solution which is also bounded and continuous  by Eq.~\eqref{eq:WbTheta}.

Next we show $r^t$ is uniquely defined. Again by Eq.~\eqref{eq:space-condition-Ct} and 
Kolmogorov continuity theorem, the sample path $w^t$ is $\alpha$-H\"{o}lder continuous for any 
$\alpha \in (0, 1)$ with probability $1$. 
Further, the equation defining $r$ is a nonlinear Volterra integral equation of second
kind with kernel $\Rt(t,s)/\delta$ that is bounded on $[0,T]^2$ by Eq.~\eqref{eq:RtBoundDef}
and Lemma \ref{lem:Exist-General-ODE}, and Lipschitz continuous nonlinearity $\ell_t(\,\cdot\,;z)$
by Assumption \ref{ass:Normal}.
By Theorem \cite[Theorem 2.6, Chapter 12]{gripenberg1990volterra} this equation
admits a unique solution $t\mapsto r^t$ that is continuous. 

The proof of uniqueness and existence for the functional derivative $\partial \ell_t(r^t;z) / \partial w^s$ is the same, provided that now the path $r^t$ is uniquely defined.

\subsection{Proof of Lemma~\ref{lem:solution-in-space}} \label{proof:solution-in-space}
\paragraph{$\trsfrmA$ maps $\mathcal{S}$ into $\wb{\mathcal{S}}_{\mathrm{cont}}$.} Directly from Eq.~\eqref{eq:trsfrm-t}, we can get
\begin{align}
\frac{\mathrm d}{\mathrm d t} \left\|\vt^t \right\|_2 & \leq \prn{\norm{\Lambda^t}+ \norm{\vGamma^t}}   \left\| \vt^t\right\|_2  + \int_0^t \left\| \vRl(t,s) \right\| \left\| \vt^s \right\|_2 \mathrm d s +  \left\| \vu^t \right\|_2 \nonumber \\
& \leq \prn{\cstlbd + \cstloss }   \left\| \vt^t\right\|_2  + \int_0^t \pRl(t-s) \left\| \vt^s \right\|_2 \mathrm d s +  \left\| \vu^t \right\|_2 \, ,
\end{align}
where the last line follows from the assumptions that $\norm{\Lambda^t} \leq \cstlbd, \norm{\Gamma^t} \leq \cstloss$ and $\left\| \vRl(t,s) \right\|  \leq \pRl(t-s)$, which further gives us
\begin{align}
& \ddt \sqrt{\mathbb{E} \norm{\vt^t}_2^2}  \leq \sqrt{\mathbb{E} \brk{\prn{\prn{\cstlbd + \cstloss}   \left\| \vt^t\right\|_2 + \int_0^t \pRl(t-s) \left\| \vt^s \right\|_2 \mathrm d s+ \left\| \vu^t \right\|_2}^2}} \nonumber \\
& \stackrel{\mathrm{(i)}}{\leq} \sqrt{\mathbb{E} \brk{\prn{\prn{\cstlbd + \cstloss}^2   \left\| \vt^t\right\|_2^2 + \int_0^t (t-s+1)^2 \pRl(t-s)^2 \left\| \vt^s \right\|_2^2 \mathrm d s + \left\| \vu^t \right\|_2^2} \cdot  \prn{1 + \int_0^t (t-s+1)^{-2} \mathrm d s + 1}}} \nonumber \\
& \leq \sqrt{3 \cdot \brc{\prn{\cstlbd + \cstloss}^2   \mathbb{E} \brk{\left\| \vt^t\right\|_2^2} + \int_0^t (t-s+1)^2 \pRl(t-s)^2 \mathbb{E} \brk{\left\| \vt^s \right\|_2^2} \mathrm d s + \frac{k}{\delta} \Phi_{\vCl}(t)}} \, , \label{eq:mid-solution-in-space-1}
\end{align}
where in (i) we use Cauchy-Schwarz inequality and in the last line it is used that
\begin{align}
\mathbb{E} \brk{\left\| \vu^t \right\|_2^2} = \Tr \prn{\mathbb{E} \brk{\vu^t {\vu^t}^\sT}} \leq k \norm{\mathbb{E} \brk{\vu^t {\vu^t}^\sT}} = \frac{k}{\delta} \norm{\vCl(t,t)} \leq \frac{k}{\delta} \Phi_{\vCl}(t) \, .
\end{align}
While since $\pCt(0) > \cstthetaz \geq \mathbb{E} \brk{\norm{\vt^0}_2^2}$ and recall
Eq.~\eqref{eq:def-pCt},
we obtain that $\mathbb{E} \brk{\norm{\vt^t}_2^2} < \Phi_{\vCt}(t)$ for all $t \in [0,T]$. We thus have
\begin{align}
\norm{\bvCt(t,t)} =  \norm{\mathbb{E} \brk{\vt^t {\vt^t}^\sT}} \leq  \mathbb{E} \brk{\norm{\vt^t}_2^2} \leq \Phi_{\vCt}(t)\, .
\end{align}
Next we look at the definition for the formal partial derivative $\partial \theta^t / \partial u^s$, as it is not a random function we have from Eq.~\eqref{eq:trsfrm-derivative-t} that
\begin{align}
\ddt \bRt(t,s) &  = - \prn{\vLambda^t + \Gamma^t} \bRt(t,s) - \int_s^t \vRl(t,s') \bRt(s',s) \mathrm d s'\, ,
\end{align}
for $0 \leq s \leq t \leq T$ and with $\bRt(s,s) = I$. Substituting in assumptions of $\mathcal{S}$ in Definition~\ref{def:space-S}, it holds that
\begin{align}
\ddt  \norm{\bRt(t,s)} 
& \leq \prn{\cstlbd + \cstloss}\norm{\bRt(t,s)}  + \int_s^t \pRl(t-s') \norm{\bRt(s',s)} \mathrm d s' \, ,
\end{align}
with $\norm{\bRt(s,s)} = 1$. Since $\pRt(0) > 1$ and by Eq.~\eqref{eq:def-pRt},
we can obtain that
\begin{align}
\norm{\bRt(t,s)} \leq \pRt(t-s).
\end{align}
Finally, we note that for any $ 0\leq s \leq t \leq T$,
\begin{align*}
\norm{\bCt(t,t) - 2 \bCt(t,s) + \bCt(s,s)} & = \norm{\Ep \brk{(\vt^t - \vt^s)(\vt^t - \vt^s)^\sT}} 
\leq \Ep \brk{\norm{\vt^t - \vt^s}_2^2} \, ,  
\end{align*}
and thus further
\begin{align}
& \norm{\bCt(t,t) - 2 \bCt(t,s) + \bCt(s,s)} \nonumber \\
& \leq \Ep \brk{\norm{\int_s^t \brc{-(\Lambda^{t'} + \Gamma^{t'}) \theta^{t'} - \int_0^{t'} R_{\ell}(t',s') \theta^{s'} \de s' + u^{t'}}\de t'}_2^2 } \nonumber \\
& \leq (t-s)^2 \sup_{0 \leq t \leq T} \Ep \brk{\prn{(\cstlbd + \cstloss) \norm{\theta^t}_2 + \int_0^t \pRl(t-s) \norm{\theta^s}_2 \de s + \norm{u^t}_2}^2} \nonumber \\
& \leq (t-s)^2  \cdot \sup_{0 \leq t \leq T} 3\brc{\prn{\cstlbd + \cstloss}^2   \pCt(t) + \int_0^t (t-s+1)^2 \pRl(t-s)^2 \pCt(s) \mathrm d s + \frac{k}{\delta} \Phi_{\vCl}(t)} \, .
\end{align}
Since $\pRl, \pRt$ and $\pCl$ are nondecreasing, we get
\begin{align}
\norm{\bCt(t,t) - 2 \bCt(t,s) + \bCt(s,s)} &\leq 3\brc{\brk{\prn{\cstlbd + \cstloss}^2  + T(T+1)^2 \pRl(T)^2 }\pCt(T) + \frac{k}{\delta} \Phi_{\vCl}(T)}(t-s)^2 \nonumber \\
& =\cstdualspaceS (t-s)^2 \, .
\end{align}
Similarly, we can see the continuity of $\bCt$ by Cauchy-Schwarz (as an even stronger result, we show $\bCt$ is Lipschitz continuous)
\begin{align}
\norm{\bCt(t,s) - \bCt(t,s')} & \leq \sqrt{\Ep \bigg[\norm{\theta^t}_2^2\bigg] \cdot \Ep \brk{\norm{\vt^s - \vt^{s'}}_2^2}} \leq \sqrt{\pCt(T)\cstdualspaceS } \cdot |s-s'|\, . 
\end{align}
This shows that $(\bCt, \bRt) \in \wb{\mathcal{S}}_{\mathrm{cont}}$ and concludes the first part.
\paragraph{$\trsfrmB$ maps $\wb{\mathcal{S}}$ into $\mathcal{S}_{\mathrm{cont}}$.}  Next we will show $\norm{\bCl(t,t)} \leq \pCl(t)$ assuming that $(\bCt, \bRt) \in \wb{\mathcal{S}}$. By Definition~\ref{def:space-S-dual} and Eq.~\eqref{eq:trsfrm-r}, it follows that
\begin{align}
\norm{\vl_t \prn{\vr^t; \veps} }_2 & \leq \norm{\vl_t \prn{0; \veps} }_2 + \cstloss \norm{\vr^t}_2  \leq \norm{\vl_t \prn{0; \veps} }_2  + \frac{\cstloss}{\delta} \int_0^t \norm{\bRt(t,s)} \big\|\vl_s \prn{\vr^s; \veps} \big\|_2 \mathrm d s + \cstloss \norm{\vw^t}_2 \nonumber \\
& \leq \norm{\vl_t \prn{0; \veps} }_2 + \frac{\cstloss}{\delta} \int_0^t \Phi_{\vRt}(t-s) \norm{\vl_s \prn{\vr^s; \veps} }_2 \mathrm d s + \cstloss\norm{\vw^t}_2 \, .
\end{align} 
Hence
\begin{align}
&\mathbb{E} \brk{\norm{\vl_t \prn{\vr^t; \veps} }_2^2}  \leq \mathbb{E} \brk{\prn{\norm{\vl_t \prn{0; \veps} }_2 + \frac{\cstloss}{\delta} \int_0^t \pRt(t-s) \norm{\vl_s \prn{\vr^s; \veps} }_2 \mathrm d s + \cstloss\norm{\vw^t}_2}^2} \nonumber \\
& \stackrel{\mathrm{(i)}}{\leq} \mathbb{E} \brk{\prn{\norm{\vl_t \prn{0; \veps} }_2^2 + \frac{\cstloss^2}{\delta^2} \int_0^t (t-s+1)^2\Phi_{\vRt}(t-s)^2 \norm{\vl_s \prn{\vr^s; \veps} }_2^2 \mathrm d s + \cstloss^2\norm{\vw^t}_2^2} \cdot \prn{1 + \int_0^t (t-s+1)^{-2} \mathrm d s + 1}} \nonumber \\
& \leq 3 \brc{\cstthetaz + \frac{\cstloss^2}{\delta^2} \int_0^t (t-s+1)^2\Phi_{\vRt}(t-s)^2 \mathbb{E} \brk{\norm{\vl_s \prn{\vr_s; \veps} }_2^2} \mathrm d s + k\cstloss^2 \Phi_{\vCt}(t)} \, ,  \label{eq:mid-solution-in-space-2}
\end{align}
where in (i) we use Cauchy-Schwarz inequality and in the last line it is used that $\Ep \brk{\norm{\vw^t}_2^2} \leq k \norm{\bCt(t,t)} \leq k \pCt(t)$. By Eq.~\eqref{eq:def-pCl}
and along with the fact that $\pCl(0) \geq 3 \cstthetaz + 3 k\cstloss^2 \pCt(0) > \cstthetaz \geq  \mathbb{E} \brk{\norm{\vl_0 \prn{\vr^0; \veps} }_2^2} $, it must follow that
\begin{align}
\norm{\bCl(t,t)} = \norm{\Ep \brk{\vl_t(\vr^t; z) \vl_t(\vr^t; z)^\sT }} \leq \mathbb{E} \brk{\norm{\vl_t \prn{\vr^t; \veps} }_2^2} \leq \pCl(t) \, .
\end{align}
Next we show $\bRt(t,s) \leq \pRt(t-s)$ for all $0 \leq s < t \leq T$. By Eq.~\eqref{eq:trsfrm-derivative-l},
\begin{align*}
\frac{\partial \ell_t(r^t;z)}{\partial w^s} & = \nabla_r \ell_t(r^t;z) \cdot \prn{- \frac{1}{\delta} \int_s^t \bRt(t,s') \frac{\partial \ell_{s'}(r^{s'};z)}{\partial w^s}  \de s' - \frac 1 \delta \bRt(t,s) \nabla_r \ell_s(r^s;z) } \, , 
\end{align*}
and the Lipschitz property of the function $\ell$, i.e. $\norm{\nabla_r \ell_t(r^t;z)} \leq \cstloss$, we have
\begin{align}
\mathbb{E}\brk{\norm{\frac{\partial \ell_t(r^t;z)}{\partial w^s}}} & \leq \frac{\cstloss}{\delta} \prn{\int_s^t \norm{\bRt(t,s')} \mathbb{E} \brk{\norm{\frac{\partial \ell_{s'}(r^{s'};z)}{\partial w^s}}} \mathrm d s' + \cstloss \norm{\bRt(t,s) }} \nonumber \\
& \leq \frac{\cstloss}{\delta} \prn{\int_s^t \pRt(t-s')  \mathbb{E} \brk{\norm{\frac{\partial \ell_{s'}(r^{s'};z)}{\partial w^s}}} \mathrm d s' + \cstloss \pRt(t-s)}.
\end{align}
Comparing to  Eq.~\eqref{eq:def-pRl},
we see $\mathbb{E}\brk{\norm{\frac{\partial \ell_t(r^t;z)}{\partial w^s}}} \leq \pRl(t-s)$ for all $0 \leq s \leq t \leq T$, and further
\begin{align}
\norm{\bRl(t,s)} = \norm{\mathbb{E}\brk{\frac{\partial \ell_t(r^t;z)}{\partial w^s}}}  \leq \mathbb{E}\brk{\norm{\frac{\partial \ell_t(r^t;z)}{\partial w^s}}} \leq \pRl(t-s).
\end{align}
Then we conclude from the Lipschitz property and Eq.~\eqref{eq:trsfrm-Gamma} that
\begin{align}
\norm{\bGamma^t} & = \norm{\Ep \brk{\nabla_r \ell_t(r^t;z)}} \leq \Ep \brk{\norm{\nabla_r \ell_t(r^t;z)}} \leq \cstloss\, .
\end{align}
We note that for any $0 \leq s \leq t \leq T$,
\begin{align*}
&\norm{\bCl(t,t) - 2 \bCl(t,s) + \bCl(s,s)} = \norm{\Ep \brk{(\ell_t(r^t;z) - \ell_s(r^s;z))(\ell_t(r^t;z) - \ell_s(r^s;z))^\sT}} \\ 
& \leq \Ep \brk{\norm{\ell_t(r^t;z) - \ell_s(r^s;z)}_2^2}  \leq \cstloss \cdot \Ep \brk{\prn{\norm{r^t - r^s }_2 + |t-s|}^2} \\
& \leq \cstloss \cdot \Ep \brk{\prn{\frac 1 \delta \int_s^t \pRt(t-s') \norm{\vl_{s'}(r^{s'};z) }_2 \de s' + \norm{w^t - w^s}_2 + |t-s| }^2} \nonumber \\
& \leq \cstloss \cdot 3\Ep \brk{\prn{\frac{1}{\delta^2} \int_s^t (t-s'+1)^2\pRt(t-s')^2 \norm{\vl_{s'}(r^{s'};z) }_2^2 \de s' + \norm{w^t - w^s}_2^2 + (t-s)^2 }} \, ,
\end{align*}
where in the last line we use the Cauchy-Schwarz inequality. Further, taking into the following inequalities
\begin{align*}
\Ep \brk{\norm{\vl_{s'}(r^{s'};z) }_2^2} &\leq \pCl(s') \, , \\
\Ep \brk{\norm{w^t - w^s}_2^2} & \leq k \norm{\bCt(t,t) - 2\bCt(t,s) + \bCt(s,s)} \nonumber \\
& \leq 3k\brc{\brk{\prn{\cstlbd + \cstloss}^2  + T(T+1)^2 \pRl(T)^2 }\pCt(T) + \frac{k}{\delta} \Phi_{\vCl}(T)}(t-s)^2 \, ,
\end{align*}
it then follows that
\begin{align*}
& \norm{\bCl(t,t) - 2 \bCl(t,s) + \bCl(s,s)} \nonumber \\
& \leq 3\cstloss (t-s)^2 \nonumber \\
& \qquad \cdot \brc{\frac{1}{\delta^2} (T+1)^2 \pRt(T)^2 \pCl(T) + 3k\brc{\brk{\prn{\cstlbd + \cstloss}^2  + T(T+1)^2 \pRl(T)^2 }\pCt(T) + \frac{k}{\delta} \Phi_{\vCl}(T)} + 1} \nonumber \\
& = \cstspaceS (t-s)^2 \, .
\end{align*}
Similar to Eq.~\eqref{eq:lip-continuity-Ct} in the previous part, we also have Lipschitz continuity for $\bCl$, namely $\forall s,s' \in [0,t]$,
\begin{align}
\norm{\bCl(t,s) - \bCl(t,s')} & \leq \sqrt{\Ep \bigg[\norm{\vl_t(r^t; z)}_2^2\bigg] \cdot \Ep \brk{\norm{\vl_s(r^s; z) - \vl_{s'}(r^{s'}; z) }_2^2}} \leq \sqrt{\pCl(T)\cstspaceS } \cdot |s-s'|\, .
\end{align}
This concludes the proof.
\subsection{Proofs for contraction property of the mapping $\mathcal{T}$}

\subsubsection{Proof of Lemma \ref{lem:diff-Cl-global}} \label{proof:diff-Cl-global}
\paragraph{Controlling the distance between \texorpdfstring{$\bCt^1$}{TEXT} and \texorpdfstring{$\bCt^2$}{TEXT}.}
By Eq.~\eqref{eq:trsfrm-t}, the equations that define $\vt_1$ and $\vt_2$ can be put as for all $t \in [0,T]$ and $i=1,2$,
\begin{equation}
\frac{\mathrm d}{\mathrm d t} \vt^t_i = -(\vLambda^t + \vGamma^t) \vt^t_i - \int_0^t \vRl(t,s) \vt^s_i \mathrm d s + \vu^t_i\ ,
\end{equation}
where $\vu^t_i$ are centered Gaussian processes with autocovariances $\vCl^i/\delta$ and $\vRl := \vRl^1 = \vRl^2$. By definition, we can couple $\vu^t_1$ and $\vu^t_2$ such that
\begin{align}
\sup_{t \in [0, T]} e^{-\lambda t} \sqrt{\mathbb{E} \brk{\norm{\vu^t_1 - \vu^t_2}_2^2}} \leq 2 \cdot \lbddst{\vu^t_1}{\vu^t_2} = 2 \cdot \lbddst{\vCl^1/\delta}{\vCl^2/\delta} = \frac{2}{\sqrt \delta} \lbddst{\vCl^1}{\vCl^2}.
\end{align}
We observe that
\begin{align}
\frac{\mathrm d}{\mathrm d t} \norm{\vt^t_1 - \vt^t_2}_2 & 
& \leq \int_0^t \Phi_{\vRl}(t-s) \norm{\vt^s_1 - \vt^s_2}_2 \mathrm d s + \prn{\cstlbd + \cstloss} \norm{\vt^t_1 - \vt^t_2}_2 + \norm{\vu^t_1 - \vu^t_2}_2 \, .
\end{align}
By Lemma~\ref{lem:Exist-General-ODE} we can choose a $\bar{\lambda}$ large enough such that $\int_0^\infty e^{-\wb{\lambda} s}\Phi_{\vRl}(s) \mathrm d s \leq \cstlbd + \cstloss$, which implies that
\begin{align}
& e^{-\wb{\lambda} t} \frac{\mathrm d}{\mathrm d t} \|\vt^t_1 - \vt^t_2\|_2 \nonumber \\
& \leq \prn{\cstlbd + \cstloss + \int_0^\infty e^{-\wb{\lambda} s}\Phi_{\vRl}(s) \mathrm d s} \cdot \sup_{0 \leq s \leq t} e^{-\wb{\lambda} s} \norm{\vt^s_1 - \vt^s_2}_2 + e^{-\wb{\lambda} t}  \norm{\vu^t_1 - \vu^t_2}_2 \nonumber \\
&\leq 2\prn{\cstlbd + \cstloss} \cdot \sup_{0 \leq s \leq t} e^{-\wb{\lambda} s} \norm{\vt^s_1 - \vt^s_2}_2 + e^{-\wb{\lambda} t}  \norm{\vu^t_1 - \vu^t_2}_2 \, . \label{eq:argument-lambda-bar}
\end{align}
Using the observation
\begin{align}
\frac{\mathrm d}{\mathrm d t} \sup_{0 \leq s \leq t} e^{-\wb{\lambda} s} \norm{\vt^s_1 - \vt^s_2}_2 & \leq \max \brc{\frac{\mathrm d}{\mathrm d t}  e^{-\wb{\lambda} t} \norm{\vt^t_1 - \vt^t_2}_2, 0} \leq \max \brc{ e^{-\wb{\lambda} t} \frac{\mathrm d}{\mathrm d t}  \norm{\vt^t_1 - \vt^t_2}_2, 0} \nonumber \\
& \leq 2\prn{\cstlbd + \cstloss} \cdot \sup_{0 \leq s \leq t} e^{-\wb{\lambda} s} \norm{\vt^s_1 - \vt^s_2}_2 + e^{-\wb{\lambda} t}  \norm{\vu^t_1 - \vu^t_2}_2 \, , \label{eq:argument-lambda-bar-2}
\end{align}
we can derive that
\begin{align}
\frac{\mathrm d}{\mathrm d t} \prn{e^{-2\prn{\cstlbd + \cstloss} t} \sup_{0 \leq s \leq t} e^{-\wb{\lambda} s} \norm{\vt^s_1 - \vt^s_2}_2 } \leq e^{-2\prn{\cstlbd + \cstloss} t -\wb{\lambda} t} \norm{\vu^t_1 - \vu^t_2}_2, \label{eq:argument-lambda-bar-3}
\end{align}
and consequently by Cauchy-Schwarz inequality
\begin{align}
e^{-4\prn{\cstlbd + \cstloss} t -2\wb{\lambda} t} \norm{\vt^t_1 - \vt^t_2}_2^2 &\leq \prn{e^{-2\prn{\cstlbd + \cstloss} t } \sup_{0 \leq s \leq t} e^{-\wb{\lambda} s} \norm{\vt^s_1 - \vt^s_2}_2}^2 \leq \prn{\int_0^t e^{-2\prn{\cstlbd + \cstloss} s -\wb{\lambda} s} \norm{\vu^s_1 - \vu^s_2}_2 \mathrm d s}^2 \nonumber \\
& \leq \prn{\int_0^t \frac{1}{(t-s+1)^2} \mathrm d s} \prn{\int_0^t (t-s+1)^2e^{-4\prn{\cstlbd + \cstloss} s -2\wb{\lambda} s} \norm{\vu^s_1 - \vu^s_2}_2^2 \mathrm d s} \nonumber \\
& \leq \int_0^t (t-s+1)^2e^{-4\prn{\cstlbd + \cstloss} s -2\wb{\lambda} s} \norm{\vu^s_1 - \vu^s_2}_2^2 \mathrm d s \, .
\end{align}
Taking expectation on both sides, and choose some $\lambda \geq 2(\cstlbd + \cstloss) +\wb{\lambda}$, we have
\begin{align}
e^{-2\lambda t} \Ep \brk{\norm{\vt^t_1 - \vt^t_2}_2^2} & \leq e^{- 2\prn{\lambda - 2\prn{\cstlbd + \cstloss}  -\wb{\lambda}} t} \cdot e^{-4\prn{\cstlbd + \cstloss} t -2\wb{\lambda} t} \Ep \brk{\norm{\vt^t_1 - \vt^t_2}_2^2} \nonumber \\
& \leq e^{- 2\prn{\lambda - 2\prn{\cstlbd + L}  -\wb{\lambda}} t} \cdot \int_0^t (t-s+1)^2 e^{-4\prn{\cstlbd + L} s -2\wb{\lambda} s} \Ep \brk{\norm{\vu^s_1 - \vu^s_2}_2^2} \mathrm d s \nonumber \\
& \leq \prn{\int_0^\infty e^{- 2\prn{\lambda - 2\prn{\cstlbd + L}  -\wb{\lambda}} t} (t+1)^2 \de t} \cdot \sup_{0 \leq s \leq t} e^{- 2\lambda s} \Ep \brk{\norm{\vu^s_1 - \vu^s_2}_2^2} \, .
\end{align}
Taking supremum on both sides for $t \in [0, T]$ and choosing a large enough $\lambda$ yields
\begin{align}
\sup_{t \in [0, T]} e^{-\lambda t} \sqrt{ \Ep \brk{\norm{\vt^t_1 - \vt^t_2}_2^2}} &\leq \prn{\int_0^\infty e^{- 2\prn{\lambda - 2\prn{\cstlbd + L}  -\wb{\lambda}} t} (t+1)^2 \de t} \cdot \frac{2}{\sqrt \delta} \cdot \lbddst{\vCl^1}{\vCl^2} \nonumber \\
& \leq \epsilon \cdot  \lbddst{\vCl^1}{\vCl^2} \, ,
\end{align}
for any prescribed $\epsilon > 0$. Consider a centered Gaussian process $\begin{bmatrix} g_1 \\ g_2 \end{bmatrix} \in \mathbb{R}^{2p}$ with covariance $\mathbb{E}\brk{\begin{bmatrix} \vt^t_1 \\ \vt^t_2 \end{bmatrix} \begin{bmatrix} \vt^s_1 \\ \vt^s_2 \end{bmatrix}^\sT}$. Clearly $\mathbb{E} \brk{\norm{{g}^t_1 - g^t_2}_2^2}= \mathbb{E} \brk{\norm{\vt^t_1 - \vt^t_2}_2^2}$ for all $t \in [0, \infty)$. Since $g_1$ and $g_2$ have covariance kernels $\bCt^1$ and $\bCt^2$, we have
\begin{align}
\lbddst{\bCt^1}{\bCt^2} \leq \sup_{t \in [0, T]} e^{-\lambda t} \sqrt{ \Ep \brk{\norm{\vt^t_1 - \vt^t_2}_2^2}} \leq \epsilon \cdot \lbddst{\vCl^1}{\vCl^2} \, .
\end{align}
\paragraph{Controlling the distance between \texorpdfstring{$\bRt^1$}{TEXT} and \texorpdfstring{$\bRt^2$}{TEXT}.}
Note that both $\bRt^1$ and $\bRt^2$ are defined by the same ODE by Eq.~\eqref{eq:trsfrm-derivative-t} for $i=1,2$,
\begin{equation}
\frac{\mathrm d}{\mathrm d t} \bRt^i(t,s) = -\Lambda^t \bRt^i(t,s) - \int_s^t \vRl(t,s') \bRt^i(s',s) \mathrm d s' \, ,
\end{equation}
and with the same boundary condition $\bRt^i(s,s) = I$. Thus $\bRt^1 = \bRt^2$ on $[0, T]^2$.

\subsubsection{Proof of Lemma \ref{lem:diff-Rl-global}} \label{proof:diff-Rl-global}
\paragraph{Controlling the distance between \texorpdfstring{$\bCt^1$}{TEXT} and \texorpdfstring{$\bCt^2$}{TEXT}.}
Since $\vCl^1 = \vCl^2$ on $[0, T]^2$, we have for all $t \in [0, T]$ and $i=1,2$,
\begin{equation}
\frac{\mathrm d}{\mathrm d t} \vt^t_i = -\prn{\vLambda^t + \Gamma^t_i} \vt^t_i - \int_0^t \vRl^i(t,s) \vt^s_i \mathrm d s + \vu_t\, ,
\end{equation}
where $\vu_t$ is a centered Gaussian process with the covariance kernel $\vCl/\delta := \vCl^1/\delta = \vCl^2/\delta$. Using
\begin{align}
&\frac{\mathrm d}{\mathrm d t} \prn{\vt^t_1 - \vt^t_2} = - \prn{\vLambda^t + \Gamma^t_1} \prn{\vt^t_1 - \vt^t_2} - \prn{\Gamma^t_1 - \Gamma^t_2} \vt^t_2 - \int_0^t \vRl^1(t,s) (\vt^s_1 - \vt^s_2) \mathrm d s - \int_0^t  (\vRl^1(t,s) - \vRl^2(t,s)) \vt^s_2 \mathrm d s\, ,
\end{align}
it follows that
\begin{align}
&\frac{\mathrm d}{\mathrm d t} \|\vt^t_1 - \vt^t_2\|_2  \leq \norm{\frac{\mathrm d}{\mathrm d t} \prn{\vt^t_1 - \vt^t_2}}_2 \nonumber \\
& = \norm{\prn{\Lambda^t + \vGamma^t_1} \prn{\vt^t_1 - \vt^t_2} + \int_0^t \vSl^1(t,s) (\vt^s_1 - \vt^s_2) \mathrm d s + \prn{\vGamma^t_1 - \vGamma^t_2}\vt^t_2 + \int_0^t  (\vSl^1(t,s) - \vSl^2(t,s)) \vt^s_2 \mathrm d s}_2 \nonumber \\
& \leq \prn{\cstlbd + \cstloss} \norm{\vt^t_1 - \vt^t_2}_2 + \int_0^t \Phi_{\vSl}(t-s) \norm{\vt^s_1 - \vt^s_2}_2 \mathrm d s + \norm{\vGamma^t_1 - \vGamma^t_2} \norm{\vt^t_2}_2   + \int_0^t  \norm{\vSl^1(t,s) - \vSl^2(t,s)} \norm{\vt^s_2}_2 \mathrm d s\, .
\end{align}
By Lemma~\ref{lem:Exist-General-ODE} we can choose a $\bar{\lambda}$ large enough such that $\int_0^\infty e^{-\wb{\lambda} s}\Phi_{\vSl}(s) \mathrm d s \leq \cstlbd + \cstloss$, and therefore
\begin{align}
& e^{-\wb{\lambda} t}\frac{\mathrm d}{\mathrm d t} \|\vt^t_1 - \vt^t_2\|_2 \nonumber \\
& \leq \prn{\cstlbd + \cstloss} e^{-\wb{\lambda} t}\norm{\vt^t_1 - \vt^t_2}_2 + \int_0^t e^{-\wb{\lambda} (t-s)}\Phi_{\vSl}(t-s) \cdot e^{-\wb{\lambda} s}\norm{\vt^s_1 - \vt^s_2}_2 \mathrm d s \nonumber \\
&\qquad  + e^{-\wb{\lambda} t}\norm{\vGamma^t_1 - \vGamma^t_2} \norm{\vt^t_2}_2  + e^{-\wb{\lambda} t}\int_0^t \norm{\vSl^1(t,s) - \vSl^2(t,s)}  \norm{\vt^s_2}_2 \mathrm d s \nonumber \\
& \leq \prn{\cstlbd + \cstloss +\int_0^\infty e^{-\wb{\lambda} s}\Phi_{\vSl}(s) \mathrm d s} \cdot \sup_{0 \leq s \leq t} e^{-\wb{\lambda} s}\norm{\vt^s_1 - \vt^s_2}_2 + e^{-\wb{\lambda} t}\norm{\vGamma^t_1 - \vGamma^t_2} \norm{\vt^t_2}_2 \nonumber \\
& \qquad  + e^{-\wb{\lambda} t}\int_0^t \norm{\vSl^1(t,s) - \vSl^2(t,s)}  \norm{\vt^s_2}_2 \mathrm d s \nonumber \\
& \leq 2\prn{\cstlbd + \cstloss} e^{-\wb{\lambda} t} \sup_{0 \leq s \leq t} e^{-\wb{\lambda} s}\norm{\vt^s_1 - \vt^s_2}_2+ e^{-\wb{\lambda} t}\norm{\vGamma^t_1 - \vGamma^t_2} \norm{\vt^t_2}_2 + e^{-\wb{\lambda} t}\int_0^t \norm{\vSl^1(t,s) - \vSl^2(t,s)}  \norm{\vt^s_2}_2 \mathrm d s \, .
\end{align}
Similar to the proof in Appendix~\ref{proof:diff-Cl-global}, we obtain
\begin{align}
&\frac{\mathrm d}{\mathrm d t} \prn{e^{-2\prn{\cstlbd + \cstloss} t} \sup_{0 \leq s \leq t} e^{-\wb{\lambda} s} \norm{\vt^s_1 - \vt^s_2}_2 } \nonumber \\
&\leq  e^{-2\prn{\cstlbd + \cstloss} t-\wb{\lambda} t} \prn{\norm{\vGamma^t_1 - \vGamma^t_2} \norm{\vt^t_2}_2 + \int_0^t \norm{\vSl^1(t,s) - \vSl^2(t,s)}  \norm{\vt^s_2}_2 \mathrm d s} \, ,
\end{align}
and consequently 
\begin{align}
&e^{-4\prn{\cstlbd + \cstloss} t -2\wb{\lambda} t} \norm{\vt^t_1 - \vt^t_2}_2^2 \nonumber \\
& \leq \prn{\int_0^t e^{-2\prn{\cstlbd + \cstloss} s-\wb{\lambda} s} \prn{\norm{\vGamma^s_1 - \vGamma^s_2}\norm{\vt^s_2}_2 + \int_0^s \norm{\vSl^1(s,s') - \vSl^2(s,s')}  \norm{\vt_{s'}^2}_2 \mathrm d s'} \mathrm d s}^2 \nonumber \\
& \stackrel{\mathrm{(i)}}{\leq} \brc{\int_0^t (t-s+1)^{-2}\prn{1 + \int_0^s (s'+1)^{-2} \mathrm d s' } \mathrm d s}\cdot \bigg\{\int_0^t e^{-4\prn{\cstlbd + \cstloss} s-2\wb{\lambda} s}(t-s+1)^2\nonumber \\
& \qquad \cdot \prn{\norm{\vGamma^s_1 - \vGamma^s_2}^2  \norm{\vt^s_2}_2^2+ \int_0^s (s'+1)^2\norm{\vSl^1(s,s') - \vSl^2(s,s')}^2  \norm{\vt_{s'}^2}_2^2 \mathrm d s'} \mathrm d s\bigg\} \nonumber \\
& \leq 2\int_0^t e^{-4\prn{\cstlbd + \cstloss} s-2\wb{\lambda} s}(t-s+1)^2 \prn{\norm{\vGamma^s_1 - \vGamma^s_2}^2 \norm{\vt^s_2}_2^2 + \int_0^s (s'+1)^2\norm{\vSl^1(s,s') - \vSl^2(s,s')}^2  \norm{\vt_{s'}^2}_2^2 \mathrm d s'} \mathrm d s\, ,
\end{align}
where we invoke Cauchy-Schwarz inequality in (i). Take expectation on both sides and use Lemma~\ref{lem:solution-in-space} which implies that $\mathbb{E} \brk{\norm{\vt_{s'}^2}_2^2} \leq k \norm{\Ep \brk{\vt_{s'}^2 {\vt_{s'}^2}^\sT }}_2 \leq k \Phi_{\vCt}(s')$, we have
\begin{align}
&e^{-4\prn{\cstlbd + \cstloss} t -2\wb{\lambda} t} \mathbb{E} \brk{\norm{\vt^t_1 - \vt^t_2}_2^2} \nonumber \\
&\leq 2\int_0^t  \bigg\{ \prn{e^{-4\prn{\cstlbd + \cstloss} s-2\wb{\lambda} s}(t-s+1)^2} \nonumber \\
& \qquad \cdot \prn{\norm{\vGamma^s_1 - \vGamma^s_2}^2 \mathbb{E} \brk{\norm{\vt_{s}^2}_2^2} + \int_0^s (s'+1)^2\norm{\vSl^1(s,s') - \vSl^2(s,s')}^2   \mathbb{E} \brk{\norm{\vt_{s'}^2}_2^2} \mathrm d s'} \bigg\} \mathrm d s  \nonumber \\
& \leq 2k \int_0^t e^{-4\prn{\cstlbd + \cstloss} s-2\wb{\lambda} s}(t-s+1)^2 \prn{\norm{\vGamma^s_1 - \vGamma^s_2}^2 \Phi_{\vCt}(s) + \int_0^s (s'+1)^2\norm{\vSl^1(s,s') - \vSl^2(s,s')}^2   \Phi_{\vCt}(s') \mathrm d s'} \mathrm d s\, .
\end{align}
Now we take $\lambda > 2(\cstlbd + \cstloss) + \wb{\lambda}$, and for any $t \in [0,T]$,
\begin{align}
&e^{-2\lambda t} \mathbb{E} \brk{\norm{\vt^t_1 - \vt^t_2}_2^2} \nonumber \\
&\leq 2k\int_0^t \Bigg[ e^{- 2\prn{\lambda - 2\prn{\cstlbd + \cstloss}  -\wb{\lambda}} (t-s)}(t-s+1)^2 \nonumber \\
& \qquad \cdot e^{-2\lambda s} \prn{\norm{\vGamma^s_1 - \vGamma^s_2}^2 \Phi_{\vCt}(s)+ k\int_0^s (s'+1)^2\norm{\vSl^1(s,s') - \vSl^2(s,s')}^2   \Phi_{\vCt}(s') \mathrm d s'} \Bigg] \mathrm d s \nonumber \\
& \leq 2k \prn{\int_0^\infty e^{-2\prn{\lambda - 2\prn{\cstlbd + \cstloss}  -\wb{\lambda}} t}(t+1)^2 \de t }  \nonumber \\
& \qquad \cdot \prn{\Phi_{\vCt}(t) \cdot \sup_{0 \leq s \leq t} e^{-2\lambda s} \norm{\vGamma^s_1 - \vGamma^s_2}^2 +  \prn{k\int_0^s (s'+1)^2 \Phi_{\vCt}(s') \mathrm d s'} \cdot \sup_{0 \leq s' \leq s} e^{-2 \lambda s}\norm{\vSl^1(s,s') - \vSl^2(s,s')}^2  }  \nonumber \\
& \leq 2k\prn{\int_0^\infty e^{-2\prn{\lambda - 2\prn{\cstlbd + \cstloss}  -\wb{\lambda}} t}(t+1)^2 \de t }\cdot  \prn{\Phi_{\vCt}(T) \cdot \lbddst{\Gamma_1}{\Gamma_2}^2 +  kT (T+1)^2 \Phi_{\vCt}(T) \cdot \lbddst{\vSl^1}{\vSl^2}^2 }.
\end{align}
Therefore, we can always take a large enough $\lambda$ such that for any $\epsilon > 0$
\begin{align}
e^{-\lambda t} \sqrt{\mathbb{E} \brk{\norm{\vt^t_1 - \vt^t_2}_2^2}} \leq \epsilon \cdot \sqrt{\lbddst{\Gamma_1}{\Gamma_2}^2 + \lbddst{\vSl^1}{\vSl^2}^2} \leq \epsilon \cdot \prn{\lbddst{\Gamma_1}{\Gamma_2} + \lbddst{\vSl^1}{\vSl^2}}.
\end{align}
Using the same argument in Appendix~\ref{proof:diff-Cl-global}, we conclude that
\begin{align}
\lbddst{\bCt^1}{\bCt^2} \leq \sup_{t \in [0, T]} e^{-\lambda t} \sqrt{ \Ep \brk{\norm{\vt^t_1 - \vt^t_2}_2^2}} \leq \epsilon \cdot \prn{\lbddst{\vSl^1}{\vSl^2} + \lbddst{\Gamma_1}{\Gamma_2}}.
\end{align}

\paragraph{Controlling the distance between \texorpdfstring{$\bRt^1$}{TEXT} and \texorpdfstring{$\bRt^2$}{TEXT}.}
Again from Eq.~\eqref{eq:trsfrm-derivative-t} we get for any $0 \leq s \leq t \leq T$ and $i=1,2$,
\begin{equation}
\frac{\mathrm d}{\mathrm d t} \bRt^i(t,s) = - \prn{\Lambda^t + \Gamma^t_i} \bRt^i(t,s) - \int_s^t \vRl^i(t,s') \bRt^i(s',s) \mathrm d s' \, ,
\end{equation}
with the same boundary conditions $\bRt^i(s,s) = I$, and thus for any $0 \leq s \leq t \leq T$,
\begin{align}
\frac{\mathrm d}{\mathrm d t} \prn{\bRt^1(t,s) - \bRt^2(t,s)} & = - \prn{\Lambda^t + \Gamma^t_1} \prn{\bRt^1(t,s) - \bRt^2(t,s)} - \prn{\Gamma^t_1 - \Gamma^t_2} \bRt^2(t,s) \nonumber \\
& \qquad - \int_{s}^t \vRl^1(t,s') (\bRt^1(s',s) - \bRt^2(s',s)) \mathrm d s' - \int_{s}^t  (\vRl^1(t,s') - \vRl^2(t,s')) \bRt^2(s',s) \mathrm d s'  \, ,
\end{align}
and $\bRt^1(s,s) - \bRt^2(s,s) = 0$. It then follows that
\begin{align}
&\frac{\mathrm d}{\mathrm d t} \norm{\bRt^1(t,s) - \bRt^2(t,s)} \leq \norm{\frac{\mathrm d}{\mathrm d t} \prn{\bRt^1(t,s) - \bRt^2(t,s)}} \nonumber \\
& \leq \prn{\cstlbd + \cstloss} \norm{\bRt^1(t,s) - \bRt^2(t,s)} + \int_s^t \Phi_{\vSl}(t-s') \norm{\bRt^1(s',s) - \bRt^2(s',s)} \mathrm d s' \nonumber \\
& \qquad + \norm{\vGamma^t_1 - \vGamma^t_2} \cdot \Phi_{\vRt}(t-s) + \int_s^t \norm{\vSl^1(t,s') - \vSl^2(t,s')} \cdot \Phi_{\vRt}(s'-s) \mathrm d s' \, ,
\end{align}
where in the last line we use $\norm{\vGamma^t_1} \leq \cstloss$ and $\norm{\bRt^2(t,s)} \leq \Phi_{\vRt}(t-s)$ by invoking Lemma~\ref{lem:solution-in-space}. We now proceed almost identically to the proof in the previous part. We find some large enough $\wb{\lambda}$ such that $\int_0^\infty e^{-\wb{\lambda} t}\pRl(t) \de t \leq \cstlbd + \cstloss$ and on which $\wb{\lambda}$ it holds that for any $ 0 \leq s \leq t \leq T$,
\begin{align}
&\frac{\mathrm d}{\mathrm d t} \prn{e^{-2\prn{\cstlbd + \cstloss}t} \sup_{s \leq s' \leq t} e^{-\wb{\lambda} s'} \norm{\bRt^1(s', s) - \bRt^2(s', s)} } \nonumber \\
& \leq  e^{-2\prn{\cstlbd + \cstloss} t-\wb{\lambda} t} \prn{\norm{\vGamma^t_1 - \vGamma^t_2} \Phi_{\vRt}(t-s) + \int_s^t \norm{\vSl^1(t,s') - \vSl^2(t,s')}  \Phi_{\vRt}(s'-s) \mathrm d s'}
\end{align}
and then we have
\begin{align}
&e^{-2\prn{\cstlbd + \cstloss} t-\wb{\lambda} t} \norm{\bRt^1(t, s) - \bRt^2(t, s)}  \nonumber \\
& \leq \int_s^t e^{-2\prn{\cstlbd + \cstloss} s'-\wb{\lambda} s'} \prn{\norm{\vGamma_{s'}^1 - \vGamma_{s'}^2} \Phi_{\vRt}(s'-s) + \int_s^{s'} \norm{\vSl^1(s',s'') - \vSl^2(s',s'')}  \Phi_{\vRt}(s''-s) \mathrm d s''} \mathrm d s' \, .
\end{align}
For any $\lambda > 2(\cstlbd + \cstloss) + \wb{\lambda}$ and $0 \leq s \leq t \leq T$,
\begin{align}
&e^{-\lambda t} \norm{\bRt^1(t, s) - \bRt^2(t, s)}  \nonumber \\
& \leq \int_s^t e^{- \prn{\lambda - 2\prn{\cstlbd + \cstloss}  -\wb{\lambda}}(t-s')} \nonumber \\
& \qquad  \cdot e^{-\lambda s'} \prn{\norm{\vGamma_{s'}^1 - \vGamma_{s'}^2} \Phi_{\vRt}(s'-s) + \int_s^{s'} \norm{\vSl^1(s',s'') - \vSl^2(s',s'')}  \Phi_{\vRt}(s''-s) \mathrm d s''} \mathrm d s' \nonumber \\
& \leq \frac{1}{\lambda - 2\prn{\cstlbd + \cstloss}  -\wb{\lambda}} \nonumber \\
& \qquad \cdot \prn{\sup_{s \leq s' \leq t} e^{-\lambda s'}\norm{\vGamma_{s'}^1 - \vGamma_{s'}^2} \Phi_{\vRt}(s'-s) + \sup_{s \leq s' \leq t} \int_s^{s'} e^{-\lambda s'} \norm{\vSl^1(s',s'') - \vSl^2(s',s'')} \Phi_{\vRt}(s''-s) \mathrm d s''  } \nonumber \\
& \leq \frac{T \Phi_{\vRt}(T)}{\lambda - 2\prn{\cstlbd + \cstloss}  -\wb{\lambda}} \cdot \prn{\lbddst{\Gamma_1}{\Gamma_2} + \lbddst{\vSl^1}{\vSl^2}} \, .
\end{align}
For any $\epsilon > 0$, we can take a large enough $\lambda$ such that
\begin{align}
\lbddst{\bRt^1}{\bRt^2} = \sup_{0 \leq s \leq t \leq T}e^{-\lambda t} \norm{\bRt^1(t, s) - \bRt^2(t, s)} \leq \epsilon \cdot \prn{\lbddst{\vSl^1}{\vSl^2} + \lbddst{\Gamma_1}{\Gamma_2} } \, .
\end{align}

\subsubsection{Proof of Lemma \ref{lem:diff-Ctheta-global}} \label{proof:diff-Ctheta-global}
\paragraph{Controlling the distance between \texorpdfstring{$\bvCl^1$}{TEXT} and \texorpdfstring{$\bvCl^2$}{TEXT}.}
Given that $\bvRt :=\bvRt^1 = \bvRt^2$ on $[0,T]^2$, we can write the equations that define $\vr_1$ and $\vr_2$ as
\begin{equation}
\vr_i^t = - \frac 1 \delta \int_0^t \bvRt(t,s) \vl_s(\vr^s_i; \veps) \mathrm d s + \vw^t_i \, ,
\end{equation}
for $i=1,2$, where $\vw^t_i$ are centered Gaussian processes with covariance kernels $\bvCt^i$. We couple $\vw^t_1$ and $\vw^t_2$ such that they achieve small $(\lambda, T)$-distance, namely
\begin{align}
\sup_{t \in [0, T]} e^{-\lambda t} \sqrt{\mathbb{E} \brk{\norm{\vw^t_1 - \vw^t_2}_2^2}}  \leq 2 \cdot \lbddst{\vw^t_1}{\vw^t_2} = 2 \cdot \lbddst{\bvCt^1}{\bvCt^2} \, .
\end{align}
For any $t \leq T$, we have
\begin{align}
& e^{-\lambda t}\left\|\vr^t_1 - \vr^t_2\right\|_2  \leq e^{-\lambda t} \prn{\frac 1 \delta \int_0^{t} \left\|\bvRt(t,s)\right\| \left\|\vl_s(\vr^s_1;\veps) - \vl_s(\vr^s_2;\veps)\right\|_2 \mathrm d s + \norm{\vw^t_1 - \vw^t_2}_2} \nonumber \\
& \leq \frac 1 \delta \int_0^{t} e^{-\lambda(t-s)} \left\|\bvRt(t,s)\right\| \cdot e^{-\lambda s}\left\|\vl_s(\vr^s_1;\veps) - \vl_s(\vr^s_2;\veps)\right\|_2 \mathrm d s +  e^{-\lambda t}\norm{\vw^t_1 - \vw^t_2}_2\nonumber \\
& \leq \frac \cstloss \delta \int_0^{t} e^{-\lambda(t-s)} \Phi_{\vRt}(t-s) \cdot e^{-\lambda s}\left\|\vr^s_1 - \vr^s_2\right\|_2 \mathrm d s +  e^{-\lambda t}\norm{\vw^t_1 - \vw^t_2}_2 \, .
\end{align}
Therefore square both sides and taking expectations, we have
\begin{align}
&e^{-2\lambda t} \Ep \brk{\left\|\vr^t_1 - \vr^t_2\right\|_2^2} \leq \Ep \brk{\prn{\frac \cstloss \delta \int_0^{t} e^{-\lambda(t-s)} \Phi_{\vRt}(t-s) \cdot e^{-\lambda s}\left\|\vr^s_1 - \vr^s_2\right\|_2 \mathrm d s +  e^{-\lambda t}\norm{\vw^t_1 - \vw^t_2}_2}^2}	\nonumber \\
& \leq \prn{\int_0^t (t-s+1)^{-2} \mathrm d s + 1} \nonumber \\
& \qquad \cdot  \Ep \brk{\frac{\cstloss^2}{\delta^2} \int_0^{t} (t-s+1)^2e^{-2\lambda(t-s)} \Phi_{\vRt}(t-s)^2 \cdot e^{-2\lambda s}\left\|\vr^s_1 - \vr^s_2\right\|_2^2 \mathrm d s +  e^{-2\lambda t}\norm{\vw^t_1 - \vw^t_2}_2^2} \nonumber \\
&\leq \frac{2\cstloss^2}{\delta^2} \cdot \prn{\int_0^t e^{-2\lambda t}(t+1)^2 \Phi_{\vRt}(t)^2 \de t} \cdot \sup_{0 \leq s \leq t} e^{-2\lambda s} \Ep \brk{\left\|\vr^s_1 - \vr^s_2\right\|_2^2} + 2 \cdot \prn{2 \cdot \lbddst{\bvCt^1}{\bvCt^2}}^2 \, .
\end{align}
Note that the right hand side is increasing in $t$. By taking $\lambda$  to be large enough such that
\begin{align}
\frac{2\cstloss^2}{\delta^2} \cdot \int_0^t e^{-2\lambda t}(t+1)^2 \Phi_{\vRt}(t)^2 \de t \leq \frac 1 2\, , \label{eq:condition-Ctheta-global-1}
\end{align}
we have
\begin{align}
\sup_{0 \leq s \leq t} e^{-2\lambda s} \Ep \brk{ \left\|\vr^s_1 - \vr^s_2\right\|_2^2} \leq 16 \cdot \lbddst{\bvCt^1}{\bvCt^2}^2\, . \label{eq:inequality-Ctheta-global-1}
\end{align}
It then can be established following the same argument in Appendix~\ref{proof:diff-Cl-global} that
\begin{align}
\lbddst{\bvCl^1}{\bvCl^2} & \leq \sup_{t \in [0,T]} e^{-\lambda t} \sqrt{\mathbb{E} \brk{\left\|\vl_t(\vr^t_1; \veps) -\vl_t(\vr^t_1; \veps) \right\|_2^2}} \leq \cstloss \cdot \sqrt{\sup_{0 \leq s \leq t} e^{-2\lambda s} \brk{\Ep \left\|\vr^s_1 - \vr^s_2\right\|_2^2}} \nonumber \\
& \leq 4\cstloss \cdot \lbddst{\bvCt^1}{\bvCt^2}\, .
\end{align}
\paragraph{Controlling the distances between \texorpdfstring{$\bvRl^1$}{TEXT} and \texorpdfstring{$\bvRl^2$}{TEXT}, \texorpdfstring{$\bGamma_1$}{TEXT} and \texorpdfstring{$\bGamma_2$}{TEXT}.} From Eq.~\eqref{eq:trsfrm-derivative-l} we can obtain for any $0 \leq s \leq t \leq T$ and $i=1,2$,
\begin{align}
\frac{\partial \vl_t(\vr_i^t; \veps)}{\partial w^s} & = \nabla_r \ell_t(r^t;z) \cdot \frac{\partial \vr_i^t}{\partial \vw^s}\, , \\
\frac{\partial \vr_i^t}{\partial \vw^s} & := - \frac 1 \delta \int_s^t \bvRt(t,s') \frac{\partial \vl_{s'}(\vr^{s'}_i; \veps)}{\partial \vw^{s'}} \mathrm d s' - \frac 1 \delta \bRt(t,s) \nabla_r \ell_s(r_i^s;z) \, . \label{eq:trsfrm-derivative-r}
\end{align}
and by Eq.~\eqref{eq:trsfrm-Gamma},
\begin{align}
\bvGamma^t_i = \mathbb{E} \brk{\nabla_r\vl_t(\vr_i^t, \veps)} \, .
\end{align}
Therefore, for any $\lambda$ satisfying Eq.~\eqref{eq:condition-Ctheta-global-1} we have
\begin{align}
e^{-\lambda t} \norm{\bvGamma^t_1 - \bvGamma^t_2} & \leq e^{-\lambda t} \mathbb{E} \brk{\norm{\nabla_r \vl_t(\vr^t_1; \veps) - \nabla_r \vl_t(\vr^t_2; \veps)}} \leq e^{-\lambda t} \sqrt{\mathbb{E} \brk{\norm{\nabla_r \vl_t(\vr^t_1; \veps) - \nabla_r \vl_t(\vr^t_2; \veps)}^2}} \, ,
\end{align}
and then we use Eq.~\eqref{eq:inequality-Ctheta-global-1} and obtain
\begin{align}
e^{-\lambda t} \norm{\bvGamma^t_1 - \bvGamma^t_2} &  \leq e^{-\lambda t} \sqrt{\mathbb{E} \brk{\norm{\nabla_r \vl_t(\vr^t_1; \veps) - \nabla_r \vl_t(\vr^t_2; \veps)}^2}}  \nonumber \\
& \leq \cstloss \cdot e^{-\lambda t} \sqrt{\mathbb{E} \brk{\norm{\vr^t_1 - \vr^t_2}_2^2}}  =  4\cstloss   \cdot \lbddst{\bvCt^1}{\bvCt^2} \, .\label{eq:Cl-diff-Gamma-lbddst-bound}  
\end{align}
Thus
\begin{align}
\lbddst{\bvGamma_1}{\bvGamma_2} = \sup_{t \in [0, T]} e^{-\lambda t} \norm{\bvGamma^t_1 - \bvGamma^t_2} \leq 4 \cstloss \cdot \lbddst{\bvCt^1}{\bvCt^2} \, .
\end{align}
Next we focus on the $\lambda$-distance between $\bvSl^1$ and $\bvSl^2$. For any $0 \leq s < t \leq T$, we have
\begin{align}
\bvSl^i(t,s) = \Ep \brk{\frac{\partial \vl_t(\vr_i^t; \veps)}{\partial \vw^s}} = \Ep \brk{\nabla_r \ell_t(r^t;z) \cdot \frac{\partial \vr_i^t}{\partial \vw^s}} \, ,
\end{align}
which implies
\begin{align}
& \norm{\bvSl^1(t,s) - \bvSl^2(t,s)} \nonumber \\
& \leq \Ep \brk{\norm{\nabla_r \ell_t(\vr^t_1;z) \cdot \frac{\partial \vr^t_1}{\partial \vw^s} - \nabla_r \ell_t(\vr^t_2;z) \cdot \frac{\partial \vr^t_2}{\partial \vw^s}}} \nonumber \\
& \leq \sqrt{\Ep \brk{\norm{\frac{\partial \vr^t_1}{\partial \vw^s}}^2}} \cdot \sqrt{\Ep \brk{\norm{\nabla_r \ell_t(\vr^t_1;z) - \nabla_r \ell_t(\vr^t_2;z)}^2}} + \cstloss \cdot \Ep \brk{\norm{\frac{\partial \vr^t_1}{\partial \vw^s} - \frac{\partial \vr^t_2}{\partial \vw^s}}} \, , 
\end{align}
where in the last line we use the fact that $\norm{\nabla_r \ell_t(\vr^t_2;z)} \leq \cstloss$. Taking in Eq.~\eqref{eq:Cl-diff-Gamma-lbddst-bound}, we have
\begin{align}
e^{-\lambda t}\norm{\bvSl^1(t,s) - \bvSl^2(t,s)} &\leq e^{-\lambda t} \Ep \brk{\norm{\frac{\partial \vl_t(\vr^t_1;\veps)}{\partial \vw^s} - \frac{\partial \vl_t(\vr^t_2;\veps)}{\partial \vw^s}}}\nonumber \\
& \leq \sqrt{\Ep \brk{\norm{\frac{\partial \vr^t_1}{\partial \vw^s}}^2}} \cdot 4 \cstloss  \cdot \lbddst{\bvCt^1}{\bvCt^2} + \cstloss \cdot e^{-\lambda t} \Ep \brk{\norm{\frac{\partial \vr^t_1}{\partial \vw^s} - \frac{\partial \vr^t_2}{\partial \vw^s}}}.\label{eq:Cl-diff-global-mid-1}
\end{align}
It only remains to bound the quantities $\sqrt{\Ep \brk{\norm{\frac{\partial \vr^t_1}{\partial \vw^s}}^2}}$ and $e^{-\lambda t}\Ep \brk{\norm{\frac{\partial \vr^t_1}{\partial \vw^s} - \frac{\partial \vr^t_2}{\partial \vw^s}}}$. Substituting in the definition of $\frac{\partial \vr^t_i}{\partial \vw^s}$ gives us 
\begin{align}
\norm{\frac{\partial \vr_i^t}{\partial \vw^s}} & \leq \frac{1}{\delta} \int_s^t \norm{\bvRt(t,s')} \norm{\frac{\partial \vl_{s'}(\vr^{s'}_i, \veps)}{\partial \vw^s}} \mathrm d s' + \frac{1}{\delta} \norm{\bvRt (t,s)} \cdot \norm{\nabla_r \vl_s(\vr^s_i, \veps)} \nonumber \\
& \leq \frac{1}{\delta} \int_s^t\Phi_{\vRt}(t-s') \cdot \norm{\nabla_r \vl_{s'}(\vr^{s'}_i, \veps)} \cdot  \norm{\frac{\partial \vr^{s'}_i}{\partial \vw^s}} \mathrm d s' + \frac{\cstloss}{\delta} \Phi_{\vRt}(t-s) \leq  \frac{\cstloss \Phi_{\vRt}(T)}{\delta} \prn{1+\int_s^t \norm{\frac{\partial \vr^{s'}_i}{\partial \vw^s}}  \mathrm d s'}\, .
\end{align}
Invoking Gronwall's inequality gives the upper-bound
\begin{align}
\norm{\frac{\partial \vr_i^t}{\partial \vw^s}} \leq \frac{\cstloss \Phi_{\vRt}(T)}{\delta} \cdot \exp \prn{\frac{\cstloss \Phi_{\vRt}(T)}{\delta} \prn{t-s}} \leq \frac{\cstloss \Phi_{\vRt}(T)}{\delta} \cdot \exp \prn{\frac{\cstloss T \Phi_{\vRt}(T)}{\delta}}\, , \label{eq:Cl-diff-drdw-upper-bound}
\end{align}
and thus
\begin{align}
\sqrt{\Ep \brk{\norm{\frac{\partial \vr^t_1}{\partial \vw^s}}^2}} \leq \frac{\cstloss \Phi_{\vRt}(T)}{\delta} \cdot \exp \prn{\frac{\cstloss T \Phi_{\vRt}(T)}{\delta}}. \label{eq:Cl-diff-global-mid-1.5}
\end{align}
On the other hand we have using Eq.~\eqref{eq:trsfrm-derivative-r},
\begin{align}
& e^{-\lambda t}\Ep \brk{\norm{\frac{\partial \vr^t_1}{\partial \vw^s} - \frac{\partial \vr^t_2}{\partial \vw^s}}}  = \frac 1 \delta \cdot e^{-\lambda t}\Ep \brk{\norm{ \int_s^t \bvRt(t,s') \prn{\frac{\partial \vl_{s'}(\vr^{s'}_1; \veps)}{\partial \vw^s} -\frac{\partial \vl_{s'}(\vr^{s'}_2; \veps)}{\partial \vw^s}} \mathrm d s'}} \nonumber \\
& \qquad  +\frac 1 \delta \cdot e^{-\lambda t} \Ep \brk{\norm{\bvRt (t,s) \prn{\nabla_r \vl_s(\vr^s_1; \veps) -  \nabla_r \vl_s(\vr^s_2; \veps)}}}. \label{eq:Cl-diff-global-mid-2}
\end{align}
We bound the two parts separately. First we have
\begin{align}
& e^{-\lambda t} \Ep \brk{\norm{\int_s^t \bvRt(t,s')  \prn{\frac{\partial \vl_{s'}(\vr^{s'}_1; \veps)}{\partial \vw^s}  - \frac{\partial \vl_{s'}(\vr^{s'}_2; \veps)}{\partial \vw^s} }\mathrm d s'}  } \nonumber \\
& \leq \int_s^t e^{-\lambda(t-s')} \norm{\bvRt(t,s')} \cdot e^{-\lambda s'} \Ep \brk{\norm{\frac{\partial \vl_{s'}(\vr^{s'}_1; \veps)}{\partial \vw^s} - \frac{\partial \vl_{s'}(\vr^{s'}_2; \veps)}{\partial \vw^s}}} \mathrm d s' \nonumber \\
& \leq \prn{\int_0^\infty e^{-\lambda t} \pRt(t) \de t} \cdot \sup_{s \leq s' \leq t} e^{-\lambda s'} \Ep \brk{\norm{\frac{\partial \vl_{s'}(\vr^{s'}_1; \veps)}{\partial \vw^s} - \frac{\partial \vl_{s'}(\vr^{s'}_2; \veps)}{\partial \vw^s}}} \, .\label{eq:Cl-diff-global-mid-3}
\end{align}
Next we get
\begin{align}
&e^{-\lambda t} \Ep \brk{\norm{\bvRt (t,s) \prn{\nabla_r \vl_s(\vr^s_1; \veps) - \nabla_r \vl_s(\vr^s_2; \veps)}}}  \leq e^{-\lambda (t-s)}\Phi_{\vRt}(t-s) \cdot  e^{-\lambda s} \Ep \brk{\norm{\nabla_r \vl_s(\vr^s_1; \veps) - \nabla_r \vl_s(\vr^s_2; \veps)}} \nonumber \\
& \leq  \Phi_{\vRt}(T) \cdot 4 \cstloss  \cdot \lbddst{\bvCt^1}{\bvCt^2}\, , \label{eq:Cl-diff-global-mid-4}
\end{align}
where we invoke Eq.~\eqref{eq:Cl-diff-Gamma-lbddst-bound} again in the last step. Take Eq.~\eqref{eq:Cl-diff-global-mid-3} and \eqref{eq:Cl-diff-global-mid-4} into \eqref{eq:Cl-diff-global-mid-2} and we get
\begin{align}
& e^{-\lambda t}\Ep \brk{\norm{\frac{\partial \vr^t_1}{\partial \vw^s} - \frac{\partial \vr^t_2}{\partial \vw^s}}} \nonumber \\
& \leq \frac{1}{\delta} \cdot\prn{\int_0^\infty e^{-\lambda t} \pRt(t) \de t} \cdot \sup_{s \leq s' \leq t} e^{-\lambda s'} \Ep \brk{\norm{\frac{\partial \vl_{s'}(\vr^{s'}_1; \veps)}{\partial \vw^s} - \frac{\partial \vl_{s'}(\vr^{s'}_2; \veps)}{\partial \vw^s}}} + \frac{4\cstloss \pRt(T)}{\delta} \cdot \lbddst{\bvCt^1}{\bvCt^2}\, .
\end{align}
Further we substitute Eq.~\eqref{eq:Cl-diff-global-mid-1.5} into Eq.~\eqref{eq:Cl-diff-global-mid-1}, we get
\begin{align}
&e^{-\lambda t} \Ep \brk{\norm{\frac{\partial \vl_t(\vr^t_1;\veps)}{\partial \vw^s} - \frac{\partial \vl_t(\vr^t_2;\veps)}{\partial \vw^s}}} \leq \sqrt{\Ep \brk{\norm{\frac{\partial \vr^t_1}{\partial \vw^s}}^2}} \cdot 4 \cstloss \cdot \lbddst{\bvCt^1}{\bvCt^2} + \cstloss \cdot e^{-\lambda t} \Ep \brk{\norm{\frac{\partial \vr^t_1}{\partial \vw^s} - \frac{\partial \vr^t_2}{\partial \vw^s}}} \nonumber \\
& \leq \frac{4 \cstloss^2 \Phi_{\vRt}(T)}{\delta} \cdot \exp \prn{\frac{\cstloss T \Phi_{\vRt}(T)}{\delta}} \cdot \lbddst{\bvCt^1}{\bvCt^2} \nonumber \\
& \qquad + \frac{\cstloss}{\delta} \cdot\prn{\int_0^\infty e^{-\lambda t} \pRt(t) \de t}  \cdot \sup_{s \leq s' \leq t} e^{-\lambda s'} \Ep \brk{\norm{\frac{\partial \vl_{s'}(\vr^{s'}_1; \veps)}{\partial \vw^s} - \frac{\partial \vl_{s'}(\vr^{s'}_2; \veps)}{\partial \vw^s}}}  + \frac{4\cstloss^2 \Phi_{\vRt}(T)}{\delta} \cdot \lbddst{\bvCt^1}{\bvCt^2} \nonumber \\
& \leq \frac{\cstloss}{\delta} \cdot \prn{\int_0^\infty e^{-\lambda t} \pRt(t) \de t} \cdot \sup_{s \leq s' \leq t} e^{-\lambda s'} \Ep \brk{\norm{\frac{\partial \vl_{s'}(\vr^{s'}_1; \veps)}{\partial \vw^s} - \frac{\partial \vl_{s'}(\vr^{s'}_2; \veps)}{\partial \vw^s}}} \nonumber \\
& \qquad + \prn{\frac{4\cstloss^2 \Phi_{\vRt}(T)}{\delta} \cdot \exp \prn{\frac{\cstloss T \Phi_{\vRt}(T)}{\delta}}+ \frac{4\cstloss^2 \Phi_{\vRt}(T)}{\delta}}  \cdot \lbddst{\bvCt^1}{\bvCt^2}.
\end{align}
Additionally we can also choose $\lambda$ such that $ \frac{\cstloss}{\delta} \cdot \prn{\int_0^\infty e^{-\lambda t} \pRt(t) \de t}  \leq \frac 1 2$ and using the fact that the right hand side of the inequality is increasing in $t$, we can get
\begin{align}
&e^{-\lambda t} \Ep \brk{\norm{\frac{\partial \vl_t(\vr^t_1;\veps)}{\partial \vw^s} - \frac{\partial \vl_t(\vr^t_2;\veps)}{\partial \vw^s}}} \leq \sup_{s \leq s' \leq t} e^{-\lambda s'} \Ep \brk{\norm{\frac{\partial \vl_{s'}(\vr^{s'}_1; \veps)}{\partial \vw^s} - \frac{\partial \vl_{s'}(\vr^{s'}_2; \veps)}{\partial \vw^s}}} \nonumber \\
& \leq \prn{\frac{8\cstloss^2 \Phi_{\vRt}(T)}{\delta} \cdot \exp \prn{\frac{\cstloss T \Phi_{\vRt}(T)}{\delta}}+ \frac{8\cstloss^2 \Phi_{\vRt}(T)}{\delta}} \cdot \lbddst{\bvCt^1}{\bvCt^2}.
\end{align}
From Eq.~\eqref{eq:Cl-diff-global-mid-1}, it follows that
\begin{align}
\lbddst{\bvSl^1}{\bvSl^2} & \leq \sup_{0 \leq s < t \leq T} e^{-\lambda t} \Ep \brk{\norm{\frac{\partial \vl_t(\vr^t_1;\veps)}{\partial \vw^s} - \frac{\partial \vl_t(\vr^t_2;\veps)}{\partial \vw^s}}_2}  \nonumber \\
& \leq  \prn{\frac{8\cstloss^2 \Phi_{\vRt}(T)}{\delta} \cdot \exp \prn{\frac{\cstloss T \Phi_{\vRt}(T)}{\delta}}+ \frac{8\cstloss^2 \Phi_{\vRt}(T)}{\delta}} \cdot \lbddst{\bvCt^1}{\bvCt^2}.
\end{align}
The proof is completed by taking
\begin{align}
M := \max \brc{4 \cstloss , \prn{\frac{8\cstloss^2 \Phi_{\vRt}(T)}{\delta} \cdot \exp \prn{\frac{\cstloss T \Phi_{\vRt}(T)}{\delta}}+ \frac{8\cstloss^2 \Phi_{\vRt}(T)}{\delta}} }.
\end{align}

\subsubsection{Proof of Lemma \ref{lem:diff-Rtheta-global}} \label{proof:diff-Rtheta-global}
\paragraph{Controlling the distance between \texorpdfstring{$\bvCl^1$}{TEXT} and \texorpdfstring{$\bvCl^2$}{TEXT}.}
Since $\bvCt^1 = \bvCt^2$ on $[0,T]^2$, we can write the equations that define $\vr^1$ and $\vr^2$ as
\begin{equation}
\vr^t_i = - \frac 1 \delta \int_0^t \bvRt^i(t,s) \vl_s(\vr^s_i; \veps) \mathrm d s + \vw^t \, ,
\end{equation}
for $i=1,2$, where $\vw^t$ is a centered Gaussian process with autocovariance $\bvCt := \bvCt^1 = \bvCt^2$. For any $t \in [0,T]$, we have
\begin{align}
& e^{-\lambda t}\left\|\vr^t_1 - \vr^t_2\right\|_2 \nonumber \\
& \leq e^{-\lambda t} \prn{\frac 1 \delta \int_0^{t} \left\|\bvRt^1(t,s)\right\| \left\|\vl_s(\vr^s_1;\veps) - \vl_s(\vr^s_2;\veps)\right\|_2 \mathrm d s + \int_0^t  \left\|\bvRt^1(t,s) - \bvRt^2(t,s)\right\| \left\|\vl_s(\vr^s_2; \veps)\right\|_2\mathrm d s} \nonumber \\
& \leq \frac \cstloss \delta \int_0^{t} e^{-\lambda(t-s)} \Phi_{\vRt}(t-s) \cdot e^{-\lambda s}\left\|\vr^s_1 - \vr^s_2\right\|_2 \mathrm d s + \prn{ \int_0^t   \left\|\vl_s(\vr^s_2; \veps)\right\|_2\mathrm d s} \cdot  \sup_{0 \leq s \leq t \leq T} e^{-\lambda t}\left\|\bvRt^1(t,s) - \bvRt^2(t,s)\right\|  \nonumber \\
& \leq \frac{\cstloss }{\delta} \cdot \prn{\int_0^\infty e^{-\lambda t} \pRt(t) \de t} \cdot \sup_{0 \leq s \leq t}  e^{-\lambda s}\left\|\vr^s_1 - \vr^s_2\right\|_2 + \prn{ \int_0^t   \left\|\vl_s(\vr^s_2; \veps)\right\|_2\mathrm d s} \cdot \lbddst{\bvRt^1}{\bvRt^2}  \, .
\end{align}
The right hand side is increasing in $t$ and therefore
\begin{align}
& \sup_{0 \leq s \leq t}  e^{-\lambda s}\left\|\vr^s_1 - \vr^s_2\right\|_2 \nonumber \\
& \leq \frac{\cstloss }{\delta} \cdot \prn{\int_0^\infty e^{-\lambda t} \pRt(t) \de t} \cdot \sup_{0 \leq s \leq t}  e^{-\lambda s}\left\|\vr^s_1 - \vr^s_2\right\|_2 + \prn{ \int_0^t   \left\|\vl_s(\vr^s_2; \veps)\right\|_2\mathrm d s} \cdot \lbddst{\bvRt^1}{\bvRt^2} \, ,
\end{align}
which by choosing $\lambda$ large enough such that $\frac{\cstloss }{\delta} \cdot \prn{\int_0^\infty e^{-\lambda t} \pRt(t) \de t} \leq \frac 1 2$ yields
\begin{align}
e^{-\lambda t}\left\|\vr^t_1 - \vr^t_2\right\|_2 \leq \sup_{0 \leq s \leq t}  e^{-\lambda s}\left\|\vr^s_1 - \vr^s_2\right\|_2 \leq 2 \prn{ \int_0^t   \left\|\vl_s(\vr^s_2; \veps)\right\|_2\mathrm d s} \cdot \lbddst{\bvRt^1}{\bvRt^2} \, .
\end{align}
Therefore
\begin{align}
& e^{-\lambda t} \sqrt{\mathbb{E} \brk{\left\|\vr^t_1 - \vr^t_2\right\|_2^2}} \nonumber \\
& \leq 2 \sqrt{\Ep \brk{ \prn{ \int_0^t   \left\|\vl_s(\vr^s_2; \veps)\right\|_2\mathrm d s}^2}} \cdot \lbddst{\bvRt^1}{\bvRt^2}
\leq  2 \sqrt{\Ep \brk{ t \cdot \int_0^t   \left\|\vl_s(\vr^s_2; \veps)\right\|_2^2\mathrm d s}} \cdot \lbddst{\bvRt^1}{\bvRt^2} \nonumber \\
& \leq 2 \sqrt{kt \cdot \int_0^t \norm{\Ep \brk{\vl_s(\vr^s_2; \veps) \vl_s(\vr^s_2; \veps)^\sT} } \mathrm d s} \cdot \lbddst{\bvRt^1}{\bvRt^2} 
\leq 2 t \sqrt{k \Phi_{\vCl}(t)}  \cdot \lbddst{\bvRt^1}{\bvRt^2} \, ,
\end{align}
and consequently by Lipschitz continuity
\begin{align}
e^{-\lambda t} \sqrt{\mathbb{E} \brk{\left\|\nabla_r \vl_t(\vr^t_1; \veps) -\nabla_r \vl_t(\vr^t_1; \veps) \right\|_2^2}} \leq \cstloss \cdot e^{-\lambda t} \sqrt{\mathbb{E} \brk{\left\|\vr^t_1 - \vr^t_2\right\|_2^2}} \leq 2 \cstloss t \sqrt{k \Phi_{\vCl}(t)}  \cdot \lbddst{\bvRt^1}{\bvRt^2} \, . \label{eq:rt-lbddst-bound}
\end{align}
We then use the same argument in Appendix~\ref{proof:diff-Cl-global} which gives us
\begin{align}
\lbddst{\bvCl^1}{\bvCl^2} \leq \sup_{t \in [0,T]} e^{-\lambda t} \sqrt{\mathbb{E} \brk{\left\|\nabla_r \vl_t(\vr^t_1; \veps) -\nabla_r \vl_t(\vr^t_1; \veps) \right\|_2^2}} \leq 2 \cstloss t \sqrt{k \Phi_{\vCl}(t)} \cdot \lbddst{\bvRt^1}{\bvRt^2} \, .
\end{align}

\paragraph{Controlling the distances between \texorpdfstring{$\bvRl^1$}{TEXT} and \texorpdfstring{$\bvRl^2$}{TEXT}, \texorpdfstring{$\bGamma_1$}{TEXT} and \texorpdfstring{$\bGamma_2$}{TEXT}.} From Eq.~\eqref{eq:trsfrm-derivative-l} we can obtain for any $0 \leq s \leq t \leq T$ and $i=1,2$,
\begin{align}
\frac{\partial \vl_t(\vr_i^t; \veps)}{\partial w^s} & = \nabla_r \ell_t(r^t;z) \cdot \frac{\partial \vr_i^t}{\partial \vw^s}\, , \\
\frac{\partial \vr_i^t}{\partial \vw^s} & := - \frac 1 \delta \int_s^t \bvRt^i(t,s') \frac{\partial \vl_{s'}(\vr^{s'}_i; \veps)}{\partial \vw^{s'}} \mathrm d s' - \frac 1 \delta \bRt^i(t,s) \nabla_r \ell_s(r_i^s;z) \, . \label{eq:trsfrm-derivative-r-alternate}
\end{align}
and by Eq.~\eqref{eq:trsfrm-Gamma},
\begin{align}
\bvGamma^t_i = \mathbb{E} \brk{\nabla_r\vl_t(\vr_i^t, \veps)} \, .
\end{align}
First, for any $\lambda$ satisfying $\frac{\cstloss }{\delta} \cdot \prn{\int_0^\infty e^{-\lambda t} \pRt(t) \de t} \leq \frac 1 2$, we have
\begin{align}
e^{-\lambda t} \norm{\bvGamma^t_1 - \bvGamma^t_2}_2 & \leq e^{-\lambda t} \mathbb{E} \brk{\norm{\nabla_r \vl_t(\vr^t_1; \veps)- \nabla_r \vl_t(\vr^t_2; \veps)}} \leq e^{-\lambda t} \sqrt{\mathbb{E} \brk{\norm{\nabla_r \vl_t(\vr^t_1; \veps)- \nabla_r \vl_t(\vr^t_2; \veps)}^2}} \nonumber \\
& =  2 \cstloss T \sqrt{k \Phi_{\vCl}(T)}  \cdot \lbddst{\bvRt^1}{\bvRt^2}\, , \label{eq:Gamma-lbddst-bound}
\end{align}
where in the last line we invoke Eq.~\eqref{eq:rt-lbddst-bound}. Thus
\begin{align}
\lbddst{\bvGamma^1}{\bvGamma^2} = \sup_{t \in [0, T]} e^{-\lambda t} \norm{\bvGamma^t_1 - \bvGamma^t_2}_2 \leq 2 \cstloss T \sqrt{k \Phi_{\vCl}(T)}  \cdot \lbddst{\bvRt^1}{\bvRt^2} \, .
\end{align}
Next we turn to the $(\lambda,T)$-distance between $\bvSl^1$ and $\bvSl^2$. Note that for any $0 \leq s < t \leq T$, we have
\begin{align}
\bvSl^i(t,s) = \Ep \brk{\frac{\partial \vl_t(\vr^t_i;\veps)}{\partial \vw^s}} = \Ep \brk{\nabla_r \vl_t(\vr^t_i; \veps) \cdot \frac{\partial \vr^t_i}{\partial \vw^s}}\, ,
\end{align}
which gives us that
\begin{align}
& \norm{\bvSl^1(t,s) - \bvSl^2(t,s)} \leq \Ep \brk{\norm{\frac{\partial \vl_t(\vr^t_1; \veps)}{\partial \vw^s} - \frac{\partial \vl_t(\vr^t_2;\veps)}{\partial \vw^s}}} \nonumber \\
& \leq \Ep \brk{\norm{\nabla_r \vl_t(\vr^t_1; \veps) - \nabla_r \vl_t(\vr^t_2; \veps)} \cdot \norm{\frac{\partial \vr^t_1}{\partial \vw^s}} + \norm{\nabla_r \vl_t(\vr^t_2; \veps)} \cdot  \norm{\frac{\partial \vr^t_1}{\partial \vw^s} - \frac{\partial \vr^t_2}{\partial \vw^s}}} \nonumber \\
& \leq \sqrt{\Ep \brk{\norm{\frac{\partial \vr^t_1}{\partial \vw^s}}^2}} \cdot \sqrt{\Ep \brk{\norm{\nabla_r \vl_t(\vr^t_1; \veps) - \nabla_r \vl_t(\vr^t_2; \veps)}^2}} + \cstloss \cdot \Ep \brk{\norm{\frac{\partial \vr^t_1}{\partial \vw^s} - \frac{\partial \vr^t_2}{\partial \vw^s}}} \, ,
\end{align}
where in the last line we use Cauchy-Schwarz inequality and  $\norm{\nabla_r \vl_t(\vr^t_2; \veps)} \leq \cstloss$. Taking in Eq.~\eqref{eq:Gamma-lbddst-bound}, we can have for all $\lambda$ that $\frac{\cstloss }{\delta} \cdot \prn{\int_0^\infty e^{-\lambda t} \pRt(t) \de t} \leq \frac 1 2$,
\begin{align}
& e^{-\lambda t}\norm{\bvSl^1(t,s) - \bvSl^2(t,s)} \leq e^{-\lambda t} \Ep \brk{\norm{\frac{\partial \vl_t(\vr^t_1; \veps)}{\partial \vw^s} - \frac{\partial \vl_t(\vr^t_2;\veps)}{\partial \vw^s}}}\nonumber \\
& \leq \sqrt{\Ep \brk{\norm{\frac{\partial \vr^t_1}{\partial \vw^s}}^2}} \cdot 2 \cstloss T \sqrt{k \Phi_{\vCl}(T)}  \cdot \lbddst{\bvRt^1}{\bvRt^2} + \cstloss \cdot e^{-\lambda t} \Ep \brk{\norm{\frac{\partial \vr^t_1}{\partial \vw^s} - \frac{\partial \vr^t_2}{\partial \vw^s}}}\, .\label{eq:Rl-diff-global-mid-1}
\end{align}
It only remains to bound the quantities $\sqrt{\Ep \brk{\norm{\frac{\partial \vr^t_1}{\partial \vw^s}}^2}}$ and $e^{-\lambda t}\Ep \brk{\norm{\frac{\partial \vr^t_1}{\partial \vw^s} - \frac{\partial \vr^t_2}{\partial \vw^s}}}$. From Eq.~\eqref{eq:trsfrm-derivative-r-alternate} we have
\begin{align}
\norm{\frac{\partial \vr^t_i}{\partial \vw^s}} & \leq \frac{1}{\delta} \int_s^t \norm{\bvRt^i(t,s')} \norm{\frac{\partial \vl_{s'}(\vr_{s'}^i; \veps)}{\partial \vw^s}} \mathrm d s' + \frac{1}{\delta} \norm{\bvRt^i (t,s)}_2 \cdot \norm{\nabla_r \vl_s(\vr^s_i; \veps)}_2 \nonumber \\
& \leq \frac{1}{\delta} \int_s^t\Phi_{\vRt}(t-s') \cdot  \norm{\nabla_r \vl_{s'}(\vr^{s'}_i; \veps)} \cdot \norm{\frac{\partial \vr_{s'}^i}{\partial \vw^s}} \mathrm d s' + \frac{\cstloss}{\delta} \Phi_{\vRt}(t-s) \nonumber \\
& \leq  \frac{\cstloss \Phi_{\vRt}(T)}{\delta} \prn{1+\int_s^t \norm{\frac{\partial \vr_{s'}^i}{\partial \vw^s}} \mathrm d s'}\, .
\end{align}
This allows us to invoke Gronwall's inequality, giving an non-random upper bound
\begin{align}
\norm{\frac{\partial \vr^t_i}{\partial \vw^s}} \leq \frac{\cstloss \Phi_{\vRt}(T)}{\delta} \cdot \exp \prn{\frac{\cstloss \Phi_{\vRt}(T)}{\delta} \prn{t-s}} \leq \frac{\cstloss \Phi_{\vRt}(T)}{\delta} \cdot \exp \prn{\frac{\cstloss T \Phi_{\vRt}(T)}{\delta}}\, , \label{eq:drdw-upper-bound}
\end{align}
and thus
\begin{align}
\sqrt{\Ep \brk{\norm{\frac{\partial \vr^t_1}{\partial \vw^s}}^2}} \leq \frac{\cstloss \Phi_{\vRt}(T)}{\delta} \cdot \exp \prn{\frac{\cstloss T \Phi_{\vRt}(T)}{\delta}}\, . \label{eq:Rl-diff-global-mid-1.5}
\end{align}
On the other hand we have
\begin{align}
& e^{-\lambda t}\Ep \brk{\norm{\frac{\partial \vr^t_1}{\partial \vw^s} - \frac{\partial \vr^t_2}{\partial \vw^s}}} = \frac 1 \delta \cdot e^{-\lambda t}\Ep \brk{\norm{ \int_s^t \prn{\bvRt^1(t,s') \frac{\partial \vl_{s'}(\vr^{s'}_1; \veps)}{\partial \vw^s} - \bvRt^2(t,s') \frac{\partial \vl_{s'}(\vr^{s'}_2; \veps)}{\partial \vw^s}} \mathrm d s'}} \nonumber \\
& \qquad +\frac 1 \delta \cdot e^{-\lambda t} \Ep \brk{\norm{\bvRt^1 (t,s) \nabla_r \vl_r(\vr^s_1; \veps) - \bvRt^2(t,s) \nabla_r \vl_r(\vr^s_2; \veps)}}\, . \label{eq:Rl-diff-global-mid-2}
\end{align}
We bound the two parts respectively, first we have
\begin{align}
&e^{-\lambda t}\Ep \brk{\norm{ \int_s^t \prn{\bvRt^1(t,s') \frac{\partial \vl_{s'}(\vr^{s'}_1; \veps)}{\partial \vw^s} - \bvRt^2(t,s') \frac{\partial \vl_{s'}(\vr^{s'}_2; \veps)}{\partial \vw^s}} \mathrm d s'}} \nonumber \\
& \leq e^{-\lambda t} \Ep \brk{\int_s^t\norm{ \bvRt^1(t,s')  \prn{\frac{\partial \vl_{s'}(\vr^{s'}_1; \veps)}{\partial \vw^s}  - \frac{\partial \vl_{s'}(\vr^{s'}_2; \veps)}{\partial \vw^s} }} \mathrm d s' }  +e^{-\lambda t} \Ep \brk{\int_s^t\norm{ \prn{\bvRt^1(t,s') -  \bvRt^2(t,s')} \frac{\partial \vl_{s'}(\vr^{s'}_2; \veps)}{\partial \vw^s} } \mathrm d s' } \nonumber \\
& \leq \int_s^t e^{-\lambda(t-s')} \Phi_{\vRt}(t-s')\cdot e^{-\lambda s'} \Ep \brk{\norm{\frac{\partial \vl_{s'}(\vr^{s'}_1;\veps)}{\partial \vw^s} - \frac{\partial \vl_{s'}(\vr^{s'}_2;\veps)}{\partial \vw^s}}} \mathrm d s' \nonumber \\
& \qquad + \cstloss  \int_s^t  e^{-\lambda t} \norm{\bvRt^1(t,s') -  \bvRt^2(t,s')} \cdot \Ep \brk{\norm{\frac{\partial \vr_{s'}^2}{\partial \vw^s}}} \mathrm d s' \nonumber \\
& \leq \prn{\int_0^\infty e^{-\lambda t} \pRt(t) \de t} \cdot \sup_{s <s' \leq t} e^{-\lambda s'} \Ep \brk{\norm{\frac{\partial \vl_{s'}(\vr^{s'}_1;\veps)}{\partial \vw^s} - \frac{\partial \vl_{s'}(\vr^{s'}_2;\veps)}{\partial \vw^s}}} \nonumber \\
& \qquad  + \frac{\cstloss^2 T \Phi_{\vRt}(T)}{\delta} \cdot \exp \prn{\frac{\cstloss T \Phi_{\vRt}(T)}{\delta}} \cdot \lbddst{\bvRt^1}{\bvRt^2}\, ,\label{eq:Rl-diff-global-mid-3}
\end{align}
where in the last line we use the upper bound from Eq.~\eqref{eq:drdw-upper-bound}. For the second term in Eq.~\eqref{eq:Rl-diff-global-mid-2} we have
\begin{align}
&e^{-\lambda t} \Ep \brk{\norm{\bvRt^1 (t,s) \nabla_r \vl_r(\vr^s_1; \veps) - \bvRt^2(t,s) \nabla_r \vl_r(\vr^s_2; \veps)}} \nonumber \\
& \leq e^{-\lambda t} \Ep \brk{\norm{\prn{\bvRt^1 (t,s) - \bvRt^2 (t,s)} \nabla_r \vl_r(\vr^s_1; \veps)}} + e^{-\lambda t} \Ep \brk{\norm{\bvRt^2 (t,s) \prn{\nabla_r \vl_r(\vr^s_1; \veps) - \nabla_r \vl_r(\vr^s_2; \veps)}}} \nonumber \\
& \leq \cstloss \cdot \lbddst{\bvRt^1}{\bvRt^2} + \Phi_{\vRt}(T) \cdot 2 \cstloss T \sqrt{k \Phi_{\vCl}(T)}  \cdot \lbddst{\bvRt^1}{\bvRt^2}\, , \label{eq:Rl-diff-global-mid-4}
\end{align}
where we invoke Eq.~\eqref{eq:Gamma-lbddst-bound} in the last line. Define
\begin{align}
\wb{M}_1 & := \frac{\cstloss^2 T \Phi_{\vRt}(T)}{\delta} \cdot \exp \prn{\frac{\cstloss T \Phi_{\vRt}(T)}{\delta}} + \cstloss +  \Phi_{\vRt}(T) \cdot 2 \cstloss T \sqrt{k \Phi_{\vCl}(T)}\, ,\\
\wb{M}_2& := \frac{\cstloss \Phi_{\vRt}(T)}{\delta} \cdot \exp \prn{\frac{\cstloss T \Phi_{\vRt}(T)}{\delta}} \cdot 2\cstloss T \sqrt{k \Phi_{\vCl}(T)} + \frac{\cstloss}{\delta} \wb{M}_1\, ,
\end{align}
and take Eqs.~\eqref{eq:Rl-diff-global-mid-3} and \eqref{eq:Rl-diff-global-mid-4} into \eqref{eq:Rl-diff-global-mid-2} and we get
\begin{align}
& e^{-\lambda t}\Ep \brk{\norm{\frac{\partial \vr^t_1}{\partial \vw^s} - \frac{\partial \vr^t_2}{\partial \vw^s}}} \nonumber \\
& \leq \frac{1}{\delta} \cdot \prn{\int_0^\infty e^{-\lambda t} \pRt(t) \de t} \cdot \sup_{s <s' \leq t} e^{-\lambda s'} \Ep \brk{\norm{\frac{\partial \vl_{s'}(\vr^{s'}_1;\veps)}{\partial \vw^s} - \frac{\partial \vl_{s'}(\vr^{s'}_2;\veps)}{\partial \vw^s}}} + \frac{1}{\delta} \wb{M}_1 \cdot \lbddst{\bvRt^1}{\bvRt^2} \, .
\end{align}
Further substituting Eq.~\eqref{eq:Rl-diff-global-mid-1.5} into Eq.~\eqref{eq:Rl-diff-global-mid-1}, it follows that
\begin{align}
&e^{-\lambda t} \Ep \brk{\norm{\frac{\partial \vl_t(\vr^t_1; \veps)}{\partial \vw^s} - \frac{\partial \vl_t(\vr^t_2;\veps)}{\partial \vw^s}}}\nonumber \\
& \leq \sqrt{\Ep \brk{\norm{\frac{\partial \vr^t_1}{\partial \vw^s}}^2}} \cdot 2\cstloss T \sqrt{k \Phi_{\vCl}(T)}  \cdot \lbddst{\bvRt^1}{\bvRt^2} + \cstloss \cdot e^{-\lambda t} \Ep \brk{\norm{\frac{\partial \vr^t_1}{\partial \vw^s} - \frac{\partial \vr^t_2}{\partial \vw^s}}} \nonumber \\
& \leq \frac{\cstloss \Phi_{\vRt}(T)}{\delta} \cdot \exp \prn{\frac{\cstloss T \Phi_{\vRt}(T)}{\delta}} \cdot 2\cstloss T \sqrt{k \Phi_{\vCl}(T)}  \cdot  \lbddst{\bvRt^1}{\bvRt^2} \nonumber \\
& \qquad + \frac{\cstloss}{\delta} \cdot \prn{\int_0^\infty e^{-\lambda t} \pRt(t) \de t} \cdot \sup_{s <s' \leq t} e^{-\lambda s'} \Ep \brk{\norm{\frac{\partial \vl_{s'}(\vr^{s'}_1;\veps)}{\partial \vw^s} - \frac{\partial \vl_{s'}(\vr^{s'}_2;\veps)}{\partial \vw^s}}}  + \frac{\cstloss}{\delta} \wb{M}_1 \cdot \lbddst{\bvRt^1}{\bvRt^2} \nonumber \\
& \leq \frac{\cstloss}{\delta} \cdot\prn{\int_0^\infty e^{-\lambda t} \pRt(t) \de t} \cdot \sup_{s <s' \leq t} e^{-\lambda s'} \Ep \brk{\norm{\frac{\partial \vl_{s'}(\vr^{s'}_1;\veps)}{\partial \vw^s} - \frac{\partial \vl_{s'}(\vr^{s'}_2;\veps)}{\partial \vw^s}}} + \wb{M}_2 \cdot \lbddst{\bvRt^1}{\bvRt^2} \, .
\end{align}
Recall that we choose $\lambda$ such that $ \frac{\cstloss}{\delta} \cdot\prn{\int_0^\infty e^{-\lambda t} \pRt(t) \de t} \leq \frac 1 2$ and the right hand side of the inequality is increasing in $t$, we can get
\begin{align}
e^{-\lambda t} \Ep \brk{\norm{\frac{\partial \vl_t(\vr^t_1; \veps)}{\partial \vw^s} - \frac{\partial \vl_t(\vr^t_2;\veps)}{\partial \vw^s}}_2} 
\leq 2 \wb{M}_2 \cdot \lbddst{\bvRt^1}{\bvRt^2} \, .
\end{align}
Again by Eq.~\eqref{eq:Rl-diff-global-mid-1},
\begin{align}
\lbddst{\bvSl^1}{\bvSl^2} & \leq \sup_{0 \leq s < t \leq T} e^{-\lambda t} \Ep \brk{\norm{\frac{\partial \vl_t(\vr^t_1; \veps)}{\partial \vw^s} - \frac{\partial \vl_t(\vr^t_2;\veps)}{\partial \vw^s}}_2}  \leq  2 \wb{M}_2 \cdot \lbddst{\bvRt^1}{\bvRt^2} \, .
\end{align}
The proof is completed by taking
\begin{align}
M := \max \brc{2\cstloss T \sqrt{k \Phi_{\vCl}(T)}, 2\wb{M}_2} \, .
\end{align}

\section{Auxiliary lemmas for the proof of Theorem \ref{thm:StateEvolution}}

\subsection{Proof of Lemma~\ref{lem:flow-approximation}} \label{proof:flow-approximation} 

Claim \eqref{eq:FlowApproxFirstClaim} immediately follows from 
basic results about Euler method, see, for instance,
\cite[Theorem II.3.6]{hairer1996solving}.
In order to apply these results, letting 
$\bF(\btheta,t) :=  -\btheta\Lambda^{t, \sT} -\frac{1}{\delta}
\bX^{\sT}\bell_t(\bX\btheta;\bz)$, we need to check that the following two 
conditions hold with high probability for some constants $L$, $C$ possibly dependent on $T$
but not on $n,d$:
\begin{enumerate}
\item $(\btheta,t)\mapsto \bF(\btheta,t)$ is $L$-Lipschitz, for a constant $L$ independent 
of $n,d$. This holds by the Lipschitz continuity of
$(u,t)\mapsto \ell_t(u;z)$ and $t\mapsto\Lambda_t$, see Assumptions \ref{ass:Normal}
and because $\|\bX\|_{\op}\le C$ with high probability.
\item Along the trajectory $\btheta^t$, we have 
$\|\btheta^{t+\eta}-\btheta^t-\eta \bF(\btheta^t,t)\|\le C_0\eta^2\sqrt{d}$. 

To prove this note that, with high probability $\|\bF(\btheta,t)\|\le C_1\sqrt{d}+C_2\|\btheta\|$
(this follows from the Lipschitz continuity, and a simple bound on $\|\bF(\bfzero,t)\|$).
This implies $\|\btheta^t\|\le C_3\exp(C_3t)$ by Gronwall, whence
$\|\btheta^{t+\eta}-\btheta^t\|\le C_4\eta\sqrt{d}$ for any $t\le T$. 
Finally, the claimed bound
$\|\btheta^{t+\eta}-\btheta^t-\eta \bF(\btheta^t,t)\|\le C_0\eta^2\sqrt{d}$
follows by using once more the Lipschitz property of $\bF$.
\end{enumerate}

Since
\begin{align}
\wdst{\est{\mu}_{\theta^{\tau_1}, \cdots, \theta^{\tau_m}}}{\est{\mu}_{\theta^{\tau_1}_\eta, \cdots, \theta^{\tau_m}_\eta}} 	& \leq \sqrt{\frac{1}{d} \sum_{j=1}^d \sum_{l=1}^m \norm{\btheta^{\tau_l}_j - \prn{\btheta^{\tau_l}_\eta}_j}_2^2} = \sqrt{\frac 1 d \sum_{l=1}^m \norm{\btheta^{\tau_l} - \btheta^{\tau_l}_\eta}_F^2} \, ,
\end{align}
the the second claim of the lemma follows immediately.

\subsection{Proof of Lemma~\ref{lem:state-evolution-discrete-flow}} \label{proof:state-evolution-discrete-flow}
We introduce the following approximate message passing (AMP) algorithm that admits an asymptotic characterization by state evolution. For sequences of Lipschitz functions $f_{i}: \reals^{k(i+1) + 1} \to \reals^{k}$ and $g_{i}: \reals^{k(i+1)} \to \reals^k$ with $i=0,1,\cdots$, we consider the following matrix sequences $\brc{\boldsymbol{a}^{i+1}, \boldsymbol{b}^i}_{i \geq 0}$ in $\reals^{d \times k}$ and $\reals^{n \times k}$ respectively, generated by
\begin{align}
\boldsymbol{a}^{i+1} & = -\frac 1 \delta \bX^\sT \boldsymbol{f}_i(\boldsymbol{b}^0, \cdots, \boldsymbol{b}^i; \bz) + \sum_{j=0}^i \boldsymbol{g}_j(\ba^1, \cdots, \ba^j; \btheta^0) \xi_{i,j} \, , \\
\bb^{i} & = \bX \boldsymbol{g}_i(\ba^1, \cdots, \ba^i; \btheta^0) + \frac1 \delta \sum_{j=0}^{i-1} \boldsymbol{f}_j (\boldsymbol{b}^0, \cdots, \boldsymbol{b}^j; \bz) \zeta_{i,j} \, , 
\end{align} 
where $\boldsymbol{f}_i, \boldsymbol{g}_i$ are functions that apply $f_i, g_i$ row-wise similar to $\bell_t$. $\brc{\xi_{i,j}}_{0 \leq j \leq i}$ and $\brc{\zeta_{i,j}}_{0 \leq j \leq i-1}$ are sequences of deterministic matrices in $\reals^{k \times k}$ that depend on the $\brc{f_i, g_i}_{i \geq 0}$ in a specific way that we shall explicitly define later. The algorithm is initialized by $\boldsymbol{g}_0(\btheta^0) = \btheta^0, \bb^0 = \bX \btheta^0$. To relate this AMP algorithm with the discretized flow system $\mathfrak{F}^\eta$, we consider the specific choice of
\begin{align}
\boldsymbol{g}_i(\ba^1, \cdots, \ba^i; \btheta^0) & := \btheta^{t_i}_\eta \, , \\
\boldsymbol{f}_i(\bb^0, \cdots, \bb^i; \bz) & := \bell_{t_i}(\bX \btheta^{t_i}_\eta; \bz) \, ,
\end{align}
where $t_i = i \eta$. We next show that $\btheta^{t_i}_\eta$ is indeed a function of $\ba^1, \cdots, \ba^i, \btheta^0$ and $- \bell_{t_i}(\bX \btheta^{t_i}_\eta; \bz)$ is indeed a function of $\bb^0, \cdots, \bb^i, \bz$. This can be seen by induction
\begin{align}
\btheta^{t_i}_\eta & = \btheta^{t_{i-1}}_\eta + \eta \cdot \brc{ -\btheta^{t_{i-1}}_\eta \Lambda^{t_{i-1}, \sT} -\frac{1}{\delta} \bX^{\sT}\bell_{t_{i-1}}(\bX\btheta^{t_{i-1}}_\eta;\bz)} \nonumber \\
&  = \btheta^{t_{i-1}}_\eta + \eta \cdot \brc{ -\btheta^{t_{i-1}}_\eta \Lambda^{t_{i-1}, \sT} -\frac{1}{\delta} \bX^{\sT}\boldsymbol{f}_{i-1}(\bb^0, \cdots, \bb^{i-1}; \bz)} \nonumber  \\
& = \boldsymbol{g}_{i-1}(\ba^1, \cdots, \ba^{i-1}; \btheta^0) \prn{I - \eta \Lambda^{t_{i-1}, \sT}} + \eta \prn{\ba^i - \sum_{j=0}^{i-1} \boldsymbol{g}_j(\ba^1, \cdots, \ba^j; \btheta^0) \xi_{i-1,j} } \, , \label{eq:def-recursive-g} \\
\bell_{t_i}(\bX \btheta^{t_i}_\eta; \bz) & = \bell_{t_i}(\bX \boldsymbol{g}_i(\ba^1, \cdots, \ba^i; \btheta^0); \bz) \nonumber \\
& = \bell_{t_i}\prn{\bb^i -\frac1 \delta \sum_{j=0}^{i-1} \boldsymbol{f}_j (\boldsymbol{b}^0, \cdots, \boldsymbol{b}^j; \bz) \zeta_{i,j} ; \bz} \, \label{eq:def-recursive-f} . 
\end{align}
By Lipschitz property of $\ell_t$ in Assumption~\ref{ass:Normal}, we can see by this inductive definition, $g_i$ and $f_i$ are all Lipschitz continuous. To apply the standard AMP result in~\cite{chen2021universality} we only need to specify the matrices $\brc{\xi_{i,j}}_{0 \leq j \leq i}$ and $\brc{\zeta_{i,j}}_{0 \leq j \leq i-1}$. To this end we iteratively define sequences of centered Gaussian vectors $\brc{\wb{u}^{t_{i+1}}_\eta, \wb{w}^{t_i}_\eta}_{i \geq 0}$ in $\reals^k$ according to
\begin{subequations}
\begin{align}
	\Ep \brk{\wb{w}^{t_i}_\eta \prn{\wb{w}^{t_j}_\eta}^\sT} & =\Ep \brk{g_i(\wb{u}^{t_1}_\eta, \cdots, \wb{u}^{t_i}_\eta; \theta^0)g_j(\wb{u}^{t_1}_\eta, \cdots, \wb{u}^{t_j}_\eta; \theta^0)^\sT}\, , & & 0 \leq j \leq i < \infty\, , \label{eq:def-amp-w} \\
	\Ep \brk{\wb{u}^{t_{i+1}}_\eta \prn{\wb{u}^{t_{j+1}}_\eta}^\sT} & = \frac 1 \delta \Ep \brk{f_i(\wb{w}^{t_0}_\eta, \cdots, \wb{w}^{t_i}_\eta; z) f_j(\wb{w}^{t_0}_\eta, \cdots, \wb{w}^{t_j}_\eta;z)^\sT}\, , & & 0 \leq j \leq i <\infty \, , \label{eq:def-amp-u} \\ 
	\zeta_{i, j} &= \Ep \brk{\frac{\partial}{\partial \wb{u}^{t_{j+1}}_\eta} g_i(\wb{u}^{t_1}_\eta, \cdots, \wb{u}^{t_i}_\eta; \theta^0)}\, ,  & & 0 \leq j \leq i - 1 \, , \label{eq:def-amp-eta} \\
	\xi_{i,j} & = \Ep \brk{\frac{\partial}{\partial \wb{w}^{t_j}_\eta} f_i(\wb{w}^{t_0}_\eta, \cdots, \wb{w}^{t_i}_\eta; z)}\, , & & 0 \leq j \leq i \, . \label{eq:def-amp_zeta}
\end{align}
\end{subequations}
Here the expectation is taking over the Gaussian random vectors $\wb{u}^{t_i}_\eta, \wb{w}^{t_i}_\eta$ and also on the independently distributed random variables $(\theta^0, z) \sim \mu_{\theta^0, z}$. 

The above equations define inductively 
the matrices $\brc{\xi_{i,j}}_{0 \leq j \leq i}$, $\brc{\zeta_{i,j}}_{0 \leq j \leq i-1}$
and also the Gaussian vectors $\brc{\wb{u}^{t_{i+1}}_\eta, \wb{w}^{t_i}_\eta}_{i \geq 0}$. 
The sequence is initialized by $\wb{w}^{t_0}_\eta \sim \normal\prn{0, \Ep \brk{\theta^0 \prn{\theta^0}^\sT}}$ and $\xi_{0,0} = \Ep \brk{\frac{\partial}{\partial \wb{w}^{t_0}_\eta } f_0(\wb{w}^{t_0}_\eta; z)}$. Suppose for some $r=0,1,\cdots$, we have define $\wb{u}^{t_i}_\eta, \wb{w}^{t_i}_\eta$ and the matrices $\zeta_{i,j}, \xi_{i,j}$ for $i \leq r$. According to Eqs.~\eqref{eq:def-recursive-g} and \eqref{eq:def-recursive-f}, the functions $f_0, \cdots, f_{r+1}; g_0, \cdots, g_{r+1}$ are all explicitly defined. Substituting into Eq.~\eqref{eq:def-amp-u} we can then determine  $\wb{w}^{t_{r+1}}_\eta$ and next by Eq.~\eqref{eq:def-amp-w} we obtain $\wb{u}^{t_{r+1}}_\eta$. Finally, by Eqs.~\eqref{eq:def-amp-eta} and \eqref{eq:def-amp_zeta} the matrices $\zeta_{i,j}, \xi_{i,j}$ for $i=r+1$ are determined.

Under the conditions of Theorem~\ref{thm:StateEvolution}, we can invoke~\cite[Theorem 2.4]{chen2021universality}
and \cite[Theorem 1]{javanmard2013state} to 
obtain\footnote{Note that \cite[Theorem 1]{javanmard2013state} only considers AMP
algorithms on which the nonlinearities depends on the last iterate. However by enlarging the dimension $k$,
this also covers the case of nonlinearities depend on any constant number of
previous times. This reduction is explained in several earlier papers, e.g. 
\cite[Appendix A]{montanari2021optimization}.} that, for any fixed
$t_1 = \eta, \cdots, t_m = m \eta$ and any Lipschitz bounded function $\psi: \reals^{k(m+1)} \to \reals$,
\begin{align}
\frac 1 d \sum_{j=1}^d \psi\prn{\prn{\ba^{t_1}_\eta}_j, \cdots,\prn{\ba^{t_m}_\eta}_j; 
	\prn{\btheta^0}_j} \stackrel{p}\to \Ep \brk{\psi \prn{\wb{u}^{t_1}_\eta, \cdots, \wb{u}^{t_m}_\eta; \theta^0 }} \, . \label{eq:amp-convergence-pseudo-lipschitz-expectation}
	\end{align}
	
	For any  Lipschitz bounded function $\widetilde{\psi} : (\reals^{k})^{m+1} \to \reals$,
	define $\psi: (\reals^{k})^{m+1} \to \reals$ via
	\begin{align*}
\psi(\wb{u}^{t_1}_\eta, \cdots, \wb{u}^{t_m}_\eta;\theta^0) := \widetilde{\psi}(\wb{\theta}^{t_1}_\eta, \cdots, \wb{\theta}^{t_m}_\eta) \, ,
\wb{\theta}^{t_1}_\eta &:= g_{i_1}(\wb{u}^1_\eta, \cdots, \wb{u}^{i_1}_\eta; \theta^0)\, ,\\
&\cdots\\
\wb{\theta}^{t_m}_\eta & := g_{i_m}(\wb{u}^1_\eta, \cdots, \wb{u}^{i_m}_\eta; \theta^0)\, .
\end{align*}
By the Lipschitz property of $g_{i_1}, \cdots, g_{i_m}$,  $\psi$ is also Lipschitz bounded. 
We thus proved that, for any Lipschitz bounded function $\tpsi$,
\begin{align}
\frac 1 d \sum_{j=1}^d
\tpsi\prn{\prn{\btheta^{t_1}_\eta}_j, \cdots,\prn{\btheta^{t_m}_\eta}_j; \prn{\btheta^0}_j} \stackrel{p}\to \Ep \brk{\tpsi \prn{\wb{\theta}^{t_1}_\eta, \cdots, \wb{\theta}^{t_m}_\eta; \theta^0 }} \, . \label{eq:amp-convergence-pseudo-lipschitz-expectation-theta}
\end{align}

The next lemma relates the random variables 
$\wb{\theta}^{t_1}_\eta, \cdots, \wb{\theta}^{t_m}_\eta$ to the DMFT
system $\mathfrak{S}^\eta$. We defer the proof to Appendix~\ref{proof:AMP-to-S-eta}.
\begin{lemma} \label{lem:AMP-to-S-eta}
The discrete-time DMFT system $\mathfrak{S}^\eta$ has a unique solution in the space
$\mathcal{S}$ and $(\theta^{t}_\eta)_{t=i\eta, i\le m} \ed (\theta^{t}_\eta)_{t=i\eta, i\le m}$.
Further $t\mapsto \theta_\eta^t$ is piecewise linear with knots $t_i = i\eta$.
\end{lemma}

Fix $T$ and set $m=T/\eta$. 
By this lemma, and since $t\mapsto \btheta^t_{\eta}$ is also piecewise linear 
with knots at $t=i\eta$, Eq.~\eqref{eq:amp-convergence-pseudo-lipschitz-expectation-theta}
implies that, for any $\ell$, any $\tau_1,\dots \tau_{\ell}\in [0,T]$, and any 
bounded Lipschitz function $\psi:(\reals^k)^{\ell}\to\reals$, we have
\begin{align}
\frac 1 d \sum_{j=1}^d
\psi\prn{\prn{\btheta^{\tau_1}_\eta}_j, \cdots,\prn{\btheta^{\tau_\ell}_\eta}_j} \stackrel{p}\to \Ep \brk{\psi \prn{\wb{\theta}^{t_1}_\eta, \cdots, \wb{\theta}^{t_m}_\eta; \theta^0 }} \, . 
\end{align}
The proof is completed by applying the following basic fact about weak convergence
to the probability measures $\nu_n=\widehat{\mu}_{\theta^{\tau_1}_\eta, \cdots, \theta^{\tau_\ell}_\eta}$,
$\nu=\mu_{\theta^{\tau_1}_\eta, \cdots, \theta^{\tau_m}_\eta}$ on $\reals^d$,
$d=\ell k$.
\begin{lemma}
Let $(\nu_n)_{n\ge 1}$ be a sequence of random probability measures on $\reals^d$, and
assume that, for any bounded Lipschitz function $\psi:\reals^d\to\reals$, we have
$\int \psi(x) \, \nu_n(\de x) \stackrel{p}{\to} \int \psi(x)\, \nu(\de x)$.

Then $\dW{\nu_n}{\nu}\stackrel{p}\to 0$.
\end{lemma}
\begin{proof}
By Lemma \ref{lemma:Elementary}, it is sufficient to show that for any subsequence
$(n_j)_{j\ge 1}$ we can construct a further subsequence $(n'_j)_{j\ge 1}$
such that $\dW{\nu_{n'_j}}{\nu}\stackrel{a.s.}\to 0$.

Fix such a subsequence $(n_j)$, and let $(\psi_i)_{i\in \naturals}$ be a countable collection of 
bounded Lipschitz functions on $\reals^d$ which determine weak convergence (i.e. such that
$\int \psi_i(x) \, q_n(\de x) \to \int \psi_i(x) \, q(\de x)$ imply $\dW{q_n}{q}\to 0$\rev{).}
One can take for instance all functions of the form $\psi(x) = (1-d(x,Q)/\eps)_+$
where $Q\subseteq\reals^d$ is a rectangle with rational corners, and $\eps>0$ is rational.

By Borel-Cantelli, we can construct a subsequence $(n^1_j)\subseteq (n_j)$ such that
$\int \psi_1(x)\nu_{n^1_j}(\de x)\to \int \psi_1(x)\nu(\de x)$. Refining this sequence, we
obtain, for each $k$ a subsequence $(n^k_j)$
such that $\int \psi_a(x)\nu_{n^k_j}(\de x)\to \int \psi_a(x)\nu(\de x)$ for all $a\le k$.
Taking the diagonal $n'_j=n^j_j$ yields a subsequence along which $\dW{\nu_{n'_j}}{\nu}$
as desired.
\end{proof}
This concludes the proof of Lemma \ref{lem:state-evolution-discrete-flow}.

\begin{remark}
By \cite{berthier2020state},
Eq.~\eqref{eq:amp-convergence-pseudo-lipschitz-expectation} holds for test functions $\psi$ which are pseudo-Lipschitz of order 2 when the matrix $\bX$ has Gaussian entries. 
In this case, using the same argument as above,
we may conclude that \eqref{eq:LP-conv} holds also for the Wasserstein distance.
\end{remark}


\subsection{Proof of Lemma~\ref{lem:integral-differential-approximation}} \label{proof:integral-differential-approximation}
First, we define the transformation $\trsfrm^\eta = \trsfrmB^\eta \circ \trsfrmA^\eta$ 
where we let $\trsfrmA^\eta: (\Cl, \Rl, \Gamma) \mapsto (\bCt, \bRt)$ and $\trsfrmB^\eta: (\bCt, \bRt) \mapsto (\bCl, \bRl, \bGamma)$. We remind the readers that $\trsfrmA^\eta$ does not necessarily map $\mathcal{S}$ into $\wb{\mathcal{S}}$ and nor does $\trsfrmB^\eta$ map $\wb{\mathcal{S}}$ into $\mathcal{S}$. We use this notation here because exactly similar to our previous definitions of $\trsfrmA$ and $\trsfrmB$, the transformation $\trsfrmA^\eta$ is defined by taking the input function triplet through Eqs.~\eqref{eq:def-theta-eta} and \eqref{eq:def-derivative-t-eta} and then we obtain $(\bCt, \bRt)$ by Eqs.~\eqref{eq:def-C-t-eta} and \eqref{eq:def-R-t-eta}; $\trsfrmB$ is defined by taking the input function pair into Eqs.~\eqref{eq:def-r-eta}, \eqref{eq:def-derivative-l-eta} and \eqref{eq:def-derivative-l-eta-2} and $(\bCl, \bRl, \bGamma)$ is obtained by Eqs.~\eqref{eq:def-C-l-eta}, \eqref{eq:def-R-l-eta} and \eqref{eq:def-Gamma-eta}. 

As we have shown in the proof of Lemma~\ref{lem:state-evolution-discrete-flow}, the mappings $\trsfrmA^\eta$ and $\trsfrmB^\eta$ are essentially determined recursively on the discrete time knots $t_i = i\eta$, $i=0,1,\cdots$, so they are uniquely defined. We express the solution of the system $\mathfrak{S}^\eta$ as the unique fixed-point of $\trsfrm^\eta$, namely if we let $X^\eta = (\Cl^\eta, \Rl^\eta, \Gamma_\eta )$ be the function triplet that solves $\mathfrak{S}^\eta$, it holds that
\begin{align}
\trsfrm^\eta(X^\eta) = X^\eta \, .
\end{align}
Suppose $\trsfrmA^\eta(X^\eta) = (\Ct^\eta, \Rt^\eta)$, we have the following lemma characterizing the unique solution of $\mathfrak{S}^\eta$.
\begin{lemma} \label{lem:X-eta-upper-bound}
Under the same conditions of Lemma~\ref{lem:integral-differential-approximation}, the unique solution of $\mathfrak{S}^\eta$ satisfies $X^\eta = (\Cl^\eta, \Rl^\eta, \Gamma_\eta ) \in \mathcal{S}$ and $ (\Ct^\eta, \Rt^\eta) \in \wb{\mathcal{S}}$.
\end{lemma}
Let $X \in \mathcal{S}$ be the unique fixed-point of $\mathcal{T}$, we can then control the distance between $X$ and $X^\eta$ by
\begin{align}
\lbddst{X}{X^\eta} & = \lbddst{\mathcal{T}(X)}{\mathcal{T}^\eta(X^\eta)} \nonumber \\
& \leq \underbrace{\lbddst{\mathcal{T}(X)}{\mathcal{T}(X^\eta)}}_{\mathrm{(I)}} + \underbrace{\lbddst{\mathcal{T}(X^\eta)}{\mathcal{T}^\eta(X^\eta)}}_{\mathrm{(II)}} \, , \label{eq:mid-integral-differential-approximation-1}
\end{align}
where by Eq.~\eqref{eq:T-contraction} we can choose $\lambda$ large enough such that
\begin{align}
\mathrm{(I)} \leq \frac 1 2 \lbddst{X}{X^\eta} \, . \label{eq:mid-integral-differential-approximation-2}
\end{align}
The following lemma controls the quantity (ii). We defer its proof to Appendix~\ref{proof:integral-differential-eta-approximation-1}.
\begin{lemma} \label{lem:integral-differential-eta-approximation-1}
Under the same conditions of Lemma~\ref{lem:integral-differential-approximation}, it holds 	for all $\lambda \geq \wb{\lambda}_5 := \wb{\lambda}_5(\mathcal{S}, \wb{\mathcal{S}})$ that
\begin{align}
	\lbddst{\mathcal{T}(X^\eta)}{\mathcal{T}^\eta(X^\eta)} \leq h(\eta) 
\end{align}
for some nondecreasing continuous function $h(\eta)$ with $h(0) = 0$. Here the function $h$ only depends on the spaces $\mathcal{S}$ and $\wb{\mathcal{S}}$.
\end{lemma}
Substituting Lemma~\ref{lem:integral-differential-eta-approximation-1} and Eq.~\eqref{eq:mid-integral-differential-approximation-2} into Eq.~\eqref{eq:mid-integral-differential-approximation-1} yields
\begin{align}
\lbddst{X}{X^\eta} \leq 2h(\eta) \to 0 
\end{align}
as $\eta \to 0$. The following lemma establishes if $X$ and $X^\eta$ are close and the step size $\eta$ is small, we can couple $\theta^t$ and $\theta^t_\eta$ such that their $(\lambda, T)$-distance is small. A proof can be found in Appendix~\ref{proof:integral-differential-eta-approximation-path}.

\begin{lemma} \label{lem:integral-differential-eta-approximation-path}
Under the same conditions of Lemma~\ref{lem:integral-differential-approximation}, for all $\lambda \geq \wb{\lambda}_6 := \wb{\lambda}_6(\mathcal{S}, \wb{\mathcal{S}})$ we can find a coupling for $\theta^t$ and $\theta^t_\eta$ such that
\begin{align}
	\sup_{0 \leq t \leq T} e^{-\lambda t} \sqrt{\Ep \brk{\normtwo{\theta^t - \theta^t_\eta}^2}} \leq H(\eta, \lbddst{X}{X^\eta}) \, , 
\end{align}
where $H$ is a nondecreasing function in each coordinate and $\lim_{(u, v) \to (0, 0)} H(u, v) \to 0$. Here the function $H$ only depends on the spaces $\mathcal{S}$ and $\wb{\mathcal{S}}$.
\end{lemma}
By coupling $\theta^t$ and $\theta^t_\eta$ as in Lemma~\ref{lem:integral-differential-eta-approximation-path}, we can then conclude the proof by invoking Lemma~\ref{lem:integral-differential-eta-approximation-1} and Lemma~\ref{lem:integral-differential-eta-approximation-path} since
\begin{align}
&\wdst{\mu_{\theta^{\tau_1}_\eta, \cdots, \theta^{\tau_m}_\eta}}{\mu_{\theta^{\tau_1}, \cdots, \theta^{\tau_m}}} \leq \sqrt{\frac{1}{m} \sum_{j=1}^m \normtwo{\theta_\eta^{\tau_j} - \theta^{\tau_j}}^2} \leq e^{\lambda T} \cdot H(\eta, 2h(\eta))\, .
\end{align}
The proof is completed by taking $\eta \to 0$.

\subsection{Proof of Lemma~\ref{lem:AMP-to-S-eta}} \label{proof:AMP-to-S-eta}
First we show any solution of $\mathfrak{S}^\eta$ must be uniquely determined by its values at discrete time knots $t_i = i\eta$ for $i=0,1,\cdots$. From Eqs.~\eqref{eq:def-C-t-eta} and \eqref{eq:def-C-l-eta} we have $u^t_\eta$ and $w^t_\eta$ must be piecewise constant, namely
\begin{align}
u^t_\eta = u^{\flr{t}}_\eta \, , \qquad w^t_\eta = w^{\flr{t}}_\eta \, ,
\end{align}
and therefore we have $\theta^t_\eta$ is piecewise linear with time knots $t_i$ and $r^t_\eta$ is piecewise constant with time knots $t_i$. Finally, from Eqs.~\eqref{eq:def-derivative-t-eta} and \eqref{eq:def-derivative-l-eta} we have $\partial \theta^t_\eta / \partial u^s_\eta$ is piecewise linear and $\partial \ell_{\flr{t}}(r^t_\eta; z)/\partial w^s_\eta$ is piecewise constant with time knots $t_i$, this then implies $\Rt^\eta$ is piecewise linear and $\Rl^\eta$ is piecewise constant with knots $t_i$. By Eq.~\eqref{eq:def-Gamma-eta} we have $\Gamma_\eta^t$ must be piecewise constant with knots $t_i$. We then conclude that $\mathfrak{S}^\eta$ is uniquely determined at $t_i = i \eta$.

We show by induction the unique solution of $\mathfrak{S}^\eta$ at discrete time knots $t_i = i\eta$ must be
\begin{subequations}
\begin{align}
	(\theta^{t_0}_\eta, \cdots, \theta^{t_r}_\eta) & \stackrel{d}{=} (\wb{\theta}_\eta^{t_0}, \cdots, \wb{\theta}_\eta^{t_r}) \, , & & \label{eq:amp-induction-1} \\
	\Rt^\eta(t_i, t_j) & = \zeta_{i,j-1} / \eta \, , & & 0 \leq j \leq i \leq r \, , \label{eq:amp-induction-2} \\
	\Rl^\eta(t_i, t_j) & = \xi_{i, j} / \eta \, , & & 0 \leq j < i \leq r \, , \label{eq:amp-induction-3} \\
	\Gamma_\eta^{t_i} & = \xi_{i, i} \, , & & 0 \leq i  \leq r \, , \label{eq:amp-induction-4}
\end{align}
\end{subequations}
where we define
\begin{align}
\zeta_{i, -1} &= \Ep \brk{\frac{\partial}{\partial \prn{\theta^0 / \eta}} g_i(\wb{u}^{t_1}_\eta, \cdots, \wb{u}^{t_i}_\eta; \theta^0)}\, .
\end{align}
For $r=0$, provided that $\theta^0_\eta \stackrel{d}{=} \wb{\theta}^0_\eta \stackrel{d}{=} \theta^0$, it follows immediately that $w_\eta^{0} \stackrel{d}{=} \wb{w}_\eta^{0} \stackrel{d}{=} \normal\prn{0, \Ep \brk{\theta^0 \prn{\theta^0}^\sT}}$ and  therefore $\Gamma_\eta^0 = \Ep \brk{\nabla_r \ell_0(r^0_\eta; z)}  = \Ep \brk{\frac{\partial}{\partial \wb{w}^{0}_\eta } f_0(\wb{w}^{0}_\eta; z)} = \xi_{0,0}$. Suppose the induction hypothesis holds for $r$, we next show Eqs.~\eqref{eq:amp-induction-1} to \eqref{eq:amp-induction-4} hold for $r+1$. 
\paragraph{Induction on Eq.~\eqref{eq:amp-induction-1}.} First, by Eqs.~\eqref{eq:def-amp-u} and \eqref{eq:def-C-t-eta} we have
\begin{align}
\Ep \brk{\wb{w}^{t_i}_\eta \prn{\wb{w}^{t_j}_\eta}^\sT} & =\Ep \brk{g_i(\wb{u}^{t_1}_\eta, \cdots, \wb{u}^{t_i}_\eta; \theta^0)g_j(\wb{u}^{t_1}_\eta, \cdots, \wb{u}^{t_j}_\eta; \theta^0)^\sT} = \Ep \brk{\wb{\vt}^{t_i}_\eta \prn{\wb{\vt}^{t_j}_\eta}^{\sT}} = \Ep \brk{{\vt}^{t_i}_\eta \prn{{\vt}^{t_j}_\eta}^{\sT}} = \Ct^\eta(t_i, t_j)\, ,
\end{align}
which implies $(\wb{w}^{t_0}_\eta, \cdots, \wb{w}^{t_r}_\eta) \stackrel{d}{=} (w^{t_0}_\eta, \cdots, w^{t_r}_\eta)$. Similarly it also holds $(\wb{u}^{t_1}_\eta, \cdots, \wb{u}^{t_{r+1}}_\eta) \stackrel{d}{=} (u^{t_0}_\eta, \cdots, u^{t_{r}}_\eta)$. Thus, substituting into Eq.~\eqref{eq:def-theta-eta} gives us for $t \in [t_r, t_{r+1})$,
\begin{align}
\frac{\de}{\de t} \theta^t_\eta & = -(\Lambda^{\flr{t}} + \Gamma^{\flr{t}}_\eta) \theta^{\flr{t}}_\eta - \int_0^{\flr{t}} R_{\ell}^\eta(\flr{t},\flr{s}) \theta^{\flr{s}}_\eta \de s + u^t_\eta \nonumber \\
& = -(\Lambda^{t_r} + \Gamma^{t_r}_\eta) \theta^{t_r}_\eta - \int_0^{t_r} R_{\ell}^\eta(\flr{t},\flr{s}) \theta^{\flr{s}}_\eta \de s + u^{t_r}_\eta \nonumber \\
& = -(\Lambda^{t_r} + \xi_{r, r}) \theta^{t_r}_\eta - \sum_{j=0}^{r-1} \xi_{r,j} \theta^{t_j}_\eta  + u_\eta^{t_r}\, ,
\end{align} 
and further
\begin{align}
\theta^{t_{r+1}}_\eta = (I - \eta \Lambda^{t_r}) \theta_\eta^{t_r} + \eta \prn{u_\eta^{t_r} - \sum_{j=0}^r \xi_{r,j} \theta^{t_j}_\eta} \, .
\end{align}
Comparing to Eq.~\eqref{eq:def-recursive-g} which asserts
\begin{align}
\wb{\theta}^{t_{r+1}}_\eta = (I - \eta \Lambda^{t_r}) \wb{\theta}_\eta^{t_r} + \eta \prn{\wb{u}_\eta^{t_{r+1}} - \sum_{j=0}^r \xi_{r,j} \wb{\theta}^{t_j}_\eta},
\end{align}
which immediately implies Eq.~\eqref{eq:amp-induction-1} holds for $r+1$. 
\paragraph{Induction on Eq.~\eqref{eq:amp-induction-2}.} With the same calculations applied to Eq.~\eqref{eq:def-derivative-t-eta}, for an $0 \leq i \leq r$ it follows that
\begin{align}
\frac{\partial \theta^{t_{r+1}}_\eta}{\partial u^{t_i}_\eta} & = (I - \eta \Lambda^{t_r}) \frac{\partial \theta^{t_{r}}_\eta}{\partial u^{t_i}_\eta} - \eta \sum_{j=i}^r \xi_{r,j} \frac{\partial \theta^{t_{j}}_\eta}{\partial u^{t_i}_\eta} \nonumber \\
& = (I - \eta \Lambda^{t_r}) \frac{\partial \theta^{t_{r}}_\eta}{\partial u^{t_i}_\eta} - \eta \sum_{j=i+1}^r \xi_{r,j} \frac{\partial \theta^{t_{j}}_\eta}{\partial u^{t_i}_\eta} - \eta \xi_{r,i} \, ,
\end{align}
where we use $\partial \theta_\eta^{t_i}/ \partial u^{t_i}_\eta = I$. We slightly abuse the notation here by taking $\wb{u}_\eta^{t_{0}} := \theta^0/\eta$, and as a direct consequence of Eq.~\eqref{eq:def-recursive-g}, we get the recursion when $0 \leq i \leq r$,
\begin{align}
& \frac{\partial}{\partial \wb{u}^{t_i}_\eta} g_{r+1}(\wb{u}^{t_1}_\eta, \cdots, \wb{u}^{t_{r+1}}_\eta; \theta^0)  = \frac{\partial \wb{\theta}^{t_{r+1}}_\eta}{\partial \wb{u}^{t_i}_\eta} \nonumber \\
& = (I - \eta \Lambda^{t_r})\frac{\partial}{\partial \wb{u}^{t_i}_\eta} g_{r}(\wb{u}^{t_1}_\eta, \cdots, \wb{u}^{t_{r}}_\eta; \theta^0) - \eta \sum_{j=i+1}^r \xi_{r,j} \frac{\partial}{\partial \wb{u}^{t_i}_\eta} g_{j}(\wb{u}^{t_1}_\eta, \cdots, \wb{u}^{t_{j}}_\eta; \theta^0) \nonumber \\
& = (I - \eta \Lambda^{t_r})\frac{\partial}{\partial \wb{u}^{t_i}_\eta} g_{r}(\wb{u}^{t_1}_\eta, \cdots, \wb{u}^{t_{r}}_\eta; \theta^0) - \eta \sum_{j=i+1}^r \xi_{r,j} \frac{\partial}{\partial \wb{u}^{t_i}_\eta} g_{j}(\wb{u}^{t_1}_\eta, \cdots, \wb{u}^{t_{j}}_\eta; \theta^0) - \eta^2 \xi_{r,i}\, ,
\end{align}
where in the last line it is used that $\partial g_i(\wb{u}^{t_1}_\eta, \cdots, \wb{u}^{t_{i}}_\eta; \theta^0) / \partial \wb{u}^{t_i}_\eta = \eta I$. We thus have
\begin{align}
&\prn{\eta^{-1} \frac{\partial}{\partial \wb{u}^{t_i}_\eta} g_{r+1}(\wb{u}^{t_1}_\eta, \cdots, \wb{u}^{t_{r+1}}_\eta; \theta^0) } \nonumber \\
& = (I - \eta \Lambda^{t_r}) \prn{\eta^{-1}\frac{\partial}{\partial \wb{u}^{t_i}_\eta} g_{r}(\wb{u}^{t_1}_\eta, \cdots, \wb{u}^{t_{r}}_\eta; \theta^0)} - \eta \sum_{j=i+1}^r \xi_{r,j} \prn{\eta^{-1} \frac{\partial}{\partial \wb{u}^{t_i}_\eta} g_{j}(\wb{u}^{t_1}_\eta, \cdots, \wb{u}^{t_{j}}_\eta; \theta^0)} - \eta \xi_{r,i} \, .
\end{align}
Together with the induction hypothesis we then show for all $0 \leq i \leq r+1$,
\begin{align}
\Rt^\eta(t_{r+1}, t_i) = \Ep \brk{\frac{\partial \theta^{t_{r+1}}_\eta}{\partial u^{t_i}_\eta}} = \Ep \brk{\eta^{-1} \frac{\partial}{\partial \wb{u}^{t_i}_\eta} g_{r+1}(\wb{u}^{t_1}_\eta, \cdots, \wb{u}^{t_{r+1}}_\eta; \theta^0) } = \zeta_{r+1, i-1} / \eta \, .
\end{align}
\paragraph{Induction on Eqs.~\eqref{eq:amp-induction-3} and \eqref{eq:amp-induction-4}.} By Eq.~\eqref{eq:def-r-eta}, for all $0 \leq i \leq r+1$,
\begin{align}
r^{t_i}_\eta & = - \frac{1}{\delta} \int_0^{t_i} R_{\theta}^\eta(\flr{t},\cil{s}) \ell_{\flr{s}}(r^{s}_\eta; z) \de s + w^{t_i}_\eta \nonumber \\
& = - \frac{1}{\delta} \sum_{j=0}^{i-1} \eta R_{\theta}^\eta(t_i,t_{j+1}) \ell_{t_j}(r^{t_j}_\eta; z)  + w^{t_i}_\eta \nonumber \\
& = - \frac{1}{\delta} \sum_{j=0}^{i-1} \zeta_{i,j} \ell_{t_j}(r^{t_j}_\eta; z) + w^{t_i}_\eta \, ,
\end{align}
which further gives
\begin{align}
\ell_{t_i}(r^{t_i}_\eta; z) & =  \ell_{t_i} \prn{- \frac{1}{\delta} \sum_{j=0}^{i-1} \zeta_{i,j} \ell_{t_j}(r^{t_j}_\eta; z) + w^{t_i}_\eta; z} \, .
\end{align}
From Eq.~\eqref{eq:def-recursive-f} we get similarly
\begin{align}
f_i(\wb{w}^{t_0}_\eta, \cdots, \wb{w}^{t_i}_\eta; z) & =  \ell_{t_i} \prn{- \frac{1}{\delta} \sum_{j=0}^{i-1} \zeta_{i,j} f_j(\wb{w}^{t_0}_\eta, \cdots, \wb{w}^{t_j}_\eta; z) + \wb{w}^{t_i}_\eta; z} \, .
\end{align}
Since Eq.~\eqref{eq:amp-induction-1} holds for $r+1$, this implies we can assume without loss of generality that $(w_\eta^{t_0}, \cdots, w_\eta^{t_{r+1}}) = (\wb{w}_\eta^{t_0}, \cdots, \wb{w}_\eta^{t_{r+1}})$. In this case, it always holds that $f_i(\wb{w}^{t_0}_\eta, \cdots, \wb{w}^{t_i}_\eta; z) = \ell_{t_i}(r^{t_i}_\eta; z)$ for $0 \leq i \leq r+1$. In particular
\begin{align}
\nabla_r \ell_{t_i}(r^{t_i}_\eta; z) & = \nabla_r \ell_{t_i} \prn{- \frac{1}{\delta} \sum_{j=0}^{i-1} \zeta_{i,j} \ell_{t_j}(r^{t_j}_\eta; z) + w^{t_i}_\eta; z} \nonumber \\
& = \nabla_r \ell_{t_i} \prn{- \frac{1}{\delta} \sum_{j=0}^{i-1} \zeta_{i,j} f_j(\wb{w}^{t_0}_\eta, \cdots, \wb{w}^{t_j}_\eta; z) + \wb{w}^{t_i}_\eta; z} \nonumber \\
& = \frac{\partial }{\partial \wb{w}^{t_i}_\eta} \ell_{t_i} \prn{- \frac{1}{\delta} \sum_{j=0}^{i-1} \zeta_{i,j} f_j(\wb{w}^{t_0}_\eta, \cdots, \wb{w}^{t_j}_\eta; z) + \wb{w}^{t_i}_\eta; z} \nonumber \\
& = \frac{\partial }{\partial \wb{w}^{t_i}_\eta} f_i(\wb{w}^{t_0}_\eta, \cdots, \wb{w}^{t_i}_\eta; z) \, .
\end{align}
Taking expectation on both sides and we obtain $\Gamma_\eta^{t_i} = \xi_{i,i}$ for $0 \leq i \leq r+1$. It then only remains to be shown that Eq.~\eqref{eq:amp-induction-3} holds for $r+1$. From Eq.~\eqref{eq:def-derivative-l-eta}, we have for all $0 \leq i \leq r$,
\begin{align}
& \frac{\partial \ell_{t_{r+1}}(r^{t_{r+1}}_\eta;z)}{\partial w^{t_i}_\eta} \nonumber \\
& = \nabla_r \ell_{t_{r+1}}(r^{t_{r+1}}_\eta;z) \cdot \prn{- \frac{1}{\delta} \int_{t_{i+1}}^{t_{r+1}} R_{\theta}^\eta(t_{r+1},\cil{s'}) \frac{\partial \ell_{\flr{s'}}(r^{s'}_\eta;z)}{\partial w^{t_i}_\eta}  \de s' - \frac 1 \delta R_\theta^\eta(t_{r+1},t_{i+1}) \nabla_r \ell_{t_i}(r^{t_i}_\eta;z) } \nonumber \\
& = \nabla_r \ell_{t_{r+1}}(r^{t_{r+1}}_\eta;z) \cdot \prn{- \frac 1 \delta \sum_{j=i+1}^{r} \eta\Rt^\eta(t_{r+1}, t_{j+1}) \frac{\partial \ell_{t_j}(r^{t_j}_\eta;z)}{\partial w^{t_i}_\eta}  - \frac 1 \delta R_\theta^\eta(t_{r+1},t_{i+1}) \nabla_r \ell_{t_i}(r^{t_i}_\eta;z)}  \nonumber \\
& = \nabla_r \ell_{t_{r+1}}(r^{t_{r+1}}_\eta;z) \cdot \prn{- \frac 1 \delta \sum_{j=i+1}^{r} \zeta_{r+1, j} \frac{\partial \ell_{t_j}(r^{t_j}_\eta;z)}{\partial w^{t_i}_\eta}  - \frac 1 \delta \zeta_{r+1,i} \nabla_r \ell_{t_i}(r^{t_i}_\eta;z) / \eta} \, .
\end{align}
Since
\begin{align}
& \frac{\partial }{\partial \wb{w}^{t_i}_\eta} f_{r+1}(\wb{w}^{t_0}_\eta, \cdots, \wb{w}^{t_{r+1}}_\eta; z) \nonumber \\
& =\frac{\partial }{\partial \wb{w}^{t_i}_\eta} \ell_{t_{r+1}} \prn{- \frac{1}{\delta} \sum_{j=0}^{r} \zeta_{r+1,j} f_j(\wb{w}^{t_0}_\eta, \cdots, \wb{w}^{t_j}_\eta; z) + \wb{w}^{t_{r+1}}_\eta; z} \nonumber \\
& = \nabla_r \ell_{t_{r+1}} \prn{- \frac{1}{\delta} \sum_{j=0}^{r} \zeta_{r+1,j} f_j(\wb{w}^{t_0}_\eta, \cdots, \wb{w}^{t_j}_\eta; z) + \wb{w}^{t_{r+1}}_\eta; z} \cdot \prn{-\frac{1}{\delta} \sum_{j=i}^r \zeta_{r+1,j}\frac{\partial }{\partial \wb{w}^{t_i}_\eta} f_j(\wb{w}^{t_0}_\eta, \cdots, \wb{w}^{t_j}_\eta; z) } \nonumber \\
& = \nabla_r \ell_{t_{r+1}}(r^{t_{r+1}}_\eta;z) \cdot \prn{-\frac{1}{\delta} \sum_{j=i}^r \zeta_{r+1,j}\frac{\partial }{\partial \wb{w}^{t_i}_\eta} f_j(\wb{w}^{t_0}_\eta, \cdots, \wb{w}^{t_j}_\eta; z) } \nonumber \\
& = \nabla_r \ell_{t_{r+1}}(r^{t_{r+1}}_\eta;z) \cdot \prn{-\frac{1}{\delta} \sum_{j=i+1}^r \zeta_{r+1,j}\frac{\partial }{\partial \wb{w}^{t_i}_\eta} f_j(\wb{w}^{t_0}_\eta, \cdots, \wb{w}^{t_j}_\eta; z) - \frac 1 \delta \zeta_{r+1,i} \nabla_r \ell_{t_i}(r^{t_i}_\eta;z) }\, ,
\end{align}
where in the last line we use $\partial f_i(\wb{w}^{t_0}_\eta, \cdots, \wb{w}^{t_i}_\eta; z) /  \partial \wb{w}^{t_i}_\eta= \nabla_r \ell_{t_i}(r^{t_i}_\eta;z)$. Comparing the above two equations, it then follows that
\begin{align}
\frac{\partial \ell_{t_{r+1}}(r^{t_{r+1}}_\eta;z)}{\partial w^{t_i}_\eta} = \eta^{-1} \frac{\partial }{\partial \wb{w}^{t_i}_\eta} f_{r+1}(\wb{w}^{t_0}_\eta, \cdots, \wb{w}^{t_{r+1}}_\eta; z) \, ,
\end{align}
which further implies Eq.~\eqref{eq:amp-induction-3} for $r+1$ by taking expectation on both sides. This concludes the induction.

Finally, we invoke Lemma~\ref{lem:X-eta-upper-bound} to show the solution of $\mathfrak{S}^\eta$ is in the space $\mathcal{S}$.

\subsection{Proof of Lemma~\ref{lem:X-eta-upper-bound}} \label{proof:X-eta-upper-bound}
Since the covariance kernels $\Cl^\eta$ and $\Ct^\eta$ are piecewise constant, the continuity conditions are automatically satisfied by Definition~\ref{def:space-S} and \ref{def:space-S-dual}. To show the lemma we only need to prove the upper bounds
\begin{subequations}
\begin{align}
	&\norm{\Rt^\eta(t,s)} \leq \pRt(t-s)\, , & & \norm{\Rl^\eta(t,s)} \leq \Ep \brk{\norm{\frac{\partial \ell_{\flr{t}}(r^t_\eta;z)}{\partial w^s_\eta}}} \leq \pRl(t-s)\, , & &  0 \leq s \leq t \leq T \, ,  \label{eq:R-eta-upper-bound} \\
	&\norm{\Ct^\eta(t,t)} \leq \pCt(t)\, , & &\norm{\Cl^\eta(t,t)} \leq \pCl(t)\, , & &  0 \leq t \leq T \, , \label{eq:C-eta-upper-bound} \\
	&\norm{\Gamma^t_\eta} \leq \cstloss \, , & & & & 0 \leq t \leq T \,.  \label{eq:Gamma-eta-upper-bound}
\end{align}
\end{subequations}
Note that
\begin{align}
\norm{\Gamma_\eta^t} \leq \Ep \brk{\norm{\nabla_r \ell_{\flr{t}}(r^t_\eta;z)}} \leq \cstloss\, ,
\end{align}
which proves Eq.~\eqref{eq:Gamma-eta-upper-bound}.
\paragraph{Upper bounds for $\Rt^\eta$ and $\Rl^\eta$.} From the definition of $\Rl^\eta$ in Eq.~\eqref{eq:def-R-l-eta} and that $r_\eta^t$ and $w_\eta^t$ are piecewise constant, we know $\Rl^\eta(t,s) = \Rl^\eta(\flr{t}, \flr{s})$. Since $\max \brc{\flr{t} - \cil{s}, 0} \leq t-s$, it suffices to prove $\norm{\Rl^\eta(\flr{t}, \flr{s})} \leq \pRl(\max \brc{\flr{t} - \cil{s}, 0})$. We prove Eq.~\eqref{eq:R-eta-upper-bound} for all $\flr{t} - \flr{s} \leq m\eta$ for all $m \in \integersp$. When $m=0$, we have
\begin{align}
\ddt \norm{\Rt^\eta(t,s)} & \leq \prn{\cstlbd + \cstloss} \norm{\Rt^\eta(\flr{t},s)} \leq \cstlbd + \cstloss \, , \\
\norm{\Rl^\eta(t,s)} & = \norm{-\frac 1 \delta \Ep \brk{\nabla_r \ell_{\flr{t}}(r^t_\eta;z)\nabla_r \ell_{\flr{t}}(r^s_\eta;z)}} \leq \frac{\cstloss^2}{\delta^2} \, .
\end{align}
Comparing to Eqs.~\eqref{eq:def-pRl} and \eqref{eq:def-pRt} we see $\norm{\Rt^\eta(t,s)} \leq \pRt(t-s), \norm{\pRl(t,s)} \leq \pRl(0)$ when $\flr{t} = \flr{s}$. Suppose Eq.~\eqref{eq:R-eta-upper-bound} holds for $\flr{t} - \flr{s} \leq m\eta$, then by Eq.~\eqref{eq:def-derivative-t-eta} one has for $\flr{t} - \flr{s} = (m+1) \eta$,
\begin{align}
\ddt \norm{\Rt^\eta(t,s)} & \leq \prn{\cstlbd + \cstloss} \norm{\Rt^\eta(\flr{t},s)} + \int_s^{\flr{t}} \norm{\Rl^\eta(\flr{t}, \flr{s'})} \norm{\Rt^\eta(\flr{s'}, s)} \de s' \nonumber \\
& \stackrel{\mathrm{(i)}}{=} \lim_{t' \uparrow \flr{t}}  \brc{\prn{\cstlbd + \cstloss} \norm{\Rt^\eta(t',s)} + \int_s^{t'} \norm{\Rl^\eta(\flr{t}, \flr{s'})} \norm{\Rt^\eta(\flr{s'}, s)}} \de s' \nonumber \\
& \leq \prn{\cstlbd +\cstloss} \pRt(t-s) + \int_s^{\flr{t}} \pRl(t-s') \pRt(s' -s) \de s' \nonumber \\
& \stackrel{\mathrm{(ii)}}{\leq} \ddt \pRt(t-s) \, ,
\end{align}
where in (i) we use that $\Rt^\eta(t,s)$ is continuous in $t$ and in (ii) we use Eq.~\eqref{eq:def-pRt}. We conclude $\norm{\Rt^\eta(t,s)} \leq \pRt(t-s)$ when $\flr{t} - \flr{s} \leq (m+1)\eta$. Similarly by Eq.~\eqref{eq:def-derivative-l-eta} when $\flr{t} - \flr{s} = (m+1) \eta$ it holds
\begin{align}
\Ep \brk{\norm{\frac{\partial \ell_{\flr{t}}(r^t_\eta;z)}{\partial w^s_\eta}}} & \leq \cstloss \cdot \prn{\frac{1}{\delta} \int_{\cil{s}}^{\flr{t}} R_{\theta}^\eta(\flr{t},\cil{s'}) \Ep \brk{\norm{ \frac{\partial \ell_{\flr{s'}}(r^{s'}_\eta;z)}{\partial w^s_\eta}}}  \de s' + \frac{\cstloss}{\delta}\cdot R_\theta^\eta(\flr{t},\cil{s})}
\nonumber \\
& \stackrel{\mathrm{(i)}}{\leq} \frac{\cstloss}{\delta} \cdot \brc{\cstloss \pRt(t) + \int_{\cil{s}}^{\flr{t}} \pRt(t-s') \pRl(s'-s) \de s'} \nonumber \\
& \stackrel{\mathrm{(ii)}}{\leq}  \pRl(t-s) \, , 
\end{align}
where we use $\flr{t} - \cil{s} \leq m \eta$ and the induction hypothesis at $m$ in (i), in (ii) we invoke Eq.~\eqref{eq:def-pRl}. Finally, note that $\norm{\Rl^\eta(t,s)}  \leq \Ep \brk{\norm{\frac{\partial \ell_{\flr{t}}(r^t_\eta;z)}{\partial w^s_\eta}}}$ holds and we complete the proof by induction.

\paragraph{Upper bounds for $\Ct^\eta$ and $\Cl^\eta$.} Since $\Ct^\eta$ and $\Cl^\eta$ are piecewise constant we only need to show $\norm{\Ct^\eta(\flr{t}, \flr{t})} \leq \pCt(\flr{t})$ and $\norm{\Cl^\eta(\flr{t}, \flr{t})} \leq \pCl(\flr{t}, \flr{t})$ and Eq.~\eqref{eq:C-eta-upper-bound} will follow by monotonicity of $\pCt$ and $\pCl$. We show this by induction on $\flr{t} = r \eta$ with hypotheses
\begin{align}
\Ep \brk{\norm{\theta^{\flr{t}}_\eta}^2} \leq \pCt(\flr{t}) \, ,
\end{align}
When $r =0$, the initial condition holds at time $0$. Suppose the inductive hypotheses hold for $\flr{t} \leq r \eta$, when $r\eta \leq t < (r+1)\eta$, we can obtain from Eq.~\eqref{eq:def-theta-eta} that
\begin{align}
\frac{\de}{\de t} \normtwo{\theta^t_\eta} = \prn{\cstlbd +\cstloss} \normtwo{\theta^{\flr{t}}_\eta} + \int_0^{\flr{t}} \pRl(t-s) \normtwo{\theta^{\flr{s}}_\eta} \de s + \normtwo{u^{\flr{t}}_\eta} \, . 
\end{align}
By the same calculations in Eq.~\eqref{eq:mid-solution-in-space-1} we get
\begin{align}
&\ddt \sqrt{\Ep \brk{\normtwo{\theta^t_\eta}^2}} \nonumber \\
&\leq \sqrt{3 \cdot \brc{\prn{\cstlbd + \cstloss}^2   \mathbb{E} \brk{\left\| \vt^{\flr{t}}_\eta\right\|_2^2} + \int_0^{\flr{t}} (t-s+1)^2 \pRl(t-s)^2 \mathbb{E} \brk{\left\| \vt^{\flr{s}}_\eta \right\|_2^2} \mathrm d s + \frac{k}{\delta} \Phi_{\vCl}(t)}} \nonumber \\
&\leq \sqrt{3 \cdot \brc{\prn{\cstlbd + \cstloss}^2   \pCt(t)^2 + \int_0^{\flr{t}} (t-s+1)^2 \pRl(t-s)^2 \pCt(s)^2 \mathrm d s + \frac{k}{\delta} \Phi_{\vCl}(t)}} \, ,
\end{align}
which together with Eq.~\eqref{eq:def-pCt} implies $\sqrt{\Ep \brk{\normtwo{\theta^t_\eta}^2}}  \leq \pCt(t)$ when $t \leq (r+1)\eta$. It then follows that
\begin{align}
\norm{\vCt^\eta(t,t)} =  \norm{\mathbb{E} \brk{\vt^t_\eta {\vt^t_\eta}^\sT}} \leq  \mathbb{E} \brk{\norm{\vt^t_\eta}_2^2} \leq \Phi_{\vCt}(t)\, .
\end{align}
By Eq.~\eqref{eq:def-r-eta} we have
\begin{align}
\norm{\vl_{\flr{t}} \prn{\vr^t_\eta; \veps} }_2 & \leq \norm{\vl_{\flr{t}} \prn{0; \veps} }_2 + \cstloss \norm{\vr^t_\eta}_2 \nonumber \\
& \leq \norm{\vl_{\flr{t}} \prn{0; \veps} }_2  + \frac{\cstloss}{\delta} \int_0^{\flr{t}} \norm{\Rt^\eta(\flr{t},\cil{s})} \big\|\vl_{\flr{s}} \prn{\vr^s_\eta; \veps} \big\|_2 \mathrm d s + \cstloss \norm{\vw^t_\eta}_2 \nonumber \\
& \leq \norm{\vl_{\flr{t}} \prn{0; \veps} }_2 + \frac{\cstloss}{\delta} \int_0^{\flr{t}} \Phi_{\vRt}(t-s) \norm{\vl_{\flr{s}} \prn{\vr^{\flr{s}}_\eta; \veps} }_2 \mathrm d s + \cstloss\norm{\vw_\eta^{t}}_2 \, ,
\end{align}
where in the last line we use the fact that $r_\eta^t$ is piecewise constant. Repeat the same argument in Eq.~\eqref{eq:mid-solution-in-space-2}, we obtain
\begin{align}
&\Ep \brk{\norm{\vl_{\flr{t}} \prn{\vr^t_\eta; \veps} }_2^2} \nonumber \\
& \leq 3 \brc{\mathbb{E} \brk{ \norm{\vl_{\flr{t}} \prn{0; \veps} }_2^2} + \frac{\cstloss^2}{\delta^2} \int_0^{\flr{t}} (t-s+1)^2\Phi_{\vRt}(t-s)^2 \mathbb{E} \brk{\norm{\vl_{\flr{s}} \prn{\vr^{\flr{s}}_\eta; \veps} }_2^2} \mathrm d s + \cstloss^2 \mathbb{E} \brk{\norm{\vw^t_\eta}_2^2}} \nonumber \\
& \leq 3 \brc{\cstthetaz + \frac{\cstloss^2}{\delta^2} \int_0^{\flr{t}} (t-s+1)^2\Phi_{\vRt}(t-s)^2 \pCl(s)\mathrm d s + k\cstloss^2 \Phi_{\vCt}(t)} \, , 
\end{align}
where we use $\mathbb{E} \brk{\norm{\vw^t_\eta}_2^2} \leq k \norm{\Ct^\eta(t,t)} \leq k \pCt(t)$. Comparing to Eq.~\eqref{eq:def-pCl} we have $\Ep \brk{\norm{\vl_{\flr{t}} \prn{\vr^t_\eta; \veps} }_2^2} \leq \pCl(t)$ and
\begin{align}
\norm{\Cl^\eta(t,t)} = \norm{\Ep \brk{\vl_{\flr{t}}(\vr^t_\eta; z) \vl_{\flr{t}}(\vr^t_\eta; z)^\sT }} \leq \mathbb{E} \brk{\norm{\vl_{\flr{t}} \prn{\vr^t_\eta; \veps} }_2^2} \leq \pCl(t) \, .
\end{align}
We conclude the proof by induction.

\subsection{Proof of Lemma~\ref{lem:integral-differential-eta-approximation-1}} \label{proof:integral-differential-eta-approximation-1}
We first introduce a lemma for mappings $\trsfrmA$ and $\trsfrmA^\eta$ on $X^\eta$. The reader can find its proof in Appendix~\ref{proof:integral-differential-eta-approximation-2}.
\begin{lemma} \label{lem:integral-differential-eta-approximation-2}
Under the same conditions of Lemma~\ref{lem:integral-differential-approximation}, suppose $\trsfrmA(X^\eta) = (\bCt^\eta, \bRt^\eta)$, $\trsfrmA^\eta(X^\eta)  = (\Ct^\eta, \Rt^\eta)$ and define $\brk{\Rt^\eta}(t,s):= \Rt^\eta(\flr{t}, \cil{s})$ when $\cil{s} \leq \flr{t}$, $\brk{\Rt^\eta}(t,s):= I$ when $\cil{s} > \flr{t}$. It then holds for all $\lambda \geq \wb{\lambda}_6 := \wb{\lambda}_6(\mathcal{S}, \wb{\mathcal{S}})$.
\begin{align}
	\lbddst{\bCt^\eta}{\Ct^\eta} \leq \wb{h}(\eta) \, ,\\
	\lbddst{\bRt^\eta}{\brk{\Rt^\eta}} \leq \wb{h}(\eta) \, ,
\end{align}
for some nondecreasing function $\wb{h}(\eta)$ with $\wb{h}(0) = 0$. Here the function $\wb{h}$ only depends on the spaces $\mathcal{S}$ and $\wb{\mathcal{S}}$.
\end{lemma}
Suppose $\mathcal{T}(X^\eta) = (\bCl^\eta, \bRl^\eta, \bGamma_\eta)$ and the fixed point equation $\mathcal{T}^\eta(X^\eta) = X^\eta = (\Cl^\eta, \Rl^\eta, \Gamma_\eta)$. Using the same notations in Lemma~\ref{lem:integral-differential-eta-approximation-2} we can then write out the equations determining $(\bCl^\eta, \bRl^\eta, \bGamma_\eta)$ and $(\Cl^\eta, \Rl^\eta, \Gamma_\eta)$ as
\begin{align*}
\wb{r}_\eta^t & = - \frac{1}{\delta} \int_0^t \bRt^\eta(t,s) \ell_s(\wb{r}_\eta^s; z) \de s + \wb{w}_\eta^t \, , & & w^t \sim \mathsf{GP}(0, \bCt^\eta) \, , \\
\frac{\partial \ell_t(\wb{r}_\eta^t;z)}{\partial \wb{w}_\eta^s} & = \nabla_r \ell_t(\wb{r}_\eta^t;z) \cdot \prn{- \frac{1}{\delta} \int_s^t \bRt^\eta(t,s') \frac{\partial \ell_{s'}(\wb{r}_\eta^{s'};z)}{\partial \wb{w}_\eta^s}  \de s' - \frac 1 \delta \bRt^\eta(t,s) \nabla_r \ell_s(\wb{r}_\eta^s;z) } \, , & & 0 \leq s < t \leq T \, , \\
\bCl^\eta(t,s) & = \Ep \brk{\ell_t(\wb{r}_\eta^t; z) {\ell_s(\wb{r}_\eta^s; z)}^\sT}\, , & & 0 \leq s \leq t \leq T \, , \\
\bRl^\eta(t,s) & = \Ep \brk{\frac{\partial \ell_t(\wb{r}_\eta^t;z)}{\partial \wb{w}_\eta^s}} \, , & & 0 \leq s < t \leq T \, , \\
\bGamma_\eta^t & = \Ep \brk{\nabla_r \ell_t(\wb{r}_\eta^t;z)}\, , & & 0 \leq t \leq T \, ,
\end{align*}
and
\begin{align*}
r^t_\eta & = - \frac{1}{\delta} \int_0^{\flr{t}} \brk{\Rt^\eta}(t,s) \ell_{\flr{s}}(r^{s}_\eta; z) \de s + w^t_\eta \, , \\
\frac{\partial \ell_{\flr{t}}(r^t_\eta;z)}{\partial w^s_\eta} & = \nabla_r \ell_{\flr{t}}(r^t_\eta;z) \cdot \prn{- \frac{1}{\delta} \int_{\min \brc{\cil{s}, \flr{t}}}^{\flr{t}} \brk{R_{\theta}^\eta}(t, s') \frac{\partial \ell_{\flr{s'}}(r^{s'}_\eta;z)}{\partial w^s_\eta}  \de s' - \frac 1 \delta \brk{R_\theta^\eta}(t,s) \nabla_r \ell_{\flr{s}}(r^s_\eta;z) } \, ,   \\
\Cl^\eta(t,s) & = \Ep \brk{\ell_{\flr{t}}(r^{\flr{t}}_\eta;z) \ell_{\flr{s}}(r^{\flr{s}}_\eta;z)^{\sT}} \, , \\
\Rl^\eta(t,s) & = \Ep \brk{\frac{\partial \ell_{\flr{t}}(r^t_\eta;z)}{\partial w^s_\eta}} \, ,  \\
\Gamma_\eta^t & = \Ep \brk{\nabla_r \ell_{\flr{t}}(r^t_\eta;z)} \, ,
\end{align*}
where $ w^t_\eta \sim \mathsf{GP}(0, \Ct^\eta)$. Note that since we set $\brk{\Rt^\eta}(t,s) = I$ when $\cil{s} >\flr{t}$, it is consistent with the definition in Eq.~\eqref{eq:def-derivative-l-eta-2}.

\paragraph{Controlling the distance between \texorpdfstring{$\bCl^\eta$}{TEXT} and \texorpdfstring{$\Cl^\eta$}{TEXT}.} By Lemma~\ref{lem:integral-differential-eta-approximation-2}, we can couple the Gaussian processes $\wb{w}^t_\eta$ and $w^t_\eta$ such that for all $\lambda \geq \wb{\lambda}_6$,
\begin{align}
\sup_{t \in [0,T]} e^{-\lambda t} \sqrt{\Ep \brk{\normtwo{\wb{w}^t_\eta - w^t_\eta}^2}} \leq 2 \cdot \lbddst{\bCt^\eta}{\Ct^\eta} \leq 2 \wb{h}(\eta) \, . \label{eq:mid-transformB-eta-1}
\end{align}
By definition of $\Ct^\eta$ we know $w_\eta^t$ is piecewise constant in the sense that $w_\eta^{t} = w_\eta^{\flr{t}}$. Hence $r^t_\eta = r^{\flr{t}}_\eta$ and
\begin{align}
\normtwo{\wb{r}^t_\eta - r^t_\eta} & = \normtwo{- \frac{1}{\delta} \int_0^t \bRt^\eta(t,s) \ell_s(\wb{r}_\eta^s; z) \de s + \wb{w}_\eta^t + \frac{1}{\delta} \int_0^{\flr{t}} \brk{\Rt^\eta}(t,s) \ell_{\flr{s}}(r^{s}_\eta; z) \de s - w^t_\eta } \nonumber \\
& \leq \frac 1 \delta \int_0^{\flr{t}} \norm{\bRt^\eta(t,s)} \normtwo{\ell_s(\wb{r}_\eta^s; z) - \ell_{\flr{s}}(r^{s}_\eta; z)} \de s + \frac 1 \delta \int_0^{\flr{t}} \norm{\bRt^\eta(t,s) - \brk{\Rt^\eta}(t,s)} \normtwo{\ell_{\flr{s}}(r^{s}_\eta; z)} \de s \nonumber \\
& \qquad + \frac 1 \delta \int_{\flr{t}}^t \norm{\bRt^\eta(t,s)} \normtwo{\ell_s(\wb{r}_\eta^s; z)} \de s + \normtwo{\wb{w}_\eta^t - w_\eta^t} \nonumber \\
& \leq \frac{\cstloss}{\delta} \int_0^{\flr{t}} \pRt(t-s) \prn{\normtwo{\wb{r}^s_\eta - r^s_\eta} + |s-\flr{s}|} \de s + \frac{1}{\delta} \int_0^{\flr{t}} \norm{\bRt^\eta(t,s) - \brk{\Rt^\eta}(t,s)} \normtwo{\ell_{\flr{s}}(r^{s}_\eta; z)} \de s \nonumber \\
& \qquad + \frac{\pRt(T)}{\delta} \int_{\flr{t}}^t \normtwo{\ell_s(\wb{r}_\eta^s; z)} \de s + \normtwo{\wb{w}_\eta^t - w_\eta^t} \, .
\end{align}
We choose $\lambda$ large enough such that Lemma~\ref{lem:integral-differential-eta-approximation-2} holds and we can get
\begin{align}
&e^{-\lambda t} \normtwo{\wb{r}^t_\eta - r^t_\eta} \nonumber \\
& \leq \frac{\cstloss}{\delta} \int_0^{\flr{t}} e^{-\lambda(t-s)} \pRt(t-s) \cdot e^{-\lambda s} \prn{\normtwo{\wb{r}^s_\eta - r^s_\eta} +\eta} \de s + \frac{\wb{h}(\eta)}{\delta} \int_0^{\flr{t}} \normtwo{\ell_{\flr{s}}(r^{s}_\eta; z)} \de s \nonumber \\
& \qquad + \frac{e^{-\lambda t}\pRt(T)}{\delta} \int_{\flr{t}}^t  \normtwo{\ell_s(\wb{r}_\eta^s; z)} \de s + e^{-\lambda t} \normtwo{\wb{w}_\eta^t - w_\eta^t}\, .
\end{align}
Square both sides and take expectations. It follows by Cauchy-Schwarz inequality that
\begin{align}
& e^{-2\lambda t} \Ep \brk{\normtwo{\wb{r}^t_\eta - r^t_\eta}^2} \nonumber \\
& \leq \brc{2\int_0^{\flr{t}} (t-s+1)^{-2} \de s + \int_0^t (t-s+1)^{-2} + 1} \nonumber \\
& \qquad \cdot \Bigg\{\frac{\cstloss^2}{\delta^2} \int_0^{\flr{t}} e^{-2\lambda(t-s)} \pRt(t-s)^2 \cdot e^{-2\lambda s} \prn{\Ep \brk{\normtwo{\wb{r}^s_\eta - r^s_\eta}^2} +\eta^2} \de s \nonumber \\
& \qquad  + \frac{\wb{h}(\eta)^2}{\delta^2} \int_0^{\flr{t}} \Ep \brk{\normtwo{\ell_{\flr{s}}(r^{s}_\eta; z)}^2} \de s + \frac{\pRt(T)^2}{\delta^2} \int_{\flr{t}}^t  \Ep \brk{\normtwo{\ell_s(\wb{r}_\eta^s; z)}^2} \de s  + e^{-2\lambda t} \Ep \brk{\normtwo{\wb{w}_\eta^t - w_\eta^t}^2} \Bigg\} \nonumber \\
& \leq 4 \cdot \Bigg\{\frac{\cstloss^2}{\delta^2} \int_0^{\flr{t}} e^{-2\lambda(t-s)} \pRt(t-s)^2 \cdot e^{-2\lambda s} \prn{\Ep \brk{\normtwo{\wb{r}^s_\eta - r^s_\eta}^2} +\eta^2} \de s + \frac{\wb{h}(\eta)^2 kT \pCl(T)}{\delta^2} \nonumber \\
& \qquad + \frac{\eta k\pCl(T)\pRt(T)^2 }{\delta^2} + 4 \wb{h}(\eta)^2 \Bigg\} \, ,
\end{align}
where in the last line we invoke Eq.~\eqref{eq:mid-transformB-eta-1} and use $\Ep \brk{\normtwo{\ell_s(\wb{r}_\eta^s; z)}^2} \leq k \norm{\bCl(t,t)} \leq k \pCl(T)$. Next, we take $\lambda$ large enough such that
\begin{align*}
\frac{\cstloss^2}{\delta^2} \int_0^{\infty} e^{-2\lambda t} \pRt(t)^2 \leq \frac 1 8\, ,
\end{align*}
which will further induce that
\begin{align}
& e^{-2\lambda t} \Ep \brk{\normtwo{\wb{r}^t_\eta - r^t_\eta}^2} \nonumber \\
& \leq \frac 1 2 \sup_{0 \leq s \leq t} e^{-2\lambda s} \Ep \brk{\normtwo{\wb{r}^s_\eta - r^s_\eta}^2} + \frac 1 2 \eta^2 +  \frac{4\wb{h}(\eta)^2 kT \pCl(T)}{\delta^2} + \frac{4\eta k\pCl(T)\pRt(T)^2 }{\delta^2} + 16 \wb{h}(\eta)^2 \, ,
\end{align}
and taking supremum over $t \in [0, T]$ on both sides
\begin{align}
\sup_{t \in [0, T]} e^{-\lambda t} \sqrt{\Ep \brk{\normtwo{\wb{r}^t_\eta - r^t_\eta}^2}} & \leq \sqrt{\eta^2+  \frac{8\wb{h}(\eta)^2 kT \pCl(T)}{\delta^2} + \frac{8\eta k\pCl(T)\pRt(T)^2 }{\delta^2} + 32 \wb{h}(\eta)^2} \, .
\end{align}
Further following the same coupling argument in Appendix~\ref{proof:diff-Cl-global}, we obtain
\begin{align}
\lbddst{\bCl^\eta}{\Cl^\eta} & \leq \sup_{t \in [0, T]} e^{-\lambda t} \sqrt{\Ep \brk{\normtwo{\ell_t(\wb{r}_\eta^t; z) - \ell_{\flr{t}}(r^{\flr{t}}_\eta;z) }^2}} \nonumber \\
& = \sup_{t \in [0, T]} e^{-\lambda t} \sqrt{\Ep \brk{\normtwo{\ell_t(\wb{r}_\eta^t; z) - \ell_{\flr{t}}(r^{t}_\eta;z) }^2}} \nonumber \\
& \leq \cstloss \sup_{t \in [0, T]} e^{-\lambda t} \sqrt{\Ep \brk{\prn{\normtwo{\wb{r}^t_\eta - r^t_\eta} + \eta}^2}} \,.
\end{align}
By triangle inequality, we then get
\begin{align}
\lbddst{\bCl^\eta}{\Cl^\eta} & \leq \cstloss	\sup_{t \in [0, T]} e^{-\lambda t} \sqrt{\Ep \brk{\normtwo{\wb{r}^t_\eta - r^t_\eta}^2}} +  \cstloss \eta \nonumber \\
& \leq\cstloss \sqrt{\eta^2+  \frac{8\wb{h}(\eta)^2 kT \pCl(T)}{\delta^2} + \frac{8\eta k\pCl(T)\pRt(T)^2 }{\delta^2} + 32 \wb{h}(\eta)^2} + \cstloss \eta \nonumber \\
& =: h_1(\eta) \, . \label{eq:mid-transformB-eta-2}
\end{align}
Clearly $h_1(\eta) \to 0$ as $\eta \to 0$.
\paragraph{Controlling the distances between \texorpdfstring{$\bRl^\eta$}{TEXT} and \texorpdfstring{$\Rl^\eta$}{TEXT}, \texorpdfstring{$\bGamma_\eta$}{TEXT} and \texorpdfstring{$\Gamma_\eta$}{TEXT}.} First we consider the distance between $\bGamma_\eta$ and $\Gamma_\eta$. By Eq.~\eqref{eq:mid-transformB-eta-2}, we can get
\begin{align}
\lbddst{\bGamma_\eta}{\Gamma_\eta} & = \sup_{t \in [0, T]} e^{-\lambda t} \norm{\Ep \brk{\nabla_r \ell_t(\wb{r}_\eta^t;z)} - \Ep \brk{\nabla_r \ell_{\flr{t}}(r^t_\eta;z)}} \nonumber \\
& \leq \sup_{t \in [0, T]} e^{-\lambda t} \Ep \brk{\norm{\nabla_r \ell_t(\wb{r}_\eta^t;z) - \nabla_r \ell_{\flr{t}}(r^t_\eta;z)}} \nonumber \\
& \leq \cstloss	\sup_{t \in [0, T]} e^{-\lambda t} \sqrt{\Ep \brk{\normtwo{\wb{r}^t_\eta - r^t_\eta}^2}} +  \cstloss \eta \nonumber \\
& \leq h_1(\eta) \, . \label{eq:mid-transformB-eta-3}
\end{align} 
Now we only need to bound the distance between $\bRl^\eta$ and $\Rl^\eta$. To this end, we introduce two auxiliary functions
\begin{align}
\frac{\partial \wb{r}_\eta^t}{\partial \wb{w}_\eta^s} & := - \frac{1}{\delta} \int_s^t \bRt^\eta(t,s') \frac{\partial \ell_{s'}(\wb{r}_\eta^{s'};z)}{\partial \wb{w}_\eta^s}  \de s' - \frac 1 \delta \bRt^\eta(t,s) \nabla_r \ell_s(\wb{r}_\eta^s;z) \, , & &0 \leq s < t \leq T\, ,  \\
\frac{\partial r_\eta^t}{\partial w_\eta^s} &:= - \frac{1}{\delta} \int_{\min \brc{\cil{s}, \flr{t}}}^{\flr{t}} \brk{R_{\theta}^\eta}(t, s') \frac{\partial \ell_{\flr{s'}}(r^{s'}_\eta;z)}{\partial w^s_\eta}  \de s' - \frac 1 \delta \brk{R_\theta^\eta}(t,s) \nabla_r \ell_{\flr{s}}(r^s_\eta;z)  \, , & &0 \leq s < t \leq T\, .
\end{align}
We can then write
\begin{align*}
\frac{\partial \ell_t(\wb{r}_\eta^t;z)}{\partial \wb{w}_\eta^s} & = \nabla_r \ell_t(\wb{r}_\eta^t;z) \cdot 	\frac{\partial \wb{r}_\eta^t}{\partial \wb{w}_\eta^s} \, , \nonumber \\
\frac{\partial \ell_{\flr{t}}(r^t_\eta;z)}{\partial w^s_\eta} & = \nabla_r \ell_{\flr{t}}(r^t_\eta;z) \cdot 	\frac{\partial r_\eta^t}{\partial w_\eta^s} \, .
\end{align*}
Therefore, we can derive
\begin{align}
& e^{-\lambda t} \Ep \brk{\norm{\frac{\partial \ell_t(\wb{r}_\eta^t;z)}{\partial \wb{w}_\eta^s} - \frac{\partial \ell_{\flr{t}}(r^t_\eta;z)}{\partial w^s_\eta}}} \nonumber \\
& \leq  e^{-\lambda t} \Ep \brk{\Bigg\|\nabla_r \ell_t(\wb{r}_\eta^t;z) -  \nabla_r \ell_{\flr{t}}(r^t_\eta;z) \Bigg\| \Bigg\|\frac{\partial \wb{r}_\eta^t}{\partial \wb{w}_\eta^s}\Bigg\| }  + e^{-\lambda t} \Ep \brk{\Bigg\|\nabla_r \ell_{\flr{t}}(r^t_\eta;z) \Bigg\| \Bigg\|\frac{\partial \wb{r}_\eta^t}{\partial \wb{w}_\eta^s} - \frac{\partial r_\eta^t}{\partial w_\eta^s} \Bigg\| } \, . \label{eq:mid-transformB-eta-4}
\end{align}
Since $(\bCt^\eta, \bRt^\eta) \in \wb{\mathcal{S}}$, we are allowed to invoke Eq.~\eqref{eq:Cl-diff-drdw-upper-bound} that helps us bound the first term
\begin{align}
& e^{-\lambda t} \Ep \brk{\Bigg\|\nabla_r \ell_t(\wb{r}_\eta^t;z) -  \nabla_r \ell_{\flr{t}}(r^t_\eta;z) \Bigg\| \Bigg\|\frac{\partial \wb{r}_\eta^t}{\partial \wb{w}_\eta^s}\Bigg\| }  \nonumber \\
& \leq \frac{\cstloss \Phi_{\vRt}(T)}{\delta} \cdot \exp \prn{\frac{\cstloss T \Phi_{\vRt}(T)}{\delta}} \cdot e^{-\lambda t}  \Ep \brk{\Bigg\|\nabla_r \ell_t(\wb{r}_\eta^t;z) -  \nabla_r \ell_{\flr{t}}(r^t_\eta;z) \Bigg\|} \nonumber \\
& \leq \frac{\cstloss \Phi_{\vRt}(T)}{\delta} \cdot \exp \prn{\frac{\cstloss T \Phi_{\vRt}(T)}{\delta}} \cdot h_1(\eta) \,  , \label{eq:mid-transformB-eta-5}
\end{align}
where in the last line we apply Eq.~\eqref{eq:mid-transformB-eta-3} when $t \in [0,T]$. By Lipschitz property in Assumption~\ref{ass:Normal}, we can upper bound the second term by
\begin{align}
& e^{-\lambda t} \Ep \brk{\Bigg\|\nabla_r \ell_{\flr{t}}(r^t_\eta;z) \Bigg\| \Bigg\|\frac{\partial \wb{r}_\eta^t}{\partial \wb{w}_\eta^s} - \frac{\partial r_\eta^t}{\partial w_\eta^s} \Bigg\| } \nonumber \\
& \leq \cstloss \cdot e^{-\lambda t} \Ep \brk{ \Bigg\|\frac{\partial \wb{r}_\eta^t}{\partial \wb{w}_\eta^s} - \frac{\partial r_\eta^t}{\partial w_\eta^s} \Bigg\|}  \nonumber \\
& \leq \frac{\cstloss}{\delta} \cdot e^{-\lambda t} \cdot \Bigg\{ \Ep \brk{\norm{\int_s^t \bRt^\eta(t,s') \frac{\partial \ell_{s'}(\wb{r}_\eta^{s'};z)}{\partial \wb{w}_\eta^s}  \de s' -   \int_{\min \brc{\cil{s}, \flr{t}}}^{\flr{t}} \brk{R_{\theta}^\eta}(t, s') \frac{\partial \ell_{\flr{s'}}(r^{s'}_\eta;z)}{\partial w^s_\eta}  \de s'} } \nonumber \\
& \qquad + \Ep \brk{\norm{\bRt^\eta(t,s) \nabla_r \ell_s(\wb{r}_\eta^s;z)  - \brk{R_\theta^\eta}(t,s) \nabla_r \ell_{\flr{s}}(r^s_\eta;z) }} \Bigg\} \, . \label{eq:mid-transformB-eta-6}
\end{align}
From Eq.~\eqref{eq:Cl-diff-drdw-upper-bound} we can also get for any $0 \leq s < t \leq T$,
\begin{align}
\norm{\frac{\partial \ell_{t}(\wb{r}_\eta^{t};z)}{\partial \wb{w}_\eta^s}} & \leq \cstloss \cdot \norm{\frac{\partial \wb{r}_\eta^t}{\partial \wb{w}_\eta^s}} \leq \frac{\cstloss^2 \Phi_{\vRt}(T)}{\delta} \cdot \exp \prn{\frac{\cstloss T \Phi_{\vRt}(T)}{\delta}} \, ,
\end{align}
and therefore
\begin{align}
&e^{-\lambda t} \Ep \brk{\norm{\int_s^t \bRt^\eta(t,s') \frac{\partial \ell_{s'}(\wb{r}_\eta^{s'};z)}{\partial \wb{w}_\eta^s}  \de s' -   \int_{\min \brc{\cil{s}, \flr{t}}}^{\flr{t}} \brk{R_{\theta}^\eta}(t, s') \frac{\partial \ell_{\flr{s'}}(r^{s'}_\eta;z)}{\partial w^s_\eta}  \de s'} } \nonumber \\
& \leq 2 \eta \pRt(T) \cdot \frac{\cstloss^2 \Phi_{\vRt}(T)}{\delta} \cdot \exp \prn{\frac{\cstloss T \Phi_{\vRt}(T)}{\delta}} \nonumber \\
& \qquad + e^{-\lambda t} \Ep \brk{\norm{\int_{\min \brc{\cil{s}, \flr{t}}}^{\flr{t}} \bRt^\eta(t,s') \frac{\partial \ell_{s'}(\wb{r}_\eta^{s'};z)}{\partial \wb{w}_\eta^s}  \de s' -   \int_{\min \brc{\cil{s}, \flr{t}}}^{\flr{t}} \brk{R_{\theta}^\eta}(t, s') \frac{\partial \ell_{\flr{s'}}(r^{s'}_\eta;z)}{\partial w^s_\eta}  \de s'}} \nonumber \\
& \leq 2 \eta \pRt(T) \cdot \frac{\cstloss^2 \Phi_{\vRt}(T)}{\delta} \cdot \exp \prn{\frac{\cstloss T \Phi_{\vRt}(T)}{\delta}} \nonumber \\
& \qquad + e^{-\lambda t} \Ep \brk{\int_{\min \brc{\cil{s}, \flr{t}}}^{\flr{t}} \norm{\bRt^\eta(t,s') - \brk{R_{\theta}^\eta}(t, s')} \cdot \norm{\frac{\partial \ell_{s'}(\wb{r}_\eta^{s'};z)}{\partial \wb{w}_\eta^s} } \de s' } \nonumber \\
& \qquad + e^{-\lambda t} \Ep \brk{\int_{\min \brc{\cil{s}, \flr{t}}}^{\flr{t}} \norm{\brk{R_{\theta}^\eta}(t,s')} \cdot \norm{\frac{\partial \ell_{s'}(\wb{r}_\eta^{s'};z)}{\partial \wb{w}_\eta^s} - \frac{\partial \ell_{\flr{s'}}(r^{s'}_\eta;z)}{\partial w^s_\eta} } \de s' } \nonumber \\
& \leq 2 \eta \pRt(T) \cdot \frac{\cstloss^2 \Phi_{\vRt}(T)}{\delta} \cdot \exp \prn{\frac{\cstloss T \Phi_{\vRt}(T)}{\delta}} + \lbddst{\bRt^\eta}{\brk{R_{\theta}^\eta}} \cdot T \cdot \frac{\cstloss^2 \Phi_{\vRt}(T)}{\delta} \cdot \exp \prn{\frac{\cstloss T \Phi_{\vRt}(T)}{\delta}} \nonumber \\
& \qquad + \int_{\min \brc{\cil{s}, \flr{t}}}^{\flr{t}} e^{-\lambda(t-s')} \pRt(t-s') \cdot e^{-\lambda s'} \Ep \brk{\norm{\frac{\partial \ell_{s'}(\wb{r}_\eta^{s'};z)}{\partial \wb{w}_\eta^s} - \frac{\partial \ell_{\flr{s'}}(r^{s'}_\eta;z)}{\partial w^s_\eta} }} \de s' \, .
\end{align}
Take $\lambda$ large enough such that Lemma~\ref{lem:integral-differential-eta-approximation-2} holds and also
\begin{align*}
\int_0^\infty e^{-\lambda t} \pRt(t) \de t \leq \frac 1 2\, ,
\end{align*}
we can further get
\begin{align}
&e^{-\lambda t} \Ep \brk{\norm{\int_s^t \bRt^\eta(t,s') \frac{\partial \ell_{s'}(\wb{r}_\eta^{s'};z)}{\partial \wb{w}_\eta^s}  \de s' -   \int_{\min \brc{\cil{s}, \flr{t}}}^{\flr{t}} \brk{R_{\theta}^\eta}(t, s') \frac{\partial \ell_{\flr{s'}}(r^{s'}_\eta;z)}{\partial w^s_\eta}  \de s'} } \nonumber \\
& \leq  2 \eta \pRt(T) \cdot \frac{\cstloss^2 \Phi_{\vRt}(T)}{\delta} \cdot \exp \prn{\frac{\cstloss T \Phi_{\vRt}(T)}{\delta}} + \wb{h}(\eta) \cdot T \cdot \frac{\cstloss^2 \Phi_{\vRt}(T)}{\delta} \cdot \exp \prn{\frac{\cstloss T \Phi_{\vRt}(T)}{\delta}} \nonumber \\
& \qquad + \frac 1 2 \sup_{0 \leq s < t \leq T} e^{-\lambda t} \Ep \brk{\norm{\frac{\partial \ell_t(\wb{r}_\eta^t;z)}{\partial \wb{w}_\eta^s} - \frac{\partial \ell_{\flr{t}}(r^t_\eta;z)}{\partial w^s_\eta}}} \nonumber \\
& =: h_2(\eta) + \frac 1 2 \sup_{0 \leq s < t \leq T} e^{-\lambda t} \Ep \brk{\norm{\frac{\partial \ell_t(\wb{r}_\eta^t;z)}{\partial \wb{w}_\eta^s} - \frac{\partial \ell_{\flr{t}}(r^t_\eta;z)}{\partial w^s_\eta}}} \, ,  \label{eq:mid-transformB-eta-7}
\end{align}
where $h_2(\eta) \to 0$ when $\eta$ approaches $0$. For the same $\lambda$, we also get
\begin{align}
& e^{-\lambda t}  \Ep \brk{\norm{\bRt^\eta(t,s) \nabla_r \ell_s(\wb{r}_\eta^s;z)  - \brk{R_\theta^\eta}(t,s) \nabla_r \ell_{\flr{s}}(r^s_\eta;z) }} \nonumber \\
& \leq e^{-\lambda t} \Ep \brk{\norm{\bRt^\eta(t,s)  - \brk{R_\theta^\eta}(t,s)}\norm{\nabla_r \ell_s(\wb{r}_\eta^s;z)}} + e^{-\lambda t} \Ep \brk{\norm{\brk{R_\theta^\eta}(t,s)} \norm{\nabla_r \ell_s(\wb{r}_\eta^s;z) - \nabla_r \ell_{\flr{s}}(r^s_\eta;z)}} \nonumber \\
& \leq \cstloss \wb{h}(\eta) + \pRt(T) h_1(\eta) \, , \label{eq:mid-transformB-eta-8}
\end{align}
in the last line we make use of Eq.~\eqref{eq:mid-transformB-eta-3}. Taking Eqs.~\eqref{eq:mid-transformB-eta-7} and \eqref{eq:mid-transformB-eta-8} into Eq.~\eqref{eq:mid-transformB-eta-6} yields
\begin{align}
& e^{-\lambda t} \Ep \brk{\Bigg\|\nabla_r \ell_{\flr{t}}(r^t_\eta;z) \Bigg\| \Bigg\|\frac{\partial \wb{r}_\eta^t}{\partial \wb{w}_\eta^s} - \frac{\partial r_\eta^t}{\partial w_\eta^s} \Bigg\| } \nonumber \\
&\leq  h_2(\eta) +  \cstloss \wb{h}(\eta) + \pRt(T) h_1(\eta) + \frac 1 2 \sup_{0 \leq s < t \leq T} e^{-\lambda t} \Ep \brk{\norm{\frac{\partial \ell_t(\wb{r}_\eta^t;z)}{\partial \wb{w}_\eta^s} - \frac{\partial \ell_{\flr{t}}(r^t_\eta;z)}{\partial w^s_\eta}}} \, .
\end{align}
Further with Eq.~\eqref{eq:mid-transformB-eta-5}, substituting into Eq.~\eqref{eq:mid-transformB-eta-4} gives us 
\begin{align}
&e^{-\lambda t} \Ep \brk{\norm{\frac{\partial \ell_t(\wb{r}_\eta^t;z)}{\partial \wb{w}_\eta^s} - \frac{\partial \ell_{\flr{t}}(r^t_\eta;z)}{\partial w^s_\eta}}} \nonumber \\
& \leq \frac{\cstloss \Phi_{\vRt}(T)}{\delta} \cdot \exp \prn{\frac{\cstloss T \Phi_{\vRt}(T)}{\delta}} \cdot h_1(\eta) +  h_2(\eta) +  \cstloss \wb{h}(\eta) + \pRt(T) h_1(\eta) \nonumber \\
& \qquad +  \frac 1 2 \sup_{0 \leq s < t \leq T} e^{-\lambda t} \Ep \brk{\norm{\frac{\partial \ell_t(\wb{r}_\eta^t;z)}{\partial \wb{w}_\eta^s} - \frac{\partial \ell_{\flr{t}}(r^t_\eta;z)}{\partial w^s_\eta}}} \, .
\end{align} 
Taking supremum on both sides for $0 \leq s < t \leq T$ it then follows that
\begin{align}
&\sup_{0 \leq s < t \leq T} e^{-\lambda t} \Ep \brk{\norm{\frac{\partial \ell_t(\wb{r}_\eta^t;z)}{\partial \wb{w}_\eta^s} - \frac{\partial \ell_{\flr{t}}(r^t_\eta;z)}{\partial w^s_\eta}}} \nonumber \\
&\leq 2 \brc{\frac{\cstloss \Phi_{\vRt}(T)}{\delta} \cdot \exp \prn{\frac{\cstloss T \Phi_{\vRt}(T)}{\delta}} \cdot h_1(\eta) +  h_2(\eta) +  \cstloss \wb{h}(\eta) + \pRt(T) h_1(\eta)} =: h_3(\eta) \, ,
\end{align}
and $h_3(\eta) \to 0$ when $\eta \to 0$. Finally
\begin{align}
\lbddst{\bCl^\eta}{\Cl^\eta} & = \sup_{0 \leq s < t \leq T} e^{-\lambda t} \norm{\Ep \brk{\frac{\partial \ell_t(\wb{r}_\eta^t;z)}{\partial \wb{w}_\eta^s}} - \Ep \brk{\frac{\partial \ell_{\flr{t}}(r^t_\eta;z)}{\partial w^s_\eta}}} \nonumber \\
& \leq e^{-\lambda t} \Ep \brk{\norm{\frac{\partial \ell_t(\wb{r}_\eta^t;z)}{\partial \wb{w}_\eta^s} - \frac{\partial \ell_{\flr{t}}(r^t_\eta;z)}{\partial w^s_\eta}}} \nonumber \\
& \leq h_3(\eta) \, . \label{eq:mid-transformB-eta-9}
\end{align}
By Eqs.~\eqref{eq:mid-transformB-eta-2}, \eqref{eq:mid-transformB-eta-3} and \eqref{eq:mid-transformB-eta-9}, we conclude the proof by taking $h(\eta) = \max \brc{h_1(\eta), h_3(\eta)}$.
\subsection{Proof of Lemma~\ref{lem:integral-differential-eta-approximation-path}} \label{proof:integral-differential-eta-approximation-path}
By Eqs.~\eqref{eq:def-theta} and \eqref{eq:def-theta-eta}, we can write
\begin{align*}
\frac{\de}{\de t} \theta^t & = -(\Lambda^t + \Gamma^t) \theta^t - \int_0^t R_{\ell}(t,s) \theta^s \de s + u^t \, , & & u^t \sim \mathsf{GP}(0, \Cl / \delta) \, , \\
\frac{\de}{\de t} \theta^t_\eta & = -(\Lambda^{\flr{t}} + \Gamma^{t}_\eta) \theta^{\flr{t}}_\eta - \int_0^{\flr{t}} R_{\ell}^\eta(t,s) \theta^{\flr{s}}_\eta \de s + u^t_\eta \, , & & u^t_\eta \sim \mathsf{GP}(0, \Cl^\eta / \delta) \, ,
\end{align*}
where we use the fact that $\Rl^\eta(\flr{t}, \flr{s})= \Rl^\eta(t,s)$ and $\Gamma_\eta^{\flr{t}} = \Gamma_\eta^t$. We can couple $u^t$ and $u^t_\eta$ such that
\begin{align}
\sup_{t \in [0, T]} e^{-\lambda t} \sqrt{\mathbb{E} \brk{\norm{\vu^t - \vu^t_\eta}_2^2}} \leq2 \cdot \lbddst{u^t}{u^t_\eta} = 2 \cdot \lbddst{\Cl / \delta}{\Cl^\eta / \delta} = \frac{2}{\sqrt \delta} \lbddst{\Cl}{\Cl^\eta} \, . \label{eq:mid-approximation-path-1}
\end{align}
We can derive the upper bound
\begin{align}
& e^{-\lambda t} \cdot \ddt \normtwo{\theta^t - \theta^t_\eta} \nonumber \\
& \leq e^{-\lambda t} \cdot \normtwo{\ddt \theta^t - \ddt \theta^t_\eta} \nonumber \\
& \leq e^{-\lambda t} \cdot \normtwo{-(\Lambda^t + \Gamma^t) \theta^t - \int_0^t R_{\ell}(t,s) \theta^s \de s + u^t + (\Lambda^{\flr{t}} + \Gamma^{t}_\eta) \theta^{\flr{t}}_\eta + \int_0^{\flr{t}} R_{\ell}^\eta(t,s) \theta^{\flr{s}}_\eta \de s - u^t_\eta} \nonumber \\
& \leq \underbrace{e^{-\lambda t} \cdot \Bigg\|(\Lambda^t + \Gamma^t) \theta^t - (\Lambda^{\flr{t}} + \Gamma^{t}_\eta) \theta^{\flr{t}}_\eta\Bigg\|_2}_{\mathrm{(I)}} + \underbrace{e^{-\lambda t} \cdot \normtwo{\int_0^t R_{\ell}(t,s) \theta^s \de s - \int_0^{\flr{t}} R_{\ell}^\eta(t,s) \theta^{\flr{s}}_\eta \de s }}_{\mathrm{(II)}} \nonumber \\
& \qquad + e^{-\lambda t} \cdot \normtwo{u^t - u^t_\eta} \, .
\end{align}
We upper bound (I) and (II) respectively
\begin{align}
\mathrm{(I)} & \leq e^{-\lambda t} \normtwo{(\Lambda^t + \Gamma^t) (\theta^t - \theta_\eta^t)} + e^{-\lambda t} \normtwo{(\Lambda^t + \Gamma^t) (\theta_\eta^t - \theta_\eta^{\flr{t}})} + e^{-\lambda t} \normtwo{(\Lambda^t + \Gamma^t - \Lambda^{\flr{t}} - \Gamma_\eta^t) \theta_\eta^{\flr{t}}} \nonumber \\
& \leq e^{-\lambda t} \prn{\cstlbd + \cstloss} \cdot \prn{\normtwo{\theta^t - \theta_\eta^t} + \normtwo{\theta_\eta^t - \theta_\eta^{\flr{t}}}} + e^{-\lambda t} \prn{\norm{\Lambda^t - \Lambda^{\flr{t}}} + \norm{\Gamma^t - \Gamma_\eta^t} } \normtwo{\theta_\eta^{\flr{t}}} \nonumber \\
& \leq e^{-\lambda t} \prn{\cstlbd + \cstloss} \cdot \prn{\normtwo{\theta^t - \theta_\eta^t} + \normtwo{\theta_\eta^t - \theta_\eta^{\flr{t}}}} +  e^{-\lambda t} \prn{ \eta \cstlbd + \norm{\Gamma^t - \Gamma_\eta^t}}  \normtwo{\theta_\eta^{\flr{t}}}  \, ,  
\end{align}
and
\begin{align}
\mathrm{(II)} &\leq  e^{-\lambda t} \cdot  \normtwo{\int_{\flr{t}}^t \Rl(t,s) \theta^s \de s} + e^{-\lambda t} \cdot \normtwo{\int_0^{\flr{t}} \prn{\Rl(t,s) - \Rl^\eta(t,s)} \theta^s  \de s} \nonumber \\
& \qquad +e^{-\lambda t} \cdot \normtwo{\int_0^{\flr{t}} \Rl^\eta(t,s) \prn{\theta^s - \theta_\eta^{s}}  \de s} + e^{-\lambda t} \cdot \normtwo{\int_0^{\flr{t}} \Rl^\eta(t,s) \prn{\theta^s_\eta - \theta_\eta^{\flr{s}}}  \de s} \nonumber \\
& \leq e^{-\lambda t} \pRl(T) \int_{\flr{t}}^t \normtwo{\theta^s} \de s + e^{-\lambda t} \int_0^{\flr{t}} \norm{\Rl(t,s) - \Rl^\eta(t,s)} \normtwo{\theta^s} \de s \nonumber \\
& \qquad +  \int_0^{\flr{t}} e^{-\lambda(t-s)} \pRl(t-s) \cdot e^{-\lambda s} \normtwo{\theta^s - \theta^s_\eta} \de s + e^{-\lambda t}\pRl(T) \int_0^{\flr{t}} \normtwo{\theta^s_\eta - \theta_\eta^{\flr{s}}} \de s \, .
\end{align}
Suppose $\wb{\lambda}$ satisfies $\int_0^\infty e^{-\wb{\lambda} t} \pRl(t) \de t \leq \cstlbd + \cstloss$, combining inequalities above yields
\begin{align}
& e^{-\wb{\lambda} t} \cdot \ddt \normtwo{\theta^t - \theta^t_\eta} \nonumber \\
& \leq e^{-\wb{\lambda} t} \prn{\cstlbd + \cstloss} \cdot \prn{\normtwo{\theta^t - \theta_\eta^t} + \normtwo{\theta_\eta^t - \theta_\eta^{\flr{t}}}} + e^{-\wb{\lambda} t} \prn{ \eta \cstlbd + \norm{\Gamma^t - \Gamma_\eta^t}}  \normtwo{\theta_\eta^{\flr{t}}} \nonumber \\
&\qquad +  e^{-\wb{\lambda} t} \pRl(T) \int_{\flr{t}}^t \normtwo{\theta^s} \de s +  e^{-\wb{\lambda} t} \int_0^{\flr{t}} \norm{\Rl(t,s) - \Rl^\eta(t,s)} \normtwo{\theta^s} \de s \nonumber \\
& \qquad + (\cstlbd +\cstloss) \cdot \sup_{0 \leq s \leq t} e^{-\wb{\lambda} s} \normtwo{\theta^s - \theta_\eta^s} + e^{-\wb{\lambda} t} \pRl(T) \int_0^{\flr{t}} \normtwo{\theta_\eta^s - \theta_\eta^{\flr{s}}} \de s  + e^{-\wb{\lambda} t} \normtwo{u^t - u^t_\eta} \nonumber \\
& \leq 2(\cstlbd + \cstloss) \cdot \sup_{0 \leq s \leq t} e^{-\wb{\lambda} s} \normtwo{\theta^s - \theta_\eta^s} + e^{-\wb{\lambda} t} \normtwo{u^t - u^t_\eta} + e^{-\wb{\lambda} t} (\cstlbd + \cstloss) \normtwo{\theta_\eta^t - \theta_\eta^{\flr{t}}} \nonumber \\
& \qquad  + e^{-\wb{\lambda} t} \pRl(T) \int_0^{\flr{t}} \normtwo{\theta_\eta^s - \theta_\eta^{\flr{s}}} \de s + e^{-\wb{\lambda} t} \prn{ \eta \cstlbd + \norm{\Gamma^t - \Gamma_\eta^t}} \normtwo{\theta_\eta^{\flr{t}}} \nonumber \\
& \qquad  + e^{-\wb{\lambda} t} \pRl(T) \int_{\flr{t}}^t \normtwo{\theta^s} \de s + e^{-\wb{\lambda} t} \int_0^{\flr{t}} \norm{\Rl(t,s) - \Rl^\eta(t,s)} \normtwo{\theta^s} \de s \, .
\end{align}
Similar to previous proofs in Eqs.~\eqref{eq:argument-lambda-bar-2} and \eqref{eq:argument-lambda-bar-3}, it follows that
\begin{align}
&e^{-2(\cstlbd + \cstloss)t  - \wb{\lambda}t} \normtwo{\theta^t - \theta^t_\eta} \nonumber \\
& \leq \int_0^t e^{-2(\cstlbd + \cstloss)s  - \wb{\lambda}s} \cdot \Bigg\{  \normtwo{u^s - u^s_\eta}  + (\cstlbd + \cstloss) \normtwo{\theta_\eta^s - \theta_\eta^{\flr{s}}} +  \pRl(T) \int_0^{\flr{s}} \normtwo{\theta_\eta^{s'} - \theta_\eta^{\flr{s'}}} \de s' \nonumber \\
& \qquad +  \prn{ \eta \cstlbd + \norm{\Gamma^s - \Gamma_\eta^s}} \normtwo{\theta_\eta^{\flr{s}}} + \pRl(T) \int_{\flr{s}}^s \normtwo{\theta^{s'}} \de s' + \int_0^{\flr{s}} \norm{\Rl(s,s') - \Rl^\eta(s,s')} \normtwo{\theta^{s'}} \de s' \Bigg\} \de s \, .
\end{align}
This implies for any $\lambda \geq 2(\cstlbd + \cstloss) + \wb{\lambda}$, one has
\begin{align}
&e^{-\lambda t} \normtwo{\theta^t - \theta^t_\eta} \nonumber \\
& \leq \int_0^t e^{-\lambda s} \cdot \Bigg\{  \normtwo{u^s - u^s_\eta}  + (\cstlbd + \cstloss) \normtwo{\theta_\eta^s - \theta_\eta^{\flr{s}}} +  \pRl(T) \int_0^{\flr{s}}  \normtwo{\theta_\eta^{s'} - \theta_\eta^{\flr{s'}}} \de s' \nonumber \\
& \qquad +  \prn{ \eta \cstlbd + \norm{\Gamma^s - \Gamma_\eta^s}} \normtwo{\theta_\eta^{\flr{s}}} + \pRl(T) \int_{\flr{s}}^s \normtwo{\theta^{s'}} \de s' + \int_0^{\flr{s}} \norm{\Rl(s,s') - \Rl^\eta(s,s')} \normtwo{\theta^{s'}} \de s' \Bigg\} \de s \nonumber \\
& \leq \int_0^t \Bigg( e^{-\lambda s} \normtwo{u^s - u^s_\eta}  + (\cstlbd + \cstloss) \normtwo{\theta_\eta^s - \theta_\eta^{\flr{s}}} +  \pRl(T) \int_0^{\flr{s}} \normtwo{\theta_\eta^{s'} - \theta_\eta^{\flr{s'}}} \de s' \nonumber \\
& \qquad +  \prn{ \eta \cstlbd + \lbddst{\Gamma}{\Gamma_\eta}} \normtwo{\theta_\eta^{\flr{s}}} + \pRl(T) \int_{\flr{s}}^s \normtwo{\theta^{s'}} \de s' + \lbddst{\Rl}{\Rl^\eta} \cdot \int_0^{\flr{s}} \normtwo{\theta^{s'}} \de s' \Bigg)\de s \, .
\end{align}
Square both sides and take expectations, we have
\begin{align}
& e^{-2\lambda t} \Ep \brk{\normtwo{\theta^t - \theta^t_\eta}^2} \nonumber \\
& \leq \Ep \Bigg[ \Bigg\{ \int_0^t \Bigg( e^{-\lambda s} \normtwo{u^s - u^s_\eta}  + (\cstlbd + \cstloss) \normtwo{\theta_\eta^s - \theta_\eta^{\flr{s}}} +  \pRl(T) \int_0^{\flr{s}} \normtwo{\theta_\eta^{s'} - \theta_\eta^{\flr{s'}}} \de s' \nonumber \\
& \qquad +  \prn{ \eta \cstlbd + \lbddst{\Gamma}{\Gamma_\eta}} \normtwo{\theta_\eta^{\flr{s}}} + \pRl(T) \int_{\flr{s}}^s \normtwo{\theta^{s'}} \de s' + \lbddst{\Rl}{\Rl^\eta} \cdot \int_0^{\flr{s}} \normtwo{\theta^{s'}} \de s' \Bigg)\de s  \Bigg\}^2 \Bigg] \nonumber \\
& \stackrel{\mathrm{(i)}}{\leq } \Ep \Bigg[ \Bigg\{ \int_0^t \prn{1 + 1 + \int_0^{\flr{s}} 1 \de s' + 1 + \int_{\flr{s}}^s 1 \de s' + \int_0^{\flr{s}} 1 \de s'} \de s \Bigg\} \nonumber \\
& \qquad \cdot \Bigg\{ \int_0^t \Bigg( e^{-2\lambda s} \normtwo{u^s - u^s_\eta}^2  + (\cstlbd + \cstloss)^2 \normtwo{\theta_\eta^s - \theta_\eta^{\flr{s}}}^2 +  \pRl(T)^2 \int_0^{\flr{s}} \normtwo{\theta_\eta^{s'} - \theta_\eta^{\flr{s'}}}^2 \de s' \nonumber \\
& \qquad +  \prn{ \eta \cstlbd + \lbddst{\Gamma}{\Gamma_\eta}}^2 \normtwo{\theta_\eta^{\flr{s}}}^2 + \pRl(T)^2 \int_{\flr{s}}^s \normtwo{\theta^{s'}}^2 \de s' + \lbddst{\Rl}{\Rl^\eta}^2 \cdot \int_0^{\flr{s}} \normtwo{\theta^{s'}}^2 \de s' \Bigg)\de s  \Bigg\} \Bigg] \nonumber \\
& \leq \Bigg\{ \int_0^t (2s + 3) \de s \Bigg\} \cdot  \Bigg\{ \int_0^t \Bigg( e^{-2\lambda s} \Ep \brk{\normtwo{u^s - u^s_\eta}^2}  + (\cstlbd + \cstloss)^2 \Ep \brk{ \normtwo{\theta_\eta^s - \theta_\eta^{\flr{s}}}^2}  \nonumber \\
& \qquad  +  \pRl(T)^2 \int_0^{\flr{s}} \Ep \brk{\normtwo{\theta_\eta^{s'} - \theta_\eta^{\flr{s'}}}^2} \de s' +  \prn{ \eta \cstlbd + \lbddst{\Gamma}{\Gamma_\eta}}^2 \Ep \brk{\normtwo{\theta_\eta^{\flr{s}}}^2} + \pRl(T)^2 \int_{\flr{s}}^s \Ep \brk{\normtwo{\theta^{s'}}^2} \de s' \nonumber \\
& \qquad  + \lbddst{\Rl}{\Rl^\eta}^2 \cdot \int_0^{\flr{s}}\Ep \brk{ \normtwo{\theta^{s'}}^2} \de s' \Bigg)\de s  \Bigg\} \, , 
\end{align}
where in (i) we use Cauchy-Schwarz inequality. Substituting in Eqs.~\eqref{eq:mid-transformA-eta-1}, \eqref{eq:mid-approximation-path-1} and $\Ep \brk{\normtwo{\theta}^2} \leq k \norm{\Ep \brk{\theta \theta^\sT}}$ yields
\begin{align}
& e^{-2\lambda t} \Ep \brk{\normtwo{\theta^t - \theta^t_\eta}^2} \nonumber \\
& \leq (T^2 + 3T) \cdot \int_0^t \Bigg( \frac{4}{\delta} \lbddst{\Cl}{\Cl^\eta}^2 + (\cstlbd + \cstloss)^2 \wb{h}_1(\eta) +  \pRl(T)^2 \int_0^{\flr{s}} \wb{h}_1(\eta) \de s' \nonumber \\
& \qquad   +  \prn{ \eta \cstlbd + \lbddst{\Gamma}{\Gamma_\eta}}^2 k \pCt(T) + \pRl(T)^2 \int_{\flr{s}}^s k \pCt(T) \de s'  + \lbddst{\Rl}{\Rl^\eta}^2 \cdot \int_0^{\flr{s}}k \pCt(T) \de s'   \Bigg)\de s   \nonumber \\
& \leq (T^3 + 3T^2) \cdot \Bigg( \frac{4}{\delta} \lbddst{\Cl}{\Cl^\eta}^2 + (\cstlbd + \cstloss)^2 \wb{h}_1(\eta) +   T\pRl(T)^2  \wb{h}_1(\eta) +  \prn{ \eta \cstlbd + \lbddst{\Gamma}{\Gamma_\eta}}^2 k \pCt(T)\nonumber \\
& \qquad    + \eta \pRl(T)^2 k \pCt(T)   + \lbddst{\Rl}{\Rl^\eta}^2 kT \pCt(T)  \Bigg) \, .
\end{align}
The proof is then completed by
\begin{align}
&e^{-\lambda t} \sqrt{\Ep \brk{\normtwo{\theta^t - \theta^t_\eta}^2}} \nonumber \\
& \leq  (T^3 + 3T^2)^{\frac 1 2} \cdot \Bigg( \frac{4}{\delta} \lbddst{X}{X^\eta}^2 + (\cstlbd + \cstloss)^2 \wb{h}_1(\eta) +   T\pRl(T)^2  \wb{h}_1(\eta)   \nonumber \\
& \qquad +  \prn{ \eta \cstlbd + \lbddst{X}{X^\eta}}^2 k \pCt(T) + \eta \pRl(T)^2 k \pCt(T)     + \lbddst{X}{X^\eta}^2 kT \pCt(T)  \Bigg)^{\frac 1 2} \nonumber \\
& =: H(\eta, \lbddst{X}{X^\eta})\, .
\end{align}

\subsection{Proof of Lemma~\ref{lem:integral-differential-eta-approximation-2}} \label{proof:integral-differential-eta-approximation-2}
We write the equations that define $(\bCt^\eta, \bRt^\eta)$ and $(\Ct^\eta, \Rt^\eta)$ as
\begin{align*}
\frac{\de}{\de t} \wb{\theta}^t_\eta & = -(\Lambda^t + \Gamma^t_\eta) \wb{\theta}^t_\eta - \int_0^t R_{\ell}^\eta(t,s) \wb{\theta}^s_\eta \de s + \wb{u}^t_\eta \, , & & \wb{u}^t_\eta \sim \mathsf{GP}(0, \Cl^\eta/\delta) \, , \\
\frac{\de}{\de t} \frac{\partial \wb{\theta}^t_\eta}{\partial \wb{u}^s_\eta} & = -(\Lambda^t + \Gamma^t_\eta) \frac{\partial \wb{\theta}_\eta^t}{\partial \wb{u}_\eta^s} - \int_s^t R_{\ell}^\eta(t,s') \frac{\partial \wb{\theta}_\eta^{s'}}{\partial \wb{u}_\eta^s} \de s'\, , & & 0 \leq s \leq t \leq T \, , \\
\bCt^\eta(t,s) & = \Ep \brk{\wb{\vt}^t_\eta {\wb{\vt}^s_\eta}^\sT}\, , & & 0 \leq s \leq t \leq T \, , \\
\bRt^\eta(t,s) & = \Ep \brk{\frac{\partial \wb{\theta}^t_\eta}{\partial \wb{u}^s_\eta}} \, , & & 0 \leq s \leq t \leq T \, ,
\end{align*}
and
\begin{align*}
\ddt  \theta^t_\eta & = -(\Lambda^{\flr{t}} + \Gamma^{\flr{t}}_\eta) \theta^{\flr{t}}_\eta - \int_0^{\flr{t}} R_{\ell}^\eta(\flr{t},\flr{s}) \theta^{\flr{s}}_\eta \de s + u^t_\eta \, , & & u^t_\eta \sim \mathsf{GP}(0, \Cl^\eta / \delta) \, , \\
\frac{\de}{\de t} \frac{\partial \theta^t_\eta}{\partial u^s_\eta} & = -(\Lambda^{\flr{t}} + \Gamma^{\flr{t}}_\eta) \frac{\partial \theta^{\flr{t}}_\eta}{\partial u^s_\eta} - \int_s^{\flr{t}} R_{\ell}^\eta(\flr{t},\flr{s'}) \frac{\partial \theta^{\flr{s'}}_\eta}{\partial u^{s}_\eta} \de s'\, , & & 0 \leq s \leq t \leq T \, ,  \\
\Ct^\eta(t,s) & = \Ep \brk{\vt^{\flr{t}}_\eta {\vt^{\flr{s}}_\eta}^\sT} \, , & & 0 \leq s \leq t \leq T \, , \\
\Rt^\eta(t,s) & = \Ep \brk{\frac{\partial \theta^t_\eta}{\partial u^s_\eta}} \, , & & 0 \leq s \leq t \leq T \, .
\end{align*}
\paragraph{Controlling the distance between \texorpdfstring{$\bCt^\eta$}{TEXT} and \texorpdfstring{$\Ct^\eta$}{TEXT}.} From Eq.~\eqref{eq:def-C-l-eta} we know that $\Cl^\eta$ is piecewise constant, i.e. $\Cl^\eta(t,s) = \Cl^\eta (\flr{t}, \flr{s})$ and $\vu^t_\eta = \vu^{\flr{t}}_\eta$. Further since $\vRl^\eta$ is piecewise constant from Eq.~\eqref{eq:def-R-l-eta}, it follows that
\begin{align}
\vt_\eta^t - \vt_\eta^{\flr{t}} = (t- \flr{t}) \cdot \prn{-(\Lambda^{\flr{t}} + \Gamma^{\flr{t}}_\eta) \theta^{\flr{t}}_\eta - \int_0^{\flr{t}} R_{\ell}^\eta(t,s) \theta^{\flr{s}}_\eta \de s + u^{\flr{t}}_\eta } \, .
\end{align}
Hence using the fact that $(\Cl^\eta, \Rl^\eta, \Gamma_\eta) \in \mathcal{S}$,
\begin{align}
\Ep \brk{\normtwo{\vt_\eta^t - \vt_\eta^{\flr{t}}}^2} & \leq \eta^2 \cdot \Ep \brk{\normtwo{-(\Lambda^{\flr{t}} + \Gamma^{\flr{t}}_\eta) \theta^{\flr{t}}_\eta - \int_0^{\flr{t}} R_{\ell}^\eta(t, s) \theta^{\flr{s}}_\eta \de s + u^{\flr{t}}_\eta}^2} \nonumber \\
& \leq \eta^2 \cdot \Ep \brk{\prn{\prn{\cstlbd + \cstloss} \cdot \normtwo{\theta_\eta^{\flr{t}}} + \int_0^{\flr{t}} \pRl(t-s) \normtwo{\theta_\eta^{\flr{s}}} ds + \normtwo{u_\eta^{\flr{t}}}   }^2} \nonumber \\
& \leq \eta^2 \cdot \Ep \brk{\prn{\prn{\cstlbd + \cstloss}^2 \cdot \normtwo{\theta_\eta^{\flr{t}}}^2 + \int_0^{\flr{t}} \pRl(t-s)^2 \normtwo{\theta_\eta^{\flr{s}}}^2 ds + \normtwo{u_\eta^{\flr{t}}}^2   } \cdot \prn{1 + T + 1}} \nonumber \\
& \leq \eta^2 \cdot (T+2) \cdot \brc{\prn{\cstlbd + \cstloss}^2 \cdot k\pCt(T) + T \pRl(T)^2 \cdot k\pCt(T) + \frac{k}{\delta} \pCl(T)} \nonumber \\
& =: \wb{h}_1(\eta) \, . \label{eq:mid-transformA-eta-1}
\end{align}
where in the last line we use the inequality $\Ep \brk{\norm{\theta}_2^2} = \Tr \prn{\Ep \brk{\theta \theta^\sT}} \leq k \norm{\Ep \brk{\theta \theta^\sT}}$ for any $k$ dimensional random vector $\theta$. Note that $\wb{h}_1(\eta) \to 0$ as $\eta \to 0$. Next by using the coupling $\wb{u}^t_\eta = u^t_\eta$ and using $\Gamma_\eta^t = \Gamma_\eta^{\flr{t}}$ by Eq.~\eqref{eq:def-Gamma-eta},
\begin{align}
\ddt \normtwo{\wb{\vt}^t_\eta - \vt^t_\eta} & \leq \normtwo{\ddt \prn{\wb{\vt}^t_\eta - \vt^t_\eta}} \nonumber \\
& = \normtwo{-(\Lambda^t + \Gamma^t_\eta) \wb{\theta}^t_\eta - \int_0^t R_{\ell}^\eta(t,s) \wb{\theta}^s_\eta \de s + (\Lambda^{\flr{t}} + \Gamma^{t}_\eta) \theta^{\flr{t}}_\eta + \int_0^{\flr{t}} R_{\ell}^\eta(t,s) \theta^{\flr{s}}_\eta \de s } \nonumber \\
& \leq \normtwo{-\prn{\vLambda^t + \vGamma_\eta^t} \prn{\wb{\vt}_\eta^t- \vt_\eta^t} - \prn{\vLambda^t + \vGamma_\eta^t} \prn{\vt_\eta^t - \vt_\eta^{\flr{t}}} - \prn{\Lambda^t - \Lambda^{\flr{t}}} \vt_\eta^{\flr{t}} } \nonumber \\
& \qquad + \normtwo{\int_0^t \vRl^\eta(t,s) \prn{\wb{\vt}_\eta^s - \vt_\eta^s} \de s + \int_0^t \vRl^\eta(t,s) \prn{\vt_\eta^s - \vt_\eta^{\flr{s}}} \de s - \int_{\flr{t}}^t \vRl^\eta(t,s)  \vt_\eta^{\flr{s}} \de s} \nonumber \\
& \leq \prn{\cstlbd + \cstloss} \normtwo{\wb{\vt}_\eta^t- \vt_\eta^t} + \prn{\cstlbd + \cstloss} \normtwo{\vt_\eta^t - \vt_\eta^{\flr{t}}} + \eta \cstlbd \normtwo{\vt_\eta^{\flr{t}}} \nonumber \\
& \qquad + \int_0^t \pRl(t-s) \normtwo{\wb{\vt}_\eta^s - \vt_\eta^s} \de s + \int_0^t \pRl(t-s) \normtwo{\vt_\eta^s - \vt_\eta^{\flr{s}}} \de s + \eta \pRl(T) \normtwo{\vt_\eta^{\flr{t}}} \, , \label{eq:mid-transformA-eta-1.5}
\end{align}
where in the last line we use Assumption~\ref{ass:Normal} which gives $\norm{\Lambda^t - \Lambda^{\flr{t}}} \leq \eta \cstlbd$, the fact $\vt_\eta^{\flr{s}} = \vt_\eta^{\flr{t}}$ when $\flr{t} \leq s \leq t$ and also $\pRl(t)$ is a nondecreasing function in $t \in \realsp$. By taking $\wb{\lambda}$ such that $\int_0^\infty e^{-\wb{\lambda}t}\pRl(t) \de \leq \cstlbd + \cstloss$ and repeating the argument in Eq.~\eqref{eq:argument-lambda-bar}, we can have
\begin{align}
&e^{-\wb{\lambda} t} \ddt \normtwo{\wb{\vt}^t_\eta - \vt^t_\eta} \nonumber \\
& \leq 2\prn{\cstlbd + \cstloss} \cdot \sup_{0 \leq s \leq t} \normtwo{\wb{\vt}^s_\eta - \vt^s_\eta} \nonumber \\
& \qquad + e^{-\wb{\lambda}t} \brc{\prn{\cstlbd + \cstloss}\normtwo{\vt_\eta^t - \vt_\eta^{\flr{t}}} + \eta \prn{ \cstlbd + \pRl(T)} \normtwo{\vt_\eta^{\flr{t}}} + \pRl(T) \int_0^t \normtwo{\vt_\eta^s - \vt_\eta^{\flr{s}}} \de s } \, .
\end{align}
Consequently
\begin{align}
&e^{- 2\prn{\cstlbd + \cstloss}t -\wb{\lambda} t}\normtwo{\wb{\vt}^t_\eta - \vt^t_\eta} \nonumber \\
& \leq \int_0^t e^{- 2\prn{\cstlbd + \cstloss}s -\wb{\lambda} s}  \nonumber \\
& \qquad \cdot \brc{\prn{\cstlbd + \cstloss}\normtwo{\vt_\eta^s - \vt_\eta^{\flr{s}}} + \eta \prn{ \cstlbd + \pRl(T)} \normtwo{\vt_\eta^{\flr{s}}} + \pRl(T) \int_0^s \normtwo{\vt_\eta^{s'} - \vt_\eta^{\flr{s'}}} \de s' } \de t \, ,
\end{align}
which further implies for any $\lambda > 2(\cstlbd + \cstloss) + \wb{\lambda}$, we have
\begin{align}
&e^{-\lambda t} \normtwo{\wb{\vt}^t_\eta - \vt^t_\eta} \nonumber \\
& \leq \int_0^t  \prn{\prn{\cstlbd + \cstloss}\normtwo{\vt_\eta^s - \vt_\eta^{\flr{s}}} + \eta \prn{ \cstlbd + \pRl(T)} \normtwo{\vt_\eta^{\flr{s}}} + \pRl(T) \int_0^s \normtwo{\vt_\eta^{s'} - \vt_\eta^{\flr{s'}}} \de s' } \de t \nonumber \\
& \leq \int_0^T \prn{\prn{\cstlbd + \cstloss + T \pRl(T)}\normtwo{\vt_\eta^t - \vt_\eta^{\flr{t}}} + \eta \prn{ \cstlbd + \pRl(T)} \normtwo{\vt_\eta^{\flr{t}}} } \de t  \, .
\end{align}
By triangle inequality,
\begin{align}
&e^{-\lambda t} \sqrt{\Ep \brk{\normtwo{\wb{\vt}^t_\eta - \vt^t_\eta}^2}} \nonumber \\
& \leq \sqrt{\Ep \brk{\prn{\int_0^T \prn{\prn{\cstlbd + \cstloss + T \pRl(T)}\normtwo{\vt_\eta^t - \vt_\eta^{\flr{t}}} + \eta \prn{ \cstlbd + \pRl(T)} \normtwo{\vt_\eta^{\flr{t}}} } \de t}^2 }} \nonumber \\
& \leq \prn{\cstlbd + \cstloss + T \pRl(T)} \cdot \sqrt{\Ep \brk{\prn{\int_0^T \normtwo{\vt_\eta^t - \vt_\eta^{\flr{t}}} \de t}^2}} + \eta \prn{ \cstlbd + \pRl(T)} \cdot \sqrt{\Ep \brk{\prn{\int_0^T \normtwo{\vt_\eta^{\flr{t}}} \de t}^2}} \nonumber \\
& \leq  \prn{\cstlbd + \cstloss + T \pRl(T)} \sqrt{T \cdot \Ep \brk{\int_0^T \normtwo{\vt_\eta^t - \vt_\eta^{\flr{t}}}^2 \de t}} + \eta \prn{ \cstlbd + \pRl(T)} \cdot \sqrt{T \cdot \Ep \brk{\int_0^T \normtwo{\vt_\eta^{\flr{t}}}^2 \de t}} \nonumber \, ,
\end{align}
where we invoke Cauchy-Schwarz inequality in the last line. Substituting in Eq.~\eqref{eq:mid-transformA-eta-1} and $\Ep \brk{\normtwo{\vt_\eta^{\flr{t}}}^2} \leq k \pCt(T)$, we get
\begin{align}
e^{-\lambda t} \sqrt{\Ep \brk{\normtwo{\wb{\vt}^t_\eta - \vt^t_\eta}^2}} & \leq T \prn{\cstlbd + \cstloss + T \pRl(T)} \wb{h}_1(\eta) + \eta T \prn{ \cstlbd + \pRl(T)} \sqrt{k \pCt(T)} \, .
\end{align}
Following the same coupling argument in Appendix~\ref{proof:diff-Cl-global}, we obtain
\begin{align}
&\lbddst{\bCt^\eta}{\Ct^\eta} \nonumber \\
& \leq \sup_{t \in [0, T]} e^{-\lambda t} \sqrt{\Ep \brk{\normtwo{\wb{\vt}^t_\eta - \vt^\flr{t}_\eta}^2}} \nonumber \\
& \leq \sup_{t \in [0, T]} e^{-\lambda t} \sqrt{\Ep \brk{\normtwo{\wb{\vt}^t_\eta - \vt^t_\eta}^2}} + \sup_{t \in [0, T]} e^{-\lambda t} \sqrt{\Ep \brk{\normtwo{\vt^t_\eta - \vt^{\flr{t}}_\eta}^2}} \nonumber \\
& \leq T \prn{\cstlbd + \cstloss + T \pRl(T)} \wb{h}_1(\eta) + \eta T \prn{ \cstlbd + \pRl(T)} \sqrt{k \pCt(T)} + \wb{h}_1(\eta) =: \wb{h}_2(\eta)\, ,
\end{align}
where $\wb{h}_2(\eta) \to 0$ as $\eta \to 0$.
\paragraph{Controlling the distance between \texorpdfstring{$\bRt^\eta$}{TEXT} and \texorpdfstring{$\brk{\Rt^\eta}$}{TEXT}.} Since $\partial \wb{\theta}^t_\eta / \partial \wb{u}^s_\eta$  and $\partial \theta^t_\eta / \partial u^s_\eta$ are not random, we can write
\begin{align}
\frac{\de}{\de t} \bRt^\eta(t,s) & = -(\Lambda^t + \Gamma^t_\eta) \bRt^\eta(t,s)  - \int_s^t R_{\ell}^\eta(t,s') \bRt^\eta(s',s)  \de s'\, , & & 0 \leq s \leq t \leq T \, , \\
\frac{\de}{\de t} \Rt^\eta(t,s) & = -(\Lambda^{\flr{t}} + \Gamma^{\flr{t}}_\eta) \Rt^\eta(\flr{t},s) - \int_s^{\flr{t}} R_{\ell}^\eta(t,s') \Rt^\eta(\flr{s'},s) \de s'\, , & & 0 \leq s \leq t \leq T \, .
\end{align}
with the same boundary conditions $\bRt^\eta(s,s) = \Rt^\eta(s,s) = I$ and the convention that $\Rt^\eta(t,s) = 0$ when $t < s$. 

First we try to control the error $\norm{\Rt^\eta(t,s) - \brk{\Rt^\eta}(t,s)}$. By definition, we have $\brk{\Rt^\eta}(t,s) = I$ when $\cil{s} >\flr{t}$. In this case, we have $\flr{t} \leq s$ and therefore $\norm{\ddt \Rt^\eta(t,s)} \leq \prn{\cstlbd + \cstloss} \norm{\Rt^\eta(\flr{t}, s)} \leq (\cstlbd + \cstloss)$. We can then control
\begin{align}
\norm{\Rt^\eta(t,s) - \brk{\Rt^\eta}(t,s)} & = \norm{\Rt^\eta(t,s) - \Rt^\eta(s,s)} \leq \eta \prn{\cstlbd + \cstloss}\, . \label{eq:mid-transformA-eta-2}
\end{align}
We can then assume $ \cil{s} \leq \flr{t}$. One then can derive
\begin{align}
\norm{\Rt^\eta(t,s) - \Rt^\eta(\flr{t}, s)} &\leq \eta \sup_{\flr{t} \leq s' \leq t} \norm{\frac{\de }{\de s'} \Rt^\eta(s', s)} \nonumber \\
& \leq \eta \cdot \norm{ -(\Lambda^{\flr{t}} + \Gamma^{\flr{t}}_\eta) \Rt^\eta(\flr{t},s) - \int_s^{\flr{t}} R_{\ell}^\eta(t,s') \Rt^\eta(\flr{s'},s) \de s'}  \nonumber \\
& \leq \eta \cdot \brc{\prn{\cstlbd + \cstloss} \pRt(T) + T \pRl(T) \pRt(T)} =: \wb{h}_3(\eta) \, . \label{eq:mid-transformA-eta-2.5}
\end{align}
In particular, take $t = \cil{s} - \epsilon$ and let $\epsilon \to 0$, it follows then
\begin{align}
\norm{\Rt^\eta(\cil{s}, s) - \Rt^\eta(s,s)} = \norm{\Rt^\eta(\cil{s}, s) - \Rt^\eta(\cil{s},\cil{s})} \leq \wb{h}_3(\eta) \, . \label{eq:mid-transformA-eta-3}
\end{align}
Note that for all $t \geq \cil{s}$,
\begin{align}
& \ddt \norm{\Rt^\eta(t, s) - \Rt^\eta(t, \cil{s})} \nonumber \\
& \leq \norm{\ddt \Rt^\eta(t, s) - \ddt \Rt^\eta(t, \cil{s})} \nonumber \\
& = \bigg\| -(\Lambda^{\flr{t}} + \Gamma^{\flr{t}}_\eta) \prn{\Rt^\eta(\flr{t},s) - \Rt^\eta(\flr{t}, \cil{s})} - \int_{\cil{s}}^{\flr{t}} R_{\ell}^\eta(t,s') \prn{\Rt^\eta(\flr{s'},s) -   \Rt^\eta(\flr{s'},\cil{s})}\de s' \nonumber \\
& \qquad - \int_s^{\cil{s}}  R_{\ell}^\eta(t,s') \Rt^\eta(\flr{s'},s) \de s' \bigg\| \nonumber \\
& \leq (\cstlbd + \cstloss)\norm{\Rt^\eta(\flr{t},s) - \Rt^\eta(\flr{t}, \cil{s})} + \int_{\cil{s}}^{\flr{t}} \pRl(t-s') \norm{\Rt^\eta(\flr{s'},s) - \Rt^\eta(\flr{s'}, \cil{s})} \de s' \nonumber \\
& \qquad + \eta \pRl(T) \pRt(T) \, .
\end{align}
Similar to what we did in Eq.~\eqref{eq:argument-lambda-bar}, taking $\wb{\lambda}$ such that $\int_0^\infty e^{-\wb{\lambda}t}\pRl(t) \de \leq \cstlbd + \cstloss$ gives us
\begin{align}
& e^{-\wb{\lambda}t}\ddt \norm{\Rt^\eta(t, s) - \Rt^\eta(t, \cil{s})} \nonumber \\
& \leq (\cstlbd + \cstloss) e^{-\wb{\lambda}t} \norm{\Rt^\eta(\flr{t},s) - \Rt^\eta(\flr{t}, \cil{s})} + \int_{\cil{s}}^{\flr{t}} e^{-\wb{\lambda}(t-s')} \pRl(t-s') \cdot e^{-\wb{\lambda}s'} \norm{\Rt^\eta(\flr{s'},s) - \Rt^\eta(\flr{s'}, \cil{s})} \de s' \nonumber \\
& \qquad + e^{-\wb{\lambda}t} \eta \pRl(T) \pRt(T) \nonumber \\
& \leq (\cstlbd + \cstloss) e^{-\wb{\lambda}\flr{t}} \norm{\Rt^\eta(\flr{t},s) - \Rt^\eta(\flr{t}, \cil{s})} + \int_{\cil{s}}^{\flr{t}} e^{-\wb{\lambda}(t-s')} \pRl(t-s') \cdot e^{-\wb{\lambda}\flr{s'}} \norm{\Rt^\eta(\flr{s'},s) - \Rt^\eta(\flr{s'}, \cil{s})} \de s' \nonumber \\
& \qquad + e^{-\wb{\lambda}t} \eta \pRl(T) \pRt(T) \nonumber \\
& \leq 2(\cstlbd + \cstloss) \cdot \sup_{\cil{s} \leq s' \leq t} e^{-\wb{\lambda}s'} \norm{\Rt^\eta(s',s) - \Rt^\eta(s', \cil{s})} + e^{-\wb{\lambda}t} \eta \pRl(T) \pRt(T) \, .
\end{align}
Further, this allows us to derive by taking in Eq.~\eqref{eq:mid-transformA-eta-3}
\begin{align}
& e^{-2(\cstlbd + \cstloss) t - \wb{\lambda}t} \norm{\Rt^\eta(t, s) - \Rt^\eta(t, \cil{s})} \nonumber \\
&\leq \norm{\Rt^\eta(\cil{s}, s) - \Rt^\eta(\cil{s}, \cil{s})} + \int_0^t e^{- 2\prn{\cstlbd + \cstloss}s -\wb{\lambda} s} \eta \pRl(T) \pRt(T) \de s  \nonumber \\
& \leq \wb{h}_3(\eta) + \eta T \pRl(T) \pRt(T) \, .
\end{align}
This implies for any $\lambda > 2(\cstlbd + \cstloss) + \wb{\lambda}$,
\begin{align}
e^{-\lambda t} \norm{\Rt^\eta(t, s) - \Rt^\eta(t, \cil{s})} \leq \wb{h}_3(\eta) + \eta T \pRl(T) \pRt(T) \, . \label{eq:mid-transformA-eta-4}
\end{align}
Combining Eqs.~\eqref{eq:mid-transformA-eta-2}, \eqref{eq:mid-transformA-eta-2.5} and \eqref{eq:mid-transformA-eta-4}, we obtain for any $0 \leq s \leq t \leq T$,
\begin{align}
e^{-\lambda t}\norm{\Rt^\eta(t,s) - \brk{\Rt^\eta}(t,s)} \leq \max \brc{\eta \prn{\cstlbd + \cstloss},  2\wb{h}_3(\eta) + \eta T \pRl(T) \pRt(T)} =: \wb{h}_4(\eta) \, . \label{eq:mid-transformA-eta-5}
\end{align}
Next we control the term $\norm{\bRt^\eta(t,s) - \Rt^\eta(t,s)}$, by definition and the fact that $\Gamma_\eta^{\flr{t}} = \Gamma_\eta^t$ we have
\begin{align}
& \ddt \norm{\bRt^\eta(t, s) - \Rt^\eta(t, s)} \nonumber \\
& \leq \norm{\ddt \bRt^\eta(t, s) - \ddt \Rt^\eta(t, s)} \nonumber \\
& = \norm{-(\Lambda^t + \Gamma^t_\eta) \bRt^\eta(t,s)  - \int_s^t R_{\ell}^\eta(t,s') \bRt^\eta(s',s)  \de s' + (\Lambda^{\flr{t}} + \Gamma^{t}_\eta) \Rt^\eta(\flr{t},s) + \int_s^{\flr{t}} R_{\ell}^\eta(t,s') \Rt^\eta(\flr{s'},s) \de s' } \nonumber \\
& \leq \norm{-(\Lambda^t + \Gamma^t_\eta) \prn{\bRt^\eta(t,s) - \Rt^\eta(t,s)} -(\Lambda^t + \Gamma^t_\eta) \prn{\Rt^\eta(t,s) - \Rt^\eta(\flr{t},s)} - \prn{\Lambda^t - \Lambda^{\flr{t}}} \Rt^\eta(\flr{t},s)} \nonumber \\
& \qquad + \bigg\|-\int_s^{\flr{t}} \Rl^\eta(t,s') \prn{\bRt^\eta(s',s) - \Rt^\eta(s',s)} \de s'  -\int_s^{\flr{t}} \Rl^\eta(t,s') \prn{\Rt^\eta(s',s) - \Rt^\eta(\flr{s'},s)} \de s' \nonumber \\
& \qquad - \int_{\flr{t}}^t \Rl^\eta(t,s') \bRt^\eta(s',s) \de s' \bigg\| \nonumber \\
& \leq (\cstlbd + \cstloss) \norm{\bRt^\eta(t,s) - \Rt^\eta(t,s)} +  (\cstlbd + \cstloss) \norm{\Rt^\eta(t,s) - \Rt^\eta(\flr{t},s)} + \eta \cstlbd \pRt(T) \nonumber \\
& \qquad + \int_s^{\flr{t}} \pRl(t-s') \norm{\bRt^\eta(s',s) - \Rt^\eta(s',s)} \de s'  + T \pRl(T) \sup_{s \leq s' \leq t} \norm{\Rt^\eta(s',s) - \Rt^\eta(\flr{s'},s)} \nonumber \\
& \qquad + \eta \pRl(T) \pRt(T) \, .
\end{align}
Further by Eq.~\eqref{eq:mid-transformA-eta-2.5}, we can obtain
\begin{align}
& \ddt \norm{\bRt^\eta(t, s) - \Rt^\eta(t, s)} \nonumber \\
& \leq  (\cstlbd + \cstloss) \norm{\bRt^\eta(t,s) - \Rt^\eta(t,s)} +  \int_s^{\flr{t}} \pRl(t-s') \norm{\bRt^\eta(s',s) - \Rt^\eta(s',s)} \de s' \nonumber \\
& \qquad + \eta \prn{\cstlbd \pRt(T) + \pRl(T) \pRt(T)} + \prn{\cstlbd + \cstloss + T \pRl(T)} \wb{h}_3(\eta) \, .
\end{align}
We get the exact same type of inequality as in Eq.~\eqref{eq:mid-transformA-eta-1.5} and we can repeat the same argument and get for all $\lambda > 2(\cstlbd + \cstloss) + \wb{\lambda}$
\begin{align}
e^{-\lambda t} \norm{\bRt^\eta(t, s) - \Rt^\eta(t, s)} \leq \eta T\prn{\cstlbd \pRt(T) + \pRl(T) \pRt(T)} + T\prn{\cstlbd + \cstloss + T \pRl(T)} \wb{h}_3(\eta) =: \wb{h}_5(\eta) \label{eq:mid-transformA-eta-6} \, .
\end{align}
Putting together Eqs.~\eqref{eq:mid-transformA-eta-5} and \eqref{eq:mid-transformA-eta-6} yields
\begin{align}
\lbddst{\bRt^\eta}{\brk{\Rt^\eta}} & \leq \sup_{0 \leq s \leq t \leq T} e^{-\lambda t} \norm{\bRt^\eta(t, s) - \brk{\Rt^\eta}(t, s)} \leq \wb{h}_4(\eta) + \wb{h}_5(\eta) =: \wb{h}_6(\eta)\,.
\end{align}
Clearly $\eta \to 0$ we have $\wb{h}_6(\eta) \to 0$. The proof is completed by taking $\wb{h}(\eta) = \max \brc{\wb{h}_2(\eta), \wb{h}_6(\eta)}$.
%
%
\subsection{Proof of Lemma \ref{lemma:Tight}}
\label{app:Tightness}

The claim of this lemma follows by establishing separately the following
two statements (possibly after adjusting the constants $M(\eps)$):
\begin{align}
&\prob\Big(\est{\mu}^{(n)}\big(\|\theta^0\|>M(\eps)\big)\ge \eps\mbox{ for infinitely many }n\Big) = 0\, ,\label{eq:Tightness1}\\
&\prob\Big(\est{\mu}^{(n)}\big( \|\rev{\theta^{[0,T]}}\|_{C^{0,\alpha}}>M(\eps) \big)\ge \eps
\mbox{ for infinitely many }n\Big) = 0\, .\label{eq:Tightness2}
\end{align}

We begin by Eq.~\eqref{eq:Tightness1}:
\begin{align*}
\est{\mu}^{(n)}\big(\|\theta^0\|>M(\eps)\big) 
= \frac{1}{d}\sum_{i=1}^d \bfone_{\{\|\theta_i^0\|>M \}}
\le \frac{1}{dM^2}\sum_{i=1}^d \|\theta_i^0\|^2 =  \frac{1}{dM^2}\|\btheta^0\|_F^2\, .
\end{align*}
Since by assumption $\Ep_{\widehat{\mu}_{\theta^0}} \brk{\norm{\theta^0}^2}
\to \Ep_{\mu_{\theta^0}} \brk{\norm{\theta^0}^2} < \infty$,
there exists a constant $C$ such that 
$\|\btheta^0\|_F^2/d\le C$ for all $n$ large enough. Therefore 
\begin{align*}
\est{\mu}^{(n)}\big(\|\theta^0\|>M(\eps)\big) \le \frac{C}{M^2}
\end{align*}
for all but finitely many values of $n$, which yields the claim \eqref{eq:Tightness1}.

Next, to prove Eq.~\eqref{eq:Tightness2}, we begin by noting that,
for any differentiable function $f:[0,T]\to \reals^k$, ans any $0\le s\le t\le T$, we have
\begin{align*}
\big\|f(t)-f(s)\big\| &= \Big\|\int_0^Tf'(u)\bfone_{[t,s]}(u)\, \de u\Big\|\\
& \le (t-s)^{1/2}\Big(\int_0^T\|f'(u)\|^2\de u\Big)^{1/2}\, ,
\end{align*}
which implies $\|f\|_{C^{0,1/2}}\le \|f'\|_{L^2}$.
Therefore
\begin{align*}
\est{\mu}^{(n)}\big( \|\rev{\theta^{[0,T]}}\|_{C^{0,1/2}}>M \big) & \le 
\est{\mu}^{(n)}\big( \|\rev{\dot\theta^{[0,T]}}\|_{L^2}>M \big)
= \frac{1}{d}\sum_{i=1}^d \bfone_{\{\|\rev{\dot\theta_i^{[0,T]}}\|_{L^2}>M \}}\\
& \le \frac{1}{M^2 d}\sum_{i=1}^d\|\rev{\dot\theta_i^{[0,T]}}\|^2_{L^2}
= \frac{1}{M^2 d}\sum_{i=1}^d \int_0^T\|\dot\theta_i^t\|^2\de t\\
& =  \frac{1}{M^2 d} \int_0^T\Big\|\btheta^t\Lambda^t+ \frac{1}{\delta}\bX\bell_t(\bX\btheta^t;\bz)
\Big\|_F^2\de t\, ,
\end{align*}
where in the last step we used the definition of the flow, per Eq.~\eqref{eq:GeneralFlow}.

By the Bai-Yin law, there exists $C=C(\delta)$ such that almost surely
$\|\bX\|\le  C(\delta)$ for all but finitely man values of $n$. Using the 
conditions on $\Lambda^t$, $\bell^t$ in Assumption \ref{ass:Normal}, we deduce
that, for all but finitely many values of $n$,
\begin{align}
\est{\mu}^{(n)}\big( \|\rev{\theta^{[0,T]}}\|_{C^{0,1/2}}>M \big) & \le 
\frac{C}{M^2 d} \int_0^T\big(\|\btheta^t\|_F^2+ \|\bell_t(\bfzero;\bz)\|_F^2\big)\, \de t\, .
\label{eq:TightnessAlmostDone}
\end{align}
It is therefore sufficient to bound $\|\btheta^t\|_F^2$. 
This follows from the Lipschitz property of $\ell$ and the fact that $\|\bX\|_{\op}$
is bounded, with high probability, implying that, for all but finitely many values of 
$n$, and for all $t$,
\begin{align*}
\|\btheta^t\|\le Ce^{Ct}\big(\|\btheta^0\|_F+\|\bell_t(\bfzero;\bz)\|_F\big)\, .
\end{align*}
By the assumptions on $\ell_t$ and $\bz$, we have 
$\|\bell_t(\b0;\bz)\|^2\le C(\|\bell_t(\bfzero;\bfzero)\|^2+\|\bz\|^2)\le C'd$.
Substituting these bounds in Eq.~\eqref{eq:TightnessAlmostDone}, we obtain
\begin{align}
\est{\mu}^{(n)}\big( \|\rev{\theta^{[0,T]}}\|_{C^{0,1/2}}>M \big) & \le  
\frac{Ce^{CT}}{M^2 d}\big(\|\btheta^0\|_F^2+ \|\bell_t(\bfzero;\bz)\|_F^2\big)\, \de t
\le \frac{C'}{M^2}\, ,
\end{align}
where the last step follows for all $n$ large enough from the assumptions
$\Ep_{\widehat{\mu}_{\theta^0}} \brk{\norm{\theta^0}^2} \to \Ep_{\mu_{\theta^0}} \brk{\norm{\theta^0}^2} < \infty$ and 
$\Ep_{\widehat{\mu}_{z}} \brk{\norm{z}^2} \to \Ep_{\mu_{z}} \brk{\norm{z}^2} < \infty$.
This concludes the proof of Eq.~\eqref{eq:Tightness2}.

\section{Proofs for fixed-point equations}
\label{app:fixed-pt}

\subsection{Proof of Corollary \ref{cor:planted-se}}

We write the flow Eq.~\eqref{eq:full-flow} as
\begin{equation}
\frac{\de \bar \btheta^t}{\de t} 
=
- \bar \btheta^t \bar \Lambda^{t, \sT} - \frac1\delta \bX^\top \bar \bell_t(\bX \bar \btheta^t;\bz),
\end{equation}
where
\begin{equation}
\label{eq:planted-Lam-ell}
\bar \Lambda^t
= 
\diag(\bar \Lambda^t_{11},0),
\qquad 
\bar \ell_t(\bX \bar \btheta^t;\bz)
=
\begin{pmatrix}
	\bar \ell_t(\bX\bar \btheta^t;\bz)_1 \\ 0   
\end{pmatrix},
\end{equation}
initialized at $\bar \btheta^0 = (\btheta^0,\btheta^*)$.
Here, we have identified $\Lambda^t = \bar \Lambda^t_{11}$ 
and $\bell_t(\bX \btheta^t,\bX \btheta^*;\bz) = \bar \bell_t(\bX\btheta^t,\bX \btheta^*;\bz)_1$.
Corollary \ref{cor:planted-se} will follow from applying Theorem \ref{thm:StateEvolution} to the special case that $\bar \Lambda^t$ and $\bar \ell_t$ take the special form given above, namely, that they contain zeros in certain coordinates.
We will show then that in this case, the unique solution to the integro-differential equations \eqref{eq:int-diff-1} and \eqref{eq:int-diff-2} are of the form
\begin{equation}\label{eq:planted-soln-form}
\begin{gathered}
	\bar \theta^t
	=
	\begin{pmatrix} 
		\bar \theta^t_1 \\ \bar \theta^0_2
	\end{pmatrix},
	\qquad 
	\bar r^t
	=
	\begin{pmatrix} 
		\bar r_1^t \\ \bar w^0_2
	\end{pmatrix},
	\qquad 
	\bar u^t
	=
	\begin{pmatrix} 
		\bar u_1^t \\ 0
	\end{pmatrix},
	\qquad 
	\bar w^t
	=
	\begin{pmatrix} 
		\bar w_1^t \\ \bar w_2^0
	\end{pmatrix},
	\\
	\bar R_\theta(t,s)
	=
	\begin{pmatrix} 
		\bar R_\theta(t,s)_{11} & \bar R_\theta(t,s)_{12} \\ 0 & I_k 
	\end{pmatrix},
	\qquad 
	\bar R_\ell(t,s)
	=
	\begin{pmatrix} 
		\bar R_\ell(t,s)_{11} & \bar R_\ell(t,s)_{12} \\ 0 & 0 
	\end{pmatrix},
	\qquad
	\bar \Gamma^t
	=
	\begin{pmatrix} 
		\bar \Gamma^t_{11} & \bar \Gamma^t_{12} \\ 0 & 0 
	\end{pmatrix},
	\\
	\frac{\partial \bar \theta^t}{\partial \bar u^s}
	= 
	\begin{pmatrix} 
		\Big(\frac{\partial \bar \theta^t}{\partial \bar u^s}\Big)_{11} & \Big(\frac{\partial \bar \theta^t}{\partial \bar u^s}\Big)_{12} \\ 
		0 & I_k 
	\end{pmatrix},
	\qquad 
	\frac{\partial \bar \ell_t(\bar r^t;z)}{\partial \bar w^s}
	=
	\begin{pmatrix} 
		\Big(\frac{\partial \bar \ell_t(\bar r^t;z)}{\partial \bar  w^s}\Big)_{11} & \Big(\frac{\partial \bar  \ell_t(\bar r^t;z)}{\partial \bar  w^s}\Big)_{12} \\ 0 & 0 
	\end{pmatrix},
	\\
	\nabla_r \bar \ell_t(\bar r^t;z)
	=
	\begin{pmatrix} 
		\nabla_r \bar \ell_t(\bar r^t;z)_{11} & \nabla_r \bar \ell_t(\bar r^t;z)_{12} \\ 0 & 0 
	\end{pmatrix},
	\\
	\bar C_\theta(t,s)
	=
	\begin{pmatrix}
		\bar C_\theta(t,s)_{11} & \bar C_\theta(t,0)_{12} \\ 
		\bar C_\theta(0,t)_{21} & \bar C_\theta(0,0)_{22}
	\end{pmatrix},
	\qquad 
	\bar C_\ell(t,s)
	= 
	\begin{pmatrix}
		\bar C_\ell(t,s)_{11} & 0 \\ 
		0 & 0
	\end{pmatrix}.
\end{gathered}
\end{equation}
Indeed, it is immediate that $\nabla_r \bar \ell_t(\bar r^t;z)$ is of the claimed form.
Then, by Eqs.~\eqref{eq:def-Gamma}, \eqref{eq:def-C-l}, and \eqref{eq:def-derivative-l},
we must have that $\frac{\partial \bar \ell_t(\bar r^t;z)}{\partial \bar w^s}$, $\bar \Gamma^t$, $\bar C_\ell$, and $\bar R_\ell$ are of the claimed form (i.e., they have zeros in the locations specified by the preceding display).
Thus, $\bar u^t_2 = 0$ for all $t$.
Moreover, by Eq.~\eqref{eq:def-derivative-t},
we must have that the final $k$-rows of $\frac{\de}{\de t} \frac{\partial \bar \theta^t}{\partial \bar u^s}$ are 0 for all $t$, whence by the \rev{initial} condition $\frac{\partial \bar \theta^t}{\partial \bar u^s} = I_{2k}$ we have that $\frac{\partial \bar \theta^t}{\partial \bar u^s}$ and thus also $\bar R_\theta(t,s)$ are of the claimed form.
Then, by Eq.~\eqref{eq:def-theta},
$\Big(\frac{\de}{\de t} \bar \theta^t\Big)_2$ is equal to zero for all $t$, whence $\bar \theta^t_2 = \bar \theta^0_2$, so that $\bar \theta^t$ is of the claimed form. 
Then, by Eq.~\eqref{eq:def-C-t}, $\bar C_\theta(t,s)$ is of the claimed form (because $\bar \theta_2^t = \bar \theta_2^0$).
Because $\bar w^t$ has covariance kernel $\bar C_\theta$,
we have that $\bar w^t_2 = \bar w^0_2$ for all $t$.
Then, by Eq.~\eqref{eq:def-r},
we have that $\bar r^t_2 = \bar w^t_2 = \bar w^0_2$, so that $\bar w^t$ is of the claimed form.
We thus conclude that the unique solution to Eqs.~\eqref{eq:int-diff-1} and \eqref{eq:int-diff-2} is of the form given in the preceding display.

To complete the proof of Corollary \ref{cor:planted-se},
we must show that $\bar \theta^t,\bar r^t,\bar u^t,\bar w^t,\bar R_\theta,\bar R_\ell,\bar \Gamma^t,\bar C_\theta,\bar C_\ell$ of the form \eqref{eq:planted-soln-form} solves Eqs.~\eqref{eq:int-diff-1} and \eqref{eq:int-diff-2} if and only if it solves Eqs.~\eqref{eq:planted-se} and \eqref{eq:planted-derivs}.
Indeed, plugging \eqref{eq:planted-soln-form} into Eqs.~\eqref{eq:int-diff-1} and \eqref{eq:int-diff-2} and simplifying where possible gives
\begin{equation}\label{eq:full-se-eqns}
\begin{aligned}
	\frac{\de}{\de t} \bar \theta_1^t 
	&= - (\bar \Lambda^t_{11} + \bar \Gamma^t_{11})\bar \theta_1^t 
	- \int_0^t \bar R_\ell(t,s)_{11} \bar \theta^s_1 \de s
	- \Big(\bar \Gamma^t_{12} + \int_0^t \bar R_\ell(t,s)_{12} \de s\Big) \bar \theta^s_2
	+ \bar u_1^t,
	\\
	\frac{\de}{\de t} \bar \theta_2^0 
	&= 0,
	\\
	\bar r^t_1
	&=
	-\frac1\delta \int_0^t \bar R_\theta(t,s)_{11} \bar \ell_s(\bar r^s_1,\bar w^0_2;z)_1 \de s + \bar w^t_1,
	\\
	\bar r^t_2 
	&= \bar w^t_2 = \bar w^0_2, 
	\\ 
	\bar R_\theta(t,s)_{11} 
	&= \E\Big[
	\Big(\frac{\partial \bar \theta^t}{\partial \bar u^s}\Big)_{11}
	\Big],
	\\
	\bar R_\theta(t,s)_{12} 
	&= \E\Big[
	\Big(\frac{\partial \bar \theta^t}{\partial \bar u^s}\Big)_{12}
	\Big],
	\\ 
	\bar R_\ell(t,s)_{11}
	&=
	\E\Big[
	\Big(\frac{\partial \bar \ell_t(\bar r^t;z)}{\partial \bar w^s}\Big)_{11}
	\Big],
	\\
	\bar R_\ell(t,s)_{12}
	&=
	\E\Big[
	\Big(\frac{\partial \bar \ell_t(\bar r^t;z)}{\partial \bar w^s}\Big)_{12}
	\Big],
	\\
	\bar \Gamma^t_{11}
	&=
	\E\Big[
	\nabla_{r_1} \bar \ell_t(\bar r^t_1,\bar w^0_2;z)_1
	\Big],
	\\
	\bar \Gamma^t_{12}
	&=
	\E\Big[
	\nabla_{w^0_2} \bar \ell_t(\bar r^t_1,\bar w^0_2;z)_1
	\Big],
	\\ 
	\bar C_\theta(t,s)_{11}
	&= \E[\bar \theta_1^t(\bar \theta_1^t)^\top],
	\qquad 
	\bar C_\theta(t,0)_{12}
	= \E[\bar \theta_1^t(\bar \theta_2^0)^\top],
	\qquad 
	\bar C_\theta(0,0)_{22}
	= \E[\bar \theta_2^0(\bar \theta_2^0)^\top],
	\\ 
	\bar C_\ell(t,s)_{11} &= \E[\bar \ell_t(\bar r^t_1,\bar w_2^0;z)\bar \ell_t(\bar r^t_1,\bar w_2^0;z)^\top]
	\\
	\frac{\de}{\de t} \Big(\frac{\partial \bar \theta^t}{\partial \bar u^s}\Big)_{11}
	&= 
	-(\bar \Gamma^t_{11} + \bar \Lambda^t_{11}) \Big(\frac{\partial \bar \theta^t}{\partial \bar u^s}\Big)_{11}
	- \int_s^t \bar R_\ell(t,s')_{11} \Big(\frac{\partial \bar \theta^{s'}}{\partial \bar u^s}\Big)_{11} \de s',
	\\ 
	\frac{\de}{\de t} \Big(\frac{\partial \bar \theta^t}{\partial \bar u^s}\Big)_{12}
	&= 
	-(\bar \Gamma^t_{11} + \bar \Lambda^t_{11}) \Big(\frac{\partial \bar \theta^t}{\partial \bar u^s}\Big)_{12}
	- \int_s^t \bar R_\ell(t,s')_{11} \Big(\frac{\partial \bar \theta^{s'}}{\partial \bar u^s}\Big)_{12} \de s'
	- \int_s^t \bar R_\ell(t,s')_{12} \de s',
	\\ 
	\Big(
	\frac{\partial \bar \ell_t(\bar r^t;z)}{\partial \bar w^s}
	\Big)_{11}
	&=
	\nabla_{r_1} \bar \ell_t(\bar r^t_1,\bar w^0_2;z)_1
	\cdot 
	\left(
	- \frac1\delta \int_s^t \bar R_\theta(t,s')_{11} \Big(\frac{\partial \bar \ell_{s'}(\bar r^{s'};z)}{\partial \bar w^s}\Big)_{11} \de s' 
	-
	\frac1\delta \bar R_\theta(t,s)_{11} \nabla_{r_1}\bar \ell_s(\bar r^s_1,\bar w^0_2;z)_{11}
	\right),
	\\ 
	\Big(
	\frac{\partial \bar \ell_t(\bar r^t;z)}{\partial \bar w^s}
	\Big)_{12}
	&=
	\nabla_{r_1} \bar \ell_t(\bar r^t_1,\bar w^0_2;z)_1
	\cdot 
	\left(
	- \frac1\delta \int_s^t \bar R_\theta(t,s')_{11} \Big(\frac{\partial \bar \ell_{s'}(\bar r^{s'};z)}{\partial \bar w^s}\Big)_{12} \de s' 
	-
	\frac1\delta \bar R_\theta(t,s)_{11} \nabla_{w_2^0} \bar \ell_s(\bar r^s_1,\bar w^0_2;z)_1
	\right).
\end{aligned}
\end{equation}
By Theorem \ref{thm:UniqueExist},
there exists a unique solution to the equations in the previous display.

We can simplify the above equations. 
In particular, integrating the last line and adding $\nabla_{w_2^0} \bar \ell_t(\bar r_1^t,\bar w_2^0;z)_1$ to both sides gives
\begin{equation}\label{eq:Rell-simplification}
\begin{aligned}
	&\nabla_{w_2^0}\bar \ell_t(\bar r^t_1, \bar w^0_2;z)_1
	+ 
	\int_0^t \Big(\frac{\partial \bar \ell_t(\bar r^t;z)}{\partial \bar w^s}\Big)_{12} \de s
	\\
	&=\nabla_{r_1} \bar \ell_t(\bar r^t_1,\bar w^0_2;z)_1
	+ 
	\int_0^t 
	\nabla_{r_1} \bar \ell_t(\bar r^t_1,\bar w^0_2;z)_1
	\cdot 
	\left(
	- \frac1\delta \int_s^t \bar R_\theta(t,s')_{11} \Big(\frac{\partial \bar \ell_{s'}(\bar r^{s'};z)}{\partial \bar  w^s}\Big)_{12} \de s' 
	-
	\frac1\delta \bar R_\theta(t,s)_{11} \nabla_{w_2^0}\bar \ell_s(\bar r^s_1,\bar w^0_2;z)_1
	\right) \de s
	\\
	&=\nabla_{w_2^0} \bar \ell_t(\bar r^t_1,\bar w^0_2;z)_1
	- \frac1\delta \nabla_{r_1} \bar \ell_t(\bar r^t_1,\bar w^0_2;z)_1 \cdot 
	\int_0^t 
	\left(
	\bar R_\theta(t,s)_{11}
	\left(
	\nabla_{w_2^0} \bar \ell_s(\bar r^s_1,\bar w^0_2;z)_1
	+
	\int_0^{s} \Big(\frac{\partial \bar \ell_{s}(\bar r^{s};z)}{\partial \bar w^{s'}}\Big)_{12} \de s' 
	\right)
	\right) \de {s}.
\end{aligned}
\end{equation}
Define 
\begin{equation}
\frac{\partial \bar \ell(\bar r^t_1,\bar w^*;z)}{\partial \bar w^*}
:=
\nabla_{w_2^0}\bar \ell_t(\bar r^t_1, \bar w^0_2;z)_1
+ 
\int_0^t \Big(\frac{\partial \bar \ell_t(\bar r^t;z)}{\partial \bar w^s}\Big)_{12} \de s.  
\end{equation}
Then we get
\begin{equation}
\frac{\partial \bar \ell(\bar r^t_1,\bar w^*;z)}{\partial \bar w^*}
=
- \frac1\delta \nabla_{r_1} \bar \ell_t(\bar r^t_1,\bar w^0_2;z)_1 \cdot 
\int_0^t 
\bar R_\theta(t,s)_{11}
\frac{\partial \bar \ell(\bar r^s_1,\bar w^*;z)}{\partial \bar w^*}
\de {s}
+
\nabla_{w_2^0} \bar \ell_t(\bar r^t_1,\bar w^0_2;z)_1.
\end{equation}
Taking expectations, we get
\begin{equation}
\bar \Gamma_{12}^t + \int_0^t \bar R_\ell(t,s)_{12} \de s
= 
\E\Big[\frac{\partial \bar \ell(\bar r^t_1,\bar w^*;z)}{\partial \bar w^*}\Big].
\end{equation}
We thus see that the state evolution equations \eqref{eq:int-diff-1} and \eqref{eq:int-diff-2} applied to $\bar \Lambda^t$ and $\bar \ell_t$ as in Eq.~\eqref{eq:planted-Lam-ell} gives the planted state evolution Eqs.~\eqref{eq:planted-se} and \eqref{eq:planted-derivs} 
under the change of variables (with the notation appearing in Eq.~\eqref{eq:planted-se} and \eqref{eq:planted-derivs} on the right)
\begin{equation}
\begin{gathered}
	\theta^t = \bar \theta_1^t,
	\qquad 
	\Lambda^t = \bar \Lambda_{11}^t,
	\qquad 
	\Gamma^t = \bar \Gamma_{11}^t,
	\qquad 
	R_\ell(t,s) = \bar R_\ell(t,s)_{11} 
	\\
	R_\ell(t,*) = \bar \Gamma_{12}^t + \int_0^t \bar R_\ell(t,s)_{12} \de s,
	\qquad 
	u^t = \bar u_1^t,
	\qquad 
	r^t = \bar r_1^t,
	\qquad 
	R_\theta(t,s) = \bar R_\theta(t,s)_{11},
	\\
	w^* = w^0_2,
	\qquad 
	w^t = w^t_1,
	\\
	\frac{\partial \theta^t}{\partial u^s} = \Big(\frac{\partial \bar \theta^t}{\partial \bar u^s}\Big)_{11},
	\qquad 
	\ell_s(r^s,w^*;z) 
	= 
	\bar \ell_s(r^s,w^*;z)_1,
	\qquad 
	R_\ell(t,s) 
	=
	\bar R_\ell(t,s)_{11},
	\\
	\frac{\partial \ell_t(r^t,w^*;z)}{\partial w^s}
	=
	\Big(\frac{\partial \bar \ell_t(\bar r^t;z)}{\partial \bar w^s}\Big)_{11},
	\qquad 
	C_\theta(t,s) = \bar C_\theta(t,s)_{11}, \quad 0 \leq s \leq t,
	\\
	C_\theta(t,*) = \bar C_\theta(t,0)_{12}, 
	\qquad 
	C_\theta(*,*) = \bar C_\theta(0,0)_{22}.
\end{gathered}
\end{equation}
Note that $\bar R_\theta(t,s)_{12}$, though defined by the state evolution equations \eqref{eq:int-diff-1} and \eqref{eq:int-diff-2}, plays no role in the dynamics of Eqs.~\eqref{eq:planted-se} and \eqref{eq:planted-derivs}, so is omitted.

In summary, we have shown that the unique solution to Eqs.~\eqref{eq:int-diff-1} and \eqref{eq:int-diff-2}with inputs \eqref{eq:planted-Lam-ell} gives a solution to Eqs.~\eqref{eq:planted-se} and \eqref{eq:planted-derivs}.
Thus, we have shown existence of a solution to these equations. 
Uniqueness requires a few more steps of argumentation. 
We have already shown that \eqref{eq:full-se-eqns} have a unique solution.
Note that any solution to Eqs.~\eqref{eq:planted-se} and \eqref{eq:planted-derivs} generates a solution to \eqref{eq:full-se-eqns} using the change of variables in the previous display, as well as setting 
\begin{equation}
\begin{gathered}
	\bar \Gamma^t_{12}
	=
	\E\Big[
	\nabla_{w^*} \ell_t(r^t, w^*;z)
	\Big],
	\qquad 
	\bar R_\ell(t,s)_{12}
	= 
	\E\Big[
	\Big(\frac{\partial \bar \ell_t(\bar r^t;z)}{\partial \bar w^s}\Big)_{12}
	\Big],
	\\
	\Big(
	\frac{\partial \bar \ell_t(\bar r^t;z)}{\partial \bar w^s}
	\Big)_{12}
	=
	\nabla_{r_1} \bar \ell_t(\bar r^t_1,\bar w^0_2;z)_1
	\cdot 
	\left(
	- \frac1\delta \int_s^t \bar R_\theta(t,s')_{11} \Big(\frac{\partial \bar \ell_{s'}(\bar r^{s'};z)}{\partial \bar w^s}\Big)_{12} \de s' 
	-
	\frac1\delta \bar R_\theta(t,s)_{11} \nabla_{w_2^0} \bar \ell_s(\bar r^s_1,\bar w^0_2;z)_1
	\right).
\end{gathered}
\end{equation}
By \eqref{eq:Rell-simplification},
we have that Eq.~\eqref{eq:def-derivative-dw} is satisfied with $\nabla_{w^*} \ell_t(r^s, w^*;z) + \int_0^s \Big(
\frac{\partial \bar \ell_t(\bar r^{s'};z)}{\partial \bar w^{s'}}
\Big)_{12} \de s'$ in place of $\frac{\partial \ell(r^t,w^*;z)}{\partial w^*}$.
This implies that 
\begin{equation}
\frac{\partial \ell(r^s,w^*;z)}{\partial w^*}
= 
\nabla_{w^*} \ell_t(r^s, w^*;z) + \int_0^s \Big(
\frac{\partial \bar \ell_t(\bar r^{s'};z)}{\partial \bar w^{s'}}
\Big)_{12} \de s',
\end{equation}
and that $R_\ell(t,*) = \bar \Gamma_{12}^t + \int_0^t \bar R_\ell(t,s)_{12} \de s$.
We have thus generated from Eqs.~\eqref{eq:planted-se} and \eqref{eq:planted-derivs} a solution to Eqs.~\eqref{eq:full-se-eqns}.
Because distinct solutions to Eqs.~\eqref{eq:planted-se} and \eqref{eq:planted-derivs} will generate distinct solutions to Eqs.~\eqref{eq:full-se-eqns},
and the solution to Eqs.~\eqref{eq:full-se-eqns} is unique,
we conclude the solution to Eqs.~\eqref{eq:planted-se} and \eqref{eq:planted-derivs} is unique.

\subsection{Proof of Theorem \ref{thm:fixed-pt}: convergence to fixed points}

\begin{proof}[Proof of Theorem \ref{thm:fixed-pt}]
Throughout the proof, we will repeatedly use that for $\mathsf{x} \in \{\theta,\ell\}$,
\begin{equation}
	\lim_{t\rightarrow \infty} \int_0^t \| R_\mathsf{x}(t,t-s) - R_\mathsf{x}(s) \|\de s = 0.
\end{equation}
Indeed, 
for any $t \geq \Delta \geq 0$, we have $\int_0^t \| R_\mathsf{x}(t,t-s) - R_\mathsf{x}(s) \| \de s
\leq \Delta \| R_\mathsf{x}(t,t-\cdot) - R_\mathsf{x}(\cdot) \|_\infty + 2Ce^{-c\Delta}$.
The previous display follows by taking $t \rightarrow \infty $ followed by $\Delta \rightarrow \infty$. 

Theorem \ref{thm:fixed-pt} will hold for 
\begin{equation}
	\label{eq:Rinf-def}
	R_\ell^\infty
	= 
	\Gamma
	+
	\int_0^\infty
	R_\ell(s) \de s,
	\qquad 
	R_\theta^\infty
	=
	\int_0^\infty R_\theta(s) \de s.
\end{equation}
We begin by establishing Eq.~\eqref{eq:fixed-pt-1}.
Note that as $t \rightarrow \infty$,
\begin{equation}
	\begin{aligned}
		\| (\Lambda + \Gamma^t) \theta^t - (\Lambda + \Gamma^\infty) \theta^\infty \|_{L^2}
		&\leq \| \Lambda + \Gamma^\infty \| \| \theta^t - \theta^\infty \|_{L^2}
		+ \|\Gamma - \Gamma^t\| \| \theta^t \|_{L^2} \rightarrow 0,
	\end{aligned}
\end{equation}
and
\begin{equation}
	\begin{aligned}
		&\Big\|
		\int_0^t R_\ell(t,t-s) \theta^{t-s} \de s
		- 
		\int_0^t R_\ell(s) \theta^\infty \de s
		\Big\|_{L^2}
		\\
		&\qquad
		\leq
		\int_0^t 
		\| R_\ell(t,t-s) \| \| \theta^{t-s} - \theta^\infty \|_{L^2}
		\de s
		+
		\int_0^t
		\| R_\ell(t,t-s) - R_\ell(s)\| \de s \, \| \theta^\infty \|_{L^2}
		\\
		&\qquad\leq 
		Ce^{-cs}Ce^{-c(t-s)} + \int_0^t
		\| R_\ell(t,t-s) - R_\ell(s)\| \de s \, \| \theta^\infty \|_{L^2} \rightarrow 0,
	\end{aligned}
\end{equation}
and
\begin{equation}
	\| R_\ell(t,*)\theta^* - R_\ell^* \theta^* \|_{L^2} \rightarrow 0,
	\qquad 
	\| u^t - u^\infty \|_{L^2} \rightarrow 0.
\end{equation}
Combining these bounds,
we conclude that 
\begin{equation}
	\begin{aligned}
		&\Big\|
		\frac{\de}{\de t}\theta^t
		-
		\Big(
		-(\Lambda + \Gamma^\infty)\theta^\infty
		-
		\int_0^t R_\ell(s) \theta^\infty \de s 
		-
		R_\ell^* \theta^*
		+
		u^\infty
		\Big)
		\Big\|_{L^2} \rightarrow 0.
	\end{aligned}
\end{equation}
Because $\Gamma^\infty + \int_0^t R_\ell(s) \theta^\infty \de s \stackrel{L^2}\rightarrow R_\ell^\infty \theta^\infty$,
we conclude that 
\begin{equation}
	\frac{\de}{\de t} \theta^t \stackrel{\mathrm{L^2}}\rightarrow 
	-(\Lambda+R_\ell^\infty)\theta^\infty
	-
	R_\ell^* \theta^*
	+
	u^\infty.
\end{equation}
Because $\theta^t$ stays bounded in $L^2$ as $t \rightarrow \infty$,
we must have that 
\begin{equation}
	0 = -(\Lambda+R_\ell^\infty)\theta^\infty
	-
	R_\ell^* \theta^*
	+
	u^\infty.
\end{equation}
Similarly, as $t \rightarrow \infty$
\begin{equation}
	\begin{aligned}
		&\Big\|
		\int_0^t R_\theta(t,t-s) \ell(r^{t-s},w^*;z) \de s
		-
		\int_0^t R_\theta(s) \ell(r^\infty,w^*;z) \de s
		\Big\|_{L^2}
		\\
		&\qquad 
		\leq 
		\int_0^t \| R_\theta(t,t-s) \| \| \ell(r^{t-s},w^*;z) - \ell(r^\infty,w^*;z) \|_{L^2} \de s
		+
		\int_0^t \| R_\theta(t,t-s) - R_\theta(s) \| \de s\, \| \ell(r^\infty,w^*;z) \|_{L^2}
		\\ 
		&\qquad 
		\leq 
		CLe^{-cs}Ce^{-c(t-s)}
		+
		\int_0^t \| R_\theta(t,t-s) - R_\theta(s) \| \de s\, \| \ell(r^\infty,w^*;z) \|_{L^2}
		\rightarrow 0.
	\end{aligned}
\end{equation}
Because $\int_0^t R_\theta(t,t-s) \de s \rightarrow R_\theta(s) \de s$,
we conclude that 
\begin{equation}
	r^\infty = - \frac1\delta R_\theta^\infty r^\infty + w^\infty.
\end{equation}
Because $u^t \stackrel{L^2}\rightarrow u^\infty$ and $r^t \stackrel{L^2}\rightarrow r^\infty$,
we have $u^\infty \sim \normal(0,C_\ell^\infty/\delta)$,
where 
\begin{equation}
	\begin{aligned}
		C_\ell^\infty = \lim_{t \rightarrow \infty} C_\ell(t,t) = \lim_{t \rightarrow \infty} \E[\ell(r^t,w^*;z)\ell(r^t,w^*;z)^\top] = \E[\ell(r^\infty,w^*;z)\ell(r^\infty,w^*;z)^\top].
	\end{aligned}
\end{equation}
Likewise, 
because $(w^t,w^*) \stackrel{L^2}\rightarrow (w^\infty,w^*)$ and $(\theta^t,\theta^*) \stackrel{L^2}\rightarrow (\theta^\infty,\theta^*)$,
we have $(w^\infty,w^*) \sim \normal(0,C_\theta^\infty)$, 
where 
\begin{equation}
	C_\theta^\infty = \lim_{t\rightarrow \infty} C_\theta(\{t,*\},\{t,*\}) = \lim_{t \rightarrow \infty} \E[(\theta^{t\top},\theta^{*\top})^\top(\theta^{t\top},\theta^{*\top})] = \E[(\theta^{\infty\top},\theta^{*\top})^\top(\theta^{\infty\top},\theta^{*\top})].
\end{equation}
We have finished the proof of Eq.~\eqref{eq:fixed-pt-1} and that $u^\infty \sim \normal(0, \Cl / \delta)$ and $(w^\infty,w^*) \sim \normal(0, \Ct)$.

We now show that 
$R_\ell^\infty,R_\theta^\infty,R_\ell^*$ satisfy Eq.~\eqref{eq:fixed-pt-2}.
By Eq.~\eqref{eq:def-derivative-t},
$\frac{\partial \theta^t}{\partial u^s}$ is deterministic, so $R_\theta(t,s) = \frac{\partial \theta^t}{\partial u^s}$.
First compute
\begin{equation}
	\frac{\de}{\de s}\int_0^s R_\theta(t+s,t+s')\de s'
	=
	I_k 
	+ \int_0^s 
	\Big(
	-(\Lambda + \Gamma^t) R_\theta(t+s,t+s')
	-
	\int_{t+s'}^{t+s}
	R_\ell(t+s,s'')
	R_\theta(s'',t+s')
	\de s''
	\Big)
	\de s',
\end{equation}
where we have used Eq.~\eqref{eq:def-derivative-t}, and that we may exchanged differentiation and integration by Definition \ref{def:exp-conv} (using the boundedness of the derivative).
Taking $s \rightarrow \infty$,
the right-hand side converges to
\begin{equation}
	I_k
	-
	(\Lambda + \Gamma^\infty)
	\int_0^\infty
	R_\theta(s')
	\de s'
	-
	\int_0^\infty R_\ell(s) \de s \int_0^\infty R_\theta(s) \de s,
\end{equation}
where to get the second term, we have used that 
\begin{equation}
	\int_0^s
	\int_{t+s'}^{t+s}
	R_\ell(t+s,s'')
	R_\theta(s'',t+s')
	\de s''
	\de s'
	= 
	\int_0^s
	\int_0^{s-s'}
	R_\ell(t+s,t+s-s')
	R_\theta(t+s-s',t+s-s'-s'')
	\de s''
	\de s',
\end{equation}
and taken $t \rightarrow \infty$ followed by $s \rightarrow \infty$, and used the convergence and decay conditions of $R_\ell,R_\theta$.
Because $R_\theta(t+s,t+s') \leq Ce^{-cs'}$, 
we must have that $\int_0^s R_\theta(t+s,t+s')\de s'$ converges as $s \rightarrow \infty$,
which implies that 
\begin{equation}
	0 = 
	I_k
	-
	(\Lambda + \Gamma^\infty)
	\int_0^\infty
	R_\theta(s')
	\de s'
	-
	\int_0^\infty R_\ell(s) \de s \int_0^\infty R_\theta(s) \de s
	=
	I_k 
	- \Lambda R_\theta^\infty - R_\ell^\infty R_\theta^\infty.
\end{equation}
(because otherwise, we would have that $\int_0^s R_\theta(t+s,t+s')\de s'$ diverges).
This gives us the second equation in Eq.~\eqref{eq:fixed-pt-2}.

Now define
\begin{equation}
	\label{eq:def-Rhat-ell-1}
	\widehat R_\ell^{(1)}(s)
	= 
	\nabla_r \ell(r^\infty,w^*;z) 
	\cdot 
	\prn{- \frac{1}{\delta} \int_0^s R_{\theta}(s-s') \widehat R_\ell(s')  \de s' - \frac 1 \delta R_\theta(s) \nabla_r \ell(r^\infty,w^*;z) }.
\end{equation}
We bound
\begin{equation}
	\begin{aligned}
		&\norm{\int_t^{t+s} R_\theta(t+s,s') \frac{\partial \ell(r^{s'},w^*;z)}{\partial w^s} \de s' -  \int_0^s R_{\theta}(s-s') \widehat R_\ell(s')  \de s'}_{\rev{L^2}}
		\\
		&\qquad \leq 
		\int_t^{t+s} 
		\| R_\theta(t+s,s') \| 
		\norm{
			\frac{\partial \ell(r^{s'},w^*;z)}{\partial w^s}
			-
			\widehat R_\ell(s')
		}_{\rev{L^2}}
		\de s'
		+ 
		\int_t^{t+s}
		\| \widehat R_\ell(s') \|_{\rev{L^2}}
		\| R_\theta(t+s,s') - R_\theta(s-s') \| \de s'
		\\
		&\qquad 
		\leq 
		Cs \int_t^{t+s} 
		\norm{
			\frac{\partial \ell(r^{s'},w^*;z)}{\partial w^s}
			-
			\widehat R_\ell(s')
		}_{\rev{L^2}} \de s'
		+ 
		C \| R_\theta(t+s,t+s-\cdot) - R_\theta(\cdot) \|_\infty
		\rightarrow 0.
	\end{aligned}
\end{equation}
where the limit is for $s$ fixed and $t \rightarrow \infty$.
One can likewise show that as $t \rightarrow \infty$, $R_\theta(t+s,t)\nabla_r\ell(r^t,w^*;z) \stackrel{\rev{L^2}}\rightarrow R_\theta(s) \nabla_r\ell(r^\infty,w^*;z)$ and $\nabla_r \ell(r^{t+s},w^*;z) \stackrel{\rev{L^2}}\rightarrow \nabla_r \ell(r^\infty,w^*;z)$.
Because each of these terms is also bounded,
using Eq.~\eqref{eq:def-Rhat-ell-1} we conclude that $\frac{\partial \ell(r^{t+s},w^*;z)}{\partial w^t}
\stackrel{\rev{L^2}}\rightarrow 
\widehat R_\ell^{(1)}(s)$.
Then,
we must have that $\widehat R_\ell^{(1)}(s) = \widehat R_\ell(s)$.
In particular, $\widehat R_\ell(s)$ satisfies the equation
\begin{equation}
	\label{eq:hatRell-integral}
	\widehat R_\ell(s) 
	=
	\nabla_r \ell(r^\infty,w^*;z) 
	\cdot 
	\prn{- \frac{1}{\delta} \int_0^s R_{\theta}(s-s') \widehat R_\ell(s')  \de s' - \frac 1 \delta R_\theta(s) \nabla_r \ell(r^\infty,w^*;z) },
\end{equation}
and moreover, 
\begin{equation}
	\E[\widehat R_\ell(s)]
	=
	\lim_{t \rightarrow \infty}
	\E\left[
	\frac{\partial \ell(r^{t+s},w^*;z)}{\partial w^t}
	\right]
	= 
	\lim_{t \rightarrow \infty} R_\ell(t+s,s) = R_\ell(s).
\end{equation}
Because $R_\theta(s),\widehat R_\ell(s) \leq Ce^{-cs}$, 
we may integrate Eq.~\eqref{eq:hatRell-integral} and apply Fubini's theorem to get
\begin{equation}
	\int_0^\infty \widehat R_\ell(s) \de s 
	= 
	\nabla_r \ell(r^\infty,w^*;z) 
	\cdot 
	\prn{- \frac{1}{\delta} \int_0^\infty R_{\theta}(s) \de s \int_0^\infty \widehat R_\ell(s)  \de s - \frac 1 \delta \int_0^\infty R_\theta(s) \de s \nabla_r \ell(r^\infty,w^*;z) }.
\end{equation}
Recalling the definition of $R_\theta$ (Eq.~\eqref{eq:Rinf-def}),
this can be rearranged to
\begin{equation}
	\begin{aligned}
		\nabla_r \ell(r^\infty,w^*;z) 
		+ 
		\int_0^\infty \widehat R_\ell(t) \de t
		&= 
		\Big(I_k + \frac1\delta \nabla_r\ell(r^\infty,w^*;z) R_\theta^\infty\Big)^{-1} \nabla_r \ell(r^\infty,w^*;z)
		\\
		&= 
		\delta\Big(I_k - \Big(I_k + \frac1\delta\nabla_r\ell(r^\infty,w^*;z)R_\theta^\infty\Big)^{-1} \Big)(R_\theta^\infty)^{-1}.
	\end{aligned}
\end{equation}
Taking expectations 
and using Eq.~\eqref{eq:Rinf-def} gives the first equation in \eqref{eq:fixed-pt-2}.

Now consider Eq.~\eqref{eq:def-derivative-dw}.
Recall
\begin{equation}
	\frac{\partial \ell(r^t,w^*;z)}{\partial w^*} 
	=  
	-\frac{1}{\delta} \nabla_r \ell(r^t,w^*;z) \int_0^t R_{\theta}(t,t-s) \frac{\partial \ell(r^{t-s},w^*;z)}{\partial w^*}  \de s + \nabla_{w^*} \ell(r^t,w^*;z).
\end{equation}
Because $\nabla_r\ell(r^t,w^*;z) \stackrel{\rev{L^2}}\rightarrow \nabla_r\ell(r^\infty,w^*;z)$,
$\|R_\theta(t,t-\cdot) - R_\theta(\cdot)\|_\infty \rightarrow 0$, 
$\frac{\partial \ell(r^{t-s},w^*;z)}{\partial w^*} \stackrel{\rev{L^2}}\rightarrow \widehat R_\ell^*$, 
$\nabla_{w^*}\ell(r^t,w^*;z) \stackrel{\rev{L^2}}\rightarrow \nabla_{w^*}\ell(r^\infty,w^*;z) $,
and we have that $\|\nabla_r\ell(r^t,w^*;z)\|,\|\nabla_r\ell(r^\infty,w^*;z)\|,\|\nabla_{w^*}\ell(r^t,w^*;z)\|,\|\nabla_{w^*}\ell(r^\infty,w^*;z)\| \leq M_\ell$, and $\|R_\theta(t,t-s)\|,\|R_\theta(s)\| \leq Ce^{-ts}$,
we can take the limit of the previous display as $t \rightarrow \infty$ to get
\begin{equation}
	\widehat R_\ell^*
	=
	- \frac1\delta \nabla_r \ell(r^\infty,w^*;z)
	\int_0^\infty R_\theta(s) \de s \widehat R_\ell^* 
	+
	\nabla_{w^*} \ell(r^\infty,w^*;z).
\end{equation}
This can be rearranged to
\begin{equation}
	\widehat R_\ell^*
	=
	\Big(
	I_k
	+
	\frac1\delta \nabla_r \ell(r^\infty,w^*;z) R_\theta^\infty
	\Big)^{-1}
	\nabla_{w^*}\ell(r^\infty,w^*;z).
\end{equation}
Taking expectations gives the third equation in Eq.~\eqref{eq:fixed-pt-2}.
This completes the proof.
\end{proof}

\subsection{Proof of Theorem \ref{thm:Variational-General}}

We will use throughout the mapping given in Eq.~\eqref{eq:FixedPointMatch}. 
We copy here for the reader's convenience the fixed point characterization of 
Theorem \ref{thm:fixed-pt}, for the case $\ell(r,w;z) =\nabla \Ls(r,w;z)$
(it is understood that gradients are taken with respect to the first argument):
\begin{align}
r^\infty &= - \frac1\delta R_\theta^\infty \nabla\Ls(r^\infty,w^*;z) + w^\infty,
\tag{FP1}\label{eq:FP1}\\
0 &= -(\Lambda + R_\ell^\infty) \theta^\infty - R_\ell^* \theta^* + u^\infty,
\tag{FP2}\label{eq:FP2}\\
C_\theta & = \Ep [ (\theta^{\infty\sT},\theta^{*\sT})^\sT (\theta^{\infty\sT},\theta^{*\sT})]\, ,
\tag{FP3}\label{eq:FP3} \\
C_\ell & =   \Ep[\nabla\Ls(r^\infty,w^*;z) \nabla\Ls(r^\infty,w^*;z)^\sT] \, ,
\tag{FP4}\label{eq:FP4}\\
(R_\theta^\infty)^{-1} &= \Lambda + R_\ell^\infty,\tag{FP1}\label{eq:FP5}\\
R_\ell^\infty&=\E\Big[ \Big(I_k + \frac1\delta\nabla^2\Ls(r^\infty,w^*;z) R_\theta^\infty\Big)^{-1}
\nabla^2\Ls(r^\infty,w^*;z) \Big],
\tag{FP5}\label{eq:FP6}\\
R_\ell^*&= \E\Big[\Big( I_k + \frac1\delta \nabla^2 \Ls(r^\infty,w^*;z) R_\theta^\infty\Big)^{-1} \nabla_{w^*} 
\Ls(r^\infty,w^*;z)\Big], 
\tag{FP6}\label{eq:FP7}
\end{align}
where $u^\infty \sim \normal(0, \Cl / \delta)$ and $(w^\infty,w^*) \sim \normal(0, \Ct)$.
By rotation invariance we can and will assume $\theta^* \sim \normal(0, Q_{00})$.

\vspace{0.2cm}
\noindent{\bf Equation \eqref{eq:FP1}}
can be rewritten as
\begin{align}
\delta (R_{\theta}^{\infty})^{-1}(r^{\infty}-w^{\infty}) +\nabla \Ls(r^{\infty},w^{\infty};z)=0\, .
\label{eq:ProxR}
\end{align}
Using the identification  $R_{\theta}^{\infty} = \delta S$ given in Eq.~\eqref{eq:FixedPointMatch}
and calculus, we see that the above is equivalent to
\begin{align}
r^{\infty} = \Prox_{ \Ls(\,\cdot\, ,w^*;z)}(w^{\infty};S)\, .
\end{align}

\vspace{0.2cm}

\noindent{\bf Equation \eqref{eq:FP2}.}
Using $R_{\ell}^*= -\frac{1}{\delta} S^{-1} Q_{10}Q_{00}^{-1}$
and
$R_{\ell}^{\infty} = (\delta \, S)^{-1}-\lambda I_k$, again as prescribed in 
Eq.~\eqref{eq:FixedPointMatch}, we can rewrite 
Eq.~\eqref{eq:FP2} as
\begin{align}
\theta^{\infty} = Q_{10}Q_{00}^{-1}\theta^* + \xi \, ,\;\;\; \xi\sim \normal(0,Q\setminus Q_{00})\, ,
\end{align}
which is equivalent to $(\theta^*,\theta^{\infty})\sim \normal(0,Q)$.
In particular $(\theta^*,\theta^{\infty})\ed (w^*,w^{\infty})$.

\vspace{0.2cm}

\noindent{\bf Equation \eqref{eq:FP3}.} Using the mapping \eqref{eq:FixedPointMatch},
this equation amounts to $Q=\E[(\theta^*,\theta^{\infty})(\theta^*,\theta^{\infty})^{\sT}]$
which is proved above.    

\vspace{0.2cm}

\noindent{\bf Equation \eqref{eq:FP4}} is equivalent to Eq.~\eqref{eq:QS1}, again using \eqref{eq:FixedPointMatch}.

\vspace{0.2cm}

\noindent{\bf Equation \eqref{eq:FP5}} 
holds because 
$R_{\theta}^{\infty} = \delta S$, $R_{\ell}^{\infty} = (\delta \, S)^{-1}-\lambda I_k$.
by Eq.~\eqref{eq:FixedPointMatch}.

\vspace{0.2cm}

We finally claim that Eq.~\eqref{eq:QS2} is equivalent to
\eqref{eq:FP6}, \eqref{eq:FP7}. The first block reads
\begin{align*}
0&=  \E\big[r_{\infty}\nabla \Ls(r_{\infty},w^*;z)^{\sT}\big] +\lambda Q_{11}\\
&\stackrel{(a)}{=} \E\big[w_{\infty}\nabla \Ls(r_{\infty},w^*;z)^{\sT}\big] - S \E[\nabla \Ls(r_{\infty},w^*;z)^{\sT}
\nabla \Ls(r_{\infty},w^*;z)^{\sT}]+\lambda Q_{11}\\
& =  \E\big[w_{\infty}\nabla \Ls(r_{\infty},w^*;z)^{\sT}\big] -\frac{1}{\delta} (Q\setminus Q_{00}) S^{-1}
+\lambda Q_{11}\, ,
\end{align*}
where in $(a)$ we used Eq.~\eqref{eq:ProxR}. Using Stein's lemma in the firs term,
and the fact that the Jacobian is
\begin{align}
D_{w^{\infty}}\Prox_{ \Ls(\,\cdot\, ,w^*;z)}(w^{\infty};S) = 
\big(I_k+S\nabla^2 \Ls(r_{\infty},w^*;z)\big)^{-1}\, ,
\end{align}
we obtain 
\begin{align}
0=  &Q_{11}\E\big[ \nabla^2 \Ls(r_{\infty},w^*;z)
\big(I_k+S\nabla^2 \Ls(r_{\infty},w^*;z)\big)^{-1}\big]\nonumber\\
&+
Q_{10}\E\big[ \nabla^2_{w^*,r} \Ls(r_{\infty},w^*;z)
\big(I_k+S\nabla^2 \Ls(r_{\infty},w^*;z)\big)^{-1}\big]\label{eq:LongEquiv}\\
&
-\frac{1}{\delta} (Q\setminus Q_{00}) S^{-1}
+\lambda Q_{11}\, .\nonumber
\end{align}
Defining 
\begin{align}
A &:= \E\big[ \nabla^2 \Ls(r_{\infty},w^*;z)
\big(I_k+S\nabla^2 \Ls(r_{\infty},w^*;z)\big)^{-1}\big]\, ,\\
B & :=  \E\big[ \nabla^2_{w_*,r} \Ls(r_{\infty},w^*;z)
\big(I_k+S\nabla^2 \Ls(r_{\infty},w^*;z)\big)^{-1}\big]\, ,
\end{align}
we can rewrite Eq.~\eqref{eq:LongEquiv} as
\begin{align}
\label{eq:ABeq1}
Q_{11}A+Q_{10}B = \frac{1}{\delta} (Q\setminus Q_{00})S^{-1} -\lambda Q_{11}\, ,
\end{align}
Proceeding in the same way for the second block in Eq.~\eqref{eq:QS2},
we obtain (note that $Q_{01}$)
\begin{align}
\label{eq:ABeq2}
Q_{01}A+Q_{00}B = -\lambda Q_{01}\, ,
\end{align}
It is easy to check that Eq.~\eqref{eq:ABeq1}, \eqref{eq:ABeq2} are solved by
\begin{align}
A = -\lambda I_k+ \frac{1}{\delta} S^{-1}\, ,\;\;\;\;
B = -\frac{1}{\delta}Q_{00}^{-1}Q_{01}S^{-1}\, .\label{eq:AB-solution}
\end{align}
These are easily seen to coincide with the  \eqref{eq:FP6}, \eqref{eq:FP7}.

Finally, recall that we assumed $Q_{00}=\E[\theta^*\theta^{*,\sT}]$
and $\Cov_z(\nabla \Ls(r,w;z))$ to be strictly positive for any $r,w$. This implies
that $Q_{00}$ and $(Q\setminus Q_{00})$ are strictly positive, and hence so is $Q$.
Therefore Eqs.~\eqref{eq:ABeq1}, \eqref{eq:ABeq2} have a unique solution that is given by 
Eq.~\eqref{eq:AB-solution}.

\rev{\subsection{Proof of Proposition \ref{prop:convergence-DMFT}} \label{proof:convergence-DMFT}
By the assumption of strong convexity of the regularized risk, there exists a unique global minimizer $\hbtheta$ to the modified cost function \eqref{eq:GeneralCost}
\begin{align*}
	\cL_n(\btheta)
	= 
	\frac{1}{\delta} \sum_{i=1}^n \Ls(\btheta^{\sT} \bx_i,(\btheta^*)^{\sT} \bx_i;z_i)
	+ 
	\frac{1}{2} \|\btheta \Lambda^{\frac{1}{2}}\|_F^2 \, ,
\end{align*}
such that $\nabla \cL_n(\hbtheta) = 0$. And for some $c_0 > 0$ independent of $n,d$, it holds that
\begin{align*}
	\langle \nabla \cL_n(\btheta^t), \btheta^t - \hbtheta \rangle = \langle \nabla \cL_n(\btheta^t) - \nabla \cL_n(\hbtheta), \btheta^t - \hbtheta \rangle \geq c_0 \| \btheta^t -  \hbtheta\|_2^2.
\end{align*}
For the Lyapunov function $V(t) = \| \btheta^t -  \hbtheta\|_2^2$, it satisfies 
\begin{align*}
	\ddt V(t) = - 2\langle \nabla \cL_n(\btheta^t), \btheta^t - \hbtheta \rangle \leq - 2c_0 V(t),
\end{align*}
and by Gronwall, $V(t) \leq e^{-2c_0t} V(0)$. This implies for some universal $C > 0$ and any $0 < t < t' < \infty$, $\| \btheta^t - \btheta^{t'}\|_2 /\sqrt{d} \leq C e^{-c_0 t}$, corresponding to in the DMFT limit that $\| \theta^t - \theta^{t'} \|_{L^2} \leq C e^{-c_0 t}$ invoking Theorem~\ref{thm:StateEvolution}. By completeness of the $L^2$ space and the Cauchy subsequence argument, we show there exists a limiting random variable $\theta^\infty$ in $\reals^k$ with bounded variance and $\| \theta^t - \theta^\infty \|_{L^2} \leq C e^{-c_0 t}$. The exponential convergence for $r^t$ follows similarly, and as well as the DMFT system parameters $R_\ell, R_\theta$.}
\end{document}